\documentclass[12pt]{article}

\usepackage{personal}

\title{Cohomological Obstructions for Mittag-Leffler Problems}
\author{Mateus Schmidt}
\date{2020}

\begin{document}

    \maketitle
    \begin{abstract}
        This is an extensive survey of the techniques used to formulate generalizations of the Mittag-Leffler Theorem from complex analysis. With the techniques of the theory of differential forms, sheaves and cohomology, we are able to define the notion of a Mittag-Leffler Problem on a Riemann surface, as a problem of passage of data from local to global, and discuss characterizations of contexts where these problems have solutions. This work was motivated by discussions found in \cite{griffithsharris}, \cite{forster}, as well as \cite{hartshorne}.
    \end{abstract}
    
    \pagebreak
    
    \tableofcontents
    \pagebreak

\part{Preliminaries}

\begin{center}
\textit{“Wir beschr\"{a}nken die Ver\"{a}nderlichkeit der Gr\"{o}ssen x, y auf ein endliches Gebiet, indem wir als Ort des Punktes O nicht mehr die Ebene A selbst,
sondern eine \"{u}ber dieselbe ausgebreitete Fl\"{a}che T betrachten. ... Wir lassen die M\"{o}glichkeit offen, dass der Ort des Punktes O \"{u}ber denselben Theil der Ebene sich mehrfach erstrecke, setzen jedoch f\"{u}r einen solchen Fall voraus, dass die auf einander liegenden Fl\"{a}chentheile nicht l\"{a}ngs einer Linie zusammenh\"{a}ngen, so dass eine Umfaltung der Fl\"{a}che, oder eine Spaltung in auf einander liegende Theile nicht vorkommt.”}

\medskip

\textit{“We restrict the variables x, y to a finite domain by considering as the locus
of the point O no longer the plane A itself but a surface T spread over the
plane. We admit the possibility ... that the locus of the point O is covering
the same part of the plane several times. However in such a case we assume
that those parts of the surface lying on top of one another are not connected
along a line. Thus a fold or a splitting of parts of the surface cannot occur."}

-Bernhard Riemann in \cite{riemannthesis}, translation provided by Reinhold Remmert in \cite{rscs}.

\end{center}

\section{Riemann Surfaces and Meromorphic functions}

\makebox[0pt]{}\vspace{-2ex}

Riemann introduced the concept of Riemann surfaces in his PhD thesis in 1851. He had noted that they are the correct context for the study of functions of one complex variables, for they provide ways to deal with some problems that we encounter during the study of functions on the plane, such as multi-valuation. Their study throughout the years became interesting from both a complex analytic and geometric point of view. Much of the analytic study of Riemann surfaces revolves around trying to transport results of the classical theory of functions of one complex variable in the plane to the new setting of surfaces. 

\begin{definition}

Let $X$ be a two dimensional real manifold, i.e., a second countable Hausdorff topological space $X$ such that every point $a \in X$ has an open neighborhood which is homeomorphic to an open subset of $\mathbb{R}^2$. A \textit{complex chart} on $X$ is a homeomorphism $	\varphi: U	\rightarrow V$ of an open set $U \subseteq X$ onto an open subset $V \subseteq \mathbb{C}$. Two complex charts $\varphi_1: U_1	\rightarrow V_1$ and $\varphi_2: U_2	\rightarrow V_2$ are said to be \textit{holomorphically compatible} if the map
\begin{equation}
    \varphi_2\circ\varphi_1^{-1}: \varphi_1(U_1 \cap U_2) \rightarrow \varphi_2(U_1 \cap U_2)
\end{equation}
is biholomorphic, i.e., it is holomorphic and its inverse is holomorphic.

A \textit{complex atlas} on X is a system $\mathfrak{A} = \{\varphi_i: U_i	\rightarrow V_i, i \in I\}$ of charts which are holomorphically compatible and which cover $X$.
\end{definition}

\begin{definition}

Two complex atlases $\mathfrak{A}$ and $\mathfrak{A}'$ on $X$ are said to be \textit{analytically equivalent} if every chart of $\mathfrak{A}$ is holomorphically compatible with every chart of $\mathfrak{A}'$. A \textit{complex structure} on a two-dimensional manifold $X$ is an equivalence class of analytically equivalent atlases on $X$.

\end{definition}

\begin{remark}
Analitycal equivalence is indeed an equivalence relation, and each equivalence class has a canonical representative, called the \textit{maximal atlas}, consisting of the union of all of the atlases in said equivalence class.
\end{remark}

\begin{definition}
A \textit{Riemann surface} is a pair $(X,\Sigma)$, where $X$ is a connected two-dimensional manifold and $\Sigma$ is a complex structure on $X$. We will usually write $X$ for $(X,\Sigma)$, whenever $\Sigma$ is clear from the context.
\end{definition}

\begin{example}
The \textit{complex plane $\mathbb{C}$}, with the atlas composed by the identity map.
\end{example}

\begin{example}
Any \textit{domain}, i.e., any connected open subset of $\mathbb{C}$. If $G \subseteq \mathbb{C}$ is a domain, it has a natural structure of a Riemann surface, namely, the chart is the identity restricted to $G$. More generally, if $X$ is a Riemann surface and $Y \subseteq X$ is an open connected subset, it is a Riemann surface itself with the structure given by the charts $\varphi: U	\rightarrow V$ on X, where $U \subseteq Y$.
\end{example}

\begin{example}\label{P1}
The \textit{Riemann sphere}, which we identify with the projective space $\mathbb{P}^1$. Define $\mathbb{P}^1 \coloneqq \mathbb{C} \cup \{\infty\}$, where $\infty \notin \mathbb{C}$ is a symbol. We then give $\mathbb{P}^1$ a topology in the following way: the open sets are the usual open sets $U \subseteq \mathbb{C}$, together with sets of the form $V \cup \{\infty\}$, where $V \subseteq \mathbb{C}$ is the complement of a compact set $K \subseteq \mathbb{C}$. With this topology, $\mathbb{P}^1$ is a compact Hausdorff topological space, the $1$-point compactification of the plane, and is homeomorphic to the 2-Sphere $S^2$. Set

\begin{align*}
    &U_1 \coloneqq \mathbb{P}^1 \setminus \{\infty\} \\
    &U_2 \coloneqq \mathbb{P}^1 \setminus \{0\} = \mathbb{C} \setminus \{0\} \cup \{\infty\} = \mathbb{C}^{*} \cup \{\infty\}
\end{align*}

Define maps $\varphi_i: U_i	\rightarrow \mathbb{C}$, as follows: $\varphi_1$ is the identity, and
    \[ \varphi_2(z) = \begin{cases} 
          \frac{1}{z} & z\in \mathbb{C}^{*} \\
          0 & z= \infty 
       \end{cases}
    \]
    
These maps are homeomorphisms $U_i	\rightarrow \mathbb{C} \cong \mathbb{R}^2$, and thus $\mathbb{P}^1$ is a two-dimensional manifold. The complex structure on $\mathbb{P}^1$ is defined by the atlas consisting of the $\varphi_i, i=1, 2$. They are holomorphically compatible, since $\varphi_1(U_1 \cap U_2) = \varphi_2(U_1 \cap U_2) = \mathbb{C}^{*}$ and
\begin{equation}
    \varphi_2\circ\varphi_1^{-1}:\mathbb{C}^{*} \rightarrow \mathbb{C}^{*}, z \mapsto \frac{1}{z}
\end{equation}
is biholomorphic.
\end{example}

\begin{definition}
Let $X$ be a Riemann surface, and $Y \subseteq X$ an open subset. A function $f: Y \rightarrow \mathbb{C}$ is called \textit{holomorphic}, if for every chart $\psi: U \rightarrow V$ on $X$, the function
\begin{equation}
    f \circ \psi^{-1}: \psi(U \cap V) \rightarrow \mathbb{C}
\end{equation}
is holomorphic in the usual sense of complex functions (since $\psi(U \cap V) \subseteq \mathbb{C}$). The set of all functions holomorphic on $Y$ is denoted $\mathcal{O}(Y)$.
\end{definition}
\begin{remark}
The set $\mathcal{O}(Y)$ has a natural structure of $\mathbb{C}$-algebra, given by pointwise addition and multiplication.
\end{remark}

\begin{remark}
Every chart $\psi: U \rightarrow \mathbb{C}$ is trivially a holomorphic function.
\end{remark}

\begin{definition}
Suppose $X$ and $Y$ are Riemann surfaces. A continuous mapping $f:X \rightarrow Y$ is called \textit{holomorphic} if for every pair of charts $\psi_1:U_1 \rightarrow V_1$ on $X$ and $\psi_2:U_2 \rightarrow V_2$ on $Y$ with $f(U_1) \subseteq U_2$ the mapping
\begin{equation}
    \psi_2 \circ f \circ \psi_1^{-1}: V_1 \rightarrow V_2
\end{equation}
is holomorphic in the usual sense (note that it is a function on the plane).
\end{definition}

We now discuss two equivalent definitions of meromorphic functions:

\begin{definition}
Let $X$ be a Riemann surfface, and $Y \subseteq X$ an open subset. A \textit{meromorphic function} on $Y$ is a holomorphic function $f: Y' \rightarrow \mathbb{C}$, where $Y' \subseteq Y$ is an open subset, such that the following hold:
    \begin{enumerate}[label=(\roman*)]
        \item $Y \setminus Y'$ contains only isolated points
        \item For every point $p \in Y \setminus Y'$ one has
        \begin{equation}
            \lim_{x\to p} \left |f(x) \right| = \infty
        \end{equation}
    \end{enumerate}
The points of $Y \setminus Y'$ are called the poles of $f$. The set of all meromorphic functions on Y is denoted by $\mathcal{K}(Y)$.
\end{definition}

Now, recall that given a point $p$ in a Riemann surface $X$, we have the set $\mathcal{O}_p$ of equivalence classes of functions $f$ which are holomorphic in some neighborhood of $p$, subject to the equivalence relation $f \sim g$ if $f = g$ in a neighborhood of $p$. It is called the \textit{ring of germs of holomorphic functions} around $p$. If $f$ is holomorphic in a neighborhood of $p$, we write $f_p$ for its equivalence class in $\mathcal{O}_p$. Moreover, $\mathcal{O}_p$ is an integral domain, and so we can form its quotient field $\mathcal{K}_p$.

\begin{definition}
A \textit{meromorphic function} $\varphi$ in an open set $U$ of a Riemann surface $X$, is a map
\begin{equation}
    \varphi: U \rightarrow \bigcup\limits_{p} \mathcal{K}_p
\end{equation}
such that $\varphi(p) \in \mathcal{K}_p$ for every $p$ and to every point in $U$ there is a neighborhood $V$ and functions $f, g \in \mathcal{O}(V)$ such that $\varphi(p) = \frac{f_p}{g_p}$ when $p \in V$.\footnote{This definition also already hints at a potential relation with sheaf theory.} The set of all meromorphic functions in $U$ is denoted by $\mathcal{K}(U)$.
\end{definition}

By the second definition, we see that locally every meromorphic function is represented by a quotient of holomorphic functions.

We now present few interesting results about Riemann surfaces and meromorphic functions, which are generalizations of classical results of complex function theory in the plane:

\begin{theorem}[Riemann's Removable Singularity Theorem]
Let $U$ be an open subset of a Riemann Surface $X$, and let $a \in U$. Suppose the function $f \in \mathcal{O}(U\setminus \{a\})$ is bounded in some neighborhood of $a$. Then f can be extended uniquely to a function $\widetilde{f} \in \mathcal{O}(U)$.
\end{theorem}

\begin{proof}
Follows directly from Riemann's Removable Singularity Theorem in the plane (cf. \cite{krantz}, §4.1.5) 
\end{proof}

Using this theorem, one can prove that $\mathcal{K}(U)$ also has the natural structure of a $\mathbb{C}$-algebra.

\begin{theorem}[Identity Theorem]
Suppose $X$ and $Y$ are a Riemann surfaces, and $f_1, f_2: X \rightarrow Y$ are holomorphic mappings which coincide on a set $A \subseteq X$ having a limit point $a \in X$. Then $f_1$ and $f_2$ are identically equal.
\end{theorem}

\begin{proof}
 \cite{forster}, Chapter 1, Theorem 1.11.
\end{proof}

Let $f$ be defined and holomorphic in a punctured neighborhood of $p \in X$. Let $\varphi:U \rightarrow V$ be a chart on $X$, with $p \in U$. If $z$ is the local coordinate on $X$ near $p$, i.e., for all $x$ near $p$ we have $\varphi(x) = z$, we have that $f \circ \varphi^{-1}$ is holomorphic in a neighborhood of $\varphi(p) = z_0 \in \mathbb{C}$. Therefore, we may expand $\varphi(p) = z_0$ in a \textit{Laurent series} about $z_0$, just as we would in the plane:

\begin{equation}
    f(\varphi^{-1}(z)) = \sum_{n} c_n(z-z_0)^n
\end{equation}

This is called the \textit{Laurent series for $f$ about $p$ with respect to $\varphi$}. It \textbf{depends} on the choice of chart $\varphi$. The $\{c_n\}$ are called the \textit{Laurent coefficients}. The Laurent series encodes information about the isolated singularities of $f$, similarly to how it did for functions in the plane:

\begin{proposition}
The function $f$ has a pole at p if, and only if, any one of its Laurent series has finitely many nonzero negative terms.
\end{proposition}


The following proposition interprets meromorphic functions on Riemann surfaces as mappings to the Riemann sphere.

\begin{theorem}\label{mappingstop1}
Suppose $X$ is a Riemann surface and $f \in \mathcal{K}(X)$. For each pole $p$ of $f$, define $f(p) \coloneqq \infty$. Then $f:X \rightarrow \mathbb{P}^1$ is a holomorphic mapping. Conversely, if $f:X \rightarrow \mathbb{P}^1$ is a holomorphic mapping, then $f$ is either identically equal to $\infty$ or $\{f^{-1}(\infty)\}$ consists of isolated points and $f:X \setminus \{f^{-1}(\infty)\} \rightarrow \mathbb{C}$ is a holomorphic function on $X$.
\end{theorem}

\begin{proof}
Let $f \in \mathcal{K}(X)$ and let $P$ be the set of poles of $f$. Then $f$ induces a mapping $f:X \rightarrow \mathbb{P}^1$, which is continuous on the topology of $\mathbb{P}^1$, because it is continuous on the usual complex topology outside of $P$, and for any $p \in P$, $\lim_{x\to p} \left |f(x) \right| = \infty$ by construction. Suppose $\varphi: U \rightarrow V$ and $\psi:U' \rightarrow V'$ are charts on $X$ and $\mathbb{P}^1$ respectively, with $f(U) \subseteq U'$. We have to show that
\begin{equation}
    g \coloneqq \psi \circ f \circ \varphi^{-1}: V \rightarrow V'
\end{equation}
is holomorphic. Since $f$ is holomorphic on $X \setminus P$ by hypothesis, it follows that $g$ is holomorphic on $V \setminus \varphi(P)$. So, by the Riemann's Removable Singularities Theorem, $g$ is holomorphic on all of $V$ (if $\psi$ is equal to the chart $\varphi_2$, $g$ is trivially bounded. If $\psi$ it is equal to $\varphi_1$, $g$ is bounded on $V \setminus \varphi(P)$ by construction).

The converse follows from the Identity Theorem, for if $f:X \rightarrow \mathbb{P}^1$ is a holomorphic mapping, it agrees with a holomorphic function $\widetilde{f}: U \rightarrow \mathbb{C}$ on the open set $X \setminus f^{-1}(\infty)$. Since $\{f^{-1}(\infty)\}$ is closed, $X \setminus f^{-1}(\infty)$ is open unless $f$ is identically equal $\infty$, in which case it is clopen. Again by the Identity Theorem, if $f$ is not identically equal $\infty$, $\{f^{-1}(\infty)\}$ is a set of isolated points.
\end{proof}

\begin{remark}
The set $\mathcal{K}(X)$ is a field. Indeed, by Theorem \ref{mappingstop1}, the Identity Theorem is true for meromorphic functions away from their set of poles, and therefore locally we have a well defined algebraic inverse for every non-zero meromorphic function on its domain of definition.
\end{remark}

Finally, we state a theorem which we will prove in the following section, in the classical context of the plane:

\begin{theorem}[Runge's Theorem for non-compact Riemann surfaces]\label{rungesurfaces}
Suppose $X$ is a \textit{non-compact} Riemann surface. Let $Y$ be an open subset of $X$, and suppose that $X \setminus Y$ has no compact connected components. Then every holomorphic function $f \in \mathcal{O}(Y)$ is an uniform limit on every compact subset of $Y$ of restrictions of holomorphic functions on $X$.
\end{theorem}

\section{The Mittag-Leffler Theorem}

\makebox[0pt]{}\vspace{-2ex}

In this section, we study in depth the Mittag-Leffler Theorem in the complex plane, the main result which we intend to generalize to the context of Riemann surfaces.

For the Mittag-Leffler Theorem in the complex plane $\mathbb{C}$, we follow the theory as presented in \cite{conway}. In this section, $\mathbb{C}_\infty = \mathbb{C} \cup \{\infty\}$ denotes the Riemann sphere, which can be identified as a Riemann surface with $\mathbb{P}^1$, as was discussed above.

\begin{theorem}[Mittag-Leffler]\label{mlplane}
Let $G \subseteq \mathbb{C}$ be an open subset of the plane, $\{a_k\}$ a sequence of distinct points in $G$ without limit point in $G$, and let $\{S_k(z)\}$ be a sequence of rational functions given by
\begin{equation}
    S_k(z) = \sum_{j=1}^{m_k} \frac{A_{jk}}{(z-a_k)^j}
\end{equation}
where $m_k$ is a positive integer, and $A_{1k},...,A_{mk}$ are arbitrary complex coefficients. Then, there exists a meromorphic function $f$ on $G$ such that its poles are precisely the points $\{a_k\}$ and the principal part of $f$ at $a_k$ is $S_k(z)$.
\end{theorem}

To prove the Mittag-Leffler Theorem, we will use some preliminary results. We begin by recalling a version of the
 
\begin{theorem}[Cauchy Integral Formula]\label{fic}
Let $K \subset G$ be a compact subset of a region (i.e., an open connected subset) $G$. Then, there exists line segments $\gamma_1,..,\gamma_n$ in $G \setminus K$ such that for any holomorphic function $f$ on $G$

\begin{equation}
    f(z) = \sum_{k=1}^{n} \frac{1}{2\pi i} \int_{\gamma_k} \frac{f(w)}{w-z} dw
\end{equation}
for every $z \in K$. The line segments form a finite number of closed polygons.
\end{theorem}

\begin{proof}
\cite{conway}, VIII, 1.1
\end{proof}

We now restate Theorem \ref{rungesurfaces} in the context of $\mathbb{C}$. It is the main theorem used in the proof of the Mittag-Leffler Theorem. Its own proof will make use of several technical lemmas.

\begin{theorem}[Runge Theorem]\label{runge}
Let $K \subset \mathbb{C}$ and let $E$ be a subset of $\mathbb{C}_\infty \setminus K$ intersecting every connected component of $\mathbb{C}_\infty \setminus K$. If $f$ is holomorphic on some open set containing $K$, and $\epsilon > 0$, then there exists a rational function $R(z)$, such that all of its poles are on $E$, and such that

\begin{equation}
    \left|f(z) - R(z)\right| < \epsilon
\end{equation}
for every $z \in K$.
\end{theorem}

\begin{lemma}\label{lema1}
Let $\gamma$ be a sectionally $C^1$ curve, and let $K$ be a compact set such that $K \cap \{\gamma\} = \emptyset$. If $f$ is a continuous function on $\{\gamma\}$ and $\epsilon > 0$, then there exists a rational function $R(z)$ with all of its poles in $\{\gamma\}$ and such that

\begin{equation}
    \Biggl\lvert \int_{\gamma} \frac{f(w)}{w-z}dw - R(z) \Biggr\rvert < \epsilon
\end{equation}
for every $z \in K$.
\end{lemma}

\begin{proof}
\cite{conway}, VIII, 1.5.
\end{proof}

A small comment: this proof is done in the Banach space $C(K)$ of all continuous functions $f: K \rightarrow \mathbb{C}$, with the metric of uniform convergence that is induced by the norm

\begin{equation}
\norm{f}_\infty = sup\{|f(z)|: z \in K\}
\end{equation}

Denote by $B(E)$ = $\{f \in C(K)$: such that $\exists \{R_n\}$ a sequence of rational functions with poles in $E$, such that $R_n \rightarrow f$ uniformly on $K\}$. The Runge Theorem tells us that if $f$ is holomorphic in a neighborhood of $K$, then the restriction of $f$ to $K$ is in $B(E)$.

\begin{lemma}\label{lema2}
$B(E)$ is a closed $\mathbb{C}$-subalgebra of $C(K)$, containing all of the rational functions with poles in $E$.
\end{lemma}

%

\begin{lemma}\label{lema3}
If $a \in \mathbb{C} \setminus K$, then $(z-a)^{-1} \in B(E)$.
\end{lemma}

\begin{proof}
This proof is done with very technical analysis. For those interested, it can be found in \cite{conway}, VIII, 1.10.
\end{proof}

\begin{proof}[Proof of the Runge Theorem]
If $f$ is holomorphic on some open set $G$ and $K \subset G$, then for every $\epsilon > 0$, \textbf{Lemma} \ref{lema1} and \textbf{Theorem} \ref{fic} imply on the existence of a rational function $R(z)$ with poles in $\mathbb{C} \setminus K$, such that $\lvert f(z) - R(z) \rvert < \epsilon$ for all $z \in K$. Then \textbf{Lemma} \ref{lema2} and \textbf{Lemma} \ref{lema3} tell us that $R \in B(E)$.
\end{proof}

Besides Theorem \ref{runge}, we will need two additional results to prove the Mittag-Leffler Theorem. Denote by $\operatorname{int}X$ the interior of $X$:

\begin{lemma}\label{lema4}
If $G \subseteq \mathbb{C}$ is an open subset, then there exists a sequence $\{K_n\}$ of compact sets in $G$ such that $G = \bigcup_{n=1}^{\infty} K_n$. Moreover, the $K_n$ can be chosen so as to satisfy
\begin{enumerate}
    \item $K_n \subset \operatorname{int}K_{n+1}$
    \item If $K \subset G$ and $K$ is compact, then $K \subset K_n$ for some $n$
    \item All of the connected components of $\mathbb{C}_\infty \setminus K_n$ contain a connected component of $\mathbb{C}_\infty \setminus G$.
\end{enumerate}
\end{lemma}

\begin{proof}
For every positive integer $n$, let
\begin{equation}
    K_n = \{z: \left|z\right| \leq n \} \cap \{z: d(z, \mathbb{C} \setminus G) \geq \frac{1}{n} \}.
\end{equation}
Every $K_n$ is bounded, since $\{z: \left|z\right| \leq n \}$ and $\{z: d(z, \mathbb{C} \setminus G) \geq \frac{1}{n} \}$ are bounded, and also are closed, and therefore their intersection is closed. So $K_n$ is in fact compact. Moreover,
\begin{equation}
     \{z: \left|z\right| < n+1 \} \cap \{z: d(z, \mathbb{C} \setminus G) > \frac{1}{n+1} \}
\end{equation}
is open, contains $K_n$ and is contained in $K_{n+1}$. Therefore, the collection of the $K_n$ we have just defined satisfies \textit{(1)}.

Also, we have that $G = \bigcup_{n=1}^{\infty} K_n$, and so $G = \bigcup_{n=1}^{\infty} \operatorname{int}K_n$. Then, if $K \subset G$ is compact, the subsets $\{\operatorname{int}K_n\}$ form an open covering of $K$. Since $K$ is compact, there exists a finite subcover of $\{\operatorname{int}K_n\}$, and its union is a compact $K_n$. This proves that the collection of the $K_n$ satisfy \textit{(2)}. 

Finally, note that by construction, the unlimited component of $\mathbb{C}_\infty \setminus K_n$ (which contains $\mathbb{C}_\infty \setminus G)$ must contain $\infty$ and therefore must contain the connected component of $\mathbb{C}_\infty \setminus G$ that contains $\infty$. Moreover, the connected component contains $\{z: \left|z\right| > n\}$. So, if $D$ is a connected component of $\mathbb{C}_\infty \setminus K_n$, $D$ contains a point $z$ such that $d(z, \mathbb{C} \setminus G) < \frac{1}{n}$. Then by definition, we have a point $w$ in $\mathbb{C} \setminus G$ with $\left|w-z\right| < \frac{1}{n}$. But then $z \in \{a: \left|a-w\right| < \frac{1}{n}\} \subset \mathbb{C}_\infty \setminus K_n$. Since discs are connected, $z$ is in the component $D$ of $\mathbb{C}_\infty \setminus K_n$, $\{a: \left|a-w\right| < \frac{1}{n}\} \subset D$. So if $D_1$ is the connected component of $\mathbb{C}_\infty \setminus K_n$ containing $w$, $D_1 \subset D$. And so we have \textit{(3)}.
\end{proof}

\begin{theorem}[Weierstraß M-Test]
Let $(X, d)$ be a metric space and $u_n: X \rightarrow \mathbb{C}$ be a sequence of functions such that $\left|u_n(x)\right| \le M_n$ for all $x \in X$ and suppose that the constants satisfy

\begin{equation}
    \sum_{n=1}^{\infty} M_n < \infty
\end{equation}
Then $\sum_{n=1}^{\infty}u_n$ converges uniformly and absolutely.
\end{theorem}

\begin{proof}
\cite{conway}, II, 6.2
\end{proof}


\begin{proof}[Proof of the Mittag-Leffler Theorem]
Using \textbf{Lemma} \ref{lema4}, we can find compacts $K_n \subset G$ such that

\begin{equation}
    G = \bigcup_{n=1}^{\infty} K_n, \text{ } K_n \subset \operatorname{int}K_{n+1}
\end{equation}
and each connected component of $\mathbb{C}_\infty \setminus K_n$ contains a connected component of $\mathbb{C}_\infty \setminus G$. Since each $K_n$ is compact and $\{a_k\}$ does not have a limit point in $G$, there exists a \textit{finite} number of points $a_k$ in each $K_n$. Define the set of integers $I_n$ as:

\begin{align*}
    &I_1 = \{k: a_k \in K_1\} \\
    &I_n = \{k: a_k \in K_n \setminus K_{n-1}\}
\end{align*}
for $n \ge 2$. Define, for each $n \ge 1$ functions $f_n$ as

\[ f_n(z) = \begin{cases} 
          \sum_{k \in I_n} S_k(z) & I_n \neq \emptyset \\
          0 & I_n = \emptyset 
       \end{cases}.
    \]
Then $f_n$ is rational since the sums are finite, and its poles are the points

\begin{equation}
    \{a_k: k \in I_n\} \subset K_n \setminus K_{n-1}.
\end{equation}
Since $f_n$ does not have poles in $K_{n-1}$ if $n \ge 2$, it is holomorphic in a neighborhood of $K_{n-1}$. By the Runge Theorem, there exists a rational function $R_n(z)$, with poles in $\mathbb{C}_\infty \setminus G$ satisfying

\begin{equation}
    \left|f_n(z) - R_n(z)\right| < \biggl(\frac{1}{2}\biggr)^n
\end{equation}
for all $z \in K_{n-1}$. We claim that

\begin{equation}
  f(z) = f_1(z) +  \sum_{n=2}^{\infty} [f_n(z) - R_n(z)]
\end{equation}
is the meromorphic function we want. Indeed, let

\begin{equation}
    K \subset G \setminus \{a_k: k \ge 1\}
\end{equation}
be a compact subset. Then by \textbf{Lemma} \ref{lema4}, there exists an integer $N$ such that $K \subset K_N$. If $n \ge N$, then $\left|f_n(z) - R_n(z)\right| < \bigl(\frac{1}{2}\bigr)^n$ for all $z \in K$. Then the series is dominated in $K$ by the convergent series $\sum_{n \ge N} (\frac{1}{2})^n$, and therefore by the Weierstraß M-Test, the function $f$ converges absolutely and uniformly in $K$. So $f$ is holomorphic in $G \setminus \{a_k: k \ge 1\}$.

Now, consider a fixed integer $k \ge 1$. By hypothesis, there exists $R > 0$ such that $\left|a_j - a_k\right| > R$ if $j \neq k$. So, since the series converges absolutely, $f(z) = S_k(z) + g(z)$ in $0 < \left|z-a_k\right| < R$, with $g$ holomorphic in $\left|z-a_k\right| < R$. Threfore, $a_k$ is a pole of $f$ and $S_k(z)$ is its principal part.
\end{proof}

\section{Vector Bundles and Singular (Co)homology}

\makebox[0pt]{}\vspace{-2ex}

In this section, we give a very brief introduction to the objects and techniques that we are later going to generalize on this work. The objects introduced will be the examples where we can apply our intuition to later, when we develop the theory on more abstract settings.

\subsection{Vector Bundles}

We begin with give a quick review of the theory of vector bundles over differentiable manifolds, following the exposition from \cite{botttu}, Chapter I, §6. Throughout, differentiable manifold means differentiable manifold of class $C^{\infty}$.

\begin{definition}
Let $E$ and $M$ be differentiable manifolds. Let $\pi: E \rightarrow M$ be surjective morphism of manifolds, such that each fiber $\pi^{-1}(x)$ has the structure of a real vector space for every $x \in M$. The map $\pi$ is a \textit{$C^{\infty}$ real vector bundle of rank r} if there exists an open cover $\{U_{\alpha}\}$ of $M$ and \textit{fiber-preserving diffeomorphisms}

\begin{equation}
    \varphi_{\alpha}:  \left. E \right|_{U_{\alpha}} = \pi^{-1}(U_{\alpha}) \rightarrow U_{\alpha} \times \mathbb{R}^{r}
\end{equation}
which induce linear isomorphisms on each fibers, that is, for every $x \in U_{\alpha}$, the linear map

\begin{equation}
    \varphi_x = \left. \varphi_{\alpha} \right|_{\pi^{-1}(x)} : \pi^{-1}(x) \rightarrow \{x\} \times \mathbb{R}^r
\end{equation}
is an isomorphism. The maps $\varphi_{\alpha}$ are called the \textit{local trivializations} of the vector bundle.
\end{definition}

That is, given a point $x \in M$, we can find a neighborhood $U_{\alpha}$ of $x$ such that the restriction of the vector bundle to $U_{\alpha}$ is isomorphic to the trivial bundle $U_{\alpha} \times \mathbb{R}^{r} \rightarrow U_{\alpha}$, given by the projection on the first coordinate. Moreover, this local isomorphism of manifolds induces vector space isomorphism on each fiber over $U_{\alpha}$. We will denote the fiber $\pi^{-1}(x)$ over $x$ by $E_x$.

The maps

\begin{equation}
    \varphi_{\alpha} \circ \varphi_{\beta}^{-1}:(U_{\alpha} \cap U_{\beta}) \times \mathbb{R}^{r} \rightarrow (U_{\alpha} \cap U_{\beta}) \times \mathbb{R}^{r}
\end{equation}
induce automorphism of dimension $r$ real vector spaces on each fiber over the intersection $U_{\alpha} \cap U_{\beta}$, and so we get maps

\begin{equation}
    g_{\alpha\beta}: U_{\alpha} \cap U_{\beta} \rightarrow GL(r, \mathbb{R})
\end{equation}
given by $g_{\alpha\beta}(x) = \left. \varphi_{\alpha} \varphi_{\beta}^{-1} \right|_{\{x\} \times \mathbb{R}^{r}}$, from $U_{\alpha} \cap U_{\beta}$ to the group of invertible $r \times r$ matrices with real coefficients. The functions $g_{\alpha\beta}$ are called \textit{transition functions}. We will later see that these transition functions are very special, for they allow us to recover much of the information about the vector bundle.

The transition functions $\{g_{\alpha\beta}\}$ of a vector bundle satisfy the \textit{cocycle condition}, i.e., on $U_{\alpha} \cap U_{\beta} \cap U_{\gamma}$ we have

\begin{equation}
    g_{\alpha\beta} \circ g_{\beta\gamma} = g_{\alpha\gamma}.
\end{equation}
The cocycle $\{g_{\alpha\beta}\}$ does \textbf{depend} on the choice of trivialization, however. We are able to relate the cocycles associated to different trivializations with the following lemma:

\begin{lemma}
If the cocycle $\{g'_{\alpha\beta}\}$ comes from some other local trivializations $\{\varphi'_{\alpha}\}$, then there exists maps $\lambda_{\alpha}: U_{\alpha} \rightarrow GL(r, \mathbb{R})$ such that for each $x \in U_{\alpha} \cap U_{\beta}$

\begin{equation}
    g_{\alpha\beta}(x) = \lambda_{\alpha}(x) g'_{\alpha\beta}(x) \lambda_{\beta}(x)^{-1}.
\end{equation}
\end{lemma}

\begin{proof}
Since trivializations are fiber-preserving, we conclude that the two trivializations differ by some linear automorphism or $\mathbb{R}^r$, i.e., an invertible linear transformation at each point. So, for each $x \in U_{\alpha}$ we have

\begin{equation}
     \left. \varphi_{\alpha} \right|_{\pi^{-1}(x)} = \lambda_{\alpha}(x) \left. \varphi'_{\alpha} \right|_{\pi^{-1}(x)}
\end{equation}
where $\lambda_{\alpha}: U_{\alpha} \rightarrow GL(r, \mathbb{R})$ is a family of change-of-basis matrices. And so

\begin{equation}
    g_{\alpha\beta}(x) = \left. \varphi_{\alpha} \varphi_{\beta}^{-1} \right|_{\{x\} \times \mathbb{R}^{r}} = \lambda_{\alpha}(x) \left. \varphi_{\alpha}^{\prime} \varphi_{\beta}^{\prime -1} \right|_{\{x\} \times \mathbb{R}^{r}} \lambda_{\beta}(x)^{-1} = \lambda_{\alpha}(x) g^{\prime}_{\alpha\beta}(x)  \lambda_{\beta}(x)^{-1}
\end{equation}
\end{proof}

The following proposition states that given an open covering of a manifold and a collection of maps that we wish to be local trivializations, we can use the transition functions to recover a vector bundle with this structure.

\begin{proposition}\label{cocyclesconstructions}
Let $E$ be a set, $M$ be a manifold, and $\pi: E \rightarrow M$ be a surjective map. Let $\{U_{\alpha}\}$ be an open cover of $M$, and $\varphi_{\alpha}: \pi^{-1}(U_{\alpha}) \rightarrow U_{\alpha} \times \mathbb{R}^r$ be bijective maps satisfying $\pi_1 \circ \varphi_{\alpha} = \pi$, where $\pi_1$ denotes the projection onto the first coordinate, such that whenever the intersection $U_{\alpha} \cap U_{\beta}$ is nonempty, the map $\varphi_{\alpha} \circ \varphi_{\beta}^{-1}$ coincides with

\begin{equation}
    \varphi_{\alpha} \circ \varphi_{\beta}^{-1}(p, v) = (p, g_{\alpha\beta}(p)v)
\end{equation}
for some smooth map $g_{\alpha\beta}: U_{\alpha} \cap U_{\beta} \rightarrow GL(r, \mathbb{R})$. Then $E$ has a \textbf{unique} structure of a $C^{\infty}$ real vector bundle of rank $r$ over $M$. Moreover, the maps $\varphi_{\alpha}$ are its local trivializations for this bundle structure.
\end{proposition}

This proposition allows us to construct vector bundles on manifolds from charts and transition functions alone.

\begin{proof}
We begin by defining a linear structure on each fiber $E_x = \pi^{-1}(x)$. Let $\alpha$ be such that $x \in U_{\alpha}$. Since the map

\begin{equation}
    \left. \varphi \right|_{E_x} : E_x \rightarrow \{x\} \times \mathbb{R}^r
\end{equation}
is a bijection, we can define the linear structure on $E_x$ as the pullback of the linear structure on $\mathbb{R}^r$ via $\left. \varphi \right|_{E_x}$. We now check that this is well defined, that is, it is independent of our pick of trivializing chart $U_{\alpha}$. Suppose $U_{\beta}$ is another trivializing neighborhood with $x \in U_{\beta}$. Then, by our hypothesis, over $U_{\alpha} \cap U_{\beta}$

\begin{equation}
    \left. \varphi_{\alpha} \varphi_{\beta}^{-1} \right|_{\{x\} \times \mathbb{R}^{r}} = g_{\alpha\beta}(x) \in GL(r, \mathbb{R})
\end{equation}
and we get isomorphic vector spaces over $x$.

Now, we can assume without loss of generality that the $U_{\alpha}$ together with the $\varphi_{\alpha}$ form a coordinate chart for $M$ (since we can enlarge the collection of the $U_{\alpha}$ to form a chart), and so we can find diffeomorphisms $\psi_\alpha : U_{\alpha} \rightarrow \Tilde{U}_{\alpha}$ to some open subsets $\Tilde{U}_{\alpha} \subseteq \mathbb{R}^r$. Composing each $\varphi_\alpha$ with such a $\psi_{\alpha}$ we get maps

\begin{equation}
    \begin{tikzcd}
    \pi^{-1}(U_{\alpha}) \ar{r}{\varphi_{\alpha}} & U_{\alpha} \times \mathbb{R}^r \ar{r}{\psi_{\alpha} \times \operatorname{Id}_{\mathbb{R}^r}} & \Tilde{U}_{\alpha} \times \mathbb{R}^r
\end{tikzcd}
\end{equation}
and we can use this composites as coordinate charts for $E$. Moreover, since by our hypothesis the collection of the $\varphi_{\alpha}$ are smooth on overlaps $U_{\alpha} \cap U_{\beta}$, the collection $\{U_{\alpha}, \psi_{\alpha} \circ \varphi_{\alpha}\}$ is a differentiable structure for the manifold $E$, and the $\varphi_{\alpha}$ are diffeomorphisms with respect to this structure.
\end{proof}

\begin{definition}
Let $U \subseteq M$ be an open subset. A smooth map $s:U \rightarrow E$ such that $\pi \circ s = Id_{U}$ is called a \textit{section} of the vector bundle $E$ over $U$. The space of all sections of $E$ over $U$ is denoted $\Gamma(U, E)$.
\end{definition}

Note that, every vector bundle has a well-defined \textit{global} zero section, i.e., a section which takes a point $x$ to the zero vector of the fiber $E_x$ over it. A collection of sections $s_1,...,s_n$ over an open set $U \subseteq M$ is called a \textit{frame} for $E$ on $U$ if for every point $x \in U$, $s_1(x),...,s_n(x)$ is a basis of $E_x$.

The set $\Gamma(U, E)$ can be given the structure of a real vector space, by point-wise addition and multiplication by real scalar: given $s_1, s_2 \in \Gamma(U, E)$ and $\lambda \in \mathbb{R}$, we have

\begin{equation}
    (s_1 + s_2)(x) = s_1(x) + s_2(x)
\end{equation} 
and 

\begin{equation}
    (\lambda s_1)(x) = \lambda s_1(x)
\end{equation}
for all $x \in U$. Moreover, given a smooth function $f$ on $U$ and $s \in \Gamma(U, E)$, we can define the product $fs$ by $(fs)(x) = f(x)s(x)$ for all $x \in U$. So $\Gamma(U, E)$ also has the structure of a $C^{\infty}(U)$-module.

\begin{definition}
A morphism of vector bundles $f: E \rightarrow E'$ is a map of manifolds which is fiber-preserving (i.e., commutes with the maps $\pi$ and $\pi'$ of $E$ and $E'$), and which is linear on corresponding fibers: the induced map $f_{x}: E_x \rightarrow E'_{x}$ is linear.
\end{definition}

\begin{theorem}
Two vector bundles $E$ and $E'$ over $M$ are isomorphic if, and only if their cocycles with respect to some open cover are equivalent.
\end{theorem}

\begin{proof}
First, suppose that $f: E \rightarrow E'$ is a vector bundle isomorphism. Let $\{U_{\alpha}, \varphi_{\alpha}\}$ be trivializing charts for $E$. Since $f$ is a diffeomorphism between $E$ and $E'$, we have that $\{U_{\alpha}, \varphi'_{\alpha}\}$ where $\varphi'_{\alpha} = \varphi_{\alpha} \circ f^{-1}$ are trivializing charts for $E'$. Now suppose that $\{g_{\alpha\beta}\}$ are the cocycles for $E$. So, over $U_{\alpha} \cap U_{\beta}$, the cocyle $g'_{\alpha\beta}$ of $E'$ will be

\begin{equation}
    g^{\prime}_{\alpha\beta} = \varphi^{\prime}_{\alpha}\varphi_{\beta}^{\prime -1} = \varphi_{\alpha} \circ f^{-1} \circ (\varphi_{\beta} \circ f^{-1})^{-1} = \varphi_{\alpha} \circ f^{-1} \circ f \circ \varphi_{\beta}^{-1} = g_{\alpha\beta}.
\end{equation}

Conversely, suppose $\{U_{\alpha}, \varphi_{\alpha}\}$ and $\{U_{\alpha}, \varphi^{\prime}_{\alpha}\}$ are trivializing charts for $E$ and $E'$ respectively, with cocycles $\{g_{\alpha\beta}\}$ and $\{g'_{\alpha\beta}\}$, and $g'_{\alpha\beta} = \lambda_{\alpha}(x)g_{\alpha\beta}(x)\lambda_{\beta}(x)^{-1}$ for some non-singular matrices $\lambda_{\alpha}, \lambda_{\beta} \in GL(r, \mathbb{R})$. Over each $U_{\alpha}$, define functions $f_{\alpha}: U_{\alpha} \times \mathbb{R}^r \rightarrow U_{\alpha} \times \mathbb{R}^r$ in the following way:

\begin{equation}
    f_{\alpha}(x, v) = (x, \lambda_{\alpha}(x)^{-1}v).
\end{equation}
We then lift the maps $f_{\alpha}$ locally to a map $f: E \rightarrow E^{\prime}$ by requiring that on $U_{\alpha} \times \mathbb{R}^r$, $f = \varphi_{\alpha}^{\prime -1} \circ f_{\alpha} \circ \varphi_{\alpha}$.

So we have a function $f$, defined locally over $\pi^{-1}(U_{\alpha})$ for each $\alpha$. We now have to show that they glue to a bundle morphism that is globally defined, and moreover, that it is an isomorphism. Suppose first that $(x, v) \in U_{\alpha} \cap U_{\beta} \times \mathbb{R}^r$, and note that the rule by which we lifted our maps $f_{\alpha}$ implies that $\varphi_{\alpha}^{\prime -1} \circ f_{\alpha} = f \circ \varphi_{\alpha}^{-1}$. We want to show that, over the intersection $U_{\alpha} \cap U_{\beta}$, if $f = \varphi_{\beta}^{\prime -1} \circ f_{\beta} \circ \varphi_{\beta}$, then $f = \varphi_{\alpha}^{\prime -1} \circ f_{\alpha} \circ \varphi_{\alpha}$. By the definition of $f_{\beta}$:

\begin{equation}
    \varphi_{\beta}^{\prime -1} f_{\beta} = \varphi_{\beta}^{\prime -1}(x, \lambda_{\beta}(x)^{-1}v).
\end{equation}
Then, the hypothesis implies that

\begin{align*}
    \varphi_{\beta}^{\prime -1}(x, \lambda_{\beta}(x)^{-1}v) &= \varphi_{\alpha}^{\prime -1}(x, g'_{\alpha\beta}(x)\lambda_{\beta}(x)^{-1}v)\\
    &=  \varphi_{\alpha}^{\prime -1}(x, \lambda_{\alpha}^{-1}(x)g_{\alpha\beta}(x)\lambda_{\beta}(x) \lambda_{\beta}(x)^{-1}v)\\
    &= \varphi_{\alpha}^{\prime -1}f_{\alpha}(x, g_{\alpha\beta}(x)v)
\end{align*}
where the last equality is the definition of $f_{\alpha}$ again. So, we have proved that 

\begin{equation}
    \varphi_{\beta}^{\prime -1} f_{\beta} = \varphi_{\alpha}^{\prime -1}f_{\alpha}(x, g_{\alpha\beta}(x)v).
\end{equation}
But, we also have $\varphi_{\beta}^{\prime -1} f_{\beta} = f \varphi_{\beta}^{-1}$, and 

\begin{equation}
    f \varphi_{\beta}^{-1}(x, v) = f \varphi_{\alpha}^{-1}(x, g_{\alpha\beta}(x)v)
\end{equation}
and so $\varphi_{\alpha}^{\prime -1}f_{\alpha}(x, g_{\alpha\beta}(x)v) = f \varphi_{\alpha}^{-1}(x, g_{\alpha\beta}(x)v)$, and therefore $f = \varphi_{\alpha}^{\prime -1}f_{\alpha}\varphi_{\alpha}$, as desired. So the functions lift to a global morphism of bundles $f:E \rightarrow E'$. To check that it is indeed an isomorphism, it suffices to note that the maps $\varphi_{\alpha}$ and $\varphi'_{\alpha}$ are invertible, and that we can repeat the exact same glueing argument to construct an inverse of $f$. And so the result is proved.
\end{proof}

\begin{example}[The tangent bundle of a manifold]\label{tangentbundle}
Let $T_{x}M$ denote the the tangent space to $M$ at $x \in M$. Set

\begin{equation}
    TM = \bigcup_{x \in M} T_{x}M.
\end{equation}
It is called the \textit{tangent bundle} of $M$. We now check that it is indeed a vector bundle, by checking that it is locally trivial\footnote{For a discussion on the topology and manifold structure of $TM$, see for example \cite{tu}, Chap. 12.}. Let $\{(U_{\alpha}, \psi_{\alpha})\}$ be an atlas for $M$. Then the diffeomorphisms $\psi_{\alpha}: U_{\alpha} \rightarrow \mathbb{R}^r$ induce maps $(\psi_{\alpha})_{*}: TU_{\alpha} \rightarrow T\mathbb{R}^r$ given by the differential of $\psi_{\alpha}$ on the fibers. So the maps $(\psi_{\alpha})_{*}: TU_{\alpha} \rightarrow T\mathbb{R}^r \cong \psi_{\alpha}(U_{\alpha}) \times \mathbb{R}^r$ give the local trivializations of $TM$, and we note that the transition functions on $TM$ are the Jacobians representing the differentials of the transition functions glueing the $\psi_{\alpha}$ on $M$.

A section of the tangent bundle $TM$ is a \textit{vector field} on $M$.
\end{example}

\begin{example}
A frame $\{s_1,...s_n\}$ on the tangent bundle $TM$, over some open set $U$ is simply called a \textit{frame on $U$}. The collection of \textit{vector fields} $\{\frac{\partial}{\partial x_1},...,\frac{\partial}{\partial x_n}\}$ (i.e., sections of the tangent bundle) is a frame on $\mathbb{R}^n$.
\end{example}

\textbf{Operations with Vector Bundles:} There are some operations that can be made with vector bundles over a differentiable manifold $M$. Most functorial constructions such as direct sums, duals and homs that apply to vector spaces apply to vector bundles:

Let $E$ and $E'$ be vector bundles over $M$, of rank $r$ and $s$ respectively. The \textit{direct sum bundle} $E \bigoplus E'$ of $E$ and $E'$ is define as the vector bundle over $M$ whose fibers at $x \in M$ are $(E \bigoplus E')_x = E_x \bigoplus E'_{x}$. The local trivializations are maps $\left. E \bigoplus E' \right|_{U_{\alpha}} \rightarrow U_{\alpha} \times (\mathbb{R}^r \oplus \mathbb{R}^s)$ induced by the local trivializations of $E$ and $E'$. In the same fashion, we can define the \textit{tensor product bundle} $E \otimes E'$, the \textit{dual bundle} $E^{*}$, and the \textit{hom bundle} $\Hom(E, E')$. Note that with these constructions $\Hom(E, E')$ is isomorphic as a bundle to $E^{*} \otimes E'$.

\subsection{Singular (co)homology}
We again follow the exposition of \cite{botttu}, Chap. III, §15, and also \cite{mercpictausk}. We begin with a discussion inspired by the wonderful discussion in the beginning of \cite{rotman}.

Suppose that we're given a bounded connected region $R$ in the plane $\mathbb{R}^2$, without some points $x_1,...,x_n$. Let $\gamma: [0, 1] \rightarrow R$ be a closed curve such that its trace $\{\gamma\}$ is a subset of $R$. Let $P$ and $Q$ be two functions on $R$, satisfying suitable differentiability conditions. The motivation for our discussion is to study the path independence of the line integral

\begin{equation}
    \int_{\gamma} Pdx + Qdy
\end{equation}
i.e., study the conditions under which the above integral is independent of the path $\gamma$. We know from the theory of integral calculus that the integral is independent of $\gamma$ if, and only if

\begin{equation}
    \int_{\gamma} Pdx + Qdy = 0.
\end{equation}
This integral is affected by the position of the points $x_1,...,x_n$ relative to the curve $\gamma$, for one of the functions $P$ or $Q$ may have a singularity at one of the $x_i$, i.e., its value at the limit as $x$ approaches one of the $x_i$ may be infinite, or worse, not even defined. Consider now a collection of curves which we shall call $\gamma_i:[0, 1] \rightarrow R$ with each $\gamma_i$ containing $x_i$ in its interior\footnote{The rigorous notion of interior and exterior of a region enclosed by the trace of a closed curve is not needed for our purposes here.}, and suppose that each $\gamma_i$ is oriented oppositely to our original curve $\gamma$. Then again, from the theory of integral calculus, Green's theorem tells us that

\begin{equation}
    \int_{\gamma} Pdx + Qdy + \int_{\gamma_1} Pdx + Qdy + ... + \int_{\gamma_n} Pdx + Qdy = \int_{R}(\frac{\partial Q}{\partial x} - \frac{\partial P}{\partial y}) dxdy.
\end{equation}

The first step into the construction of an algebraic analog for the expression of this path invariance is to introduce a way to express the left-hand side of the above integral as a linear combination of curves: we write 

\begin{equation}
\int_{\gamma} Pdx + Qdy + \int_{\gamma_1} Pdx + Qdy + ... + \int_{\gamma_n} Pdx + Qdy
\end{equation}
as 

\begin{equation}
\int_{\gamma + \gamma_1 + ... + \gamma_n} Pdx + Qdy
\end{equation}
where the sign of the coefficient on each curve determines its orientation, and we may write some "weighted" linear combination, so as to take into account the number of times the curve winds around itself. We do this by introducing \textbf{integer} coefficients, different from $\pm 1$. So we have succesfully found a way to express algebraically the union of the curves $\gamma$ and $\gamma_i$, so as to take the orientation and the "winding multiplicity" into account.

We now restrict our attention to pairs of functions $P$ and $Q$ satisfying $(\frac{\partial Q}{\partial x} - \frac{\partial P}{\partial y})$ so that the right-hand side of the integral from Green's Theorem vanishes. We then have determined a whole class of linear combinations of paths such that

\begin{equation}
    \int_{\gamma + \gamma_1 + ... + \gamma_n} Pdx + Qdy = 0
\end{equation}
namely, all of the linear combinations of curves that fit the construction of the $\gamma_i$. This suggests the introduction of a \textbf{equivalence relation} on linear combinations of curves: we say that two linear combinations $\alpha$ and $\beta$ will be equivalent if

\begin{equation}
    \int_{\alpha} Pdx + Qdy = \int_{\beta} Pdx + Qdy
\end{equation}
and call the equivalence class $[\alpha]$ its \textbf{homology class}.

Note that for different homology classes such that the value is nonzero, we will not have path independence. The discussion served merely as a motivation for the consideration of the introduction of algebra for expressing finite unions of curves. We may have homology classes of curves which are not even closed! Arbitrary finite unions of curves are generally called \textbf{chains}, and closed finite unions are called \textbf{cycles}. 

Note that if $\alpha$ is the \textbf{boundary} of some two-dimensional region in the plane, then the integral is zero, and so we can consider an \textbf{abelian group} of cycles, in which the homology class of boundaries of two-dimensional regions is trivial. We call this its \textbf{homology group}. We will now develop the abstract theory of homological algebra and algebraic topology which was motivated by, and generalizes our discussion.

\textbf{Singular Homology:} Denote by $e_i$ the ith standard basis vector in $\mathbb{R}^n$, i.e., $e_i = (0,...,1,...,0)$ with $1$ in the ith component and $0$ in all other components. By convention, let $e_0$ denote the origin in $\mathbb{R}^n$. Define the \textit{standard q-simplex} $\Delta_q$ to be convex hull of the set $\{e_i\}_{i=0}^{q}$, i.e., as a set:

\begin{equation}
    \Delta_q = \Biggl\{ \sum_{j=0}^{q} t_j e_j : \sum_{j=0}^{q} t_j = 1, \text{ } t_j \ge 0 \Biggr\}.
\end{equation}

So $\Delta_0$ is the origin, $\Delta_1$ is the unit interval $[0, 1]$, $\Delta_2$ is a triangle in $\mathbb{R}^2$, $\Delta_3$ is a tetrahedron in $\mathbb{R}^3$, $\Delta_4$ is a hypertetrahedron in $\mathbb{R}^4$, and so on and so forth. We make use of the natural identification of $\mathbb{R}^n$ as a subspace of $\mathbb{R}^{n+1}$ via the map $(x_1,...,x_n) \mapsto (x_1,...,x_n,0)$.

Now, let $X$ be a topological space. A \textit{singular q-simplex} in $X$ is a continuous map $s: \Delta_q \rightarrow X$. A \textit{singular q-chain} in $X$ is a \textit{finite} linear combination of singular q-simplices, with integer coefficients. One may recognize this idea from the discussion above. A singular q-simplex is, intuitively, a "deformed q-dimensional triangle on $X$" (of course, this picture is not accurate, for we have for instance, singular simplices that are constant maps. Hence, the name \textit{singular}). These q-chains form an abelian group $S_q(X)$. We now wish to turn these groups into a \textit{complex}, and define the homology groups of $X$ as the homology groups of this complex.

For this, we need a way to define the notion of a \textit{boundary} of a singular q-simplex, as a linear combination of its faces (i.e., a singular (q-1)-chain). We do this by introducing a map that identifies the standard (q-1)-simplex with the faces of the standard q-simplex.

\begin{definition}
    The ith \textit{face map} of the standard q-simplex is the function
    
    \begin{equation}
        \partial_{q}^i : \Delta_{q-1} \rightarrow \Delta_q
    \end{equation}
    given by
    \begin{equation}
        \partial_{q}^i \Biggl(\sum_{j=0}^{q-1} t_j e_j \Biggr) = \sum_{j=0}^{i-1} t_j e_j + \sum_{j=i+1}^{q} t_j e_j.
    \end{equation}
\end{definition}

So by omitting the ith vertex, we identified the standard (q-1)-simplex with the ith face of the standard q-simplex. More generally, one can think of the face map as a \textit{linear simplex}: given some vectors $v_0,...,v_q \in V$, where $V$ is some finite dimensional real vector spaces, consider the map $L(v_0,...,v_q): \Delta_q \rightarrow V$ to be the restriction of the linear map $\mathbb{R}^q \rightarrow V$ taking $e_i$ to $v_i$. The image of $L(v_0,...,v_q)$ is evidently the convex hull of the $v_i$, since $L(v_0,...,v_q)$ is linear, and so maps $\sum_{j=0}^{q} t_j e_j \mapsto \sum_{j=0}^{q} t_j v_j$. So the ith face map becomes $L(e_0,...,\hat{e_i},...,e_q)$. In this way, we can define the ith face of singular q-simplex $s$ to be $s \circ L(e_0,...,\hat{e_i},...,e_q)$. This perspective of the linear simplex also works for more elegant computations, as we will see bellow.

We now wish to turn the groups of singular chains into a differential complex. For this, we need to define a \textit{boundary operator} $\partial: S_{q}(X) \rightarrow S_{q-1}(X)$. Consider, for $s \in S_q(X)$

\begin{equation}
    \partial s = \sum_{i=0}^{q} (-1)^{i}s \circ \partial_{q}^{i}.
\end{equation}
For every $s \in S_q(X)$, $\partial s \in S_{q-1}(X)$ is called its \textit{boundary}.

Now we need to check that the boundary of the boundary is $0$, or $\partial \circ \partial = 0$, to check that the sequence is indeed a complex. To see this, let $s \in S_q(x)$ and compute

\begin{equation}
    \partial^{2} s = \sum_{i=0}^{q}(-1)^i \sum_{j=0}^{i-1} (-1)^j s \circ L(e_0,...,\hat{e_j},...,\hat{e_i},...,e_q) + \sum_{j=i+1}^{q} (-1)^{j-1} s \circ L(e_0,...,\hat{e_i},...,\hat{e_j},...,e_q)
\end{equation}
and on the outer sum, each $L(e_0,...,\hat{e_i},...,\hat{e_j},...,e_q)$ appears exactly twice, with opposite signs, and so they cancel each other out. So the sequence

\begin{equation}\label{singchaincomp}
\begin{tikzcd}
    {\ldots} \ar{r}{\partial} & S_{q+1}(X) \ar{r}{\partial} & S_{q}(X) \ar{r}{\partial} & S_{q-1}(X) \ar{r}{\partial} & {\ldots}
\end{tikzcd}
\end{equation}
is in fact a complex. From now on, denote by $\partial_q$ the restriction $\left. \partial \right|_{S_{q}(X)}$. The group $Z_q(X) = \ker \partial_q \subset S_{q}(X)$ is called the group of \textit{singular q-cycles}, and the group $B_q(X) = \Ima \partial_{q+1} \subset S_{q}(X)$ is called the group of \textit{singular q-boundaries}, for all $q \ge 0$. By convention, we set $Z_{0}(X) = S_{0}(X)$. We define the \textit{qth singular homology group of $X$} with integer coefficients, to be $\mathrm{H}_q(X, \mathbb{Z}) = Z_q(X)/B_q(X)$. If $\alpha$ is a singular q-cycle, i.e., $\alpha \in Z_{q}(X)$, the equivalence class $[\alpha] = \alpha + B_{q}(X)$ is called its \textit{homology class}. By taking the coefficients of the linear combinations of simplices in an abelian group $G$, we get the qth singular homology groups with coefficients in $G$, $\mathrm{H}_q(X, G)$. When the group $G$ is $\mathbb{Z}$, we sometimes denote $\mathrm{H}_q(X, \mathbb{Z})$ simply by $\mathrm{H}_q(X)$.

\textbf{Singular Cohomology:} Applying the \textbf{contravariant functor} $\Hom (-, \mathbb{Z})$ to the sequence in (\ref{singchaincomp}), we get a sequence

\begin{equation}\label{singcochaincomp}
\begin{tikzcd}
    {\ldots} \ar{r} & \Hom(S_{q-1}(X), \mathbb{Z}) \ar{r}{d} & \Hom (S_{q}(X), \mathbb{Z}) \ar{r}{d} & \Hom (S_{q+1}(X), \mathbb{Z}) \ar{r} & {\ldots}
\end{tikzcd}
\end{equation}
of groups of $\mathbb{Z}$-linear functionals on each $S_q(X)$. A \textit{singular q-cochain} is such a linear functional. We denote by $S^q(X) = \Hom (S_{q}(X), \mathbb{Z})$ the group of singular q-cochains on $X$. The \textit{coboundary operator} $d$ is defined by

\begin{equation}
    d(\omega)(s) = \omega (\partial s)
\end{equation}
for $\omega \in S^q(X)$ and $s \in S_q(X)$. With this coboundary operator, the sequence (\ref{singcochaincomp}) also becomes a differential complex. Denoting by $d_q$ the restriction $\left. d \right|_{S^{q}(X)}$, $Z^q(X) = \ker d_q \subset S^{q}(X)$ and $B^q(X) = \Ima d_{q-1} \subset S^{q}(X)$, the groups $\mathrm{H}^q(X, \mathbb{Z}) = Z^q(X)/B^q(X)$ are called the \textit{qth singular cohomology group of $X$} with integer coefficients. In the same way, taking the coefficients of the linear combinations in an abelian group $G$, we get the qth singular cohomology groups with coefficients in $G$, $\mathrm{H}^q(X, G)$. To avoid confusions later in this text, we will denote the qth singular cohomology groups by $\mathrm{H}^{q}_{Sing}(X)$.

\begin{example}
Let $X$ be the one point space, say $X = \{*\}$. Then every cochain degenerates to a point, and so is a linear combination over the one point $*$, and so

\begin{equation}
    \mathrm{H}^0_{Sing}(\{*\}, \mathbb{Z}) \cong \mathbb{Z}
\end{equation}
is a free $\mathbb{Z}$-module of rank $1$. Moreover, since $X$ has only one point, every cocycle is a coboundary and we have

\begin{equation}
    \mathrm{H}^q_{Sing}(\{*\}, \mathbb{Z}) = 0
\end{equation}
for every $q > 0$. Moreover, this is also true for any space which is homotopically equivalent to a one-point space, namely, any \textbf{contractible space}.
\end{example}

\medskip

\noindent\textbf{The cup and cap products:} One advantage of studying cohomology is the possibility of having more structure on the groups $\mathrm{H}^{q}_{Sing}(X)$, making them into \textit{cohomology rings}. We now give a very brief overview of this additional structure, which is given by the \textbf{cup product}. The reference is \cite{hatcher}, 3.2.

For cochains $\varphi \in S^p(X)$, $\psi \in S^q(X)$, the \textbf{cup product} $\varphi \smile \psi \in S^{p+q}(X)$ is the cochain mapping a singular $p+q$-simplex $\sigma \in S_q(X)$ to

\begin{equation}
    (\varphi \smile \psi)(\sigma) = \varphi(\left. \sigma \right|_{[e_0,...,e_p]})\psi(\left. \sigma \right|_{[e_p,...,e_{p+q}]}) \in \mathbb{Z}
\end{equation}
where the multiplication is that of the integers. For the induced product on cohomology, we have the following formula relating the cup product $\smile$ with the coboundary operator $d$:

\begin{equation}\label{dercup}
    d(\varphi \smile \psi) = d\varphi \smile \psi + (-1)^p \varphi \smile d\psi.
\end{equation}
The formula gives that if $\varphi \in Z^p(X)$ and $\psi \in Z^q(X)$, then their cup product is in $Z^{p+q}(X)$. Moreover, it follows from equation (\ref{dercup}) that if $\varphi \in Z^p(X)$ and $\psi \in B^q(X)$, then

\begin{equation}
    d(\varphi \smile \psi) = \pm (\varphi \smile d\psi)
\end{equation}
and if $\varphi \in B^p(X)$ and $\psi \in Z^q(X)$

\begin{equation}
    d(\varphi \smile \psi) = d\varphi \smile \psi
\end{equation}
and so the cup product of a cocycle and a coboundary in any order is always a coboundary. So we have an induced cup product

\begin{equation}
    \smile: \mathrm{H}^p_{Sing}(X, \mathbb{Z}) \times \mathrm{H}^q_{Sing}(X, \mathbb{Z}) \rightarrow \mathrm{H}^{p+q}_{Sing}(X, \mathbb{Z}).
\end{equation}
It makes $\bigoplus_q \mathrm{H}^q_{Sing}(X, \mathbb{Z})$ into an associative commutative\footnote{Up to sign.} graded $\mathbb{Z}$-algebra.

There is also the more subtle construction of the \textbf{cap product}, which will allow us to relate the singular homology and cohomology groups of a manifold $X$, so long as it is closed\footnote{A closed manifold is a manifold compact and without boundary.} and orientable. For any topological space $X$, we can define, similarly to the cup product, a $\mathbb{Z}$-bilinear operation

\begin{equation}
    \frown: S_p(X) \times S^q(X) \rightarrow S_{p-q}(X)
\end{equation}
for $p \ge q$, given by

\begin{equation}
    \sigma \frown \varphi = \varphi(\left. \sigma \right|_{[e_0,...,e_q]})\left. \sigma \right|_{[e_p,...,e_q]}
\end{equation}
for a singular $p$-simplex $\sigma \in S_p(X)$ and a cochain $\varphi \in S^q(X)$. We again have a formula

\begin{equation}
    \partial(\sigma \frown \varphi)= (-1)^q (\partial \sigma \frown \varphi - \sigma \frown d\varphi)
\end{equation}
which again readily implies that the cap product of a cycle $\sigma \in Z_p(X)$ and a cocycle $\varphi \in Z^P(X)$ is again a cycle in $S_{p-q}(X)$, and that the cap product between a cycle and a coboundary, or a boundary and a cocycle are all boundaries. And so we have an induced $\mathbb{Z}$-bilinear product

\begin{equation}
    \frown: \mathrm{H}_p(X, \mathbb{Z}) \times \mathrm{H}^q_{Sing}(X, \mathbb{Z}) \rightarrow \mathrm{H}_{p-q}(X, \mathbb{Z})
\end{equation}

The cap product allows us to state

\begin{theorem}[Poincaré Duality]\label{poincareduality}
Let $X$ be an oriented closed manifold of dimension $n$. Then we have isomorphisms

\begin{equation}
    \mathrm{H}^p_{Sing}(X, \mathbb{Z}) \rightarrow \mathrm{H}_{n-p}(X, \mathbb{Z})
\end{equation}
for all $p \ge 0$.
\end{theorem}

The isomorphisms are given by mapping a class $\alpha \in \mathrm{H}^p(X, \mathbb{Z})$ to $[X] \frown \alpha$, where $[X]$ is a \textit{fundamental class}\footnote{If $X$ is closed and oriented, a fundamental class of $X$ is a generator of $\mathrm{H}_{n}(X, \mathbb{Z})$.} for $X$. One then shows that if $X$ is closed and oriented, a fundamental class $[X] \in \mathrm{H}_{n}(X, \mathbb{Z})$ always exists (cf. \cite{hatcher}, Theorem 3.26).

\medskip

Singular (co)homology are very important \textit{topological invariants} of $X$. Their definition is very general, and one should note that we made no assumptions about the topological space $X$. This makes the actual computation of the homology and cohomology groups using the singular construction to be quite hard.  One alternative, for instance, is to make further assumptions on $X$, and restrict the class of spaces on which we can compute homology, and use alternative constructions, such as \textit{simplicial}, or \textit{cellular} homology. The search for alternatives with which one can make easier topological computations will be a central topic of this work.

\pagebreak

\part{Cohomology}

\begin{center}
\textit{“Je suis quelque peu affolé par ce d´eluge de cohomologie, mais j’ai courageusement tenu le coup. Ta suite spectrale me paraît raisonnable (je croyais, sur un cas particulier, l’avoir mise en défaut, mais je m’étais trompé, et cela marche au contraire admirablement bien)."}

\medskip

\textit{“I am a bit panic-stricken by this flood of cohomology, but have borne up courageously. Your spectral sequence seems reasonable to me (I thought I had shown that it was wrong in a special case, but I was mistaken, on the contrary it works remarkably well)."}

-Jean-Pierre Serre in a letter to Alexander Grothendieck from March 14, 1956 

(avaliable in \cite{grothendieckserre}).

\end{center}

\section{De Rham Cohomology}

\makebox[0pt]{}\vspace{-2ex}

De Rham cohomology is an invaluable tool for the study of differential forms on smooth manifolds, and we will later see that it is an important topological invariant. In this section, we briefly develop the general theory needed to discuss the de Rham cohomology groups, and then state some general results about them.

\subsection{Differential Forms}

\makebox[0pt]{}\vspace{-2ex}

The theory of differential forms appear as an answer to the question of \textit{``what is the right object to integrate over abstract manifolds?"}. For the general theory, we follow the exposition from \cite{warner}.

\begin{definition}
Let $V$ be a finite dimensional real vector space. The \textit{tensor space} $V_{r,s}$ of type $(r,s)$ associated to $V$ is the vector space
\begin{equation}
    \underbrace{V \otimes V \otimes \text{...} \otimes V \otimes V}_\text{r copies}\text{ }\otimes\text{ }\underbrace{V^{*} \otimes V^{*} \otimes \text{...} \otimes V^{*} \otimes V^{*}}_\text{s copies}.
\end{equation}
The direct sum
\begin{equation}
    T(V) = \sum V_{r,s} \text{    } (r,s \ge 0)
\end{equation}
where $V_{0,0} = \mathbb{R}$, is called the \textit{tensor algebra} of $V$. The elements of $T(V)$ are linear combinations over $\mathbb{R}$ of elements of the $V_{r,s}$ and are called \textit{tensors}.
\end{definition}

The algebra $T(V)$ is a noncommutative, associative graded algebra, with multiplication given by $\otimes$. If $u=u_1 \otimes u_2 \otimes \text{ ... } \otimes u_{1}^{*}\otimes u_{2}^{*} \in V_{r_1,s_1}$ and $v=v_1 \otimes v_2 \otimes \text{ ... } \otimes v_{1}^{*}\otimes v_{2}^{*} \in V_{r_2,s_2}$, then $u \otimes v$ is given by
\begin{equation}
    u \otimes v = u_1 \otimes \text{...} \otimes u_{r_1} \otimes v_1\otimes \text{...} \otimes v_{r_2} \otimes u_{1}^{*} \otimes \text{...} \otimes u_{s_1}^{*} \otimes v_{1}^{*} \otimes \text{...} \otimes v_{s_2}^{*}
\end{equation}
and $u \otimes v \in V_{r_1+r_2,s_1+s_2}$.

\begin{definition}
Let $C(V)$ be the subalgebra $\sum_{k=0}^{\infty} V_{k,0}$ of $T(V)$. Let $I(V)$ be the bilateral ideal in $C(V)$ generated by the elements of the form $v \otimes v$, $v \in V$, and denote by
\begin{equation}
    I_k(V) = I(V) \cap V_{k,0}
\end{equation}
its homogeneous component. Then
\begin{equation}
    I(V) = \sum_{k=0}^{\infty}I_{k}
\end{equation}
is a graded ideal of $C(V)$. The \textit{exterior algebra} $\bigwedge(V)$ of $V$ is the graded algebra ${C(V)}/{I(V)}$. Denote

\begin{equation}
    \bigwedge^0(V) = \mathbb{R}, \text{    } \bigwedge^1(V) = V, \text{    } \bigwedge^k(V) = {V_{k,0}}/{I_k(V)}, \text{    } (k \ge 2)
\end{equation}
then

\begin{equation}
    \bigwedge(V) = \sum_{k=0}^{\infty} \bigwedge^{k}(V).
\end{equation}
\end{definition}

We denote the multiplication in the algebra $\bigwedge(V)$ by $\wedge$, the exterior product or \textit{wedge}. The equivalence class of $u_1 \otimes \text{ ... } \otimes u_{k}$ in $\bigwedge(v)$ is denoted by $u_1 \wedge \text{ ... } \wedge u_{k}$.

\begin{definition}
Let $M$ be a real smooth manifold and let $p \in M$. Denote by $T_p M$ the \textit{tangent space} to $M$ at $p$. Define

\begin{align*}
    &\bigwedge^{k}(M) = \bigcup_{p \in M} \bigwedge^k(T_{p}^{*} M) \\
    &\bigwedge(M) = \bigcup_{p \in M} \bigwedge(T_{p}^{*} M)
\end{align*}
$\bigwedge^{k}(M)$ and $\bigwedge(M)$ have canonical structures of smooth vector bundles over $M$.
\end{definition}

Suppose that $(U, \varphi)$ is a coordinate system on $M$ with coordinates $y_1, ... ,y_d$. Then the basis $\{\partial/ \partial y_i\}$ of $T_p M$ and $\{dy_i\}$ of $T_{P}^{*} M$, for $p \in U$ induce basis of $(T_p M)_{r,s}$, $\bigwedge^{k}(T_{P}^{*} M)$ and $\bigwedge(T_{P}^{*} M)$. For example, given a basis $y_1, ... ,y_d$, the basis in $\bigwedge^{k}(T_{P}^{*} M)$ is $\{dy_{i_1} \wedge \text{...} \wedge dy_{i_k} : i_1 < \text{...} < i_k\}$.
\par With these basis we can define maps
\begin{equation}
    \pi^{-1}(U) \rightarrow \varphi(U) \times \mathbb{R}^d
\end{equation}
where $\pi$ is the canonical projection in $T_{r,s}(M)$, $\bigwedge^{k}(M)$ and $\bigwedge(M)$. If we impose that these trivializations are coordinate systems, we have the manifold structures on our bundles.

\begin{definition}
A section of the bundle $\bigwedge^{k}(M)$ is called a \textit{differential k-form} on $M$. A section of $\bigwedge(M)$ will be simply called a \textit{differential form} on $M$.
\end{definition}

\begin{definition}
Let $M$ be a smooth manifold. Denote by $TM$ the tangent bundle on $M$. A ($C^{\infty}$) \textit{vector field} $X$ on an open set $U \subset M$ is a smooth section of $TM$.
We denote by $\mathfrak{X}(M)$ the vector space of all vector fields on $M$. Besides the vector space structure, $\mathfrak{X}(M)$ also has the structure of a $C^{\infty}(M)$-module.
\end{definition}

If $m \in U$ and $X \in \mathfrak{X}(U)$, we denote $X(m) \in T_m M$, the value of $X$ at $m$, by $X_m$. If $f \in C^{\infty}(U)$, then $X(f)$ will denote the function in $U$ such that its value at $m$ is $X_m(f)$.

We now estabilish an important duality between multilinear functions on $V$ and its exterior algebra: denote by
\begin{equation}
    A_k(V) = \text{ set of all multilinear alternated functions on }\underbrace{V \times ... \times V}_\text{k copies}.
\end{equation}
Then
\begin{equation}\label{dualidade1}
    \bigwedge^k(V^{*}) \cong \bigwedge^k(V)^{*} \cong A_k(V) \\ 
\end{equation}
and it follows that
\begin{equation}\label{dualidade2}
    \bigwedge(V^{*}) \cong \bigwedge(V)^{*} \cong A(V) = \sum_{k=0}^{\infty}A_k(V)
\end{equation}
We denote by $\mathcal{A}^k(M)$ the set of all differential k-forms, and by $\mathcal{A}^{*}(M)$ the set of all differential forms on $M$. We can add forms, multiply forms (with the \textit{wedge} product) and multiply forms by scalars. Let $\omega, \varphi \in \mathcal{A}^{*}(M)$, $m \in M$ and $c \in \mathbb{R}$. Denote the value of $\omega$ at $m$ by $\omega_m$. The value of the forms $\omega + \varphi$, $c\omega$ and $\omega \wedge \varphi$ at $m$ is, respectively, $\omega_m + \varphi_m$, $c\omega_m$ and $\omega_m \wedge \varphi_m$.
\par Let $f \in \mathcal{A}^{0}(M) \cong C^{\infty}(M)$, and let $\omega \in \mathcal{A}^{*}(M)$. Then we denote $f \wedge \omega$ by $f\omega$. That way, $\mathcal{A}^{*}(M)$ has at the same time the structure of a graded algebra over $\mathbb{R}$ and of a  $C^{\infty}(M)$-module.
\par Let $\omega \in \mathcal{A}^{k}(M)$. Then $\omega_m \in \bigwedge^{k}(T_{m}^{*}M)$, and by the duality (\ref{dualidade1}), (\ref{dualidade2}), can be considered as a multilinear alternated function in $T_m M$. Therefore, if $X_1,...,X_k \in \mathfrak{X}(M)$, we define $\omega(X_1,...,X_k)$ as the function such that its value at $m$ is

\begin{equation}
    \omega(X_1,...,X_k)(m) = \omega_m(X_1(m),...,X_k(m)).
\end{equation}

So we can consider $\omega$ as a map of modules

\begin{equation}
    \omega: \underbrace{\mathfrak{X}(M) \times \text{...} \times \mathfrak{X}(M)}_\text{k copies} \rightarrow C^{\infty}(M)
\end{equation}
that is \textbf{$C^{\infty}(M)$-multilinear and alternated}.
\par Moreover, any alternate, $C^{\infty}(M)$-multilinear map from $\mathfrak{X}$ to $C^{\infty}(M)$ defines a form (cf. \cite{warner} Cap. 2, 2.18).

\begin{remark}\label{dim}
Let $V$ be a vector space of dimension $d$. Then
\begin{equation}
    dim\bigwedge^{k}(V) = \binom{d}{k}.
\end{equation}
In particular, if $k > d$, $\bigwedge^{k}(V) = \{0\}$
\end{remark}

\begin{remark}[Canonical representation of forms]\label{canonicalrepforms}
A map $\beta: M \rightarrow \bigwedge^{k}(M)$ is a differential k-form if, and only if for each coordinate system $(U, y_1,...,y_d)$ on $M$
\begin{equation}
    \left. \beta \right|_U = \sum_{i_1<...<i_k} b_{i_1,...,i_k}dy_{i_1} \wedge ... \wedge dy_{i_k}
\end{equation}
where $b_{i_1,...,i_k}$ are $C^{\infty}$ functions on $U$.
\end{remark}

Our goal now is to construct the theory of \textbf{De Rham Cohomology}. For this, we need to extend the notion of derivative to differential forms:

\begin{definition}
Let $A = \bigoplus_{k \in \mathbb{Z}}A_k$ be a graded algebra, and let $l \in End(A)$. $l$ is called

\begin{enumerate}
    \item a \textit{derivation}, if $l(ab) = l(a)b + al(b)$ (Leibniz rule) \\
    \item an \textit{anti-derivation}, if $l(ab) = l(a)b + (-1)^{k}al(b)$ $(a \in A_k; b \in A)$ \\
    \item of degree $n$ if $l:A_k \rightarrow A_{k+n}$ for all $k$
\end{enumerate}

\end{definition}
The \textit{exterior derivative} will then be an \textit{anti-derivation of degree +1} $d:\mathcal{A}^{*}(M) \rightarrow \mathcal{A}^{*}(M)$ such that

\begin{enumerate}\label{extder}
    \item $d^2 = 0$
    \item if $f \in \mathcal{A}^{0}(M) = C^{\infty}(M)$, $df$ is the differential of $f$.
\end{enumerate}

\begin{theorem}
The exterior derivative exists and is unique.
\end{theorem}

\begin{proof}
\cite{warner}, Chap. 2, 2.20.
\end{proof}

A sketch of the construction of the exterior derivative is as follows: given $m \in M$ and $(U, x_1,...,x_d)$ a coordinate system about $p$, in view of Remark \ref{canonicalrepforms}, given $\omega \in \mathcal{A}^{*}(U)$, it can be written uniquely as

\begin{equation}
    \left. \omega \right|_U = \sum_I b_{I}dx_{I}
\end{equation}
for some $C^{\infty}$ functions $b_i$, and $dx_{I} = dx_{i_1} \wedge ... \wedge dx_{i_k}$, $\{i_1 <...< i_k\}$ where the index set $I$ runs over the subsets of $\{1,...,d\}$\footnote{If $I$ is empty, then $dx_I = 1$.}. The exterior derivative of $\omega$ at $p$ is then defined as

\begin{equation}
    d\omega_p = \sum_I  \left. da_I\right|_p \wedge  \left. dx_I \right|_p
\end{equation}
where by the second condition on the above discussion, $da$ is the differential of $a$. One then has to check that this definition is independent of the chart $(U, x_1,...,x_d)$.

\subsection{De Rham Cohomology}

\begin{definition}
Let $M$ be a smooth manifold. A differential p-form $\alpha \in \mathcal{A}^{p}(M)$ is called \textit{closed} if $d\alpha = 0$. It is called \textit{exact} if there exists $\beta \in \mathcal{A}^{p-1}(M)$ such that $d\beta = \alpha$.
\end{definition}

Denote by $Z^p(M), \subset \mathcal{A}^{p}(M)$ the subspace of all closed differential p-forms. Since $d^2 = 0$, $d(E^{p-1}(M)) \subset Z^p(M)$. Define

\begin{equation}
    \mathrm{H}_{dR}^{p}(M) = \frac{Z^p(M)}{d(\mathcal{A}^{p-1}(M))}
\end{equation}
the group of \textbf{closed forms modulo exact forms}, called the pth \textit{De Rham cohomology group} of $M$.

\begin{example}
Consider $M = S^1$. Since $dim_\mathbb{R}S^1 = 1$, by Remark \ref{dim}, $\mathcal{A}^p(S^1) = \{0\}$ if $p>1$. Then for all $p>1$, $\mathrm{H}_{dR}^{p}(S^1) = 0$. Since there aren't any exact $0$-forms ($C^{\infty}$ functions), and $S^1$ is connected, closed $0$-forms on $S^1$ are constant functions, we have
\begin{equation}
    \mathrm{H}_{dR}^{0}(S^1) \cong \mathbb{R}.
\end{equation}
To compute $\mathrm{H}_{dR}^{1}(S^1)$, we need to give meaning to the notion of integration of forms.
\end{example}

\noindent\textbf{Integration of $n$-forms over $\mathbb{R}^n$}: Denote by $dx$ the Lebesgue measure in $\mathbb{R}^n$ and denote its canonical coordinates by $x_1,...,x_n$. Let $U \subset \mathbb{R}^n$ be an open set and let $\omega = fdx_1 \wedge ... \wedge dx_n \in \mathcal{A}^{n}(U)$. Let $V \subset U$. Define

\begin{equation}
    \int_{V} \omega = \int_{V} f dx
\end{equation}

That is, integrating an n-form in $\mathbb{R}^n$ consists simply of integrating the coefficient of the form against the Lebesgue measure on $\mathbb{R}^n$. Note that the integral depends on the coordinates $x_1,...,x_n$.

\begin{theorem}[\textbf{Change of Variable Formula}]
Let $\varphi$ be a diffeomorphism from a bounded open set $D \subset \mathbb{R}^n$ to a bounded open set $\varphi(D)$. Denote
\begin{equation}
    J_{\varphi} = det\biggl(\frac{\partial \varphi_i}{\partial x_j}\biggr).
\end{equation}
Let $f$ continuous and bounded in $\varphi(D)$, and let $A \subset D$. Then
\begin{equation}
    \int_{\varphi(A)} f = \int_{A} f \circ \varphi |J_{\varphi}|
\end{equation}
\end{theorem}

\begin{definition}
Let $M, N$ be smooth manifolds and let $F \in C^{\infty}(N, M)$. Denote by $F_{*}: T_p N \rightarrow T_{F(p)} M$ its differential at $p \in N$. If $X_p \in T_p N$, we call $F_{*}(X_p)$ the \textit{push-forward} of the vector $X_p$ at $p$. Note that in general, this notion does not extend well to vector fields.
\end{definition}

\begin{remark}
If $F$ is a diffeomorphism, it extends well to vector fields, since the differential is injective at every point.
\end{remark}

\begin{definition}
Let $M, N$ be smooth manifolds and let $\omega \in \mathcal{A}^{1}(M)$. Let $F \in C^{\infty}(N, M)$. We define the \textit{pullback} of $\omega$ by $F$ as the form $F^{*}\omega \in \mathcal{A}^{1}(N)$ given by
\begin{equation}
    (F^{*}\omega)_p(X_p) = \omega_{F(p)}(F_{*}X_{p})
\end{equation}
for any $p \in N$ and $X_p \in T_p N$. More generally, if $\omega \in \mathcal{A}^{k}(M)$, define $F^{*}\omega \in \mathcal{A}^{k}(N)$ in the following manner: if $p \in N$ and $v_1,...,v_k \in T_p N$
\begin{equation}
    (F^{*} \omega)_p(v_1,...,v_k) = \omega_{F(p)}(F_{*}v_1,...,F_{*}v_k).
\end{equation}
\end{definition}

\begin{remark}
If $f \in \mathcal{A}^{0}(M)$, (i.e., if $f$ is a smooth function on $M$) the pullback $F^{*}f$ is defined as the composition
\begin{equation}
    N \xrightarrow{F} M \xrightarrow{f} \mathbb{R}, \text{  } F^{*}f = f \circ F \in \mathcal{A}^{0}(N).
\end{equation}
\end{remark}

\begin{proposition}
If $\varphi: M \rightarrow N$ is a smooth mapping, then

\begin{equation}
    \varphi^{*}(d\omega) = d(\varphi^{*}\omega),
\end{equation}
\end{proposition}

\begin{proof}
Let $\omega = fd_{x_1} \wedge ... \wedge d_{x_n}$. Then

\begin{align*}
    \varphi^{*}(d\omega) &= \varphi^{*}(df) \wedge \varphi^{*}(d_{x_1} \wedge ... \wedge d_{x_n})\\
                         &= d(f \circ \varphi) \wedge d(x_1 \circ \varphi) \wedge ... \wedge d(x_n \circ \varphi)\\
                         &= d(\varphi^{*}\omega).
\end{align*}
\end{proof}

Now consider the notion of a \textit{smooth singular $q$-simplex}, that is, a map $\sigma: \Delta^{q} \rightarrow M$ which extends to a smooth map from a neighborhood of $\Delta^q$ on $\mathbb{R}^n$ to $M$. By abuse of notation, in this section we will also denote the group of smooth singular $q$-chains on $M$ by $S_q(M)$.\footnote{The theory could also be developed for contiuous chains and continous forms.}

Let $\sigma \in S_q(M)$ be a smooth singular $q$-simplex and $\omega \in \mathcal{A}^q(M)$. If $q \ge 1$, $\omega$ can be pulled back via $\sigma$ to a $q$-form $\sigma^{*}\omega \in \mathcal{A}^{q}(U)$,\footnote{Since we are supposing that $\sigma$ extends to a smooth function on $U$.} where $U$ is an open neighborhood of $\Delta^q$ in $\mathbb{R}^n$. We then define the integral of $\omega$ over $\sigma$ to be

\begin{equation}
    \int_\sigma \omega = \int_{\Delta^q} \sigma^{*}\omega.
\end{equation}

If $\delta = \sum_i a_i \sigma_i \in S_q(M)$, then we use the linearity of the integral and extend the definition of the integral of $\omega$ over the chain $\delta$:

\begin{equation}
    \int_\delta \omega = \sum_i a_i \int_{\Delta^q} \sigma_{i}^{*}\omega
\end{equation}

Finally, if $q = 0$, then $\Delta^0 = \{0\}$, a $0$-simplex is just a point in $M$, and a $0$-form is just a smooth function on $M$. We then define, for a $0$-simplex $\sigma \in S_0(M)$

\begin{equation}
    \int_\sigma \omega = \omega(\sigma(0)).
\end{equation}

Before returning to the example, we present the most important theorem of the theory of integration of forms:

\begin{theorem}[Stokes' Theorem]
Let $\omega \in \mathcal{A}^{q-1}(U)$ be a $(q-1)$-form defined in a neighborhood of the image of $\delta \in S_q(M)$. Then
\begin{equation}
    \int_{\delta} d\omega = \int_{\partial \delta} \omega.
\end{equation}
\end{theorem}

Now returning to the example of $S^1$: the diffeomorphism $\theta$ of change of coordinates to polar coordinates is not a well defined function globally on $S^1$, because it is not injective. However, its differential $d\theta$ is well defined globally, and is a $1$-form on $S^1$. Every $1$-form on $S^1$ is closed, for if $\omega \in \mathcal{A}^{1}(S^1)$, $d\omega \in \mathcal{A}^{2}(S^1) =\{0\}$, and therefore $\mathcal{A}^{1}(S^1) = Z^1(S^1)$. Consider the map $\varphi: \mathcal{A}^{1}(S^1) \rightarrow \mathbb{R}$ given by
\begin{equation}
    \varphi(\omega) = \int_{S^1} \omega.
\end{equation}
Note that since $S^1$ is compact, $\varphi$ is well-defined, and by the linearity of the integral, is a linear functional on $\mathcal{A}^{1}(S^1)$. Moreover, $\varphi(\frac{d\theta}{2\pi i}) = 1$, and so $\varphi(c\frac{d\theta}{2\pi}) = c\varphi(\frac{d\theta}{2\pi}) = c$, for all $c \in \mathbb{R}$, and therefore $\varphi$ is surjective. By the first isomorphism theorem, we have that
\begin{equation}
    \frac{\mathcal{A}^{1}(S^1)}{ker\varphi} \cong \mathbb{R}
\end{equation}
and since $\mathcal{A}^{1}(S^1) = Z^1(S^1)$, if we show that $\ker \varphi = \left. \Ima d \right|_{\mathcal{A}^{0}(S^1)}$, the left-hand side will be, by definition $\mathrm{H}_{dR}^{1}(S^1)$, and we will have that $\mathrm{H}_{dR}^{1}(S^1) \cong \mathbb{R}$. If $\omega \in \left. \Ima d \right|_{\mathcal{A}^{0}(S^1)}$, $\omega = \left. d \right|_{\mathcal{A}^{0}(S^1)}f$, for some $f \in C^{\infty}(S^1)$, which we identify with the $C^{\infty}(\mathbb{R})$ functions that are $2\pi$-periodic. And so, from this last observation we conclude that
\begin{equation}
    \int_{S^1}df = 0
\end{equation}
and $\left. \Ima d \right|_{\mathcal{A}^{0}(S^1)} \subseteq \ker \varphi$. Moreover, if $\omega \in \ker \varphi$, define
\begin{equation}
    g(\theta) = \int_{0}^{\theta} \omega
\end{equation}
and note that $\omega \in \ker \varphi$, $\int_{S^1} \omega = 0$ and so the functions $g$ defined above are well-defined, and $\left. d \right|_{\mathcal{A}^{0}(S^1)}g = \omega$. Therefore $\omega \in \left. \Ima d \right|_{\mathcal{A}^{0}(S^1)}$. So $\ker \varphi \subseteq \left. \Ima d \right|_{\mathcal{A}^{0}(S^1)}$ and $\ker \varphi = \left. \Ima d \right|_{\mathcal{A}^{0}(S^1)}$. We conclude that $\mathrm{H}_{dR}^{1}(S^1) \cong \mathbb{R}$.

Afterwards, we will compute again the cohomology of the circle, using some more elegant machinery. 

We now present the first version of a result about the cohomology of open convex subsets of $\mathbb{R}^n$:

\begin{theorem}[Poincaré Lemma]\label{poincarelemma}
Let $U \subseteq \mathbb{R}^n$ be an open convex set. Then
\begin{equation}
    \mathrm{H}_{dR}^{k}(U) = 0
\end{equation}
for all $k \ge 1$.
\end{theorem}

To end this section, we give a brief overview of the theory of \textit{partitions of unity}:

\begin{definition}
A $C^{\infty}$ \textit{partition of unity} on a manifold $M$ is a collection of $C^{\infty}$ functions $\{\rho_{\alpha}\}_{\alpha \in A}$ such that
\begin{enumerate}
    \item the collection $\{supp(\rho_{\alpha})\}_{\alpha \in A}$ is \textit{locally finite} (a collection of subsets $\{A_{\alpha}\}$ in a topological space $S$ is locally finite if every point $q \in S$ has a neighborhood $q \in V$ such that $V \cap A_{\alpha} = \emptyset$ for all but a finite number of $\alpha$'s).
    \item $\sum \rho_{\alpha} = 1$
\end{enumerate}
Given a manifold $M$ and an open covering $\{U_{\alpha}\}_{\alpha \in A}$ of $M$, we say that the partition of unity $\{\rho_{\alpha}\}_{\alpha \in A}$ is \textit{subordinate to the covering $\{U_{\alpha}\}_{\alpha \in A}$} if $supp(\rho_{\alpha}) \subset U_{\alpha}$ for all $\alpha \in A$.
\end{definition}

Since the collection of supports $\{supp(\rho_{\alpha})\}$ is locally finite by the first condition, each point $q \in M$ is only in a finite number of the $supp(\rho_{\alpha})$, and therefore $\rho_{\alpha}(q) \neq 0$ only for a finite number of the $\alpha$'s. So the sums on the second condition are finite.

\begin{theorem}
Let $M$ be a manifold and $\{U_{\alpha}\}_{\alpha \in A}$ an open covering of $M$.
\begin{enumerate}
    \item there exists a $C^{\infty}$ partition of unity $\{\varphi_k\}_{k=1}^{\infty}$ with compact support such that for each $k$, $supp(\varphi_k) \subset U_{\alpha}$ for some $\alpha \in A$.
    \item if we let go of the condition of compact support, there exists a $C^{\infty}$ partition of unity $\{\rho_{\alpha}\}$ subordinate to $U_{\alpha}$.
\end{enumerate}
\end{theorem}

\subsection{The Real De Rham Complex}

\makebox[0pt]{}\vspace{-2ex}

In this section we begin discussing a few facts about homological algebra which will serve as the general foundation for all of the cohomology theories that we will develop in this text. We will work over an abelian category $\mathcal{C}$ (like the category \textbf{Vec$_k$} of vector spaces over $k$, or the category \textbf{Mod$_A$} of modules over a commutative ring $A$). For a historical overview of the formulation of the axioms and of the general theory of abelian categories, we refer the reader to (\cite{mclarty}). During the first part of this section the reference will be (\cite{vakil}, Cap. 1, §1.6.5).

Let $\mathcal{C}$ be an abelian category, and let $A_i$ be a collection of objects of $\mathcal{C}$. A sequence

\begin{equation}\label{complexo}
\begin{tikzcd}
{\ldots} \arrow{r}{f_{i-2}} & A_{i-1} \arrow{r}{f_{i-1}} & A_{i} \arrow{r}{f_{i}} & A_{i+1} \arrow{r}{f_{i+1}} & {\ldots}
\end{tikzcd}
\end{equation}
is called a \textit{complex at $A_i$} if $f_{i} \circ f_{i-1} = 0$, that is, if $\Ima f_{i-1} \subseteq \ker f_{i}$. It is said to be \textit{exact at $A_i$} if $\Ima f_{i-1} = \ker f_{i}$. It is called a \textit{complex} if it is a complex at $A_i$ for all $i$. Similarly, it is \textit{exact} if it is exact at $A_i$ for all $i$.
Given a complex like in (\ref{complexo}), its \textit{cohomology} at $A_i$, denoted by $\mathrm{H}^{i}$ is $\ker f_{i}/ \Ima f_{i-1}$. Therefore, a complex is exact at $A_i$ if, and only if, its cohomology at $A_i$ is $0$.

A complex like in (\ref{complexo}) will be denoted by $A^{\bullet}$.

Suppose that $\mathcal{C}$ is an abelian category. Define the category \textbf{Com$_{\mathcal{C}}$} of complexes in $\mathcal{C}$ in the following way: the objects are infinite complexes

\begin{equation}
\begin{tikzcd}
{\ldots} \arrow{r} & A_{i-1} \arrow{r}{f_{i-1}} & A_{i} \arrow{r}{f_{i}} & A_{i+1} \arrow{r}& {\ldots}
\end{tikzcd}
\end{equation}
in $\mathcal{C}$, and the morphisms $f: A^{\bullet} \rightarrow B^{\bullet}$ are diagrams

\begin{equation}
\begin{tikzcd}
{\ldots} \arrow{r} & A_{i-1} \arrow{r}{f_{i-1}} \arrow{d} & A_{i} \arrow{r}{f_{i}} \arrow{d} & A_{i+1} \arrow{r} \arrow{d} & {\ldots} \\
{\ldots} \arrow{r} & B_{i-1} \arrow{r}{g_{i-1}} & B_{i} \arrow{r}{g_{i}} & B_{i+1} \arrow{r}& {\ldots}
\end{tikzcd}
\end{equation}

\begin{remark}
The category \textbf{Com$_{\mathcal{C}}$} is also abelian.
\end{remark}
We then have a family of \textit{(covariant) functors} $\mathrm{H}^{i}(-): \textbf{Com}_{\mathcal{C}} \rightarrow \mathcal{C}$. That is, each morphism of complexes $A^{\bullet} \rightarrow B^{\bullet}$ induces a map on cohomology $\mathrm{H}^{i}(A^{\bullet}) \rightarrow \mathrm{H}^{i}(B^{\bullet})$ (in fact, the precise definition would involve a discussion of when two maps of complexes are \textit{homotopic}: homotopic maps are maps that induce the same map on cohomology. We would then have the above definition up to homotopy. For a discussion (cf. \cite{kashiwara}, §1.3)).

\begin{theorem}\label{lescohomology1}
A short exact sequence

\begin{equation}
\begin{tikzcd}
0 \arrow{r} & A^{\bullet} \arrow{r} & B^{\bullet} \arrow{r} & C^{\bullet} \arrow{r} & 0
\end{tikzcd}
\end{equation}

on \textbf{Com$_{\mathcal{C}}$}, that is, a diagram

\begin{equation}
\begin{tikzcd}
{\ldots} \arrow{r} & 0 \arrow{r} \arrow{d} & 0 \arrow{r} \arrow{d} & 0 \arrow{r} \arrow{d} & {\ldots} \\
{\ldots} \arrow{r} & A_{i-1} \arrow{r}{f_{i-1}} \arrow{d} & A_{i} \arrow{r}{f_{i}} \arrow{d} & A_{i+1} \arrow{r} \arrow{d} & {\ldots} \\
{\ldots} \arrow{r} & B_{i-1} \arrow{r}{g_{i-1}} \arrow{d} & B_{i} \arrow{r}{g_{i}} \arrow{d} & B_{i+1} \arrow{r} \arrow{d} & {\ldots} \\
{\ldots} \arrow{r} & C_{i-1} \arrow{r}{h_{i-1}} \arrow{d} & C_{i} \arrow{r}{h_{i}}  \arrow{d}& C_{i+1} \arrow{r} \arrow{d} & {\ldots} \\
{\ldots} \arrow{r} & 0 \arrow{r} & 0 \arrow{r} & 0 \arrow{r} & {\ldots}
\end{tikzcd}
\end{equation}

induces a long exact sequence

\begin{equation}
\begin{tikzcd}
{\ldots} \arrow{r} & \mathrm{H}^{i-1}(C^{\bullet}) \arrow{r} & \mathrm{H}^{i}(A^{\bullet}) \arrow{r} & \mathrm{H}^{i}(B^{\bullet}) \arrow{r} & \mathrm{H}^{i}(C^{\bullet}) \arrow{r}{\delta} & \mathrm{H}^{i+1}(A^{\bullet}) \arrow{r} & {\ldots}
\end{tikzcd}
\end{equation}

on $\mathcal{C}$.
\end{theorem}

\begin{proof}[Sketch of Proof]
The idea is to construct the connecting morphism $\delta$ using the \textit{Snake Lemma} and an opportune commutative diagram (cf. \cite{kashiwara}, Cap. 1, Prop. 1.3.6). Recall that given a morphism $f: A \rightarrow B$, its \textit{cokernel} is $\coker f = B/ \Ima f$. Then $\mathrm{H}^{i}(A^{\bullet}) = \coker (\Ima f_{i-1} \rightarrow \ker f_{i})$. There exists an exact sequence

\begin{equation}
\begin{tikzcd}
    0 \arrow{r} & \mathrm{H}^{i}(A^{\bullet}) \cong \coker(\Ima f_{i-1} \rightarrow \ker f_{i}) \arrow{r} & \coker f_{i-1} \arrow{r}{f_{i}} & \ker f_{i+1} \\ \arrow{r} & \mathrm{H}^{i+1}(A^{\bullet}) \cong   \coker (\Ima f_{i} \rightarrow \ker f_{i+1}) \arrow{r} & 0
\end{tikzcd}
\end{equation}

Consider the commutative diagram with exact rows:

\begin{tikzpicture}[>=triangle 60]
\matrix[matrix of math nodes,column sep={60pt,between origins},row
sep={60pt,between origins},nodes={asymmetrical rectangle}] (s)
{
&|[name=A]| \coker f_{i-1} &|[name=B]| \coker g_{i-1} &|[name=C]| \coker h_{i-1} &|[name=01]| 0 \\
|[name=02]| 0 &|[name=A']| \ker f_{i+1} &|[name=B']| \ker g_{i+1} &|[name=C']| \ker h_{i+1} \\
};
\draw[overlay,->, font=\scriptsize,>=latex] 
          (A) edge (B)
          (B) edge node[auto] {\( \)} (C)
          (C) edge (01)
          (A) edge node[auto] {\(f_{i}\)} (A')
          (B) edge node[auto] {\(g_{i}\)} (B')
          (C) edge node[auto] {\(h_{i}\)} (C')
          (02) edge (A')
          (A') edge node[auto] {\( \)} (B')
          (B') edge (C')

;

;

\end{tikzpicture}

then applying the Snake Lemma

\begin{tikzpicture}[>=triangle 60]
\matrix[matrix of math nodes,column sep={60pt,between origins},row
sep={60pt,between origins},nodes={asymmetrical rectangle}] (s)
{
&|[name=ka]| \mathrm{H}^{i}(A^{\bullet}) &|[name=kb]| \mathrm{H}^{i}(B^{\bullet}) &|[name=kc]| \mathrm{H}^{i}(C^{\bullet}) \\
&|[name=A]| \coker f_{i-1} &|[name=B]| \coker g_{i-1} &|[name=C]| \coker h_{i-1} &|[name=01]| 0 \\
|[name=02]| 0 &|[name=A']| \ker f_{i+1} &|[name=B']| \ker g_{i+1} &|[name=C']| \ker h_{i+1} \\
&|[name=ca]| \mathrm{H}^{i+1}(A^{\bullet}) &|[name=cb]| \mathrm{H}^{i+1}(B^{\bullet}) &|[name=cc]| \mathrm{H}^{i+1}(C^{\bullet}) \\
};
\draw[overlay,->, font=\scriptsize,>=latex]
          (ka) edge (A)
          (kb) edge (B)
          (kc) edge (C)
          (A) edge (B)
          (B) edge node[auto] {\( \)} (C)
          (C) edge (01)
          (A) edge node[auto] {\(f_{i}\)} (A')
          (B) edge node[auto] {\(g_{i}\)} (B')
          (C) edge node[auto] {\(h_{i}\)} (C')
          (02) edge (A')
          (A') edge node[auto] {\( \)} (B')
          (B') edge (C')
          (A') edge (ca)
          (B') edge (cb)
          (C') edge (cc)
;

\draw[overlay,->, font=\scriptsize,>=latex] (ka) edge (kb)
               (kb) edge (kc)
               (ca) edge (cb)
               (cb) edge (cc)
;

\draw[->,gray,rounded corners] (kc) -| node[auto,text=black,pos=.7]
{\(\delta\)} ($(01.east)+(.5,0)$) |- ($(B)!.35!(B')$) -|
($(02.west)+(-.5,0)$) |- (ca);

\end{tikzpicture}

we get the connecting morphism

\begin{equation}
\begin{tikzcd}
\mathrm{H}^{i}(A^{\bullet}) \arrow{r} & \mathrm{H}^{i}(B^{\bullet}) \arrow{r} & \mathrm{H}^{i}(C^{\bullet}) \arrow{r}{\delta} & \mathrm{H}^{i+1}(A^{\bullet}) \arrow{r} & \mathrm{H}^{i+1}(B^{\bullet}) \arrow{r} & \mathrm{H}^{i+1}(C^{\bullet})
\end{tikzcd}
\end{equation}

\end{proof}

The main reference for the discussions about the \textit{Real De Rham Complex} is (\cite{botttu}, Cap. 1).

Let $M$ be a real smooth manifold, and let $U \subseteq M$ be an open set. Define $d_i$ as $d_i = \left. d \right|_{\mathcal{A}^{i}(U)}$. We then have a natural complex

\begin{equation}
\begin{tikzcd}
    C^{\infty}(U) \cong \mathcal{A}^{0}(U) \arrow{r}{d_0} & \mathcal{A}^{1}(U) \arrow{r}{d_1} & \mathcal{A}^{2}(U) \arrow{r}{d_2} & {\ldots} \arrow{r}{d_{k-1}} & \mathcal{A}^{k}(U) \arrow{r}{d_k} & {\ldots}
\end{tikzcd}
\end{equation}
and the De Rham Cohomology groups we defined last section are the cohomology groups of this complex.

\begin{equation}
    \mathrm{H}_{dR}^{i}(M) = \frac{\ker d_i}{\Ima d_{i-1}} = \frac{Z^i(M)}{d(\mathcal{A}^{i-1}(M))}
\end{equation}

\noindent\textbf{The Mayer-Vietoris Sequence for De Rham Cohomology}: The Mayer-Vietories sequence is a powerful instrument for computing the cohomology of the union of two open sets:

\begin{proposition}
$\mathcal{A}^{*}(-)$ is a functor from the category $\textbf{Man}^{\circ}$ of smooth manifolds and smooth maps with reverted arrows to the category $\textbf{Alg}_{\mathbb{R}}$ of $\mathbb{R}$-algebras.
\end{proposition}

Therefore, the de Rham complex shares all of the functorial properties induced by $\mathcal{A}^{*}(-)$. For more about the functoriality of $\mathcal{A}^{*}(-)$ (cf. \cite{botttu}, Cap. 1, §2).

Suppose now that $M = U \cup V$, with $U$ and $V$ open sets. There exists a sequence

\begin{equation}
\begin{tikzcd}
    M & \ar{l}{i} U \bigsqcup V & \ar[l,shift left=.75ex,"\partial_1"]
  \ar[l,shift right=.75ex,swap,"\partial_0"]  U \cap V
\end{tikzcd}
\end{equation}

Where the first map is the surjection of $U \bigsqcup V$ onto $M$, and the pair of arrows $\partial_i$ are the inclusions of $U \cap V$ in $U$ and $V$. Applying the contravariant functor $\mathcal{A}^{*}(-)$, we get

\begin{equation}
\begin{tikzcd}
    \mathcal{A}^{*}(M) \ar{r}{i^{*}} & \mathcal{A}^{*}(U) \bigoplus  \mathcal{A}^{*}(V) \ar[r,shift left=.75ex,"\partial_{0}^{*}"]
  \ar[r,shift right=.75ex,swap,"\partial_{1}^{*}"] & \mathcal{A}^{*}(U \cap V)
\end{tikzcd}
\end{equation}
where the two restriction maps $\partial_{i}^{*}$ are the \textit{pullbacks} induced by the inclusions. By taking the difference of the pair of arrows $\partial_{0}^{*}$ and $\partial_{1}^{*}$, $(\omega, \tau) \in \mathcal{A}^{*}(U) \bigoplus  \mathcal{A}^{*}(V) \mapsto \tau - \omega \in \mathcal{A}^{*}(U \cap V)$, we obtain the \textbf{Mayer-Vietoris sequence}:

\begin{equation}
\begin{tikzcd}
    0 \ar{r} & \mathcal{A}^{*}(M) \ar{r}{i^{*}} & \mathcal{A}^{*}(U) \bigoplus  \mathcal{A}^{*}(V) \ar[r,shift left=.75ex,"\partial_{0}^{*}"]
  \ar[r,shift right=.75ex,swap,"\partial_{1}^{*}"] & \mathcal{A}^{*}(U \cap V) \ar{r} & 0
\end{tikzcd}
\end{equation}

\begin{remark}
The disjoint union $\bigsqcup$ is the \textit{coproduct} in the category $\textbf{Man}$, and the direct sum $\bigoplus$ is the \textit{coproduct} in the category $\textbf{Vec}_{\mathbb{R}}$.
\end{remark}

\begin{proposition}
The Mayer-Vietoris sequence is exact.
\end{proposition}

\begin{proof}
First we need to prove that $i^{*}$ is injective. Let $\omega_1, \omega_2 \in \mathcal{A}^{*}(M)$ such that $i^{*}\omega_1 = i^{*}\omega_2$. Note that the map $U \bigsqcup V \rightarrow M$ is a submersion, that is, its differential is surjective. Precisely, given $v \in TM$, there exists $u \in T (U \bigsqcup V)$ such that $v = i_{*}(u)$. We get:
\begin{equation}
    \omega_1(v) = \omega_1(i_{*}(u)) = i^{*}\omega_1(u) = i^{*}\omega_2(u) = \omega_2(i_{*}(u)) = \omega_2(v)
\end{equation}
and so $\omega_1 = \omega_2$.
\par Now we need to show that the difference of the last pair of arrows is surjective. For this, begin by considering the case of $M = \mathbb{R}$. Let $f \in C^{\infty}(U \cap V)$. We want to write $f$ as the difference between a function  on $U$ and a function on $V$. Let $\{\rho_{U}, \rho_{V}\}$ be a partition of unity subordinate to the covering $\{U, V\}$. Note then that since $f$ is defined on $U \cap V$ and $supp(\rho_{V}) \subset V$, defining $\rho_{V} f = 0$ on $U \cap V^{c}$, we have the function $\rho_{V} f$ defined on $(U \cap V) \cup (U \cap V^{c}) = U$. Then, since
\begin{equation}
    (\rho_{U}f)-(-\rho_{V}f) = f
\end{equation}
$\mathcal{A}^{0}(U) \bigoplus \mathcal{A}^{0}(V) \rightarrow \mathcal{A}^{0}(U \cap V)$ is surjective. The same argument for any manifold $M$ tells us that if $\omega \in \mathcal{A}^{q}(U \cap V)$, then $(-\rho_{V}\omega, \rho_{U}\omega) \in \mathcal{A}^{q}(U) \bigoplus \mathcal{A}^{q}(V)$ is mapped to $\omega$.
\end{proof}

\begin{example}\label{circlecohomology}

We compute $\mathrm{H}_{dR}^{1}(S^1)$ again, using the machinery of the Mayer-Vietoris sequence: Cover the circle with two open sets $\{U, V\}$. We have a sequence:

\begin{equation}
\begin{tikzcd}
    0 \ar{r} & \mathcal{A}^{*}(M) \ar{r}{i^{*}} & \mathcal{A}^{*}(U) \bigoplus  \mathcal{A}^{*}(V) \ar[r,shift left=.75ex,"\partial_{0}^{*}"]
  \ar[r,shift right=.75ex,swap,"\partial_{1}^{*}"] & \mathcal{A}^{*}(U \cap V) \ar{r} & 0.
\end{tikzcd}
\end{equation}

We know that given a short exact sequence of complexes, we get a long exact sequence on cohomology:

\begin{equation}
\begin{tikzcd}
\mathrm{H}_{dR}^{0}(S^1) \arrow{r} & \mathrm{H}_{dR}^{0}(U) \bigoplus \mathrm{H}_{dR}^{0}(V) \arrow{r} & \mathrm{H}_{dR}^{0}(U \cap V) \\ \arrow{r}{\delta} & \mathrm{H}_{dR}^{1}(S^1) \arrow{r} & \mathrm{H}_{dR}^{1}(U) \bigoplus \mathrm{H}_{dR}^{1}(V) \arrow{r} & \mathrm{H}_{dR}^{1}(U \cap V)
\end{tikzcd}
\end{equation}

We already know from last time that

\begin{equation}
\begin{tikzcd}
\mathbb{R} \arrow{r} & \mathbb{R} \bigoplus \mathbb{R} \arrow{r} & \mathrm{H}_{dR}^{0}(U \cap V) \arrow{r}{\delta} & \mathrm{H}_{dR}^{1}(S^1) \\ \arrow{r} & \mathrm{H}_{dR}^{1}(U) \bigoplus \mathrm{H}_{dR}^{1}(V) \arrow{r} & \mathrm{H}_{dR}^{1}(U \cap V)
\end{tikzcd}
\end{equation}

Now note that $\mathrm{H}_{dR}^{0}$ counts the number of connected components of the space: in fact, just note that $\mathrm{H}_{dR}^{0}(U)$ are the functions with differential zero, that is, locally constant functions. Therefore, it is generated as an $\mathbb{R}$-vector space by locally constant functions, and so its dimension is precisely the number of connected components of $U$. Moreover, note that if $\{U, V\}$ is a covering of $S^1$ and $U \neq V \neq S^1$, $U \cap V$ has two connected components. Then

\begin{equation}
\begin{tikzcd}
\mathbb{R} \arrow{r} & \mathbb{R} \bigoplus \mathbb{R} \arrow{r} & \mathbb{R} \bigoplus \mathbb{R} \arrow{r}{\delta} & \mathrm{H}_{dR}^{1}(S^1) \arrow{r} & \mathrm{H}_{dR}^{1}(U) \bigoplus \mathrm{H}_{dR}^{1}(V) \arrow{r} & \mathrm{H}_{dR}^{1}(U \cap V)
\end{tikzcd}
\end{equation}

Since the intervals are balls on $\mathbb{R}$, by the Poincaré Lemma,

\begin{equation}
\begin{tikzcd}
\mathbb{R} \arrow{r} & \mathbb{R} \bigoplus \mathbb{R} \arrow{r} & \mathbb{R} \bigoplus \mathbb{R} \arrow{r}{\delta} & \mathrm{H}_{dR}^{1}(S^1) \arrow{r} & 0 \arrow{r} & 0
\end{tikzcd}
\end{equation}
 and therefore the map $\delta$ is surjective. Finally, the image of $\mathbb{R} \rightarrow \mathbb{R} \bigoplus \mathbb{R}$ is $\mathbb{R}$, and so the kernel of $\mathbb{R} \bigoplus \mathbb{R} \rightarrow \mathbb{R} \bigoplus \mathbb{R}$ is $\mathbb{R}$. So
 
 \begin{equation}
 \mathrm{H}_{dR}^{1}(S^1) \cong \frac{(\mathbb{R} \bigoplus \mathbb{R})}{\mathbb{R}} \cong \mathbb{R}.
\end{equation}
\end{example}

To end this section, we discuss an important result which states that the de Rham cohomology groups of $M$ are in fact a topological invariant:

\begin{theorem}[De Rham Theorem]\label{derhamthm}
Let $M$ be a smooth manifold. Then we have isomorphisms

\begin{equation}
    \mathrm{H}_{dR}^{k}(M) \cong \mathrm{H}_{Sing}^{k}(M; \mathbb{R})
\end{equation}
for all $k\ge 0$.
\end{theorem}

We'll talk about the main ingredients used in the proof of the theorem. W e first need to find a way to relate de Rham cohomology and singular cohomology. The following discussion is taken from \cite{bredon}.

By the way we defined integration, the integral gives us a map\footnote{We are of course dealing with smooth singular simplices.}

\begin{equation}
    \psi_\omega: S_q(M) \rightarrow \mathbb{R}
\end{equation}
given by $\psi_\omega(\delta) = \int_{\delta} \omega$. Moreover, this map is linear in both $\omega$ and $\delta$, and so we get a map

\begin{equation}
    \psi: \mathcal{A}^p(M) \rightarrow \Hom(S_q(M), R) = S^q(M).
\end{equation}

If $\omega \in \mathcal{A}^{p-1}(M)$ and $\sigma$ is a smooth singular $p$-simplex, we have

\begin{equation}
    \psi_{d\omega}(\sigma) = \int_{\sigma} d\omega = \int_{\Delta^{p}} \sigma^{*}(d\omega) = \int_{\Delta^{p}} d(\sigma^{*}(\omega))
\end{equation}
and from Stokes' Theorem

\begin{equation}
    \int_{\Delta^{p}} d(\sigma^{*}(\omega)) = \int_{\partial \Delta^p} \sigma^{*}(\omega).
\end{equation}

Expanding into the explicit form, the integral on the right is

\begin{align*}
     \int_{\partial \Delta^p} \sigma^{*}(\omega) &= \sum_i (-1)^i \int_{\Delta^{p-1}} (\partial^i)^{*}\sigma^{*}\omega \\
     &= \sum_i (-1)^i \int_{\Delta^{p-1}} (\sigma \circ \partial^i)^{*}\omega \\
     &= \sum_i (-1)^i \int_{\sigma \circ \partial^i} \omega = \psi_{\omega} (\partial\sigma) = d(\psi_{\omega}(\sigma))
\end{align*}
and we see that the exterior derivative commutes with the differential of the singular cochain complex. More precisely, $\psi$ induces a map of complexes

\begin{equation}
\begin{tikzcd}
{\ldots} \arrow{r} & \mathcal{A}^{p-1}(M) \arrow{r}{d} \arrow{d} & \mathcal{A}^{p}(M) \arrow{r} \arrow{d} & {\ldots} \\
{\ldots} \arrow{r} & S^{q-1}(M) \arrow{r}{d} & S^q(M) \arrow{r} & {\ldots}
\end{tikzcd}
\end{equation}
which in turn induces a map on cohomology

\begin{equation}
    \psi^{*}: \mathrm{H}_{dR}^{k}(M) \rightarrow \mathrm{H}_{Sing}^{k}(M; \mathbb{R}).
\end{equation}

The de Rham Theorem says that this map is an isomorphism.

Let $\mathcal{U} = \{U_\alpha\}$ be a covering of $M$. Define $S^\mathcal{U}_q(M)$ to be the subgroup of smooth singular $q$-simplexes with image contained in $U_\alpha$ for some $\alpha$, and $S^q_\mathcal{U} (M)$ to be its dual. We have an inclusion map $S^\mathcal{U}_q(M) \rightarrow S_q(M)$, and the induced map on homology \textbf{is an isomorphism} (cf. \cite{bredon} Chap. IV, Theorem 17.7).

Suppose now for simplicity that $\mathcal{U} = \{U, V\}$. We then get an exact sequence

\begin{equation}
\begin{tikzcd}
    0 \ar{r} & S^q_\mathcal{U} (M) \ar{r} & S^q (U) \bigoplus S^q(V) \ar[r] & S^q(U \cap V) \ar{r} & 0
\end{tikzcd}
\end{equation}
analogous to the Mayer-Vietoris sequence for differential $q$-forms. If we compose the map $\psi: \mathcal{A}^q(M) \rightarrow S^q_{\mathcal{U}} (M)$ with the restriction map $S^q(M) \rightarrow S^q_\mathcal{U} (M)$, we get a commutative diagram

\begin{equation}
\begin{tikzcd}
    0 \ar{r} & \mathcal{A}^q (M) \ar{r} \ar{d} & \mathcal{A}^q (U) \bigoplus \mathcal{A}^q(V) \ar{r} \ar{d} & \mathcal{A}^q(U \cap V) \ar{r} \ar{d} & 0 \\
    0 \ar{r} & S^q_\mathcal{U} (M) \ar{r} & S^q (U) \bigoplus S^q(V) \ar{r} & S^q(U \cap V) \ar{r} & 0
\end{tikzcd}
\end{equation}
which in turn, induces long exact sequences on cohomology, the \textbf{Mayer-Vietoris ladder}:

\begin{equation}
\begin{tikzcd}
    \ldots \ar{r} & \mathrm{H}^q_{dR} (M) \ar{r} \ar{d} & \mathrm{H}^q_{dR} (U) \bigoplus  \mathrm{H}^q_{dR}(V) \ar{r} \ar{d} & \mathrm{H}^q_{dR}(U \cap V)  \ar{d}  \\
    \ldots \ar{r} & \mathrm{H}^q_{Sing} (M) \ar{r} & \mathrm{H}^q_{Sing} (U) \bigoplus  \mathrm{H}^q_{Sing}(V) \ar{r} & \mathrm{H}^q_{Sing}(U \cap V) \\
     \ar{r} & \mathrm{H}^{q+1}_{dR} (M) \ar{r} \ar{d} & \mathrm{H}^{q+1}_{dR} (U) \bigoplus  \mathrm{H}^{q+1}_{dR}(V) \ar{r} \ar{d} & \mathrm{H}^{q+1}_{dR}(U \cap V) \ar{r} \ar{d} & \ldots \\
     \ar{r} & \mathrm{H}^{q+1}_{Sing} (M) \ar{r} & \mathrm{H}^{q+1}_{Sing} (U) \bigoplus  \mathrm{H}^{q+1}_{Sing}(V) \ar{r} & \mathrm{H}^{q+1}_{Sing}(U \cap V) \ar{r} & \ldots
\end{tikzcd}
\end{equation}
where the vertical maps are all $\psi^{*}$. 

Now suppose that for $U$, $V$, and $U \cap V$, $\psi^{*}$ is an isomorphism. That is, if we have

\begin{equation}
\begin{tikzcd}
     \mathrm{H}^q_{dR} (U) \bigoplus  \mathrm{H}^q_{dR}(V) \ar{r} \ar[d, "\sim" {anchor=south, rotate=90, inner sep=.5mm}] & \mathrm{H}^q_{dR}(U \cap V)  \ar[d, "\sim" {anchor=south, rotate=90, inner sep=.5mm}]  \\
    \mathrm{H}^q_{Sing} (U) \bigoplus  \mathrm{H}^q_{Sing}(V) \ar{r} & \mathrm{H}^q_{Sing}(U \cap V) \\
     \ar{r} & \mathrm{H}^{q+1}_{dR} (M) \ar{r} \ar{d} & \mathrm{H}^{q+1}_{dR} (U) \bigoplus  \mathrm{H}^{q+1}_{dR}(V) \ar{r} \ar[d, "\sim" {anchor=south, rotate=90, inner sep=.5mm}] & \mathrm{H}^{q+1}_{dR}(U \cap V) \ar[d, "\sim" {anchor=south, rotate=90, inner sep=.5mm}] \\
     \ar{r} & \mathrm{H}^{q+1}_{Sing} (M) \ar{r} & \mathrm{H}^{q+1}_{Sing} (U) \bigoplus  \mathrm{H}^{q+1}_{Sing}(V) \ar{r} & \mathrm{H}^{q+1}_{Sing}(U \cap V) 
\end{tikzcd}
\end{equation}
Then the 5-Lemma implies that $\psi^{*}$ is also an isomorphism $\mathrm{H}^{q+1}_{dR} (M) \cong \mathrm{H}^{q+1}_{Sing} (M)$. Next, we need to state a more general version of the Poincaré Lemma:

\begin{theorem}[Poincaré Lemma revisited]
Let $U \subset \mathbb{R}^n$ be an open convex subset. Then

\begin{equation}
    \psi^{*}: \mathrm{H}_{dR}^{k}(U) \rightarrow \mathrm{H}_{Sing}^{k}(U; \mathbb{R})
\end{equation}
is an isomorphism for all $k \ge 0$.
\end{theorem}

And the final ingredient we'll need is the following proposition, and its proof can be found in \cite{bredon} Chap. V, Lemma 9.4:

\begin{proposition}
If $\psi^{*}$ is an isomorphism for disjoint open sets, then it is also an isomorphism for $\bigcup_\alpha U_\alpha$.
\end{proposition}

And now the theorem will follow directly from the following general result:

\begin{lemma}
Let $M$ be a smooth manifold of dimension $n$. Let $P(U)$ be a property of open subsets of $M$ such that we have the following:

\begin{enumerate}
    \item For all $U \subset M$ diffeomorphic to some open convex subset of $\mathbb{R}^n$, $U$ satisfies $P(U)$;\\
    \item If $U$, $V$ and $U \cap V$ all satisfy $P(U)$, $P(V)$ and $P(U \cap V)$, then $U \cup V$ satisfies $P(U \cup V)$;\\
    \item If $\{U_\alpha\}$ is a disjoint collection of open sets and for every $\alpha$, $U_\alpha$ satisfies $P(U_\alpha)$, then their disjoint union satisfies $P(\cup U_\alpha)$. 
\end{enumerate}
Then $M$ satisfies $P(M)$.
\end{lemma}

This general type of argument is known as a \textbf{Mayer-Vietoris argument}.

\begin{proof}
Suppose first that $M$ is diffeomorphic to some open subset $U \subset \mathbb{R}^n$. With effect, in this case we'll assume that $M$ is in fact some open subset of $\mathbb{R}^n$.

We will first prove by induction the case in which $U$ is a finite union of convex subsets. If $U$ is convex, then the first condition immediately yields the result. Suppose now that if $U = \bigcup_{i = 1}^n U_k$ with each $U_k$ open and convex, then $P(U)$ holds. If $U = \bigcup_{i = 1}^{n+1} U_k$, then by the first and second conditions, $P(\bigcup_{i = 1}^{n} U_k)$, $P(U_{n+1})$ and $P(\bigcup_{i = 1}^{n} U_k \cap U_{n+1}) = P(\bigcup_{i = 1}^{n} (U_k \cap U_{n+1}))$ hold, and therefore so does $P(\bigcup_{i = 1}^{n+1} U_k)$.

Now let $f$ be a proper map from $M$ to the non-negative reals. Define

\begin{equation}
    A_n = f^{-1}([n, n+1]).
\end{equation}

Since each $A_n$ is compact, we can cover each one by a finite union of convex open sets contained in $f^{-1}(n - \frac{1}{2}, n + \frac{3}{2})$. Call this union $U_n$. Then $A_n \subset U_n$, and we note that the collections of sets $U_{2n}$, and $U_{2n+1}$ are disjoint.

Now, since each $U_n$ is a finite union of convex open sets, the above induction tells us that $P(U_n)$ holds. So the third condition grants that $P(\bigcup U_{2n})$ and $P(\bigcup U_{2n+1})$ hold. But then we have $(\bigcup U_{2n}) \cap (\bigcup U_{2n+1}) = \bigcup_{k, l} (U_{k} \cap U_{2l+1})$, and since each $(U_{k} \cap U_{2l+1})$ is a finite union of convex open sets, $P$ holds for the last disjoint union. So, finally by the second condition, we have $P(\bigcup U_{2n})$, $P((\bigcup U_{2n+1}))$ and $P(\bigcup U_{2n}) \cap (\bigcup U_{2n+1})$, and therefore $P(\bigcup U_{2n}) \cup (\bigcup U_{2n+1}) = P(M)$ holds.

So we can substitute the first condition by the following condition: \textit{For all $U \subset M$ diffeomorphic to some open subset of $\mathbb{R}^n$, $U$ satisfies $P(U)$}, and prove the general case by repeating the exact same argument for general open subsets instead of convex open subsets. Indeed, by what we just proved, we can repeat the same induction without the hypothesis of convexity, and subsequently derive the general case.
\end{proof}

The framework provided by the Mayer-Vietoris argument is a widespread technique in topology. Some of its applications include proving that smooth singular simplexes and continuous singular simplexes induce the same cohomology groups; Poincaré Duality; the Thom-Leray-Gysin Isomorphism, and so on. For further discussions, see \cite{botttu}, \cite{bredon}.

\section{Holomorphic Vector Bundles and Dolbeault Cohomology}

\makebox[0pt]{}\vspace{-2ex}

In this section we develop the theory of complex manifolds and holomorphic vector bundles, generalizing the earlier concepts discussed about Riemann surfaces, and emphasizing their relation with algebraic geometry. We then proceed to develop the theory of forms on complex manifolds, and Dolbeault cohomology.

\subsection{Holomorphic Vector Bundles and Dolbeault Cohomology}

We begin with a result of the theory of several complex variables which will be useful later. For a full overview of the basic theory of several complex variables, we refer the reader to \cite{griffithsharris}, Chapter 0, as well as \cite{gunningrossi} and \cite{hormander}.

\begin{theorem}[Hartogs' Theorem]\label{hartogs}
Any holomorphic function $f$ in a neighborhood of $\Delta(r) \setminus  \Delta(r')$, where $\Delta(r)$ and $\Delta(r')$ are polydiscs in $\mathbb{C}^n$, with $r' < r$, extends to a holomorphic function on all of $\Delta(r)$.
\end{theorem}

In particular, the theorem states that for $n \ge 2$, holomorphic functions have no isolated singularities. The author finds the proof given in \cite{griffithsharris}, Chap. 0 to be particularly elegant.

The main references for complex manifolds and vector bundles is \cite{voisin_2002}, I, 2 and \cite{huybrechts}, Chapter 2.

Let $X$ be a $2n$-dimensional smooth manifold. We say that $X$ is equipped with a \textit{complex structure} if there is a covering $\{U_i\}$ and diffeomorphisms $\varphi_i: U_i \rightarrow V \subset \mathbb{C}^n$ with $V$ open, in such a way that the \textit{transition maps}

\begin{equation}
    \varphi_{j} \circ \varphi_{i}^{-1}: \varphi_i(U_i \cap U_j) \rightarrow \varphi_j(U_i \cap U_j)
\end{equation}
are holomorphic in the usual sense of functions on $\mathbb{C}^n$. Its \textit{complex dimension} is $n$ by definition. A function $f:U \rightarrow \mathbb{C}$ on some open subset $U \subseteq X$ is \textit{holomorphic} if $f \circ \varphi_{i}^{-1}$ is a holomorphic function on $\varphi(U \cap U_i)$. This definition is independent of the choice of chart on $X$, since the transition functions are holomorphic by definition and compositions of holomorphic functions are also holomorphic.

A \textit{meromorphic function} on $U \subseteq \mathbb{C}^n$ is a function defined on the complement of some nowhere dense subset $S$ of $U$, with the property that there exists an open covering $\{U_i\}$ of $U$ and holomorphic functions $g_i, h_i$ on $U_i$ with $\left. f \right|_{U_i \setminus S} = \frac{\left. g_i \right|_{U_i \setminus S}}{\left. h_i \right|_{U_i \setminus S}}$, where $h_i \neq 0$. A meromorphic function on a complex manifold $X$ is a map

\begin{equation}
    f: X \rightarrow \coprod_{x \in X} \mathcal{K}(\mathcal{O}_{X,x})
\end{equation}
where $\mathcal{K}(\mathcal{O}_{X,x})$ is the fraction field of the ring of germs of holomorphic functions around $x \in X$. The map associates to each $x \in X$ an element $f_x \in \mathcal{K}(\mathcal{O}_{X,x})$ such that for any $x_0 \in X$ there exists an open neighborhood $U$ of $x_0$ in $X$ and two holomorphic functions $f,g:U \rightarrow \mathbb{C}$ with $f_x = \frac{g}{h}$ for all $x \in U$ (and so again f is locally represented by a quotient of holomorphic functions on $U$). We denote the \textbf{field\footnote{It is indeed a field, if $U$ is connected.} of meromorphic functions} by $\mathcal{K}_X(U)$.\footnote{Like $\mathcal{O}_X$, $\mathcal{K}_X$ is in fact a \textit{sheaf of rings} on $X$.} If $X$ is connected, we denote the field of global meromorphic functions on $X$ by $K(X)$ and call it its function field.

\begin{example}[Projective Space]\label{Pn}
Complex projective space, denoted $\mathbb{P}^n_{\mathbb{C}}$ or simply $\mathbb{P}^n$ is one of the most important examples of complex manifolds, and is constructed as follows: it is the set of lines through the origin in $\mathbb{C}^{n+1}$ or equivalently,

\begin{equation}
    \mathbb{P}^n = (\mathbb{C}^{n+1} \setminus \{0\})/\mathbb{C}^{*}
\end{equation}
where we think of the action of $\mathbb{C}^{*}$ on $\mathbb{C}^{n+1}$ as multiplication by a nonzero complex scalar. The coordinates of points of $\mathbb{P}^n$ are written $(z_0: z_1:...:z_n)$. The notation is intended to indicated that for any $\lambda \in \mathbb{C}^{*}$, the two points $(\lambda z_0: \lambda z_1:...:\lambda z_n)$ and $(z_0: z_1:...:z_n)$ give the same point in $\mathbb{P}^n$. We now discuss the complex manifold structure on $\mathbb{P}^n$: the \textit{standard open covering} of $\mathbb{P}^n$ is given by the $n+1$ subsets $U_i = \{(z_0:...:z_n) | z_i \neq 0\} \subset \mathbb{P}^n$, for $0 \le i \le n$. The standard open sets are open in the quotient topology, given by

\begin{equation}
    \pi: \mathbb{C}^{n+1} \setminus \{0\} \rightarrow (\mathbb{C}^{n+1} \setminus \{0\})/\mathbb{C}^{*} = \mathbb{P}^n.
\end{equation}
The local charts $\varphi_i: U_i \rightarrow \mathbb{C}^n$ are given by

\begin{equation}
    (z_0:...:z_n) \mapsto (\frac{z_0}{z_i},...,\frac{z_{i-1}}{z_i},\hat{\frac{z_i}{z_i}},\frac{z_{i+1}}{z_i},...,\frac{z_n}{z_i})
\end{equation}
(since $z_i \neq 0$ on $U_i$). For the glueing with the transition maps

\begin{equation}
\varphi_{ij} = \varphi_i \circ \varphi_{j}^{-1}: \varphi_{j}(U_i \cap U_j) \rightarrow \varphi_{i}(U_i \cap U_j)
\end{equation}
we have

\begin{equation}
    \varphi_{ij}(w_1,...,w_n) = (\frac{w_1}{w_i},...,\frac{w_{i-1}}{w_i},\frac{w_{i+1}}{w_i},...,\frac{w_{j-1}}{w_i},\frac{1}{w_i},\frac{w_j}{w_i},...,\frac{w_n}{w_i}).
\end{equation}

Note that $U_i \cap U_j$ is given by the points with nonzero ith and jth coordinate. Since $\varphi_i$ ommits the ith coordinate, $\varphi_i(U_i \cap U_j)$ is $\varphi_i(U_i \cap U_j) = \mathbb{C}^n \setminus \{w_j = 0\}$.
\end{example}

\begin{definition}
A \textit{complex vector bundle} $\pi:E \rightarrow X$ over a complex manifold $X$ is a differentiable vector bundle whose fibers are $\mathbb{C}$-vector spaces, and the transition functions are $\mathbb{C}$-linear. A complex vector bundle is said to be equipped with a \textit{holomorphic structure} if we have a covering $U_i$ of $X$ and biholomorphisms

\begin{equation}
    \psi_i: \pi^{-1}(U_i) \rightarrow U_i \times \mathbb{C}^r
\end{equation}
that make the diagram

\begin{center}

\begin{tikzpicture}[>=triangle 60]
\matrix[matrix of math nodes,column sep={60pt,between origins},row
sep={60pt,between origins},nodes={asymmetrical rectangle}] (s)
{
|[name=A]| \pi^{-1}(U_i) &|[name=B]| U_i \times \mathbb{C}^r \\
&|[name=A']| U_i\\
};
\draw[overlay,->, font=\scriptsize,>=latex] 
          (A) edge node[auto] {\(\psi_i\)} (B)
          (A) edge node[auto] {\(\pi\)} (A')
          (B) edge node[auto] {\(p\)} (A')
;

;

\end{tikzpicture}

\end{center}
commutes, where $p$ is the $\mathbb{C}$-linear projection to $U_i$.

\end{definition}

This way, for every $x \in U_i \cap U_i$, we have a $\mathbb{C}$-linear map

\begin{equation}
    g_{ij}(x): \mathbb{C}^r \rightarrow \mathbb{C}^r
\end{equation}
where $g_{ij} = \psi_i \circ \psi_j^{-1}: U_i \cap U_j \rightarrow GL(r, \mathbb{C})$. The collection $\{U_i, g_{ij}\}$ is called the \textit{holomorphic cocycle} of the bundle. In the same way that differentiable vector bundles (both real or complex), a rank $r$ holomorphic vector bundle is uniquely determined by its defining holomorphic cocycle.

If $\pi:E \rightarrow X$ is a holomorphic vector bundle, then we can assume that the $U_i$ in the above definition are charts on $X$ since they are identified with open sets of $\mathbb{C}^n$ via the charts $\varphi_i$. Then 

\begin{equation}
(\varphi_i \times id_{\mathbb{C}^n}) \circ \psi_i : \pi^{-1}(U_i) \rightarrow U_i \times \mathbb{C}^r \rightarrow V_i \times \mathbb{C}^r
\end{equation}
give charts for $E$ whose transition functions are holomorphic. So in particular, $E$ is a complex manifold, and $\pi$ is a holomorphic map, and we can conclude that a complex structure on the manifold $X$ induces a holomorphic structure on the bundle $E$.

\begin{definition}
Given a holomorphic vector bundle $\pi:E \rightarrow X$, a \textit{holomorphic section} over an open set $U \subseteq X$ is a section $s: U \rightarrow E$ which is a holomorphic map.
\end{definition}

Later (Theorem \ref{locfreesheavesvecbundles}) we will see that holomorphic functions on $X$ are equivalent to holomorphic sections of some line bundle $L$ over $X$.

For example, the holomorphic local trivializations $g_i$ of $E$ above are given by a choice of a family of holomorphic sections of $E$, whose values at $x \in U_i$ form a basis of the fiber $E_{x} = \pi^{-1}(x)$ over $x$, as a complex vector space (i.e., are given by a \textit{holomorphic frame} over U).

\begin{example}[The Holomorphic Tangent Bundle]

Let $X$ be a complex manifold, and let $\varphi_i: U_i \rightarrow \mathbb{C}^n$ be holomorphic charts on $X$. The \textit{holomorphic tangent bundle} of $X$ is defined as the bundle given by the cocycle $\{U_i, J(\varphi_{ij}) \circ \varphi_j\}$, where $J(\varphi_{ij})$ is the jacobian of the map $\varphi_{ij} = \varphi_i \circ \varphi_j^{-1}: \varphi_j(U_i \cap U_j) \rightarrow \varphi_i(U_i \cap U_j)$.
\end{example}

So the map $g_{ij}$ takes $x \in U_i \cap U_j$ and maps it to the matrix $J(\varphi_{ij})(\varphi_j(x))$. The above definition is independent of the choice of charts on $X$. We also define the \textit{holomorphic cotangent bundle} $TX^{*}$ as the dual bundle of $TX$. The \textit{bundle of holomorphic $p$-forms} is $\bigwedge^{p} TX^{*}$. Its sheaf of sections is denoted $\Omega_X^{p}$.

\begin{example}[The tautological bundle on $\mathbb{P}^n$]\label{tautbundle}
One can prove that over $\mathbb{P}^n$ there exists essentially only one holomorphic line bundle: the \textbf{tautological line bundle}. We now describe it and its dual:

\begin{proposition}
Consider the set $\mathcal{O}(-1) \subset \mathbb{P}^n \times \mathbb{C}^{n+1}$ consisting of all pairs $(L, z) \in \mathbb{P}^n \times \mathbb{C}^{n+1}$ with $z \in L$ (so, all pairs consisting of a line $L$, or a point in $\mathbb{P}^n$, and the points on $\mathbb{C}^{n+1}$ that lie on that line). So $\mathcal{O}(-1)$ has in a natural way the structure of a holomorphic line bundle over $\mathbb{P}^n$.
\end{proposition}

\begin{proof}
The projection $\pi: \mathcal{O}(-1) \subset \mathbb{P}^n \times \mathbb{C}^{n+1} \rightarrow \mathbb{P}^n$ is given by the projection to the first factor. Consider the standard open covering $\mathbb{P}^n = \bigcup_{i=0}^n U_i$. The trivialization of $\mathcal{O}(-1)$ over $U_i$, $\psi_i:\pi^{-1}(U_i) \rightarrow U_i \times \mathbb{C}$ is given by $(L, z) \mapsto (L, z_i)$. The transition maps over some line $L = (z_0:...:z_n)$ are given by $\psi_{ij}(L): \mathbb{C} \rightarrow \mathbb{C}$, and map (recall that since its a line bundle, the transition matrices are $1 \times 1$, i.e., are given by a complex scalar) $z \mapsto \frac{z_i}{z_j} z$.
\end{proof}

So we have constructed a line bundle $\mathcal{O}(-1)$ over $\mathbb{P}^n$ such that its fiber over $L \in \mathbb{P}^n$ is naturally isomorphic (as a complex vector space of dimension $1$) to $L$.

The line bundle $\mathcal{O}(1)$ is then defined as the dual of $\mathcal{O}(-1)$, that is $\mathcal{O}(1) = \mathcal{O}(-1)^{*}$. It is called the \textit{hyperplane bundle on $\mathbb{P}^n$}. For $k > 0$, we denote by $\mathcal{O}(k)$ the line bundle $\mathcal{O}(1)^{\otimes k}$. It is the dual of $\mathcal{O}(-k) = \mathcal{O}(-1)^{\otimes k}$. We will further study line bundles on projective space below.

\end{example}
\begin{example}[The pullback bundle]
Let $X$ and $Y$ be complex manifolds, and let $f: Y \rightarrow X$ be a holomorphic map. Let $E$ be a holomorphic vector bundle on $X$, given by some cocycle  $\{U_i, \psi_{ij}\}$. Then the \textit{pullback bundle} $f^{*}E$ of $E$ by $f$ is the bundle on $Y$ given by the cocycle $\{(f^{-1}(U_i)), \psi_{ij} \circ f\}$. For any $y \in Y$ we have an isomorphism $f^{*}E_y \cong E_{f(y)}$.
\end{example}

\begin{definition}
Let $L$ be a line bundle on a complex manifold $X$ with cocycle $\{g_{ij}\}$, and $\{U_i\}$ a cover of $X$ by trivializing neighborhoods of $L$. A \textit{meromorphic section of $L$} is a collection of meromorphic functions $f_i \in \mathcal{K}_X(U_i)$ that are compatible with the cocycle of $L$, that is, on overlaps we have $f_i = g_{ij} f_j$.
\end{definition}

\noindent\textbf{Almost Complex Structures and Integrability:} Let $X$ be a differentiable manifold. Let $TX_{\mathbb{R}}$ denote the real tangent bundle. An \textit{almost complex structure} on $X$ is an endomorphism $I$ of $TX_{\mathbb{R}}$ such that $I^2 = -Id$. Equivalently, it is the \textbf{structure of a complex vector bundle on $TX_{\mathbb{R}}$}.

By the observation on the definition of holomorphic vector bundle, we noted that a complex structure on $X$ naturally gives rise to an almost complex structure on $TX_{\mathbb{R}}$.

\begin{definition}
An almost complex structure on $TX_{\mathbb{R}}$ is said to be \textit{integrable} if there exists a complex structure on $X$ that induces it.
\end{definition}

\begin{proposition}
Every complex manifold admits a natural almost complex structure.
\end{proposition}

Let $\varphi_i: U_i \rightarrow \mathbb{C}^n$ be holomorphic charts covering $X$. Consider the complex tangent bundle $TX_{\mathbb{C}} = TX_{\mathbb{R}} \otimes \mathbb{C}$, the complexified tangent bundle of $X$. If we set the complex coordinates as $z_i = x_i + iy_i$, we have a complex basis for each fiber $T_x U_i$ given by

\begin{equation}
    \frac{\partial}{\partial x_1},...,\frac{\partial}{\partial x_n},\frac{\partial}{\partial y_1},...,\frac{\partial}{\partial y_n}.
\end{equation}
Consider, on each $T_x U_i$

\begin{equation}
    I_x: T_x U_i \rightarrow T_x U_i
\end{equation}
given by

\begin{equation}
    \frac{\partial}{\partial x_i} \mapsto \frac{\partial}{\partial y_i} \text{,   } \frac{\partial}{\partial y_i} \mapsto -\frac{\partial}{\partial x_i}.
\end{equation}

One can prove that the construction is independent of the $\varphi_i$, and moreover that the endomorphisms on the fibers glue via trivializations to an endomorphism $I$ of the base space of the complexified tangent bundle $TX_{\mathbb{C}}$ (cf. \cite{huybrechts} Proposition 1.3.2, Proposition 2.6.2).

If $V$ is a real finite dimensional vector space equipped with a $\mathbb{C}$-linear operator $I: V \rightarrow V$ such that $I^2 = -Id$, then the complexification $V_{\mathbb{C}} = V \otimes \mathbb{C}$ of $V$ decomposes naturally as $V_{\mathbb{C}} = V^{1,0} \bigoplus V^{0,1}$, where $V^{1,0}$ is the eigenspace associated to the eigenvalue $i$ of $I$, and $V^{0,1}$ is the eigenspace associated to the eigenvalue $-i$ of $I$. Complex conjugation acts naturally on $V_{\mathbb{C}}$, inducing an isomorphism $V^{1,0} \cong V^{0,1}$. This decomposition allows us to prove:

\begin{proposition}
Let $(X, I)$ be an almost complex manifold. Then $TX_{\mathbb{C}}$ decomposes naturally as a direct sum

\begin{equation}
    TX_{\mathbb{C}} = T^{1,0}X \bigoplus T^{0,1}X
\end{equation}
of complex vector bundles on $X$, such that the extension of $I$ as a $\mathbb{C}$-linear operator on the fibers of $T^{1,0}X$ (respec. of $T^{0,1}X$) acts as multiplication by $i$ (respec. acts as multiplication by $-i$).
\end{proposition}

\begin{proof}
We define $T^{1,0}X$ as the vector bundle whose fibers are the vector spaces spanned by the eigenvectors associated to the eigenvalue $i$, i.e., as the kernel of $I-i Id$, and similarly, $T^{0,1}X$ as the kernel of $I+i Id$. Then, by the above discussion and by the definition of direct sum of bundles, we have the decomposition $TX_{\mathbb{C}} = T^{1,0}X \bigoplus T^{0,1}X$.
\end{proof}

\begin{proposition}\label{ident1}
If $X$ is a complex manifold then we have a canonical isomorphism between $T^{1,0}X$ and the \textbf{holomorphic} tangent bundle $TX$.
\end{proposition}

\begin{proof}
We start by noting that if $f: U \rightarrow V$ is a biholomorphism between open subsets, the differential $df$ respects the above decomposition, i.e., $df(T^{1,0}_{x}U) = T^{1,0}_{f(x)}V$, and $df(T^{0,1}_{x}U) = T^{0,1}_{f(x)}V$ (cf. \cite{huybrechts} Proposition 1.3.2). So we have an isomorphism

\begin{equation}
    df: T^{1,0}_{x}U \bigoplus T^{0,1}_{x}U \rightarrow T^{1,0}_{f(x)}V \bigoplus T^{0,1}_{f(x)}V
\end{equation}
and, on coordinates, the Jacobian of $f$ has the form

\begin{equation}
\begin{pmatrix}
    J(f) & 0 \\
    0 & \overline{J(f)} \\
\end{pmatrix}
\end{equation}
Now, let $\varphi_i:U_i \rightarrow \mathbb{C}^n$ be holomorphic charts covering $X$. Denote by $V_i = \varphi_i(U_i) \subset \mathbb{C}^n$. Consider the pullback bundle $(\varphi_{i}^{-1})^{*}(\left. T^{1,0}X \right|_{U_i})$. It is canonically isomorphic to $T^{1,0}\varphi(U_i) = T^{1,0}V_i$. Now, $T^{1,0}V_i$ is canonically trivial, and the linear isomorphism between $T^{1,0}_{\varphi_{i}(x)}V_i$ and $T^{1,0}_{\varphi_{j}(x)}V_j$ at each $x \in U_i \cap U_j$ is given by $J(\varphi_{ij}) \circ \varphi_j(x)$. So $T^{1,0}X$ is the bundle associated to the cocycle $\{J(\varphi_{ij}) \circ \varphi_j\}$. Therefore $T^{1,0}X \cong TX$.
\end{proof}

So we have identified the bundle $TX$ as a complex subbundle of the complexified tangent bundle $TX_{\mathbb{C}}$. We will now refer to $T^{1,0}X$ as the \textit{holomorphic tangent bundle} of $X$, and to $T^{0,1}X$ as the \textit{antiholomorphic tanget bundle} of $X$.

\begin{definition}
Let $(X, I)$ be an almost complex manifold. We have the following complex vector bundles:

\begin{equation}
    \bigwedge^{k}_{\mathbb{C}}X = \bigwedge^{k}(TX_{\mathbb{C}})^{*} \text{  and  } \bigwedge^{p,q}X = \bigwedge^{p}(T^{1,0}X)^{*} \otimes_{\mathbb{C}} \bigwedge^{q}(T^{0,1}X)^{*}.
\end{equation}
\end{definition}

Their modules of sections are denoted $\mathcal{A}_{\mathbb{C}}^{k}(X)$ and $\mathcal{A}^{p,q}(X)$ respectively. An element of $\mathcal{A}^{p,q}(X)$ is called a \textit{$(p,q)$-form} on $X$ or a \textit{form of type $(p,q)$} (or of \textit{bidegree $(p,q)$}). Our goal now is to define the \textit{holomorphic de Rham Complex} and the \textit{Dolbeault Complex} of $X$.

There is a canonical decomposition

\begin{equation}
    \bigwedge^{k}_{\mathbb{C}}X = \bigoplus_{p+q = k} \bigwedge^{p,q}X
\end{equation}
and consequently

\begin{equation}
    \mathcal{A}^{k}_{\mathbb{C}}(X) = \bigoplus_{p+q = k} \mathcal{A}^{p,q}(X).
\end{equation}

\begin{remark}
It follows from Proposition \ref{ident1} that we can identify the bundles $\bigwedge^{p} TX^{*}$ and $\bigwedge^{p,0}X$ of a complex manifold $X$. In particular, we can identify the modules $\Omega_X^{p}$ and $\mathcal{A}^{p,0}$, and any holomorphic section of $\bigwedge^{p} TX^{*}$ defines a section of $\bigwedge^{p,0}X$.
\end{remark}

Now, by abuse of notation, denote by $d$ the $\mathbb{C}$-linear extension of the exterior differential. Denote by $\pi^{p,q}: \mathcal{A}^{k}_{\mathbb{C}}(X) \rightarrow \mathcal{A}^{p,q}(X)$ the canonical projection. Then we define the operators 

\begin{equation}
    \partial = \pi^{p+1,q} \circ d: \mathcal{A}^{p,q}(X) \rightarrow \mathcal{A}^{p+1,q}(X), \text{    } \bar{\partial} = \pi^{p,q+1} \circ d: \mathcal{A}^{p,q}(X) \rightarrow \mathcal{A}^{p,q+1}(X).
\end{equation}
where we consider the restriction of $d$ to $\mathcal{A}^{p,q}(X)$. It follows from this definition that $d = \partial + \bar{\partial}$, and that $\partial^2 = \bar{\partial}^2 = 0$, $\partial \bar{\partial} + \bar{\partial} \partial = 0$. They also satisfy the Leibniz rule: for $\alpha \in \mathcal{A}^{p,q}(U)$ and $\beta \in \mathcal{A}^{r,s}(U)$ we have

\begin{align*}
    &\partial(\alpha \wedge \beta) = \partial(\alpha) \wedge \beta + (-1)^{p+q} \alpha \wedge \partial(\beta) \\
    &\bar{\partial}(\alpha \wedge \beta) = \bar{\partial}(\alpha) \wedge \beta + (-1)^{p+q} \alpha \wedge \bar{\partial}(\beta).
\end{align*}

Indeed, since the exterior differential $d$ satisfies the Leibniz rule $d(\alpha \wedge \beta) = d(\alpha) \wedge \beta + (-1)^{p+q} \alpha \wedge d(\beta)$, taking the $(p+r+1, q+s)$-parts on both sides (respec. $(p+r, q+s+1)$-parts on both sides ), we get the Leibniz rule for $\partial$ (respec. for $\bar{\partial}$).

We can then define the \textit{complexified de Rham complex}: it is the complex

\begin{equation}
\begin{tikzcd}
    \mathcal{A}^{0}_{\mathbb{C}}(X) \arrow{r}{d_0} & \mathcal{A}^{1}_{\mathbb{C}}(X) \arrow{r}{d_1} & \mathcal{A}^{2}_{\mathbb{C}}(X) \arrow{r}{d_2} & {\ldots} \arrow{r}{d_{k-1}} & \mathcal{A}^{k}_{\mathbb{C}}(X) \arrow{r}{d_k} & {\ldots}.
\end{tikzcd}
\end{equation}

The cohomology groups of this complex

\begin{equation}
    \mathrm{H}_{dR}^{k}(X; \mathbb{C}) = \frac{\ker d_k}{\Ima d_{k-1}}
\end{equation}
are the \textit{de Rham Cohomology groups with coefficients in} $\mathbb{C}$. They satisfy

\begin{equation}
    \mathrm{H}_{dR}^{k}(X; \mathbb{C}) = \mathrm{H}_{dR}^{k}(X) \otimes_{\mathbb{R}} \mathbb{C}
\end{equation}
since $\mathbb{C}$ is flat as an $\mathbb{R}$-module (cf. Example \ref{univcoeffthm}).

Now, if we fix $p$, we get a differential complex $\mathcal{A}^{p,\bullet}X$, with differential given by $\bar{\partial}$. That is, the complex

\begin{equation}
\begin{tikzcd}
    \mathcal{A}^{p,0}(X) \arrow{r}{\bar{\partial}_0} & \mathcal{A}^{p,1}(X) \arrow{r}{\bar{\partial}_1} & \mathcal{A}^{p,2}(X) \arrow{r}{\bar{\partial}_2} & {\ldots} \arrow{r}{\bar{\partial}_{k-1}} & \mathcal{A}^{p,k}(X) \arrow{r}{\bar{\partial}_k} & {\ldots}.
\end{tikzcd}
\end{equation}
where $\bar{\partial}_k$ denotes $\left. \bar{\partial} \right|_{\mathcal{A}^{p,k}(X)}$. This is called the \textit{Dolbeault Complex} of $X$. The cohomology groups of this complex

\begin{equation}
    \mathrm{H}_{\bar{\partial}}^{p,q}(X) = \frac{\ker \bar{\partial}_{q}}{\Ima \bar{\partial}_{q-1}}
\end{equation}
are called the \textit{Dolbeault cohomology groups} of $X$.

Note that the definition of the Dolbeault cohomology groups of $X$ depend on the decomposition $TX_{\mathbb{C}} = T^{1,0}X \bigoplus T^{0,1}X$, and consequently, on the complex structure of $X$. So the \textbf{Dolbeault cohomology groups of $X$ are not a purely topological invariant of $X$}, for their depend on the additional structure of $X$ as a complex manifold.

\subsection{Divisors and Line Bundles}\label{divlb}

\makebox[0pt]{}\vspace{-2ex}

In this section we further study line bundles and tie together some ideas by introducing divisors on higher dimensional complex manifolds. This section should be read after Section \ref{sheaves}, and it relies heavily on its constructions and results. The reader unfamiliar with sheaf cohomology should skip it for now. It is only included here for contextual completeness and some proofs are omitted or just sketched. The main references are \cite{huybrechts}, 2 and \cite{griffithsharris}, Chap. 1.

Let $\mathcal{O}(0)$ denote the trivial line bundle $X \times \mathbb{C} \rightarrow \mathbb{C}$ on $\mathbb{P}^n$. We want to endow set of isomorphism classes of holomorphic line bundles on $\mathbb{P}^n$ with an abelian group strucutre. In fact, we will later show that \textit{any} holomorphic line bundle on $\mathbb{P}^n$ is isomorphic to $\mathcal{O}(k)$, for some $k$. 
This abelian group structure is a special case of the following proposition:

\begin{proposition}
Let $X$ be a complex manifold. The isomorphism classes of holomorphic line bundles over $X$ form an abelian group under the operation $\otimes$. The inverse operation is given by taking duals. This group is called the \textbf{Picard group} of $X$, and is denoted $Pic(X)$.
\end{proposition}

\begin{proof}
We need to show that given a line bundle $L$ on $X$, $L \otimes L^{*}$ is isomorphic to the trivial bundle. For this, we use the cocycle description of $L \otimes L^{*}$. Recall that if $L$ and $L'$ are holomorphic line bundles over $X$ given by the cocycles $\{g_{ij}\}$ and $\{g'_{ij}\}$ respectively, the cocycle of $L \otimes L'$ is given by $\{g_{ij}g'_{ij}\}$. Moreover, the line bundle $L^{*}$ is given by the cocyle $\{g_{ij}^{-1}\}$. So the line bundle $L \otimes L^{*}$ is given by the cocycle $\{g_{ij}g_{ij}^{-1}\}$ and so it is isomorphic to the trivial bundle over $X$. 
\end{proof}

\begin{corollary}
The Picard group $Pic(X)$ is naturally isomorphic to $\mathrm{H}^{1}(X, \mathcal{O}_{X}^{*})$, the first cohomology group of $X$ with coefficients in the subsheaf of invertible elements of the sheaf of holomorphic functions on $X$.
\end{corollary}

Indeed, the cocycle description of the elements of $Pic(X)$ associates to each line bundle $L$ an element of $\check{\mathrm{H}}^1(X, \mathcal{O}_{X}^{*})$. The isomorphism then follows from Theorem \ref{paracompact}.

\begin{corollary}
Let $f: X \rightarrow Y$ be a holomorphic map. Then pullback of holomorphic line bundles by $f$ defines a group homomorphism

\begin{equation}
    f^{*}:Pic(Y) \rightarrow Pic(X).
\end{equation}
\end{corollary}

Thinking of $Pic(X)$ as $\mathrm{H}^1(X, \mathcal{O}_{X}^{*})$, the map $\mathrm{H}^1(Y, \mathcal{O}_{Y}^{*}) \rightarrow \mathrm{H}^1(X, \mathcal{O}_{X}^{*})$ is induced by the sheaf morphism $f^{*}:\mathcal{O}_{Y}^{*} \rightarrow \mathcal{O}_{X}^{*}$ induced by $f: X \rightarrow Y$.

We now introduce the \textit{exponential sequence of $X$}, as a way to study the Picard group $Pic(X)$:

\begin{definition}
Let $X$ be a complex manifold. The \textit{exponential sequence} is the short exact sequence

\begin{equation}\label{expsequence}
\begin{tikzcd}
    0 \arrow{r} & \underbar{$\mathbb{Z}$} \arrow{r} & \mathcal{O}_X \arrow{r} & \mathcal{O}_{X}^{*} \arrow{r} & 0
\end{tikzcd}
\end{equation}
where $\underbar{$\mathbb{Z}$}$ is the constant sheaf with value $\mathbb{Z}$ and $\underbar{$\mathbb{Z}$} \rightarrow \mathcal{O}_{X}$ is the natural inclusion. The second map $\mathcal{O}_{X}\rightarrow \mathcal{O}_{X}^{*}$ is given by exponentiating $f$: $f \mapsto e^{2 \pi i f}$.\footnote{Note that this indeed gives us a sheaf morphism $\mathcal{O}_{X}\rightarrow \mathcal{O}_{X}^{*}$: $e^{2 \pi i f}$ is a nonzero holomorphic function, and moreover $e^{2 \pi i (f+g)} = e^{2 \pi i f}e^{2 \pi i g}$} It is an epimorphism because around each point $x \in X$ we can find a small simply connected neighborhood where we have a locally defined branch of the complex logarithm, the inverse of the exponential.
\end{definition}

We now look at the long exact sequence in cohomology induced by the short exact sequence (\ref{expsequence}):

\begin{equation}
\begin{tikzcd}
    \mathrm{H}^{0}(X, \mathcal{O}_{X}^{*}) \arrow{r} & \mathrm{H}^1(X, \underbar{$\mathbb{Z}$}) \arrow{r} & \mathrm{H}^1(X, \mathcal{O}_{X}) \arrow{r} & \mathrm{H}^1(X, \mathcal{O}_{X}^{*}) \arrow{r} & \mathrm{H}^2(X, \underbar{$\mathbb{Z}$})
\end{tikzcd}
\end{equation}

First, recall that $\mathrm{H}^{0}(X, \mathcal{O}_{X}^{*}) = \Gamma(X, \mathcal{O}_{X}^{*})$, and so if $X$ is compact, by the \textbf{maximum principle}, it has \textbf{no nonconstant holomorphic functions}. So $\mathrm{H}^{0}(X, \mathcal{O}_{X}^{*}) = 0$, and we can conclude that \textbf{if $X$ is compact, $\mathrm{H}^1(X, \underbar{$\mathbb{Z}$}) \rightarrow \mathrm{H}^1(X, \mathcal{O}_{X})$ is injective}. 

This sequence is a useful tool to compute the Picard group of $X$. One way in which this sequence gives information about $Pic(X)$ is through its image in $\mathrm{H}^{2}(X, \underbar{$\mathbb{Z}$})$. The image of a line bundle under this map is called its \textbf{first Chern class}:

\begin{definition}
Let $L \in \mathrm{H}^{1}(X, \mathcal{O}_{X}^{*})$ be a line bundle. Its \textbf{first Chern class} $c_{1}(L)$ is the image of $L$ under the boundary map

\begin{equation}
    Pic(X) \cong \mathrm{H}^{1}(X, \mathcal{O}_{X}^{*}) \longrightarrow \mathrm{H}^{2}(X, \underbar{$\mathbb{Z}$}).
\end{equation}
\end{definition}

\noindent\textbf{Analytic Subvarieties:} Let $X$ be a complex manifold. An \textit{analytic subvariety} of $X$ is a closed subset $V \subset X$ such that for any point $x \in X$ there exists some open neighborhood $x \in U \subset X$ of $x \in X$ such that $V \cap U$ is the zero locus of finitely many holomorphic functions $f_1,...,f_k \in \mathcal{O}_X(U)$.

A point $x$ of an analytic subvariety $V$ is called \textit{smooth} or \textit{regular} if the functions $f_1,...,f_k$ can be chosen such that $\varphi(x) \in \varphi(U)$ is a regular point of the holomorphic map $f = (f_1 \circ \varphi_1,...,f_k \circ \varphi_k): \varphi(U) \rightarrow \mathbb{C}^k$ (i.e., its Jacobian at $x$ has maximum rank, where $(U, \varphi)$ is a chart around $x$). A point $x \in V$ is called \textit{singular} if it is not regular. The set of regular points of $V$ is called $V_{reg} = V \setminus V_{sing}$, and it is a \textbf{complex submanifold of $X$}. The set $V_{sing}$ of singular points of $V$ is closed in $V$, and $V_{reg}$ is open and dense in $V$.

An analytic subvariety $V$ is said to be \textit{irreducible} if it cannot be written as the union $V = V_1 \cup V_2$ of two analytic subvarieties $V_1$ and $V_2$, with $V_1 \neq V$ and $V_2 \neq V$. The dimension of an analytic variety is the dimension of $V_{reg}$ as a submanifold of $X$.

An \textit{analytic hypersurface} is an analytic subvariety of codimension $1$, that is

\begin{equation}
\dim X - \dim V = 1.
\end{equation}
An analytic hypersurface $V \subset V$ is given as the zero locus of a single nontrivial holomorphic function. Every hypersurface is the union of its irreducible components, i.e., $V = \bigcup_{i} V_i$ with $V_i$ irreducible. This union is always finite if $X$ is compact. In general, it is only locally finite.

\begin{definition}
A \textbf{divisor} $D$ on a complex manifold $X$ is a formal linear combination over $\mathbb{Z}$ of irreducible analytic hypersurfaces $Y_i$ of $X$

\begin{equation}
    D = \sum a_i [Y_i], \text{ } a_i \in \mathbb{Z}.
\end{equation}
The \textbf{divisor group} $Div(X)$ is the set of all divisors endowed with the group structure induced by that of the additive group structure of $\mathbb{Z}$, that is, for $D = \sum a_i [Y_i]$, $D' = \sum a'_i [Y_i]$, $(D+D')$ is given by $(D+D') = \sum (a_i + a'_i) [Y_i]$.
\end{definition}

We will assume that the sums in the above definition are locally finite, i.e., for any $x \in X$, there is an open neighborhood $U$ of $x$ in $X$ such that only a finite number of the $a_i$ are nonzero, with $Y_i \cap U \neq \emptyset$.

\begin{example}\label{divisorhypersurface}
Every hypersurface $V$ defines a divisor $\sum [V_i] \in Div(X)$, where the $V_i$ are its irreducible components.
\end{example}

\begin{definition}
A divisor $D = \sum a_i [Y_i]$ is said to be \textit{effective} if $a_i \ge 0$ for all $i$.
\end{definition}

The divisor associated to a hypersurface in the example is an effective divisor, since $a_i = 1$ for all $i$.

We now describe a way to associate to each meromorphic function on $X$ a divisor:

\begin{definition}
Let $Y$ be a hypersurface, and $f$ be a meromorphic function in a neighborhood of $x \in Y$. We define the \textit{order} of $f$ with respect to $Y$ in $x$ to be $ord_{Y,x}(f)$ as the number given by the equality

\begin{equation}
    f = g^{ord_{Y,x}(f)}h
\end{equation}
for some $h \in \mathcal{O}_{X,x}^{*}$.
\end{definition}

The above definition \textbf{does not} depend on $g$, so long as $g$ is irreducible (cf. \cite{huybrechts}, Remarks 2.3.6i). We define the order $ord_Y(f)$ of some $f \in K(X)$ along a hypersurface $Y$ as $ord_Y(f) = ord_{Y,x}(f)$ on some point $x \in Y$ such that there is a neighborhood $U$ of $x$ in $Y$ where $Y$ is irreducible (cf. \cite{huybrechts}, Remarks 2.3.6ii).

A meromorphic function $f \in K(X)$ has zeros (respec. poles) of order $d \ge 0$ along an irreducible hypersurface $Y$ of $X$ if

\begin{equation}
    ord_Y(f) = d \text{ (respec. if $ord_Y(f) = -d$).}
\end{equation}
Indeed, this definition is similar to the classical definition of zeros and poles of meromorphic functions on the complex plane, except that now we're studying zeros and poles along divisors. Much of the motivation and intuition for this definition can be gained from classical function theory, cf. \cite{ahlfors}, \cite{conway}.

Similarly, given a line bundle $L$ and a global section $s \in \mathrm{H}^0(X, L)$ such that $s \neq 0$, we can define the order of $s$ along $Y$, $ord_Y(s)$. It is defined as $ord_Y(\varphi_i(s))$, for some trivialization $\varphi_i$. This definition does not depend on the choice of trivialization.

\begin{definition}
Let $f \in K(X)$. Then we associate to $f$ a divisor $(f) \in Div(X)$ given by

\begin{equation}
    (f) = \sum ord_Y(f)[Y]
\end{equation}
where the sum is taken over all of the \textit{irreducible hypersurfaces} $Y$ of $X$.
\end{definition}

A divisor $D$ associated to a meromorphic function is called \textit{principal}.

A principal divisor $(f)$ can always be written as a difference of two effective divisors: if we define

\begin{equation}
    Z(f) = \sum_{ord_Y(f) > 0}ord_Y(f)[Y]
\end{equation}

and

\begin{equation}
    P(f) = - \sum_{ord_Y(f) < 0}ord_Y(f)[Y].
\end{equation}
the \textit{zero divisor} and the \textit{pole divisor} of $f$, then $Z(f)$ and $P(f)$ are principak and $(f) = Z(f) - P(f)$.

One may notice that throughout our construction, we relied on the existance of globally defined meromorphic functions, and the existence of a function field $K(X)$. One may generalize this constructions to a more general setting where we do not have such objects available. We may think about locally defined meromorphic functions, and consider nowhere vanishing holomorphic functions as invertible meromorphic functions, i.e., an inclusion of sheaves $\mathcal{O}_X^{*} \rightarrow \mathcal{K}_X^{*}$. We then have an isomorphism between the global sections of the sheaf $\mathcal{K}_X^{*}/\mathcal{O}_X^{*}$,  $\mathrm{H}^0(X, \mathcal{K}_X^{*}/\mathcal{O}_X^{*}) = \Gamma(X, \mathcal{K}_X^{*}/\mathcal{O}_X^{*})$ and $Div(X)$ (\cite{huybrechts}, Proposition 2.3.9).\footnote{In a context of schemes, we may not even have ordinary divisors (i.e. the originally defined \textbf{Weil divisors}). On a general scheme in which we do not have function fields (for example, a scheme which is not \textit{integral}) one considers the notion of \textbf{Cartier divisor} as a global section of the subsheaf of invertible elements of the \textit{sheaf of total quotients of the structure sheaf $\mathcal{O}_X$}, which in replaces the notion of function field for general schemes (cf. \cite{hartshorne}, II, 6). This proposition guarantees that, on a complex manifold, Cartier divisors and Weil divisors agree.} The isomorphism is given as follows: we associate to each meromorphic function its divisor. A global section $f \in \mathrm{H}^0(X, \mathcal{K}_X^{*}/\mathcal{O}_X^{*}) = \Gamma(X, \mathcal{K}_X^{*}/\mathcal{O}_X^{*})$ is locally a nontrivial meromorphic function $f_i \in \mathcal{K}_{X}^{*}(U_i)$, where $U_i$ is a covering of $X$ such that $f_i f_{j}^{-1}$ is a holomorphic function without zeros. And so, for any hypersurface $Y$ with $Y \cap U_i \cap U_j \neq \emptyset$ we have a well defined order of $f$, since $ord_Y(f_i) = ord_Y(f_j)$. So for any $f \in \mathrm{H}^0(X, \mathcal{K}_X^{*}/\mathcal{O}_X^{*})$ we can compute $(f) = \sum ord_Y(f)[Y]$. The fact that it is a group isomorphism follows from the additivity of the order (cf. \cite{huybrechts}, Remarks 2.3.6iii) and the explicitation of an inverse (\cite{huybrechts}, Proposition 2.3.9).

\begin{proposition}\label{divpic}
There exists a natural group homomorphism $Div(x) \longrightarrow Pic(X)$, given by

\begin{equation}
    D \mapsto \mathcal{O}(D).
\end{equation}
\end{proposition}

The construction of $\mathcal{O}(D)$ is as follows: if $D \in Div(X)$ is corresponds to $f \in \mathrm{H}^0(X, \mathcal{K}_X^{*}/\mathcal{O}_X^{*})$, which is given by functions $f_i \in \mathcal{K}_{X}^{*}(U_i)$ for any open covering $\{U_i\}$, we define $\mathcal{O}(D)$ as the line bundle with transition functions $\varphi_{ij} = f_i f_{j}^{-1} \in \mathcal{O}^{*}_{X}(U_i \cap U_j) = \mathrm{H}^0(U_i \cap U_j, \mathcal{O}^{*}_{X})$. By construction, since $D$ is the line bundle associated to $f$, the pair $\{(U_i, \varphi_{ij})\}$ satisfies the cocycle condition, i.e., is an element of $\check{\mathrm{H}}^1(\{U_i\}, \mathcal{O}^{*}_{X})$, and so indeed $\mathcal{O}(D)$ is an element of $\mathrm{H}^1(X, \mathcal{O}^{*}_{X}) \cong Pic(X)$.

Given a nonzero section $0 \neq s \in \mathrm{H}^0(X, L)$, we can describe its zero divisor\footnote{It in fact encodes more information than the usual zero divisor of some holomorphic function $f$, cf. \cite{huybrechts}, Remarks 2.3.17} $Z(s)$: we have that $Z(s) = \sum ord_Y(s) [Y]$. With this, we state the following

\begin{proposition}\label{zerolocussection}
For any \textit{effective} divisor $D \in Div(X)$, there exists some section $0 \neq s \in \mathrm{H}^0(X, \mathcal{O}(D))$ with $Z(s) = D$.
\end{proposition}


\begin{proof}
Let $D \in Div(X)$. By the above discussion, since $\mathrm{H}^0(X, \mathcal{K}_X^{*}/\mathcal{O}_X^{*}) \cong Div(X)$, we can take $D$ to be given by $\{f_i \in \mathcal{K}^{*}_X(U_i)\}$. Since we are supposing $D$ to be effective, all of the $f_i$ are holomorphic functions, since they have no poles. So the $f_i$ are elements of $\mathcal{O}_X(U_i)$.

On the other hand, the line bundle $\mathcal{O}(D) \in Pic(X)$ is, by construction, the bundle given by the cocycle $\{g_{ij = }f_i f_j^{-1}\} \in \mathrm{H}^1(X, \mathcal{O}_X^{*})$.

The functions $f_i \in \mathcal{O}_X(U_i)$ then define an element of $\mathrm{H}^0(X, \mathcal{O}(D))$ by making $g_{ij}f_j = f_i$. Finally, $Z(s) \cap U_i = Z(\left. s \right|_{U_i})$, and $Z(\left. s \right|_{U_i}) = Z(f_i) = D \cap U_i$. Therefore $Z(s) = D$.
\end{proof}

\begin{definition}
Two divisors $D, D' \in Div(X)$ are called \textit{linearly equivalent}, denoted by $D \sim D'$ if their difference $D - D'$ is a principal divisor. 
\end{definition}

We finish this section discussing a few results concerning the projective space $\mathbb{P}^n$ which are proved as an application of the theory we just developed. The reference is \cite{griffithsharris}, Chapter 1, §3.\footnote{I spent so much time of my undergraduate life reading bits of Chapter 0 of Griffiths-Harris, that I never thought I'd eventually get to Chapter 1.}

\begin{theorem}\label{projdivpic}
Let $X \subset \mathbb{P}^n$ be a submanifold. Then,

\begin{equation}
    Pic(X) \cong Div(X)/\sim
\end{equation}
that is, the Picard group of $X$ is isomorphic to the divisor group of $X$, modulo linear equivalence.
\end{theorem}

This theorem states that, any line bundle $L$ on a projective submanifold is the image of a divisor through the map of Proposition \ref{divpic} (i.e., is of the form $\mathcal{O}(D)$ for some $D \in Div(X)$). Its proof can be found in \cite{griffithsharris}, Chapter 1, Section 2.

\medskip

\noindent\textbf{The Gysin Sequence:} We now quickly discuss a result which we will need for computations. Let $X$ be an oriented smooth manifold of dimension $n$, and let $Y$ be an oriented smooth submanifold of codimension $k$. Then, the inclusion

\begin{equation}
    i: Y \rightarrow X
\end{equation}
induces a map

\begin{equation}
    i_{*}: \mathrm{H}^p(Y, \mathbb{Z}) \rightarrow \mathrm{H}^{p+k}(X, \mathbb{Z})
\end{equation}
called the \textit{Gysin map}, which in turn gives us a long exact sequence in cohomology

\begin{equation}
\begin{tikzcd}
    \ldots \arrow{r} & \mathrm{H}^{p-r}(Y, \mathbb{Z}) \arrow{r} & \mathrm{H}^{p}(X, \mathbb{Z}) \arrow{r} & \mathrm{H}^{p}(X \setminus Y, \mathbb{Z}) \arrow{r} & \mathrm{H}^{p-r+1}(Y, \mathbb{Z}) \arrow{r} & \ldots
\end{tikzcd}
\end{equation}
called the \textbf{Gysin sequence}.

\begin{proposition}\label{projcohomology}
For the complex projective space $\mathbb{P}^n$, we have

\begin{equation}
    \mathrm{H}^p(\mathbb{P}^n, \mathbb{Z}) =
    \begin{cases}
    0, & \text{if $p$ is odd or $p > 2n$} \\
    \mathbb{Z}, & \text{if $p$ is even.}
    \end{cases}
\end{equation}
\end{proposition}

The idea of the proof is as follows: we proceed by induction on $n$. We know that we can identify $\mathbb{P}^{n-1}$ inside $\mathbb{P}^n$, by making the last coordinate be zero. Then $\mathbb{P}^n \setminus \mathbb{P}^{n-1}$ is the standard open $U_n$, given by the points with the last coordinate non-zero, and so it is isomorphic to $\mathbb{C}^n$ via the chart $\varphi_n$ on $\mathbb{P}^n$. We then apply the Gysin sequence\footnote{The real codimension of $\mathbb{P}^{n-1}$ in $\mathbb{P}^n$ is $2$.}
 
\begin{equation}
\begin{tikzcd}
    \ldots \arrow{r} & \mathrm{H}^{p-2}(\mathbb{P}^{n-1}, \mathbb{Z}) \arrow{r} & \mathrm{H}^{p}(\mathbb{P}, \mathbb{Z}) \arrow{r} & \mathrm{H}^{p}(\mathbb{C}^n, \mathbb{Z}) \arrow{r} & \mathrm{H}^{p-1}(\mathbb{P}^{n-1}, \mathbb{Z}) \arrow{r} & \ldots
\end{tikzcd}
\end{equation}

Since $\mathbb{C}^n$ is contractible, all of its cohomology groups are $0$ in positive degree, and $\mathrm{H}^{0}(\mathbb{C}^n, \mathbb{Z}) = \mathbb{Z}$. So the sequence allows us to conclude that

\begin{equation}
    \mathrm{H}^{i}(\mathbb{P}^{n}, \mathbb{Z}) \cong \mathrm{H}^{i-2}(\mathbb{P}^{n-1}, \mathbb{Z})
\end{equation}
for all $i > 0$. So now we divide in cases: if $i$ is even $i - 2$ is also even, and by the induction hypothesis 

\begin{equation}
    \mathrm{H}^{i}(\mathbb{P}^{n}, \mathbb{Z}) \cong \mathrm{H}^{i-2}(\mathbb{P}^{n-1}, \mathbb{Z}) \cong \mathbb{Z}.
\end{equation}
On the other hand, if $i$ is odd, $i - 2$ is also odd, and so

\begin{equation}
    \mathrm{H}^{i}(\mathbb{P}^{n}, \mathbb{Z}) \cong \mathrm{H}^{i-2}(\mathbb{P}^{n-1}, \mathbb{Z}) = 0.
\end{equation}

Finally, if $i = 0$, note that

\begin{equation}
    \mathrm{H}^{0}(\mathbb{P}^{n}, \mathbb{Z}) \cong \mathrm{H}^{0}(\mathbb{C}^{n}, \mathbb{Z})
\end{equation}
and the result follows.

This result can be used to prove the following:

\begin{proposition}\label{chernclassfundamental}
The first chern class $c_1(\mathcal{O}_{\mathbb{P}^n}(1))$ of the hyperplane bundle is a generator of $\mathrm{H}^{2}_{Sing}(\mathbb{P}^n, \mathbb{Z})$.
\end{proposition}

Its proof can be found in \cite{voisin_2002}, Theorem 7.14 or \cite{milnstas}, Section 14.

Next, we state a version of the \textbf{Hodge Decomposition Theorem}, which is the analogous of the decomposition on last section for cohomology:

\begin{theorem}
Let $X = \mathbb{P}^n$. Then we have\footnote{The theorem is in fact a statement about \textbf{compact Kähler manifolds}, such as the complex projective space $\mathbb{P}^n$.}

\begin{equation}
    \mathrm{H}^k_{dR}(X, \mathbb{C}) \cong \bigoplus_{p+q = k} \mathrm{H}^{p,q}_{\overline{\partial}}(X).
\end{equation}

Moreover, $\mathrm{H}^{p,q}_{\overline{\partial}}(X) = \overline{\mathrm{H}^{q, p}_{\overline{\partial}}(X)}$, where the space in the right hand side is the complex-conjugate of $\mathrm{H}^{q, p}_{\overline{\partial}}(X)$.
\end{theorem}

Its proof is complex analytical, and uses the theory of \textbf{harmonic forms} (cf. \cite{griffithsharris}, Chap. 0, Section 7). We use the Hodge Decomposition Theorem  to prove the following:

\begin{proposition}\label{cohdiffPn}
Let $X = \mathbb{P}^n$, and let $\Omega_{X}^{P}$ denote its sheaf of holomorphic $p$-forms. We have

\begin{equation}
    \mathrm{H}^q(X, \Omega_{X}^{p}) =
    \begin{cases}
    0, & p \neq q \\
    \mathbb{C}, & p=q.
    \end{cases}
\end{equation}
In particular, $\mathrm{H}^k(X, \mathcal{O}_{X}) = 0$, for $k \ge 1$.
\end{proposition}

Let's introduce some notation: we call

\begin{equation}
    h^{p,q}(X) = \dim \mathrm{H}^{p,q}_{\overline{\partial}}(X)
\end{equation}
the \textbf{Hodge numbers} of $X$. We also put $b_q(X)$ the rank off the torsion-free part of $\mathrm{H}^q_{Sing}(X, \mathbb{Z})$, in the case where it is finitely generated. The numbers $b_q(X)$ are called the \textbf{Betti numbers} of $X.$ In the case where $\mathrm{H}^q_{Sing}(X, \mathbb{Z})$ is not finitely generated, we put $b_q(X) = \infty$. The \textbf{Universal Coefficient Theorem} in algebraic topology tells us that for any field $k$ of characteristic $0$, the dimension of the vector space $\mathrm{H}^q_{Sing}(X, k)$ coincides with $b_q(X)$, whenever it is finite (see Example \ref{univcoeffthm}).

The Hodge Decomposition then tells us that

\begin{equation}\label{hodgenumberscommute}
    h^{p,q}(\mathbb{P}^n) = h^{q,p}(\mathbb{P}^n)
\end{equation}
and it also follows from the de Rham Theorem that

\begin{equation}\label{bettihodge}
    b_k(\mathbb{P}^n) = \sum_{p + q = k} h^{p,q}(\mathbb{P}^n).
\end{equation}

With this, we can conclude that

\begin{equation}
    \mathrm{H}^q(\mathbb{P}^n, \Omega_{\mathbb{P}^n}^{p}) = 0
\end{equation}
when $p + q$ is odd, since $\mathrm{H}^q(\mathbb{P}^n, \Omega_{\mathbb{P}^n}^{p}) \cong  \mathrm{H}^{p,q}_{\overline{\partial}}(\mathbb{P}^n)$ by the Dolbeault Theorem, and then Hodge Decomposition, the de Rham Theorem and Proposition \ref{projcohomology} give us the result. Now, since $\mathrm{H}^{2k}_{Sing}(\mathbb{P}^n, \mathbb{Z}) \cong \mathbb{Z}$, also by Proposition \ref{projcohomology}, whenever $p \neq k$ we get from Equation \ref{bettihodge} that

\begin{equation}
    b_{2k}(\mathbb{P}^n) \ge h^{p, 2k-p}(\mathbb{P}^n)  + h^{2k - p, p}(\mathbb{P}^n)  = 2 \cdot  h^{p, 2k-p}(\mathbb{P}^n) 
\end{equation}
by Equation \ref{hodgenumberscommute}. But since $b_{2k}(\mathbb{P}^n) = 1$, this implies that $h^{p, 2k-p} = h^{2k - p, p} = 0$. And so we conclude that

\begin{equation}
    h^{p, p}(\mathbb{P}^n) = 1
\end{equation}
and so $\mathrm{H}^{p,p}_{\overline{\partial}}(\mathbb{P}^n) \cong \mathrm{H}^{2p}(\mathbb{P}^n, \Omega_{\mathbb{P}^n}^{p}) \cong \mathrm{H}^{2p}_{dR}(\mathbb{P}^n, \mathbb{C}) \cong \mathbb{C}$.

\medskip

Now we turn back to the long exact sequence associated to the exponential sheaf sequence of $\mathbb{P}^n$. We have

\begin{equation}
\begin{tikzcd}
    \mathrm{H}^1(X, \mathcal{O}_{X}) \arrow{r} & \mathrm{H}^1(X, \mathcal{O}_{X}^{*}) \arrow{r} & \mathrm{H}^2(X, \underbar{$\mathbb{Z}$}) \arrow{r} & \mathrm{H}^2(X, \mathcal{O}_{X})
\end{tikzcd}
\end{equation}
and so since $\mathrm{H}^k(X, \mathcal{O}_{X}) = 0$, for $k \ge 1$, $Pic(X) \cong \mathrm{H}^1(X, \mathcal{O}_{X}^{*}) \cong \mathrm{H}^2(X, \underbar{$\mathbb{Z}$})$ and \textbf{every line bundle $L$ on $\mathbb{P}^n$ is determined by its chern class $c_1(L)$}. Moreover, Proposition \ref{chernclassfundamental} yields that

\begin{theorem}\label{lbprojspace}
Every holomorphic line bundle on $\mathbb{P}^n$ is a multiple (i.e., a tensor power) of the hyperplane bundle $\mathcal{O}(1)$.
\end{theorem}

We now state and prove a special case of the last theorem in this section, beginning with a definition:

\begin{definition}
A \textbf{complex projective algebraic variety} is the zero locus in $\mathbb{P}^n$ of a collection of homogeneous polynomials $\{F_{\alpha}(X_0,...,X_n)\}$.
\end{definition}

Now, lets devote our attention to describing the global sections of $\mathcal{O}(d)$ for any $d$, i.e., lets describe the group $\mathrm{H}^0(\mathbb{P}^n, \mathcal{O}(d))$:

\begin{proposition}\label{projh0}
For any $d \ge 0$, $\mathrm{H}^0(\mathbb{P}^n, \mathcal{O}(d))$ is canonically isomorphic to the space $\mathbb{C}[x_0,...,x_n]_d$ of all homogeneous polynomials of degree $d$.
\end{proposition}

\begin{proof}
    We start by noting that an inclusion of bundles $E \subset F$ induces an inclusion on all tensor powers, and therefore 
    
    \begin{equation}
        \mathcal{O}(-d) \subset \mathbb{P}^n \times (\mathbb{C}^{n+1})^{\otimes d}.
    \end{equation}
     Now we note that any form $s \in \mathbb{C}[x_0,...,x_n]_d$ induces a $\mathbb{C}$-linear map $(\mathbb{C}^{n+1})^{\otimes d} \rightarrow \mathbb{C}$, given by evaluation. This gives a map
     
     \begin{equation}
     \Tilde{s}: \mathbb{P}^n \rightarrow \mathbb{P}^n \times ((\mathbb{C}^{n+1})^{\otimes d})^{\lor}
     \end{equation}
     which in turn gives a map, that we will also call $s$:
     
     \begin{equation}
        s: \mathbb{P}^n \times (\mathbb{C}^{n+1})^{\otimes d} \rightarrow \mathbb{C}
     \end{equation}
     linear over the fibers of  $\mathbb{P}^n \times (\mathbb{C}^{n+1})^{\otimes d} \rightarrow \mathbb{P}^n$. The restriction of this map $s$ to $\mathcal{O}(-d) \subset \mathbb{P}^n \times (\mathbb{C}^{n+1})^{\otimes d} \rightarrow \mathbb{C}$ (see the construction in Example \ref{tautbundle}) is linear over each fiber, and the composititon of the map $\Tilde{s}$ with the projection $\mathbb{P}^n \times (\mathbb{C}^{n+1})^{\otimes d} \rightarrow \mathbb{P}^n$ is the identity. Therefore, we get a holomorphic section of $\mathcal{O}(d)$. This associates to any homogeneous polynomial degree $d$ a global holomorphic section of $\mathcal{O}(d)$.
     
     Moreover, note that if $0 \neq s \in \mathbb{C}[x_0,...,x_n]_d$, then the associated section is not trivial, since the compositition 
     \begin{equation}
     \mathcal{O}(-1) \longrightarrow \mathbb{P}^n \times (\mathbb{C}^{n+1}) \longrightarrow \mathbb{C}^{n+1}
     \end{equation}
     where the first map is the inclusion and the second is the projection on the second coordinate, is surjective, and so any map inducing the trivial section of $\mathcal{O}(d)$ must induce a trivial polynomial function $\mathbb{C}^{n+1} \rightarrow \mathbb{C}$, and therefore must be trivial. 
     
     So we have a linear map of complex vector spaces  $\mathbb{C}[x_0,...,x_n]_d \rightarrow \mathrm{H}^0(\mathbb{P}^n, \mathcal{O}(d))$ mapping $s$ to $\Tilde{s}$, whose kernel is trivial. So to show that it is in fact and isomorphism, we are only left to show surjectivity.
     
    We first need to describe a way in which we can associate to a pair nonzero sections $s_1, s_2 \in \mathrm{H}^0(X, L)$ of some line bundle over a complex manifold $X$, a meromorphic function in $K(X)$.  For this, let $\bigcup U_i$ be a covering of $X$ by trivializing neighborhoods $(U_i, \varphi_i)$ of the line bundle $L$. In this way,
    
    \begin{equation}
        \varphi_i: \pi^{-1}(U_i) \rightarrow X \times \mathbb{C}
    \end{equation}
    are holomorphic isomorphisms. Therefore, the sections $\varphi_i(\left. s_1 \right|_{U_i})$ and $\varphi_i(\left. s_2 \right|_{U_i})$ are holomorphic functions on $U_i$. Define the meromorphic function $s_1/s_2$ associated to the pair of sections $s_1, s_2$ as
    
    \begin{equation}
        \frac{\varphi_i(\left. s_1 \right|_{U_i})}{\varphi_i(\left. s_2 \right|_{U_i})}.
    \end{equation}
    This definition is independent of the choice of trivializing neighborhoods $(U_i, \varphi_i)$ of $L$.
    
    Let $0 \neq t \in \mathrm{H}^0(\mathbb{P}^n, \mathcal{O}(d))$ be a section and let $s_0 \in \mathrm{H}^0(\mathbb{P}^n, \mathcal{O}(d))$ be a nontrivial section in the image of the map $\mathbb{C}[x_0,...,x_n]_d \rightarrow \mathrm{H}^0(\mathbb{P}^n, \mathcal{O}(d))$, that is, $s_0$ is a section induced by some homogeneous polynomial $f$ of degree $d$. Consider the meromorphic function $F = \frac{t}{s_0} \in K(\mathbb{P}^n)$, as defined above. Pulling back by the projection $\pi: \mathbb{C}^{n+1} \setminus \{0\} \rightarrow \mathbb{P}^n$, we get a meromorphic function $\Tilde{F} = F \circ \pi \in K(\mathbb{C}^{n+1} \setminus \{0\})$. Moreover, if we set $\Tilde{G} = f\Tilde{F}$, by the above construction of the section $s_0$ from $f$, we get a function $\Tilde{G}: \mathbb{C}^{n+1} \setminus \{0\} \rightarrow \mathbb{C}$ holomorphic everywhere in $\mathbb{C}^{n+1} \setminus \{0\}$. Then by Hartogs' Theorem (Theorem \ref{hartogs}), $\Tilde{G}$ extends to a holomorphic function $G: \mathbb{C}^{n+1} \rightarrow \mathbb{C}$. Now, by construction, $\Tilde{F}(\lambda x) = \Tilde{F}(x)$ and $f(\lambda x) = \lambda^{d}f(x)$ for all $x \in \mathbb{C}^{n+1}$, $\lambda \in \mathbb{C}$, and so
    
    \begin{equation}
        G(\lambda x) = \lambda^{d}G(x)
    \end{equation}
    for all $x \in \mathbb{C}^{n+1}$, that is, \textbf{$G$ is homogeneous of degree $d$}.
    
    So letting $y$ be a real variable and taking a line through the origin $i:y \mapsto (z_0y,...,z_{n}y) = zy$ in $\mathbb{C}^{n+1}$, the pullback $i^{*}G$ is either identically $0$ if $z_0=...=z_n=0$, or
    
    \begin{equation}\label{pullbackg}
        i^{*}G = zy^d
    \end{equation}
    and it has a zero of order $d$ at $y = 0$. Then, equation (\ref{pullbackg}) implies that the power series expansion of $G$ around the origin only has terms of degree $d$, and so $G$ is a homogeneous polynomial of degree $d$. Moreover, it follows from the construction of $F$ and the construction of the section associated to a homogeneous polynomial that the section $t$ is induced by the polynomial $G$, and this concludes the proof.
\end{proof}

Its important to note that given a nonzero section $0 \neq s \in \mathrm{H}^0(\mathbb{P}^n, \mathcal{O}(d))$, the zero divisor $Z(s)$ of $s$ is precisely the image in $\mathbb{P}^n$ of the zero locus of $F$ in $\mathbb{C}^{n+1}$, where $F$ is the homogeneous polynomial associated to the section $s$. Indeed, $Z(s)$ is given by $\sum ord_Y (\varphi_i(s))[Y]$ for any trivializations $\varphi_i$, and the order of the holomorphic function $\varphi_i(s)$ is $ord_Y (\varphi_i(s)) > 0$ precisely where $F$ vanishes.

This result allows us to prove a special case of

\begin{theorem}[Chow's  Theorem]
    Every analytic subvariety of $\mathbb{P}^n$ is algebraic.
\end{theorem}

\begin{proof}
We present a proof in the special case of hypersurfaces: let $V$ be an analytic hypersurface in $\mathbb{P}^n$, and let $D$ be the associated divisor, as in Example \ref{divisorhypersurface}. Consider the associated line bundle $\mathcal{O}(D)$. By Theorem \ref{lbprojspace}, $\mathcal{O}(D)$ is of the form $\mathcal{O}(d)$ for some $d \ge 0$. Moreover, by Proposition \ref{zerolocussection}, $V$ is the zero locus of some section $s \neq 0$ of $\mathcal{O}(D)$. But by Proposition \ref{projh0}, every nonzero section $s \in \mathrm{H}^0(\mathbb{P}^n, \mathcal{O}(D))$ is induced by a homogeneous polynomial $F$, and so we have

\begin{equation}
    V = Z(s) = \{F = 0\}
\end{equation}
and $V$ is algebraic.
\end{proof}

In the spirit of Proposition \ref{projh0} one can also compute $\mathrm{H}^q(\mathbb{P}^n, \mathcal{O}(d))$ for any $q > 0$, and conclude that

\begin{equation}
    \mathrm{H}^q(\mathbb{P}^n, \mathcal{O}(d)) =
    \begin{cases}
    \mathbb{C}[x_0,...,x_n]_d, & q = 0 \\
    0, & q \neq 0, q \neq n \\
    (\frac{1}{x_0...x_n}\mathbb{C}[\frac{1}{x_0},...,\frac{1}{x_n}])_d, & q = n.
    \end{cases}
\end{equation}
as an application of the \textbf{Kodaira Vanishing Theorem}, which is a statement about vanishing of cohomology of \textit{ample line bundles} on \textit{Kähler manifolds}. But maybe some other time...

\section{Sheaves and Sheaf Cohomology}\label{sheaves}

\subsection{Sheaves on Spaces}

Sheaves are among the most important tools developed in the 20th century for the study of topology and geometry. They are useful in situations when the passage of data from local to global is necessary. We give a concise revision of the theory, with the main reference being (\cite{hartshorne}, Chap. 2). Throughout, all sheaves and presheaves will be of abelian groups, unless otherwise stated.

\begin{definition}
Let $X$ be a topological space. Define the category $\textbf{Op(X)}$ in the following way: the objects are the open sets of $X$, and the morphisms between two open sets are given by the inclusion - that means that, for $U, V$ opens, $\Hom(U, V) = \{i\}$, with $i$ being the inclusion of $U$ in $V$ if $U \subset V$ and $\Hom(U, V) = \emptyset$ otherwise. A \textit{presheaf of abelian groups} $\mathscr{F}$ on $X$ is a functor $\textbf{Op(X)}^{\circ} \rightarrow \textbf{Ab}$. 
\end{definition}

We call $\mathscr{F}(U)$ or $\Gamma(U, \mathscr{F})$ the group of \textit{sections} of $\mathscr{F}$ over $U$. Given an inclusion $U \rightarrow V$, the induced map $\rho_{VU}: \mathscr{F}(V) \rightarrow \mathscr{F}(U)$ is called the \textit{restriction map} from $V$ to $U$. Sometimes, if $s \in \mathscr{F}(V)$ and $U \subseteq V$, we denote $\rho_{VU}(s)$ by the usual restriction notation, namely $\left. s \right|_U$. Because of the functoriality of $\mathscr{F}$, it is subject to
\begin{enumerate}
    \item $\rho_{UU} \in \Hom(U, U)$ is the identity map;
    \item $W \subseteq U \subseteq V$ implies that $\rho_{VW} = \rho_{UW} \circ \rho_{VU}$.
\end{enumerate}

\begin{definition}
A \textit{sheaf} on $X$, is a presheaf satisfying the two extra conditions:
\begin{enumerate}
    \item if $U \in \textbf{Op(X)}$, $\{V_i\}$ is an open covering of $U$ and $s \in \mathscr{F}(U)$ is an element such that $\rho_{UV_{i}}(s) = \left. s \right|_{V_{i}} = 0$ for all $i$, then $s = 0$. This condition can also be formulated as follows: if $\{V_i\}$ is an open covering of $U$, $s_1, s_2 \in \mathscr{F}(U)$ and $\rho_{UV_{i}}(s_1) = \rho_{UV_{i}}(s_2)$ for all $i$, then $s_1 = s_2$. This is called the \textbf{identity axiom}. A presheaf satisfying the identity axiom is called a \textit{separated presheaf}.
    \item if $U \in \textbf{Op(X)}$, $\{V_i\}$ is an open covering of $U$ and we have elements $s_i \in \mathscr{F}(V_i)$ for each $i$ with the property that for each $i, j$, $\rho_{UV_{i} \cap V_{j}}(s_i) = \rho_{UV_{i} \cap V_{j}}(s_j)$, then there is an element $s \in \mathscr{F}(U)$ such that $\rho_{UV_{i}}(s) = s_i$ for each $i$. This is called the \textbf{gluability axiom}.
\end{enumerate}
\end{definition}

These two axioms together guarantee that sheaves can be glued along intersections, and that such a glueing is unique. They can be cleanly expressed in terms of a short exact sequence:

\begin{equation}
\begin{tikzcd}
    0 \ar{r} & \mathscr{F}(U) \ar{r} & \prod \mathscr{F}(V_i) \ar[r,shift left=.75ex]
  \ar[r,shift right=.75ex,swap] & \prod \mathscr{F}(V_i \cap V_j)
\end{tikzcd}
\end{equation}
where exactness at $\mathscr{F}(U)$ implies on the validity of the identity axiom, and exactness at $\prod \mathscr{F}(V_i)$ implies on the gluability axiom.

\begin{definition}
The \textit{stalk} of a presheaf at $p \in X$ is 
\begin{equation}
   \mathscr{F}_p = \varinjlim_{p \in U} \mathscr{F}(U).
\end{equation}
Alternatively, it can be explicitly described in a \textit{germ}-like manner: the stalk of $\mathscr{F}$ at $p \in X$ is
\begin{equation}
    \mathscr{F}_p = \{(f, U): U \in \textbf{Op(X)}, p \in U, f \in \mathscr{F}(U)\}
\end{equation}
modulo the relation that $(f, U) \sim (g, V)$ if there exists some $W \subset U, V$ with $p \in W$ and $\rho_{UW}(f) = \rho_{VW}(g)$.
\end{definition}

\begin{definition}
A \textit{morphism of presheaves} $\varphi: \mathscr{F} \rightarrow \mathscr{G}$ on $X$ is a morphism of $\mathscr{F}$ and $\mathscr{G}$ as functors - namely, a \textit{natural transformation} of functors. Explicitly, for $U \subseteq V$ opens, it is a commutative diagram of abelian groups 
\end{definition}

\begin{center}
\begin{tikzpicture}[>=triangle 60]
\matrix[matrix of math nodes,column sep={60pt,between origins},row
sep={60pt,between origins},nodes={asymmetrical rectangle}] (s)
{
&|[name=A]| \mathscr{F}(V) &|[name=B]| \mathscr{G}(V) \\
&|[name=A']| \mathscr{F}(U) &|[name=B']| \mathscr{G}(U) \\
};
\draw[overlay,->, font=\scriptsize,>=latex] 
          (A) edge node[auto] {\(\varphi(V)\)} (B)
          (A) edge node[auto] {\(\rho_{VU}\)} (A')
          (B) edge node[auto] {\(\rho'_{VU}\)} (B')
          (A') edge node[auto] {\(\varphi(U)\)} (B')

;

;

\end{tikzpicture}
\end{center} 

We then have a category $\textbf{PSh(X)}$ of presheaves of abelian groups on $X$. If $\mathscr{F}$ and $\mathscr{G}$ are sheaves, we use the same definition of morphism between $\mathscr{F}$ and $\mathscr{G}$. The category of sheaves of abelian groups on $X$ is denoted $\textbf{Sh(X)}$. It is a full subcategory of $\textbf{PSh(X)}$. An isomorphism of (pre)sheaves is a morphism that has a two-sided inverse. The next proposition is a strong characterisation of isomorphisms of sheaves, and it very well illustrates the local nature of the object. It is \textbf{not true} for presheaves. Note that if $\varphi: \mathscr{F} \rightarrow \mathscr{G}$ is a morphism of sheaves, it induces a a morphism $\varphi_p: \mathscr{F}_p \rightarrow \mathscr{G}_p$ on the stalks, for any $p \in X$.

\begin{proposition}\label{isostalks}
Let $\varphi: \mathscr{F} \rightarrow \mathscr{G}$ be a morphism of sheaves on $X$. Then $\varphi$ is an isomorphism if, and only if the induced map on the stalk $\varphi_p: \mathscr{F}_p \rightarrow \mathscr{G}_p$ is an isomorphism for all $p \in X$.
\end{proposition}

\begin{proof}
The proof needs the identity axiom to show injectivity, and the gluability axiom to show surjectivity. This is the reason why the proposition is not valid for presheaves in general (\cite{hartshorne}, Chap. 2, Prop. 1.1).
\end{proof}

It would be interesting for our cohomological purposes if the category $\mathbf{Sh(X)}$ was an abelian category. For this to be true, we need to start by defining kernels, cokernels and images:

\begin{definition}
Let $\mathscr{F}, \mathscr{G} \in \mathbf{PSh(X)}$, and let $\varphi \in \Hom (\mathscr{F}, \mathscr{G})$. We define the \textit{presheaf cokernel} of $\varphi$, \textit{presheaf kernel} of $\varphi$ and the \textit{presheaf image} of $\varphi$ to be the presheaves given by $U \mapsto \ker \varphi(U)$, $U \mapsto \coker \varphi(U)$ and $U \mapsto \Ima \varphi(U)$ respectively, in the sense of abelian groups.
\end{definition}

\begin{remark}\label{sheafexactness}
A map $\varphi: \mathscr{F} \rightarrow \mathscr{G}$ is a monomorphism (respec. epimorphism, isomorphism) in the category of sheaves if, and only if, the map $\varphi_{x}: \mathscr{F}_{x} \rightarrow \mathscr{G}_{x}$ is injective (respec. surjective, isomorphism), for all $x \in X$. Moreover, if the map $\mathscr{F}_{x} \rightarrow \mathscr{G}_{x}$ is injective for all $x \in X$, then $\mathscr{F}(X) \rightarrow \mathscr{G}(X)$ is injective. This is in general \textbf{not} true for surjections. That means that we can have a \textbf{epimorphism of sheaves} which is \textbf{not surjective on sections}. This phenomenon will be further examinated on our discussion about the exactness of the \textbf{global section functor}.
\end{remark}

We note however that if $\mathscr{F}, \mathscr{G}$ are sheaves, the presheaf cokernel and the presheaf image are in general \textbf{not} sheaves. This motivates the following discussion:

\textbf{Sheafification of a presheaf:} Given a presheaf $\mathscr{F}$, we would like to associate to it a sheaf, with the universal property that any sheaf map factors through it. Moreover, if $\mathscr{F}$ is already a sheaf, we would like that this associated sheaf be $\mathscr{F}$ itself. There are several ways to describe this association. We keep following \cite{hartshorne}:

\begin{proposition}
Given a presheaf $\mathscr{F}$, there is a sheaf $\mathscr{F}^{+}$ and a morphism $\theta: \mathscr{F} \rightarrow \mathscr{F}^{+}$, with the property that for any sheaf $\mathscr{G}$ and any morphism $\varphi: \mathscr{F} \rightarrow \mathscr{G}$, there is a unique morphism $\psi: \mathscr{F}^{+} \rightarrow \mathscr{G}$ such that $\varphi = \psi \circ \theta$. Furthermore, the pair $(\mathscr{F}^{+}, \theta)$ is unique up to unique isomorphism. $\mathscr{F}^{+}$ is called the \textit{sheafification} of the presheaf $\mathscr{F}$.
\end{proposition}

\begin{proof}
We construct the sheaf $\mathscr{F}^{+}$ as follows: for any open set $U$, let $\mathscr{F}^{+}(U)$ be the set of functions $s$ from $U$ to the union $\bigcup_{p \in U} \mathscr{F}_p$ of the stalks of $\mathscr{F}$ over all the points of $U$, i.e., $\mathscr{F}^{+} = \{s: U \rightarrow \bigcup_{p \in U} \mathscr{F}_p\} = \prod_{p \in U} \mathscr{F}_p$ such that
\begin{enumerate}
    \item for each $p \in U$, $s(p) \in \mathscr{F}_p$
    \item for each $p \in U$, there is a neighborhood $V$ of $p$ contained in $U$, and an element $t \in \mathscr{F}(V)$, such that for every point $q \in V$, the equivalence class $t_q$ of $t$ in $\mathscr{F}_q$ is equal to $s(q)$.
\end{enumerate}
One now has to show that this is a sheaf with the desired properties. The best proof of this, using this approach that the author has ever came across can be found in (\cite{stacks-project}, section 6.17).
\end{proof}

We can then take the \textit{sheaf cokernel} and \textit{sheaf image} to be the sheafification associated to the presheaf cokernel and presheaf image, respectively.

\begin{definition}
A \textit{subsheaf} $\mathscr{F}'$ of a sheaf $\mathscr{F}$ is a sheaf such that for every open set $U$ of $X$, the section $\mathscr{F}'(U)$ is a subgroup of $\mathscr{F}(U)$, with the restriction morphisms being those induced by $\mathscr{F}$. Given a subsheaf $\mathscr{F}'$ of a sheaf $\mathscr{F}$, the \textit{quotient sheaf} $\mathscr{F}/\mathscr{F}'$ is the sheafification of the presheaf defined by the rule $U \mapsto \mathscr{F}(U)/\mathscr{F}'(U)$.
\end{definition}

\begin{definition}
If $\varphi: \mathscr{F} \rightarrow \mathscr{G}$ is a morphism of sheaves, we say that it is \textit{injective} if the \textit{sheaf kernel} $\ker \varphi = 0$. Since sheafification was not needed for the sheaf kernel, a map $\varphi: \mathscr{F} \rightarrow \mathscr{G}$ is injective if and only if the induced map $\varphi(U): \mathscr{F}(U) \rightarrow \mathscr{G}(U)$ is injective for every $U \subseteq X$ open. We say that a morphism $\varphi: \mathscr{F} \rightarrow \mathscr{G}$ is \textit{surjective} if the \textit{sheaf image} (i.e. the sheafification of the presheaf image) $\Ima \varphi = \mathscr{G}$. Remark \ref{sheafexactness} tells us that monomorphisms (respec. epimorphisms) of sheaves are precisely injective (respec. surjective) maps of sheaves.

\begin{equation}
\begin{tikzcd}
   {\ldots} \ar{r} & \mathscr{F}^{i-1} \ar{r}{\varphi_{i-1}} & \mathscr{F}^{i} \ar{r}{\varphi_i} & \mathscr{F}^{i+1} \ar{r} & {\ldots}
\end{tikzcd}
\end{equation}
is \textit{exact} at $\mathscr{F}^i$ if $\ker \varphi_i = \Ima \varphi_{i-1}$. A sequence

\begin{equation}
\begin{tikzcd}
   0 \ar{r} & \mathscr{F} \ar{r}{f} & \mathscr{G}
\end{tikzcd}
\end{equation}
is exact if and only if $f$ is injective. A sequence

\begin{equation}
\begin{tikzcd}
   \mathscr{F} \ar{r}{g} & \mathscr{G} \ar{r} & 0
\end{tikzcd}
\end{equation}
is exact if and only if $g$ is surjective.

\end{definition}


Here are some examples of sheaves:

\begin{example}
Let $M$ be a complex manifold, and let $U \subseteq M$ be an open set. We have $\mathcal{O}(U)$ the sheaf of \textit{holomorphic functions} on $U$, $\Omega^{p}(U)$ the sheaf of \textit{holomorphic p-forms} on $U$, $\mathcal{A}^{p,q}(U)$ the sheaf of $C^{\infty}$ \textit{forms of type (p,q)} on $U$ and $\mathcal{A}_{\Bar{\partial}}^{p,q}(U)$ the sheaf of \textit{$\Bar{\partial}$-closed $C^{\infty}$ forms of type (p,q)} on $U$. We also have the sheaf $\mathcal{K}$ of meromorphic functions on $M$. Remember that a meromorphic function is given locally as the quotient of two holomorphic functions: for some covering $\{U_i\}$ of $U \subseteq M$, $\left. f \right|_{U_i} = \frac{g_i}{h_i}$, for $g_i$ and $h_i$ relatively prime in $\mathcal{O}(U_i)$ and $g_{i}h_{i} = g_{j}h_{j}$ in $\mathcal{O}(U_i \cap U_j)$, for the sheaf $\mathcal{O}$ defined above. A meromorphic function is only \textit{partially defined} - i.e. it is undefined in points where $g_i = h_i = 0$.
\end{example}

\begin{example}\label{structsheaf}
By an \textit{algebraic variety} $X$, we will always mean an \textit{integral, separated scheme of finite-type}, over a field $k$ (which, in our case, can always be taken to be the field of complex numbers $\mathbb{C}$). Notice in particular that varieties will \textbf{always} be \textit{irreducible}. In this definition, the variety is equipped with a canonical sheaf of rings - its \textit{structure sheaf} $\mathcal{O}_X$. It generalizes the classical notion of \textit{ring of regular functions} on an algebraic variety. Its sections $\mathcal{O}_X(U)$ are \textit{reduced} (i.e., \textit{nilpotent-free}) \textit{commutative $k$-algebras} with unity, for all Zariski open subsets $U \subseteq X$. Its stalk at $p \in X$, denoted $\mathcal{O}_{X,p}$, is a local ring.
\end{example}

\begin{example}\label{skyscrapper}
Let $X$ be a topological space, $p \in X$ and let $S$ be a set. Consider the \textit{sheaf of sets} defined as

    \[ \begin{cases} 
          S & p \in U \\
          \{*\} & p \notin U
       \end{cases}
    \]
where $\{*\}$ is any one element set. This is indeed a sheaf, and is called the \textit{skyscrapper sheaf}. Given an inclusion $V \rightarrow U$, its restriction maps are $\rho_{UV} = Id$, if $p \in V$, and $\rho_{UV}(s) = *$ for all $s 
\in V$, if $p \notin V$. This can be generalized to a sheaf in any category having a final object, and its restriction maps are always the identity, and the maps to said final object. For example, in \textbf{Ab}, the one element set would be replaced by $\{0\}$, and the restriction maps would be the map to $0$, if $p \notin V$.
\end{example}

\begin{example}
Let X be a topological space, and $S$ be a set. Define $\underbar{S}_{pre}(U) = S$ for all open sets $U$. This is a presheaf, called the constant presheaf associated to $S$. It is not, in general, a sheaf: suppose $S$ has more than one element, and $X$ is a two-point discrete space. Then $\underbar{S}_{pre}$ fails to satisfy the gluability axiom, even if we set $\underbar{S}_{pre}(\emptyset) = \{*\}$.

A more interesting object is the \textit{sheafification} of the constant presheaf $\underbar{S}_{pre}$, the \textit{constant sheaf} $\underbar{S}$. An explicit description of the constant sheaf is as follows: endow $S$ with the discrete topology, and consider $\underbar{S}_{pre}(U)$ to be the maps $U \rightarrow S$ that are continuous (i.e. locally constant in the general case).
\end{example}

\subsection{Sheaf Cohomology}


Our first approach to sheaf cohomology is largely based on the following observation, which can be found in (\cite{godement}, Chap. 5, §5.10, Théorème 5.10.1):

\begin{theorem}\label{paracompact}
Let $X$ be a paracompact topological space, $\mathscr{A}$ a sheaf on $X$, and $\mathcal{U}$ be a cover of $X$. Then, there are canonical isomorphisms
\begin{equation}
    \check{\mathrm{H}}^n(\mathcal{U},\mathscr{A}) \rightarrow {\mathrm{H}}^n(X,\mathscr{A}).
\end{equation}

\end{theorem}

The theorem states that if our space is paracompact, \v{C}ech theory and derived-functor theory of Grothendieck coincide (notice that there are no hypotheses on the sheaf $\mathscr{A}$ - this is true for \textbf{any} sheaf on $X$). Since smooth manifolds are always paracompact, we first devote our attention to the \v{C}ech cohomology theory. The main references for this section are (\cite{FAC}, I, §3, \cite{griffithsharris}, Chap. 0, §3, \cite{vakil}, Chap. 18, §18.2, \cite{gallierquaintance}, Chap. 10).

We need to start by defining the \textit{\v{C}ech complex} of a cover $\mathcal{U}$. Let $X$ be a topological space, $\mathscr{F}$ a sheaf on $X$ and $\mathcal{U} = \{U_{\alpha}\}_{\alpha \in J}$ an open cover. Denote $\mathcal{U}_I = \bigcap_{\alpha \in I} U_{\alpha}$. Define

\begin{align*}
    C^{0}(\mathcal{U}, \mathscr{F}) &= \prod_{|I| = 1} \mathscr{F}(\mathcal{U}_{I})\\
    C^{1}(\mathcal{U}, \mathscr{F}) &= \prod_{|I| = 2} \mathscr{F}(\mathcal{U}_{I})\\
    &\;\;\vdots \notag \\
    C^{i}(\mathcal{U}, \mathscr{F}) &= \prod_{|I| = i+1} \mathscr{F}(\mathcal{U}_{I}),\\
    &\;\;\vdots \notag \\
\end{align*}
that is, for $I \subset J$ finite sets, with $|I| = i+1$, we have

\begin{equation}
    C^{i}(\mathcal{U}, \mathscr{F}) = \prod_{|I| = i+1} \mathscr{F}(U_{\alpha_{0}} \cap ... \cap U_{\alpha_{i}}).
\end{equation}

An element $\sigma_{\alpha_{0},...,\alpha_{i}} \in C^{i}(\mathcal{U}, \mathscr{F})$ is an association which takes every subset of $i+1$ indices of $J$ to an element $s \in \mathscr{F}(U_{\alpha_{0}} \cap ... \cap U_{\alpha_{i}})$. Such an element is called an \textit{i-cochain} of $\mathscr{F}$. To make our sequence of groups into a complex, we need to define the \textit{coboundary maps}. Given an inclusion

\begin{equation}
    U_{\alpha_{0}} \cap ... \cap U_{\alpha_{i+1}} \rightarrow U_{\alpha_{0}} \cap ...\cap \hat{U}_{\alpha_j} \cap ... \cap U_{\alpha_{i+1}}
\end{equation}
we get an induced map (i.e., a restriction morphism)

\begin{equation}
    \mathscr{F}(U_{\alpha_{0}} \cap ...\cap \hat{U}_{\alpha_j} \cap ... \cap U_{\alpha_{i+1}}) \rightarrow \mathscr{F}(U_{\alpha_{0}} \cap ... \cap U_{\alpha_{i+1}})
\end{equation}

We define the map $\delta$ as being the alternating sum of such morphisms:

\begin{equation}
    \delta: C^{i}(\mathcal{U}, \mathscr{F}) \rightarrow C^{i+1}(\mathcal{U}, \mathscr{F})
\end{equation}
if $\sigma \in C^{i}(\mathcal{U}, \mathscr{F})$ is an \textit{i-cochain}, then

\begin{equation}
    (\delta\sigma)_{\alpha_{0},...,\alpha_{i+1}} = \sum_{j=0}^{i+1} (-1)^{j} \left. \sigma_{\alpha_{0},...,\hat{\alpha}_{j},...,\alpha_{i+1}} \right|_{U_{\alpha_{0}} \cap ... \cap U_{\alpha_{i+1}}}
\end{equation}
where the restriction morphism is denoted here as the usual restriction of maps. To check that the coboundary map $\delta$ in fact makes our sequence into a complex, we need to check that the composition of two consecutive mappings is zero: let $\sigma \in C^{i}(\mathcal{U}, \mathscr{F})$. Making the following computation, using the definition above

\begin{align*}
    (\delta^{2}\sigma)_{\alpha_{0},...,\alpha_{i+2}} &= \sum_{j=0}^{i+2} (-1)^{j} \left. (\delta\sigma)_{\alpha_{0},...,\hat{\alpha}_{j},...,\alpha_{i+2}} \right|_{U_{\alpha_{0}} \cap ... \cap U_{\alpha_{i+2}}} \\
    &= \sum_{j=0}^{i+2} (-1)^{j} (\sum_{k=0}^{i+1} (-1)^{k}\left. \sigma_{\alpha_{0},...,\hat{\alpha}_{j},...,\hat{\alpha}_{k},...,\alpha_{i+2}}) \right|_{U_{\alpha_{0}} \cap ... \cap U_{\alpha_{i+2}}} \\
    &= \sum_{j=0}^{i+2}  \sum_{k=0}^{i+1} (-1)^{j+k} (\left.\sigma_{\alpha_{0},...,\hat{\alpha}_{j},...,\hat{\alpha}_{k},...,\alpha_{i+2}}) \right|_{U_{\alpha_{0}} \cap ... \cap U_{\alpha_{i+2}}}
\end{align*}

Fixing $\alpha_m < \alpha_n$, note that the terms

\begin{equation}
(\left.\sigma_{\alpha_{0},...,\hat{\alpha}_{n},...,\hat{\alpha}_{m},...,\alpha_{i+1}}) \right|_{U_{\alpha_{0}} \cap ... \cap U_{\alpha_{i+2}}}
\end{equation}
and

\begin{equation}
(\left.\sigma_{\alpha_{0},...,\hat{\alpha}_{m},...,\hat{\alpha}_{n-1},...,\alpha_{i+1}}) \right|_{U_{\alpha_{0}} \cap ... \cap U_{\alpha_{i+2}}}
\end{equation}
appear one time each on the last sum, with opposite signs and so they cancel each other out. Since the fixed pair of indices was arbitrary, we get that
\begin{equation}
    (\delta^{2}\sigma) = 0
\end{equation}
and so $C^{\bullet}(\mathcal{U}, \mathscr{F})$ is indeed a complex.

We can now define the \textit{\v{C}ech cohomology groups of the cover $\mathcal{U}$ with values in $\mathscr{F}$} as the cohomology groups of this \v{C}ech complex. More precisely, denote by $\delta^{p} = \left.\delta \right|_{C^{p}(\mathcal{U}, \mathscr{F})}$. Then the group $B^{p}(\mathcal{U}, \mathscr{F}) = \Ima \delta^{p-1}$ is called the \textit{group of \v{C}ech p-coboundaries}, and the group $Z^{p}(\mathcal{U}, \mathscr{F}) = \ker \delta^{p}$ is called the \textit{group of \v{C}ech p-cocycles}. The \textit{$p^{th}$-\v{C}ech cohomology group of the cover $\mathcal{U}$ with values in $\mathscr{F}$} is then defined as the group of \v{C}ech p-cocycles modulo the {group of \v{C}ech p-coboundaries}

\begin{equation}
    \check{\mathrm{H}}^p(\mathcal{U},\mathscr{F}) = \frac{Z^{p}(\mathcal{U}, \mathscr{F})}{B^{p}(\mathcal{U}, \mathscr{F})}.
\end{equation}

\begin{remark}
Some authors develop the theory for \textit{alternated cochains}, which form a subgroup $C'^{p}(\mathcal{U}, \mathscr{F}) \subset C^{p}(\mathcal{U}, \mathscr{F})$, for all $p$, and subsequently form a complex $C'^{\bullet}(\mathcal{U}, \mathscr{F})$. However, one shows that the inclusion $C'^{p}(\mathcal{U}, \mathscr{F}) \rightarrow C^{p}(\mathcal{U}, \mathscr{F})$ induces an isomorphism between the two cohomology theories of these complexes, for all $p$.
\end{remark}

The first thing that we will prove is that the group $\check{\mathrm{H}}^0(\mathcal{U},\mathscr{F})$ can be identified with the global section of $\mathscr{F}$, independently of the choice of cover $\mathcal{U}$:

\begin{proposition}\label{zerocohglobalsecn}
If $\mathscr{F}$ is a sheaf on $X$, then
\begin{equation}
    \check{\mathrm{H}}^0(\mathcal{U},\mathscr{F}) = \Gamma(X, \mathscr{F})
\end{equation}
\end{proposition}

\begin{proof}
A 0-cocycle $(f_j)_{j \in J} \in C^{0}(\mathcal{U}, F)$ is a family $f_j \in \mathscr{F}(U_j)$ such that $\delta^0(f_j) = 0$, i.e., $\rho_{XU_{j}}(f_j) = 0$, for all $j$. Since $\mathcal{U}$ is a covering of $X$, the gluability and identity axioms guarantee that we can glue all of these $f_j$ to a global section $f \in \Gamma(X, \mathscr{F})$.
\end{proof}

\begin{definition}
We say that a cover $\mathcal{U} = \{U_i\}_{i \in I}$ is a refinement of a cover $\mathcal{V} = \{V_j\}_{j \in J}$, and denote it $\mathcal{V} \prec \mathcal{U}$ if there exists a function $\pi: I \rightarrow J$ such that $U_i \subset V_{\pi(i)}$ for all $i \in I$. We say that $\mathcal{U}$ and $\mathcal{V}$ are equivalent if $\mathcal{V} \prec \mathcal{U}$ and $\mathcal{U} \prec \mathcal{V}$. If $\sigma \in C^{p}(\mathcal{V}, \mathscr{F})$, define

\begin{equation}
    \pi^{p}(\sigma)_{\alpha_0,...,\alpha_p} = \rho_{VU}(\sigma_{\pi(\alpha_0),...,\pi(\alpha_p)})
\end{equation}
where $\rho_{VU}$ denotes the restriction morphism induced by the inclusion $U_{\alpha_0} \cap ... \cap U_{\alpha_p} \rightarrow V_{\pi(\alpha_0)} \cap ... \cap V_{\pi(\alpha_p)}$.
\end{definition}

The map $\pi^{p}: C^{p}(\mathcal{V}, \mathscr{F}) \rightarrow C^{p}(\mathcal{U}, \mathscr{F})$ is a homomorphism, and it commutes with the coboundary map $\delta$. So we get induced morphisms

\begin{equation}
    \pi^{*p}: \check{\mathrm{H}}^p(\mathcal{V},\mathscr{F}) \rightarrow \check{\mathrm{H}}^p(\mathcal{U},\mathscr{F})
\end{equation}

\begin{proposition}
The morphisms above are independent of $\mathcal{U}, \mathcal{V}$ and the map $\pi$.
\end{proposition}

\begin{proof}
\cite{FAC}, Chapter 1, §3, Proposition 3.
\end{proof}

This proposition tells us that if $\mathcal{V} \prec \mathcal{U}$, then there is a morphism

\begin{equation}
    \rho_{\mathcal{V} \mathcal{U}}: \check{\mathrm{H}}^p(\mathcal{V},\mathscr{F}) \rightarrow \check{\mathrm{H}}^p(\mathcal{U},\mathscr{F})
\end{equation}

We then check that the covers with the relation that $\mathcal{V} \prec \mathcal{U}$ form a \textit{preordered index category} (i.e., a \textit{filtered index category}): indeed, given covers $\mathcal{U} = \{U_i\}_{i \in I}$ and $\mathcal{V} = \{V_j\}_{j \in J}$, the cover $\mathcal{W} = \{U_i \cap V_j\}_{(i,j) \in I \times J}$ is a refinement of both $\mathcal{V}$ and $\mathcal{U}$,and so it is in fact filtered. Also, if $\mathcal{U} \prec \mathcal{V} \prec \mathcal{W}$,

\begin{equation}
    \rho_{\mathcal{U} \mathcal{W}} = \rho_{\mathcal{V} \mathcal{W}} \circ \rho_{\mathcal{U} \mathcal{V}}
\end{equation}
and $\rho_{\mathcal{U} \mathcal{U}} = Id$. Finally, we see that if $\mathcal{U}$ and $\mathcal{V}$ are equivalent,

\begin{equation}
    \check{\mathrm{H}}^p(\mathcal{U},\mathscr{F}) \rightarrow \check{\mathrm{H}}^p(\mathcal{V},\mathscr{F})
\end{equation}
is an isomorphism. So, the family $(\check{\mathrm{H}}^p(\mathcal{U},\mathscr{F}))_{\mathcal{U}}$ is a direct system of groups indexed by the filtered index category of covers. However, we arrive at the set-theoretic issue that \textbf{this index category is not necessarily small}, for it allows for arbitrary sets of indices. Serre proposes a way to avoid this issue in FAC: given a cover $\mathcal{U}$ we can always find an equivalent cover whose index set is a subset of the power set of $X$ (cf. \cite{FAC}, Chapter 1, §3, sec. 22) and so is bounded by $2^{2^{X}}$. Another alternative is proposed by Godement in (\cite{godement}, Chap. 5, §5.8). We finally arrive at the definition:

\begin{definition}
If $X$ is a topological space and $\mathscr{F}$ is a presheaf on $X$, the \textit{$p^{th}$-\v{C}ech cohomology group with values in $\mathscr{F}$} is defined as the direct limit

\begin{equation}
    \check{\mathrm{H}}^p(X,\mathscr{F}) = \varinjlim_{\mathcal{U}} \check{\mathrm{H}}^p(\mathcal{U},\mathscr{F}).
\end{equation}
with respect to the coverings $\mathcal{U}$ whose index set is a subset of the power set of $X$.
\end{definition}

In particular, the \textit{\v{C}ech cohomology groups with coefficients in the constant sheaf $\underbar{G}$} $\check{\mathrm{H}}^p(X,\underbar{G})$ allow us to recover the groups $\mathrm{H}^p_{Sing}(X, G)$.


Looking at this definition, one quickly notices that it is virtually impossible to work with, and to make computations. 
A strong characterisation of the cover $\mathcal{U}$ so as to make computing $\check{\mathrm{H}}^p(\mathcal{U},\mathscr{F})$ the same as computing $\check{\mathrm{H}}^p(X,\mathscr{F})$ is the following result:

\begin{definition}
Let $\mathcal{U} = \{U_i\}_{i \in I}$ be a cover. We say that the cover $\mathcal{U}$ is \textit{acyclic with respect to the sheaf $\mathscr{F}$}, if

\begin{equation}
    \check{\mathrm{H}}^q(U_{i_1} \cap ...\cap U_{i_p}, \mathscr{F}) = 0
\end{equation}
for all $q > 0$ and any $\{i_1,...,i_p\} \subset I$.
\end{definition}

\begin{theorem}[Leray's Theorem for Acyclic covers]\label{leraycech}
If the covering $\mathcal{U}$ is acyclic for the sheaf $\mathscr{F}$, then $\check{\mathrm{H}}^q(\mathcal{U}, \mathscr{F}) = \check{\mathrm{H}}^q(X, \mathscr{F})$.
\end{theorem}

Leray's Theorem reduces the problem of computing the cohomology of some space $X$ to computing the cohomology of some acyclic cover of $X$. For example, one can prove that a covering admitting a locally finite refinement is acyclic (\cite{godement}, Chap. 5, §5.10, Théorème 5.10.2), and this is precisely why Theorem \ref{paracompact} is useful. This theorem will later be generalized to a version of Leray's Theorem for acyclic covers for the cohomology of derived functors.

\begin{remark}
It is worth noticing that, given $X$ and fixing a covering $\mathcal{U}$, if

\begin{equation}
\begin{tikzcd}
    0 \ar{r} & \mathscr{F} \ar{r} & \mathscr{G} \ar{r} & \mathscr{H} \ar{r} & 0
\end{tikzcd}
\end{equation}
is an exact sequence of sheaves on $X$, then in general we \textbf{do not} get a long exact sequence in the \v{C}ech cohomology groups. One way to see this is, given a short exact sequence as above, if we were to get an exact sequence of \v{C}ech complexes through the \textit{sheafified} version of the \v{C}ech complex as will be described below, we may fail to get the connecting morphisms, since $\Gamma(X, -)$ is not exact. This is further illustrated in Example \ref{extohoku} (also, \cite{hartshorne}. Chap. 3, Caution 4.0.2 for a discussion involving $\delta$-functors).
\end{remark}

Now let $\mathcal{U}$ be a covering on $X$ and $\mathscr{F}$ be some abelian sheaf on $X$. Let $\sigma \in \check{\mathrm{H}}^0(\mathcal{U},\mathscr{F})$ be a $0$-cocycle. So $\sigma$ is giving by a collection of sections of $\mathscr{F}$ agreeing on the pairwise intersections of elements of $\mathcal{U}$. Let's describe the image of the differential operator $\delta$ in $C^1(\mathcal{U}, \mathscr{F})$, that is, the group $Z^1(\mathcal{U}, \mathscr{F})$ of \v{C}ech $1$-cocycles: For each $U_{i_1}, U_{i_2}, U_{i_3} \in \mathcal{U}$, the operator $\delta$ maps a $0$-cochain, say $f \in C^0(\mathcal{U}, \mathscr{F})$ to the alternated sum

\begin{equation}
    \left. f_{i_2 i_3} \right|_{U_{i_1} \cap U_{i_2} \cap U_{i_3}} - \left. f_{i_1 i_3} \right|_{U_{i_1} \cap U_{i_2} \cap U_{i_3}} + \left. f_{i_1 i_2} \right|_{U_{i_1} \cap U_{i_2} \cap U_{i_3}}
\end{equation}
and moreover, this sum is $0$ in $\check{\mathrm{H}}^1(\mathcal{U}, \mathscr{F})$. So, the class $[\delta f]$ in $\check{\mathrm{H}}^1(\mathcal{U}, \mathscr{F})$ satisfies

\begin{equation}
        \left. f_{i_2 i_3} \right|_{U_{i_1} \cap U_{i_2} \cap U_{i_3}}  + \left. f_{i_1 i_2} \right|_{U_{i_1} \cap U_{i_2} \cap U_{i_3}} = \left. f_{i_1 i_3} \right|_{U_{i_1} \cap U_{i_2} \cap U_{i_3}}
\end{equation}
and we have a description of the \textit{cocycle condition} in the language of \v{C}ech theory. This hints at the idea of \v{C}ech $1$-cocyles giving the datum for the construction of some object like in Proposition \ref{cocyclesconstructions}, except that this time it is depending on the sheaf $\mathscr{F}$ and the covering $\mathcal{U}$.

\subsection{Quasicoherent Sheaves}

We start our journey into generalization with a quick review of the theory of vector bundles over projective varieties over $\mathbb{C}$, following (\cite{lepotier}, Part I, Sect. 1). Throughout, all rings are commutative with unity.

\begin{definition}
Let $X$ be a complex projective variety. A \textit{complex linear fibration} over $X$ is given by an algebraic variety $E$ and a surjective \textbf{morphism of algebraic varieties} $\pi:E \rightarrow X$ such that for every $x \in X$, $\pi^{-1}(x)$ has the structure of a complex vector space. We again denote the spaces $\pi^{-1}(x)$ by $E_x$.
\end{definition}

\begin{definition}
Given two linear fibrations $\pi:E \rightarrow X$ and $\pi':E' \rightarrow X$, a \textit{morphism of linear fibrations} $f$ is a morphism of varieties that commutes with the projections $\pi$ and $\pi'$, i.e., the following diagram

\begin{center}

\begin{tikzpicture}[>=triangle 60]
\matrix[matrix of math nodes,column sep={60pt,between origins},row
sep={60pt,between origins},nodes={asymmetrical rectangle}] (s)
{
|[name=A]| E &|[name=B]| E' \\
&|[name=A']| X\\
};
\draw[overlay,->, font=\scriptsize,>=latex] 
          (A) edge node[auto] {\(f\)} (B)
          (A) edge node[auto] {\(p\)} (A')
          (B) edge node[auto] {\(p'\)} (A')
;

;

\end{tikzpicture}

\end{center}
commutes and the induced map $f_x: E_x \rightarrow E'_x$ on the fibers, is linear. An \textit{isomorphism} of linear fibrations is a morphism with a two-sided inverse.

\end{definition}

The fibration $X \times \mathbb{C}^r \rightarrow X$ given by the projection in the first coordinate is called the \textit{trivial fibration of rank $r$}. For each open set $U \subseteq X$, write $\left. E \right|_U$ for the restriction fibration given by $\pi: \pi^{-1}(U) \rightarrow U$.

\begin{definition}
An algebraic vector bundle of rank $r$ over $X$ is a linear fibration $\pi: E \rightarrow X$ such that for each $x \in X$ there exists an open neighborhood $U$ of $x$ and an isomorphism of fibrations

\begin{equation}
    \varphi: \left. E \right|_U \rightarrow U \times \mathbb{C}^r.
\end{equation}
The isomorphism $\varphi$ is called a trivialization of the vector bundle. A vector bundle is a fibration that is \textit{locally trivial}.
\end{definition}

Given two trivializations $\varphi_i: \left. E \right|_{U_{i}} \rightarrow U_i \times \mathbb{C}^r$ and $\varphi_j: \left. E \right|_{U_{j}} \rightarrow U_j \times \mathbb{C}^r$ over $U_i$ and $U_j$ respectively, then the change of coordinate map defined over $U_i \cap U_j$ by

\begin{equation}
    \varphi_i \circ \varphi^{-1}_j : U_i \cap U_j \times \mathbb{C}^r \rightarrow U_i \cap U_j \times \mathbb{C}^r
\end{equation}
maps $(u, v) \in U_i \cap U_j \times \mathbb{C}^r$ to $(x, g_{ij}(x)v)$, where $g_{ij}: U_i \cap U_j \rightarrow GL(r, \mathbb{C})$ is a map of varieties between the open subset $U_i \cap U_j$ and $GL(r, \mathbb{C})$ as algebraic varieities.
The maps $g_{ij}$ are called \textit{transition functions}, and they satisfy the cocycle condition:

\begin{equation}
    g_{ij} \circ g_{jk} =  g_{ik}
\end{equation}
over $U_i \cap U_j \cap U_k$.

Conversely, like in the differentiable case, given the data of the transition functions $g_{ij}$ on the open cover $\{U_i\}$, we can retrieve the bundle $E$ by gluing the $U_i \times \mathbb{C}^r$ along the intersections $U_i \cap U_j$, using $g_{ij}$ similar to Proposition \ref{cocyclesconstructions}.

\begin{definition}
Let $E$ and $F$ be algebraic vector bundles over $X$, of ranks $r$ and $s$ respectively. A \textit{map of vector bundles} $f$ is a map between $E$ and $F$ as fibrations.

\end{definition}

Given two bundles $E$ and $F$ over $X$, we can construct: $E \oplus F$, the \textit{direct sum} of $E$ and $F$; \underbar{$\Hom$}$(E, F)$, the \textit{internal-hom} bundle; $E \otimes F$, the \textit{tensor product} bundle; $E^{*}$ the \textit{dual} bundle; $\Lambda^{k}(E)$ the \textit{$k$-th exterior power} bundle.

Given a bundle $E$ of rank $r$ over $X$, a sub-bundle $F$ of rank $m$ is an \textit{algebraic subvariety} $F \subset E$ such that for every $x \in X$, $F \cap E_x$ is a complex vector space of dimension $m$, and the induced fibration $\left. \pi \right|_F : F \rightarrow X$ is locally trivial.

Given a bundle $E$ and a sub-bundle $F$ of $E$, we can construct the \textit{quotient bundle} $E/F$, and it has a unique structure of an algebraic vector bundle.

\begin{definition}
Let $\pi: E \rightarrow X$ be an algebraic vector bundle of rank $r$ over $X$. A \textit{regular section} of $E$ over some open set $U \subseteq X$ is a \textit{map 
$s: U \rightarrow E$ of varieties} such that $\pi (s(x)) = x$ for all $x \in U$. The set of regular sections of $E$ over $U$ is denoted $\Gamma(U, E)$.
\end{definition}

By now the similarities between the theory of bundles over smooth manifolds and algebraic varieties should tell that there is a common underlying principle that connects the studies of vector bundles over  geometric spaces: it consists of a space of the same kind as the base geometric space, such that \textit{locally}, it looks like whatever \textit{trivial} space is, varying from setting to setting. Such similarities suggest a generalization. We extend this idea to the more general language of sheaves, which will serve as a general framework to apply the cohomological machinery that we've been developing to varieties and manifolds, and more general geometric objects such as schemes and later, analytic spaces.

Recall that a \textit{ringed space} is a pair $(X, \mathcal{O}_X)$, where $X$ is a topological space and $\mathcal{O}_X$ is a sheaf of rings on $X$. A \textit{morphism of ringed spaces} $f: (X, \mathcal{O}_X) \to (Y, \mathcal{O}_Y)$ is a pair $(f, f^{\#})$, consisting of a continuous map $f: X \to Y$ and a map $f^{\#}: \mathcal{O}_Y \to f_{*}\mathcal{O}_X$ of sheaves of rings on $Y$ (the sheaf $f_{*}\mathcal{O}_X$ will be defined below). A \textit{sheaf of $\mathcal{O}_X$-modules} on $X$ is a sheaf $\mathscr{F}$ of abelian groups, with an $\mathcal{O}_X(U)$ module structure on every section $\mathscr{F}(U)$ for any open set $U \subseteq X$, that commutes with restriction maps. Precisely, given an inclusion $U \subset V$ we have a commutative diagram

\begin{center}
\begin{tikzpicture}[>=triangle 60]
\matrix[matrix of math nodes,column sep={90pt,between origins},row
sep={60pt,between origins},nodes={asymmetrical rectangle}] (s)
{
&|[name=A]| \mathcal{O}(V) \times \mathscr{F}(V) &|[name=B]| \mathscr{F}(V) \\
&|[name=A']| \mathcal{O}(U) \times \mathscr{F}(U) &|[name=B']| \mathscr{F}(U) \\
};
\draw[overlay,->, font=\scriptsize,>=latex] 
          (A) edge ++(2.5cm,0) node[auto] {\(\)} (B)
          (A) edge node[auto, swap] {\(\rho_{VU} \times \rho'_{VU}\)} (A')
          (B) edge node[auto] {\(\rho'_{VU}\)} (B')
          (A') edge ++(2.5cm,0) node[auto] {\(\)} (B')

;

;

\end{tikzpicture}
\end{center} 

\begin{proposition}
For a ringed space $(X, \mathcal{O}_X)$, the category \textbf{Mod$_{\mathcal{O}_X}$} of sheaves of $\mathcal{O}_X$-modules is abelian.
\end{proposition}
\noindent\textbf{Operations with sheaves of modules}: 
There are some constructions that can be made within the category \textbf{Mod$_{\mathcal{O}_X}$} of sheaves of $\mathcal{O}_X$-modules. Throughout, let's fix a ringed space $(X, \mathcal{O}_X)$. Let $\mathscr{F}_{i}$ be a family of sheaves of $\mathcal{O}_X$-modules indexed by a set $I$. Define the \textit{product} 

\begin{equation}
    \prod_{i \in I} \mathscr{F}_i
\end{equation}
to be the sheaf that associates to each open set $U$, the product $\prod_{i \in I} \mathscr{F}_i(U)$ of the sections $\mathscr{F}_i (U)$ of each $\mathscr{F}_i$ over $U$. Define the \textit{direct sum}

\begin{equation}
    \bigoplus_{i \in I} \mathscr{F}_i
\end{equation}
to be the \textit{sheafification} of the presheaf defined by the rule

\begin{equation}
    U \mapsto \bigoplus_{i \in I}\mathscr{F}_i (U)
\end{equation}
that associates to each open set $U$ the direct sum of the modules $\mathscr{F}_i (U)$. Note that if the index set $I$ is \textit{finite}, the direct sum and the direct product coincide, and so the \textbf{sheafification for finite direct sums is not necessary}.

Let $\mathscr{F}$ and $\mathscr{G}$ be two sheaves of $\mathcal{O}_X$-modules. Define the \textit{presheaf tensor product} $\mathscr{F} \otimes_{Pre\mathcal{O}_{X}} \mathscr{G}$ to be the presheaf defined by the rule
\begin{equation}
    U \mapsto \mathscr{F}(U) \otimes_{Pre\mathcal{O}_{X}} \mathscr{G}(U)
\end{equation}
and the \textit{sheaf tensor product} $\mathscr{F} \otimes_{\mathcal{O}_{X}} \mathscr{G}$ as its sheafification. It shares the usual properties of the tensor product of modules, such as the universal property, the commutativity and associativity, and the right-exactness of the functor $- \otimes_{\mathcal{O}_{X}} \mathscr{G}$. Another interesting property is that, for any $x \in X$, the stalk of the tensor product at $x$ is the tensor product of the stalks at $x$, namely 

\begin{equation}
    (\mathscr{F} \otimes_{\mathcal{O}_{X}} \mathscr{G})_x \cong \mathscr{F}_x \otimes_{\mathcal{O}_{X,x}} \mathscr{G}_x.
\end{equation}
The category \textbf{Mod$_{\mathcal{O}_X}$} also has an \textit{internal hom}: given $\mathscr{F}$ and $\mathscr{G}$, the association 
\begin{equation}
U \mapsto \Hom _{\mathcal{O}_{U}}(\left. \mathscr{F} \right|_{U}, \left. \mathscr{G} \right|_{U})
\end{equation}
is a sheaf of abelian groups. Given $g \in \Hom _{\mathcal{O}_{U}}(\left. \mathscr{F} \right|_{U}, \left. \mathscr{G} \right|_{U})$ and $f \in \mathscr{O}_U$, we can define $fg \in \Hom _{\mathcal{O}_{U}}(\left. \mathscr{F} \right|_{U}, \left. \mathscr{G} \right|_{U})$ by precomposing with the morphism of multiplication by $f$ on $\left. \mathscr{F} \right|_{U}$. And so we have a canonical $\mathcal{O}_X$-module structure on $\Hom _{\mathcal{O}_{U}}(\left. \mathscr{F} \right|_{U}, \left. \mathscr{G} \right|_{U})$. This \textit{sheaf hom} will be denoted \underbar{$\Hom$}$_{\mathcal{O}_{X}}(\mathscr{F}, \mathscr{G})$. There is also a morphism 

\begin{equation}
    (\underbar{$\Hom$} _{\mathcal{O}_{X}}(\mathscr{F}, \mathscr{G}))_x \rightarrow \underbar{$\Hom$} _{\mathcal{O}_{X}}(\mathscr{F}_x, \mathscr{G}_x)
\end{equation}
but in general it is \textbf{not} an isomorphism - it is, however, in the case that $\mathscr{F}$ is \textit{finitely presented}. The sheaf \underbar{$\Hom$}$ _{\mathcal{O}_{X}}(\mathscr{F}, \mathscr{G})$ also shares the property of letf-exactness of the usual $\Hom$ functor. The sheaves $\mathscr{F} \otimes_{\mathcal{O}_{X}} \mathscr{G}$ and \underbar{$\Hom$}$ _{\mathcal{O}_{X}}(\mathscr{F}, \mathscr{G})$ form an \textit{adjoint pair}, in the sense of the internal hom: namely, for any other sheaf of $\mathcal{O}_X$-modules $\mathscr{H}$ there is an isomorphism

\begin{equation}
    \underbar{$\Hom$}_{\mathcal{O}_{X}}(\mathscr{F} \otimes_{\mathcal{O}_{X}} \mathscr{G}, \mathscr{H}) \cong \underbar{$\Hom$} _{\mathcal{O}_{X}}(\mathscr{F}, \underbar{$\Hom$} _{\mathcal{O}_{X}}(\mathscr{G}, \mathscr{H}))
\end{equation}
functorial in $\mathscr{F}$, $\mathscr{G}$ and $\mathscr{H}$.

Now, let $f: X \rightarrow Y$ be a continuous map. For any sheaf $\mathscr{F}$ on $X$, we can define the \textit{direct image} sheaf $f_{*}\mathscr{F}$ on $Y$ by the rule 

\begin{equation}
    f_{*}\mathscr{F}(V) = \mathscr{F}(f^{-1}(V))
\end{equation}
for any open set $V \subseteq Y$. We can also define, for any \textit{abelian} sheaf $\mathscr{G}$ on $Y$ the \textit{inverse image} sheaf $f^{-1}\mathscr{G}$ as the sheafification of the presheaf given by the rule 

\begin{equation}
    U \mapsto \varinjlim_{f(U) \subset V} \mathscr{G}(V)
\end{equation}
where $U \subseteq X$ is any open set and the limit is taken over all open sets $V \subseteq Y$ such that $f(U) \subset V$. We have that $f_{*}$ is a functor from the category of sheaves of abelian groups on $X$ to the category of sheaves of abelian groups on $Y$ and $f^{-1}$ is a functor from the category of sheaves of abelian groups on $Y$ to the category of sheaves of abelian groups on $X$. They also form an adjoint pair: we have an isomorphism

\begin{equation}
    \Hom_X (f^{-1}\mathscr{G}, \mathscr{F}) \cong \Hom_X (\mathscr{G}, f_{*}\mathscr{F}).
\end{equation}

Given a map of ringed spaces $f: (X, \mathcal{O}_X) \rightarrow (Y, \mathcal{O}_Y)$, if $\mathscr{F}$ is an $\mathcal{O}_X$-module, then $f_{*}\mathscr{F}$ is an $f_{*}\mathcal{O}_X$-module. Since by the definition of morphism of ringed spaces we have a morphism $f^{\#}: \mathcal{O}_Y \rightarrow f_{*}\mathcal{O}_X$ of sheaves of rings on $Y$, we get a natural $\mathcal{O}_Y$-module structure on $f_{*}\mathscr{F}$. Now, we would like to define an inverse image functor for sheaves of modules: we would like to define a functor that behaves like the inverse image functor $f^{-1}:\textbf{Sh(Y)} \rightarrow \textbf{Sh(X)}$ for general sheaves of modules, such that said inverse image has the structure of an $\mathcal{O}_X$-module, and such that this new functor has the adjoint property with the functor $f_{*}: \textbf{Mod$_{\mathcal{O}_X}$} \rightarrow \textbf{Mod$_{\mathcal{O}_Y}$}$. Let $\mathscr{G}$ be a sheaf of $\mathcal{O}_Y$-modules. So $f^{-1}\mathscr{G}$ has the structure of an $f^{-1}\mathcal{O}_Y$-module. We would like to make it an $\mathcal{O}_X$-module. Because of the adjoint property between $f^{-1}$ and $f_{*}$, giving a morphism $\mathcal{O}_Y \rightarrow f_{*}\mathcal{O}_X$ is the same as giving a morphism $f^{-1}\mathcal{O}_Y \rightarrow \mathcal{O}_X$. We define the sheaf $f^{*}\mathscr{G}$ to be the tensor product

\begin{equation}
    f^{-1}\mathscr{G} \otimes_{f^{-1}\mathcal{O}_Y} \mathcal{O}_X
\end{equation}
and thus $f^{*}\mathscr{G}$ is an $\mathcal{O}_X$-module. It is called the \textit{inverse image} of $\mathscr{G}$ by $f$. The functors $f_{*}: \textbf{Mod$_{\mathcal{O}_X}$} \rightarrow \textbf{Mod$_{\mathcal{O}_Y}$}$ and $f^{*}: \textbf{Mod$_{\mathcal{O}_Y}$} \rightarrow \textbf{Mod$_{\mathcal{O}_X}$}$ again form an adjoint pair. This fact allows us to prove the following

\begin{lemma}
Let $f: (X, \mathcal{O}_X) \rightarrow (Y, \mathcal{O}_Y)$ be a morphism of ringed spaces. Then
\begin{enumerate}
    \item the functor $f_{*}: \textbf{Mod$_{\mathcal{O}_X}$} \rightarrow \textbf{Mod$_{\mathcal{O}_Y}$}$ is \textit{left exact}, and commutes with all limits;
    \item the functor $f^{*}: \textbf{Mod$_{\mathcal{O}_Y}$} \rightarrow \textbf{Mod$_{\mathcal{O}_X}$}$ is right exact, and commutes with all colimits.
\end{enumerate}
\end{lemma}

\begin{proof}
The result follows from the adjoint property of the pair $f^{*}$ and $f_{*}$ and two categorical facts: if the source category has all finite colimits (resp. the target category has all finite limits), then the left adjoint (resp. right adjoint) is right exact (resp. left exact) (\cite{stacks-project}, Lemma 4.24.6); and left adjoints (resp. right adjoints) commute with all representable colimits (resp. representable limits) (\cite{stacks-project}, Lemma 4.24.5). Since the categories \textbf{Mod$_{\mathcal{O}_X}$} and \textbf{Mod$_{\mathcal{O}_Y}$} have all limits and colimits\footnote{This is a stronger condition then being abelian - being an abelian category only guarantees the existence of all \textbf{finite} limits and colimits.} (\cite{stacks-project}, Lemma 17.3.2), this completes the proof.
\end{proof}

\begin{definition}
A \textit{free sheaf} on a ringed space $(X, \mathcal{O}_X)$ is a sheaf $\mathscr{F}$ of $\mathcal{O}_X$-modules isomorphic to $\mathcal{O}^{\oplus I}_X$, for some index set $I$. A free sheaf is of \textit{rank r} if it is isomorphic to $\mathcal{O}^{\oplus r}_X$. A \textit{locally free sheaf} is a sheaf of $\mathcal{O}_X$-modules locally isomorphic to a free sheaf.
\end{definition}
Recall that a variety $X$ has a canonical sheaf of rings, its \textit{sheaf of regular functions} $\mathcal{O}_X$ (see Example \ref{structsheaf}). The set $\Gamma(U, E)$ has a canonical structure of an $\mathcal{O}_X$-module: for every $U \subseteq X$, the operations are given by

\begin{align*}
    (s+t)(x) &= s(x)+t(x) \\
    (\alpha s)(x) &= \alpha(x)s(x)
\end{align*}
for $s, t \in \Gamma(U, E)$ and $\alpha \in \mathcal{O}_X(U)$. Moreover, given a trivialization over $U$, we can naturally identify a section over $U$ with an $r$-tuple of functions over $U$, by composing with the trivialization. So $\Gamma(U, E) \cong \mathcal{O}_{X}^{\oplus r}(U)$, and \textbf{$\Gamma(U, E)$ is a locally free sheaf of rank r}. This observation leads us to an important result, which kickstarts the theory of (quasi)coherent sheaves:

\begin{theorem}\label{locfreesheavesvecbundles}
The functor which associates the sheaf of regular sections to a vector bundle $E$ over $X$ is an equivalence of categories between the category of algebraic vector bundles over $X$ and the category  of locally free sheaves of finite rank on $X$.
\end{theorem}

Note that the theorem speaks of the functor from the category of bundles $\Gamma(X, -)$ which receives as input an algebraic vector bundle, and returns as output a locally finite rank $\mathcal{O}_X$-module. So we're estabilishing an equivalence between \textbf{algebraic vector bundles and locally free $\mathcal{O}_X$-modules of finite rank}.

\begin{proof}
The proof proceeds in two steps: first, we embbed the category of bundles fully faithfully in the category of $\mathcal{O}_X$-modules of finite rank; then we show that the functor is essentially surjective, that is, given an $\mathcal{O}_X$-module of finite rank, we find a bundle that is isomorphic to it.

So first, we start by showing that the functor is fully faithful: given any two vector bundles $E$ and $F$ on $X$, of ranks $p$ and $q$ respectively, the natural linear map

\begin{equation}
    \Hom (E, F) \rightarrow \Hom_{\mathcal{O}_{X}} (\Gamma (X, E), \Gamma (X, F))
\end{equation}
is an isomorphism - indeed, since the question is local on $X$, and since we can always find a trivial neighborhooud of every point $x \in X$, we can suppose that $E = X \times \mathbb{C}^p$ and $F = X \times \mathbb{C}^q$. Then any map $f: E \rightarrow F$ can then be regarded as a map $X \rightarrow L(\mathbb{C}^p, \mathbb{C}^q)$ from $X$ to the space of $\mathbb{C}$-linear maps $\mathbb{C}^p \rightarrow \mathbb{C}^q$. Since we are assuming $X$ to be projective, $\Gamma(X, E) \cong \mathcal{O}_X^{\oplus p} \cong \mathbb{C}^p$ and $\Gamma(X, F) \cong \mathcal{O}_X^{\oplus q} \cong \mathbb{C}^q$ and so each linear map in $L(\mathbb{C}^p, \mathbb{C}^q)$ gives us a morphism of $\mathcal{O}_X$-modules $\mathcal{O}_{X}^{\oplus p} \rightarrow \mathcal{O}_{X}^{\oplus q}$. 

Now, let $\mathcal{E}$ be a locally free sheaf of rank $r$. We need to construct a vector bundle $E$ and an isomorphism $\Gamma(U, E) \cong \mathcal{E}$. Let $x \in X$, and let $\mathcal{E}_x$ be the stalk of $\mathcal{E}$ at $x$. Consider the vector space $\mathcal{E}(x) = \mathcal{E}_x/ \mathfrak{m}_x \mathcal{E}_x$, where $\mathfrak{m}_x$ is the unique maximal ideal of the local ring $\mathcal{O}_{X, x}$. Since there exists a neighborhood for which we can find a local trivialization of $\mathcal{E}$, $\mathcal{E}_x$ is free of rank $r$. We can then give $\mathcal{E}_x/ \mathfrak{m}_x \mathcal{E}_x$ a $\mathcal{O}_{X,x} / \mathfrak{m}_x$-vector space structure, for which it has dimension $r$. Define

\begin{equation}
    E = \bigsqcup_{x \in X} \mathcal{E}(x) \rightarrow X
\end{equation}
Let $\{U_i\}$ be an open cover of $X$ for which we have isomorphisms $\left. \mathcal{E} \right|_{U_{i}} \cong \mathcal{O}_{X}^{\oplus r}$ for the locally free sheaf $\mathcal{E}$. On the intersections $U_i \cap U_j$, we then get maps 

\begin{equation}
    f_{ij} = \left. f_j \right|_{U_{i} \cap U_{j}} \circ \left. f_{i}^{-1} \right|_{U_{i} \cap U_{j}} : \left. \mathcal{O}_{X}^{\oplus r} \right|_{U_{i} \cap U_{j}} \rightarrow \left. \mathcal{O}_{X}^{\oplus r} \right|_{U_{i} \cap U_{j}}.
\end{equation}
By the previous discussion on fully faithfulness, each of these $f_{ij}$ induces a map

\begin{equation}
    g_{ij}: (U_i \cap U_j) \times \mathbb{C}^r \rightarrow (U_i \cap U_j) \times \mathbb{C}^r.
\end{equation}
We then have the gluing data $(\{U_i\}, g_{ij})$, and by the previous discussion, we obtain a vector bundle by gluing along the intersections using the $g_{ij}$. Finally, by construction the sections of $E$ over $U$ are given by sections $s_i$ over $U_i$ agreeing on $U_i \cap U_j$, and so $\Gamma(U, E)$ is isomorphic to $\mathcal{E}$ as an $\mathcal{O}_X$-module.
\end{proof}

\begin{remark}
We also have this equivalence in contexts more general than bundles over projective varieties over $\mathbb{C}$.
\end{remark}

\begin{remark}
A special case is that of vector bundles of rank $1$, or line bundles. The associated locally free sheaf of rank $1$ is called an \textbf{invertible sheaf}.
\end{remark}

And so, we have embedded our category of algebraic vector bundles over $X$ into the category of sheaves of $\mathcal{O}_X$-modules over $X$, which is an abelian category. The problem is, the subcategory of locally free sheaves of finite rank is \textbf{not} an abelian category - cokernels of maps of $\mathcal{O}_X$-modules between locally free sheaves need not be locally free\footnote{In fact, in \cite{EGA}, Grothendieck and Dieudonné define quasicoherent sheaves this way - they are the sheaves that are the \textit{cokernel} of a map $\mathcal{O}_{X}^{\oplus I} \rightarrow \mathcal{O}_{X}^{\oplus J}$, for some index sets $I$ and $J$, not necessarily finite.}. So we have to enlarge our category of bundles to a reasonable abelian category which contains it, so we can use our homological algebra machinery. Turns out that this larger category is the category of \textit{quasicoherent sheaves} \textbf{QCoh($X$)}, which is an abelian category which contains locally free sheaves and its reasonable generalisations, such as locally free sheaves not necessarily of finite rank. Besides the usual references of \cite{hartshorne}, \cite{vakil}, we will also follow \cite{FAC}, \cite{stacks-project} and \cite{ag2}. We begin with a general discussion, and later specialize:

\begin{definition}
Let $(X, \mathcal{O}_X)$ be a ringed space. A sheaf of $\mathcal{O}_X$-modules $\mathscr{F}$ is called a \textit{quasicoherent sheaf of $\mathcal{O}_X$-modules} if for every point $x \in X$ there is an open neighborhood $U$ of $x$ in $X$ such that $\left. \mathscr{F} \right|_{U}$ is isomorphic to the \textit{cokernel} of some map

\begin{equation}
\begin{tikzcd}
    \mathcal{O}_{U}^{\oplus I} \ar{r} & \mathcal{O}_{U}^{\oplus J} 
\end{tikzcd}
\end{equation}
of $\mathcal{O}_X$-modules, for some arbitrary index sets $I$ and $J$.
\end{definition}

The category of quasicoherent sheaves of $\mathcal{O}_X$-modules is denoted \textbf{QCoh($X$)}, and it is an abelian category. The above definition can be interpreted in the following way: a sheaf of $\mathcal{O}_X$-modules $\mathscr{F}$ is quasicoherent if there exists a covering of $X$ by open sets $U$ such that $\left. \mathscr{F} \right|_{U}$ has a (not necessarily finite) \textit{presentation} of the form

\begin{equation}
\begin{tikzcd}
    \mathcal{O}_{U}^{\oplus I} \ar{r} & \mathcal{O}_{U}^{\oplus J} \ar{r} & \left. \mathscr{F} \right|_{U} \ar{r} & 0
\end{tikzcd}
\end{equation}
that is, the above sequence is exact.

\begin{definition}
We say that a sheaf $\mathscr{F}$ of $\mathcal{O}_X$-modules is \textit{generated by global sections} if there exists a set $I$, and global sections $s_i \in \Gamma(X, \mathscr{F})$ such that the map
\begin{equation}
\begin{tikzcd}
    \mathcal{O}_{X}^{\oplus I} \ar{r} & \mathscr{F} \ar{r} & 0
\end{tikzcd}
\end{equation}
given by  $f \mapsto fs_i$ i.e., multiplication by $s_i$, on the $ith$ summand of $\mathcal{O}_{X}^{\oplus I}$, is surjective. In this case, we say that the $s_i$ generate $\mathscr{F}$. We say that $\mathscr{F}$ is of \textit{finite type} if for every $x \in X$ there exists a neighborhood such that $\left. \mathscr{F} \right|_{U}$ is generated by finitely many sections,

\begin{equation}
\begin{tikzcd}
    \mathcal{O}_{U}^{\oplus n} \ar{r} & \left. \mathscr{F} \right|_{U} \ar{r} & 0.
\end{tikzcd}
\end{equation}
\end{definition}

\begin{definition}
A quasicoherent sheaf of $\mathcal{O}_X$-modules $\mathscr{F}$ is said to be \textit{coherent} if it satisfies the following two conditions:

\begin{enumerate}
    \item $\mathscr{F}$ is of finite type and;
    \item for every open set $U \subset X$, and every finite collection $s_i \in \mathscr{F}(U)$, for $i = 1,...,n$, the kernel of the associated map 
    
    \begin{equation}
    \begin{tikzcd}
    \mathcal{O}_{U}^{\oplus n} \ar{r} & \left. \mathscr{F} \right|_{U}
    \end{tikzcd}
    \end{equation}
    is of finite type, i.e., the \textit{module of relations} between the $s_i$ is of finite type.
\end{enumerate}
the subcategory of coherent $\mathcal{O}_X$-modules is denoted by $\textbf{Coh(X)}$. It is an abelian category.
\end{definition}

\begin{proposition}
Any coherent $\mathcal{O}_X$-module is of \textit{finite presentation}, that is, we can find an exact sequence of the form 

\begin{equation}
\begin{tikzcd}
    \mathcal{O}_{U}^{\oplus n} \ar{r} & \mathcal{O}_{U}^{\oplus m} \ar{r} & \left. \mathscr{F} \right|_{U} \ar{r} & 0.
\end{tikzcd}
\end{equation}
\end{proposition}

\begin{proof}
By the first condition of the definition, for every $x\in X$ there is a neighborhood $x \in U \subset X$ and a sequence of the form

\begin{equation}
\begin{tikzcd}
    \mathcal{O}_{U}^{\oplus n} \ar{r} & \left. \mathscr{F} \right|_{U} \ar{r} & 0.
\end{tikzcd}
\end{equation}
By the second condition of the definition, we can find an open neighborhood $x \in V \subset U \subset X$ and sections $s_i$, $i = 1,...,m$ that generate the kernel of the map $\mathcal{O}_{V}^{\oplus n} \to \left. \mathscr{F} \right|_{V}$. So, over $V$, we get the presentation

\begin{equation}
\begin{tikzcd}
    \mathcal{O}_{V}^{\oplus m} \ar{r} & \mathcal{O}_{V}^{\oplus n} \ar{r} & \left. \mathscr{F} \right|_{V} \ar{r} & 0
\end{tikzcd}
\end{equation}
and this concludes the proof.
\end{proof}

We will now discuss quasicoherent sheaves over \textit{affine schemes}. In the case of schemes, quasicoherent sheaves of $\mathcal{O}_X$-modules have a very explicit description. We begin with several equivalent constructions in an abstract commutative algebraic setting, all of which will be useful later at some point, and then later we apply these constructions to describe quasicoherent sheaves over affine schemes with the Zariski topology:

\noindent\textbf{The $\Tilde{M}$ functor}: Let $R$ be a ring (commutative with unity as usual), and consider the affine scheme $X = \operatorname{Spec} R$. The idea is to associate to every $R$-module $M$ a sheaf $\Tilde{M}$ over $X$, and to prove an equivalence between these new sheaves $\Tilde{M}$ and quasicoherent sheaves on $X$. We follow the construction in \cite{stacks-project} and \cite{ag2}.

\begin{lemma}
Let $R$ be a ring and $M$ be an $R$-module. Let $\alpha: R \rightarrow \Gamma(X, \mathcal{O}_X)$ be a ring homomorphism from $R$ to the ring of global sections of the sheaf $\mathcal{O}_X$ on $X$. Then the following three constructions give isomorphic sheaves of $\mathcal{O}_X$-modules:

\begin{enumerate}
    \item  Let $(X, \mathcal{O}_X)$ be a ringed space, and let $(\{*\}, R)$ be the ringed space consisting of the one point set and the constant sheaf of rings with value $R$. Let $\pi: (X, \mathcal{O}_X) \rightarrow (\{*\}, R)$ be a morphism of ringed spaces with $\pi: X \rightarrow \{*\}$ being the unique map of topological spaces, and $\pi^{\#}$ being the given map $\alpha: R \rightarrow \Gamma(X, \mathcal{O}X)$. Denote the first sheaf $\mathscr{F}_1$ to be the inverse image of the module $M$ by this map $\pi$ defined above:
    
    \begin{equation}
        \mathscr{F}_1 = \pi^{*}M
    \end{equation}
    \item Recall that by the universal property of free modules, a module is always the quotient of a free module. So choose a presentation of $M$
    
    \begin{equation}
    \begin{tikzcd}
    R^{\oplus I} \ar{r} & R^{\oplus J} \ar{r} & M \ar{r} & 0.
    \end{tikzcd}
    \end{equation}
    Set the second sheaf $\mathscr{F}_2$ to be the cokernel
    
    \begin{equation}
    \mathscr{F}_2 = \coker (\mathcal{O}_{X}^{\oplus I} \rightarrow \mathcal{O}_{X}^{\oplus J})
    \end{equation}
    of the map given on the $ith$ component by $\sum_{j} \alpha(r_{ji})$, where $\alpha$ is the given map and $r_{ji}$ are the coefficients of the matrix of the map
    
    \begin{equation}
    \begin{tikzcd}
    R^{\oplus I} \ar{r} & R^{\oplus J}
    \end{tikzcd}
    \end{equation}
    in the presentation of $M$;
    \item Define the third sheaf $\mathscr{F}_3$ to be the sheafification of the presheaf given by the rule
    \begin{equation}
        U \mapsto \mathcal{O}_X(U) \otimes_R M
    \end{equation}
    where the map of the $R$-modules sctructure $R \rightarrow \mathcal{O}_X(U)$ is given by the composition of the given map $\alpha$ with the restriction morphism $\mathcal{O}_X(X) \rightarrow \mathcal{O}_X(U)$.
\end{enumerate}
\end{lemma}

\begin{proof}
The isomorphism between $\mathscr{F}_1$ and $\mathscr{F}_3$ comes from the definition of the inverse image sheaf - $\pi^{*}$ is defined as the sheafification of the presheaf defined by the rule in the third construction, if we see $M$ as the constant presheaf with value $M$; the isomorphism between $\mathscr{F}_1$ and $\mathscr{F}_2$ comes from the right exactness of the functor $\pi^{*}$: applying $\pi^{*}$ to the sequence in the second construction, we obtain an exact sequence

\begin{equation}
\begin{tikzcd}
   \pi^{*}(R^{\oplus I}) \ar{r} & \pi^{*}(R^{\oplus J}) \ar{r} & \pi^{*}M \ar{r} & 0.
\end{tikzcd}
\end{equation}
Since, $\pi^{*}$ commutes with all colimits, it commutes with direct sums, and the above sequence becomes

\begin{equation}
\begin{tikzcd}
   (\pi^{*}R)^{\oplus I} \ar{r} & (\pi^{*}R)^{\oplus J} \ar{r} & \pi^{*}M \ar{r} & 0.
\end{tikzcd}
\end{equation}
Finally, by definition, $\pi^{*}R = \pi^{-1}R \otimes_{\pi^{-1}R} \mathcal{O}_X = \mathcal{O}_X$, and so the sequence becomes

\begin{equation}
\begin{tikzcd}
   \mathcal{O}_X^{\oplus I} \ar{r} & \mathcal{O}_X^{\oplus J} \ar{r} & \pi^{*}M \ar{r} & 0
\end{tikzcd}
\end{equation}
and so $\pi^{*}M$ is the cokernel defined in the second construction, and this completes the proof.
\end{proof}

The sheaf $\mathscr{F}_M \cong \mathscr{F}_1 \cong \mathscr{F}_2 \cong \mathscr{F}_3$ constructed above is called the \textit{sheaf associated to the module $M$ and the ring map $\alpha$}. If $\Gamma(X, \mathcal{O}_X) = R$ and $\alpha$ is the identity, we simply call $\mathscr{F}_M$ the \textit{sheaf associated to the module $M$}.

\begin{proposition}\label{modprop}
The sheaf $\mathscr{F}_M$ associated to an $R$-module $M$ has the following properties:

\begin{enumerate}
    \item $\mathscr{F}_M$ is quasicoherent;
    \item the constructions induce a functor from \textbf{Mod$_R$} to \textbf{QCoh($X$)} that commutes with arbitrary colimits;
    \item for any $x \in X$, $\mathscr{F}_{M,x} = \mathcal{O}_{X,x} \otimes_R M$, functorial in $M$;
    \item given any $\mathcal{O}_X$-module $\mathscr{G}$, we have
    \begin{equation}
        \Hom _{\mathcal{O}_X}(\mathscr{F}_M, \mathscr{G}) \cong \Hom _{R}(M, \Gamma(X, \mathscr{G}))
    \end{equation}
    where the $R$-module structure on $\Gamma(X, \mathscr{G})$ comes from the $\Gamma(X, \mathcal{O}_X)$-module structure via $\alpha$.
\end{enumerate}
\end{proposition}

\begin{proof}
Item $(1)$ follows directly from the definition and the construction of $\mathscr{F}_2$: we explicitly constructed our sheaf as the cokernel of a map between locally free sheaves, and so it is quasicoherent; item $(2)$ follows directly from the construction of $\mathscr{F}_1$, since $\pi^{*}$ is a functor from \textbf{Mod$_R$} to \textbf{QCoh($X$)} that commutes with arbitrary colimits; item $(3)$ follows directly from the construction of $\mathscr{F}_1$ and $\mathscr{F}_3$, and the facts that the former is the sheafification of the latter and the stalk of the tensor product is the tensor product of the stalks. We then have

\begin{equation}
    \mathscr{F}_{M,x} = (\pi^{*}M)_x = (\pi^{-1}M \otimes_{\pi^{-1} R} \mathcal{O}_X)_x = (M \otimes_{R} \mathcal{O}_X)_x = M_x \otimes_{R_x} \mathcal{O}_{X,x} = M \otimes_R \mathcal{O}_{X,x} = \mathcal{O}_{X,x} \otimes_R M.
\end{equation}
For item $(4)$, on the left hand side we use the definition of $\mathscr{F}_1 = \pi^{*}M$, and on the right hand side we use the fact that, by definition, the pushforward of any sheaf to the one-point space is the constant sheaf with value equal to the global sections of the original sheaf, and so for any $\mathcal{O}_X$-module $\mathscr{G}$, $\pi_{*}\mathscr{G} = \Gamma(X, \mathscr{G})$. The result follows from the adjointness of $\pi^{*}$ and $\pi_{*}$.
\end{proof}

\begin{theorem}\label{qcohaffine1}
Let $(X, \mathcal{O}_X)$ be a locally ringed space, and suppose that $x \in X$ has a \textbf{fundamental system of quasicompact neighborhoods}.\footnote{The term quasicompact is used by algebraic geometers to mean compact in the non-Hausdorff setting, i.e., every open cover of the space admits a finite refinement.} Let $\mathscr{F}$ be a quasicoherent sheaf of $\mathcal{O}_X$-modules on $X$. Then, there exists an open neighborhood $x \in U \subset X$ such that on $(U, \mathcal{O}_U)$, $\left. \mathscr{F} \right|_{U} \cong \mathscr{F}_M$ for some $\Gamma(U, \mathcal{O}_U)$-module $M$.
\end{theorem}

\begin{proof}
The idea of the proof is as follows: we take a local presentation of our $\mathcal{O}_X$-module and construct the module $M$ explicitly, by taking it to be the cokernel of the induced map of modules between the modules of sections $\Gamma(U, \mathcal{O}_{X}^{\oplus J})$ and $\Gamma(U, \mathcal{O}_{X}^{\oplus I})$, for some suitable neighborhood $U$. The problem is, the modules $\Gamma(U, \mathcal{O}_{X}^{\oplus J})$ and $\Gamma(U, \mathcal{O}_{X})^{\oplus J}$ need not be equal, or equivalently, the map

\begin{equation}
    \Gamma(U, \mathcal{O}_{X})^{\oplus J} \rightarrow \Gamma(U, \mathcal{O}_{X}^{\oplus J})
\end{equation}
need not be bijective. This is where being able to find a quasicompact neighborhood is important - if $U$ is quasicompact, the above map is bijective, and this is what we need to prove:

First, consider $x \in U \subset X$ a quasicompact neighborhood of $x$, and the map

\begin{equation}
\Gamma(U, \mathcal{O}_{X})^{\oplus I} \rightarrow \Gamma(U, \mathcal{O}_{X}^{\oplus I}).
\end{equation}
Let $s \in \Gamma(U, \mathcal{O}_{X}^{\oplus I})$. Then, there exists an open covering $U_j$ of $U$ such that $\left. s \right|_{U_j}$ is a \textit{finite} sum $\sum_{i \in I_j} s_{ji}$, with $s_{ji} \in \mathscr{F}_i (U_j)$. Since $U$ is quasicompact, there exists a finite refinement of the covering $U_j$, and we can assume that the index set $J$ is finite. Set $I' = \cup_{j \in J} I_j$. This is a finite set, and $s \in \Gamma(U, \mathcal{O}_{X}^{\oplus I'})$. Then since finite direct sums are equal to finite direct products, the bijection follows.

Now, we try to apply this idea to our original problem. Since our sheaf is quasicoherent, we can work locally and replace $X$ by some open neighborhood $U$ of $x$, and assume that it is isomorphic to the cokernel of some map

\begin{equation}
\begin{tikzcd}
   \mathcal{O}_X^{\oplus J} \ar{r} & \mathcal{O}_X^{\oplus I}.
\end{tikzcd}
\end{equation}
Let $x \in E \subset X$ be a quasicompact neighborhood of $x \in X$. Let $x \in U \subset E$ be an open neighborhood of $x$ in $E$. Consider the map

\begin{equation}
\begin{tikzcd}
   \Gamma(U, \mathcal{O}_X^{\oplus J}) \ar{r} & \Gamma(U, \mathcal{O}_X^{\oplus I}),
\end{tikzcd}
\end{equation}
and denote by $s_j \in \Gamma(U, \mathcal{O}_X^{\oplus I})$ the image of $1 \in \mathcal{O}_X$ corresponding to the summand indexed by $j \in J$. Then, there exists a finite collection of open sets $U_{jk}$ indexed by $k \in K_j$ such that:

\begin{enumerate}
    \item $E \subset \bigcup_{k \in K_j} U_{jk}$;
    \item each restriction $\left. s_j \right|_{U_{jk}}$ is a finite sum $\sum_{i \in I_{jk}} f_{jki}$, with $I_{jk} \subset I$ finite, and $f_{jki} \in \mathcal{O}_X$, in the summand indexed by $i \in I$.
\end{enumerate}
Again, we set $I'_{j} = \cup_{k \in K_{j}} I_{jk}$. This is a finite set. Since $U \subset E \subset \bigcup_{k \in K_j} U_{jk}$, the section $\left. s_j \right|_{U}$ is a section of the finite direct sum $\mathcal{O}_{X}^{\oplus I'_{j}}$, and so it is a finite sum $\sum_{i \in I_{j}} f_{ij}$, with $f_{ij} \in \Gamma(U, \mathcal{O}_U)$.

Finally, we can define $M$: define it as the cokernel of the map

\begin{equation}
\begin{tikzcd}
   \Gamma(U, \mathcal{O}_U)^{\oplus J} \ar{r} & \Gamma(U, \mathcal{O}_U)^{\oplus I},
\end{tikzcd}
\end{equation}
given by the matrix with the entries $(f_{ij})$. By the construction $\mathscr{F}_2$, the sheaf $\mathscr{F}_M$ associated to the module $M$ has the same local presentation as our original quasicoherent sheaf $\mathscr{F}$, and so $\left. \mathscr{F} \right|_{U} \cong \mathscr{F}_M$.
\end{proof}

Before we explicitly construct the sheaves associated to modules over some affine scheme $\operatorname{Spec} R$ we will need a result in sheaf theory, that is used in the construction of the structure sheaf $\mathcal{O}_{\operatorname{Spec} R}$ on $\operatorname{Spec} R$:

\begin{lemma}\label{sheafonabasis}
Let $X$ be a topological space, and $\mathcal{B}$ be a base for the topology on $X$. Let $\mathscr{F}$ be a sheaf on this base. Then there exists a sheaf $\mathscr{F}^{ext}$ such that $\mathscr{F}(U) = \mathscr{F}^{ext}(U)$ for all $U \in \mathcal{B}$, and the restriction morphisms for elements of $\mathcal{B}$ are the same as those of $\mathscr{F}$.
\end{lemma}

In other words, given an association that is a \textbf{sheaf} if restricted to some \textbf{base} of the topology, we can extend it to a sheaf on the whole space.

\begin{proof}[Idea of proof:]
Define $\mathscr{F}^{ext}(U)$ as $(s_x)_{x \in U} \in \prod_{x \in U} \mathscr{F}_x$, satisfying the condition that for every $x \in U$, there exists $V \in \mathcal{B}$ such that $x \in V \subset U$ and $\sigma \in \mathscr{F}(V)$ such that for all $y \in V$, we have $s_y = \left. \sigma \right|_{V}$ in $\mathscr{F}_y$. We can than check that it agrees with $\mathscr{F}$ when restricted to $\mathcal{B}$, that it is indeed a sheaf, and that it extends the sheaf to the rest of the space (\cite{stacks-project}, Lemma 6.30.5, Lemma 6.30.6).
\end{proof}

It is also worth noting that, by construction $\mathscr{F}_x = \mathscr{F}_{x}^{ext}$, because the base for the topology is also in particular a fundamental system of neighborhoods of every $x \in X$. So now, let $R$ be a ring, $M$ an $R$-module, and consider the affine scheme $\operatorname{Spec} R$. Let $D(f)$ be a distinguished open set, and recall that \textbf{distinguished opens form a base for the Zariski topology on} $\operatorname{Spec} R$. Define

\begin{equation}
    \Tilde{M}(D(f)) = M_f
\end{equation}
that is, on the open set $D(f)$, the (a priori) presheaf $\Tilde{M}$ has value the localized module $M_f$. Lets compute the stalk $\Tilde{M}_x$, for $x \in X$ corresponding to the prime ideal $\mathfrak{p}$ of $R$. By definition, we have

\begin{align*}
    \Tilde{M}_x &= \varinjlim_{x \in U,\text{ $U$ dist. open}} \Tilde{M}(U)\\
                &= \varinjlim_{f \in R,\text{ } f(x) \neq 0 } M_f \\
                &= \varinjlim_{f \in R \setminus \mathfrak{p}} M_f \\
                &= M_{\mathfrak{p}}.
\end{align*}

Note in particular that from this definition, it follows that the global sections

\begin{equation}
\Tilde{M}(X) = \Gamma(X, \Tilde{M}) = M.
\end{equation}
Now, we have to check that $\Tilde{M}$ is indeed a sheaf on the distinguished open sets. For ease of notation, denote the distinguished open set $D(f) \subset X$ by $X_f$:

\begin{lemma}
Suppose $X_f = \cup_{i = 1}^{N} X_{g_i}$. Then

\begin{enumerate}
    \item if $\frac{b}{f^k} \in M_f$ maps to $0$ in each localization $M_{g_i}$, then $\frac{b}{f^k}$;
    \item if $\frac{b_i}{g_{i}^{k_i}} \in M_{g_i}$ is a set of elements such that $\frac{b_i}{g_{i}^{k_i}} = \frac{b_j}{g_{j}^{k_j}}$ in $M_{g_{i}g_{j}}$, then there exists $\frac{b}{f^k} \in M_f$ which maps to $\frac{b_i}{g_{i}^{k_i}}$ in each $i$.
\end{enumerate}
\end{lemma}

Note that this are just the identity and gluability axioms stated for the distinguished open sets as a basis for the Zariski topology.

\begin{proof}
This proof is pure algebraic manipulation. Start by noticing that our hypothesis implies that, for some $m \ge 1$, we have

\begin{equation}
    f^m = \sum a_i g_i
\end{equation}
for some $a_i \in R$. Raising this equation to some power, and factoring out the highest powers, we see that for each $n$ we get $m'$ and $a'_i$ such that

\begin{equation}
    f^{\prime m} = \sum a'_i g_{i}^{n}.
\end{equation}
Now, to prove the first part, suppose $\frac{b}{f^k} \in M_f$ is $0$ in $M_{g_i}$ for all $i$. Since for all $i$ we have that $X_{g_i} \subset X_f$, then for some $a \in R$, $m \ge 1$ we have that $g_{i}^m = af$, and so for large enough $n$, $g_{i}^{n}b = 0$ for all $i$. We then replace in last equation and get

\begin{equation}
    f^{m'}b = \sum a'_i (g_{i}^{n}b) = 0
\end{equation}
and so $\frac{b}{f^k} = 0$ in $M_f$.

To prove the second part, begin by noticing that $\frac{b_i}{g_{i}^{k_i}} = \frac{b_j}{g_{j}^{k_j}}$ in $M_{g_{i}g_{j}}$ means that, by taking the "least common divisor", we get, for some ${m_{ij}} \ge 1$

\begin{equation}
    (g_{i}g_{j})^{m_{ij}} g_{j}^{k_j} b_i = (g_{i}g_{j})^{m_{ij}} g_{i}^{k_i} b_j.
\end{equation}
Set $M = max({m_{ij}}) + max(k_i)$. Then

\begin{equation}
    \frac{b_i}{g_{i}^{k_i}} = \frac{b_i}{g_{i}^{k_i}} \frac{g_{i}^M}{g_{i}^M} = \frac{b_i g_{i}^{M-k_{i}}}{g_{i}^M}
\end{equation}
and call the numerator on the last term $b'_i$. We then have,

\begin{align*}
    g_{j}^M b'_i &= g_{j}^M b_i g_{i}^{M-k_{i}} = (g_{j}^{M-k_{j}} g_{i}^{M-k_{i}}) g_{j}^{k_{j}} b_i \\
    &= (g_{j}^{M-k_{j}} g_{i}^{M-k_{i}}) g_{i}^{k_{i}} b_j \text{  by the first identity, since we took $M-k_i$ and $M-k_j$ to be } \ge m_{ij} \\
    &= g_{i}^{M} b'_j
\end{align*}
Now, take $k$ and $a'_i$ such that

\begin{equation}
    f^k = \sum a'_i g_{i}^M.
\end{equation}
Define $b$ as $b = \sum a'_j b'_j$. We claim that with this $b$ that we have just defined

\begin{equation}
    \frac{b}{f^k} = \frac{b'_i}{g_{i}^M}
\end{equation}
in $M_{g_i}$. In fact:

\begin{align*}
    g_{i}^M b &= \sum_{j} g_{i}^M a'_j b'_j \\
              &= \sum{j} g_{j}^M a'_j b'_i \\
              &= f^k b'_i
\end{align*}
and so $\frac{b}{f^k} = \frac{b'_i}{g_{i}^M}$. Finally, since $\frac{b'_i}{g_{i}^M} = \frac{b_i}{g_{i}^{k_{i}}}$ by construction, the proof is finished.
\end{proof}

And so $\Tilde{M}$ is a sheaf on the basis of distinguished open sets, and by the above Lemma, we can extend it to a sheaf on all open sets of $\operatorname{Spec} R$, which we will also denote $\Tilde{M}$. We now show that the sheaf $\Tilde{M}$ is isomorphic to the sheaf $\mathscr{F}_M$ constructed above, and so, it is quasicoherent.

\begin{theorem}\label{qcohaffine2}
Let $X = \operatorname{Spec} $ be an affine scheme, $M$ be an $R$-module. There is a canonical sheaf isomorphism between the sheaves $\mathscr{F}_M$ associated to the module $M$, and the sheaf $\Tilde{M}$ defined above. This isomorphism is functorial in $M$. In particular, $\Tilde{M}$ is quasicoherent.
\end{theorem}

\begin{proof}
By the last item of Proposition \ref{modprop}, there is a morphism $\mathscr{F}_M \rightarrow \Tilde{M}$, corresponding to the map $M \rightarrow \Gamma(X, \Tilde{M}) = M$. Recall by Proposition \ref{isostalks} that we only need to check if a sheaf morphism is an isomorphism if it is an isomorphism on the induced map on the stalks. Let $x \in X$ be the point corresponding to the prime ideal $\mathfrak{p}$. By the third item in Proposition \ref{modprop}, $\mathscr{F}_{M,x} = \mathcal{O}_{X,x} \otimes_R M$, and so $\mathscr{F}_{M,x} \rightarrow \Tilde{M}_x$ is equal to $\mathcal{O}_{X,x} \otimes_R M \rightarrow \Tilde{M}_x$, which in turn is equal to $R_{\mathfrak{p}} \otimes_R M \rightarrow M_{\mathfrak{p}}$, which is an isomorphism by definition. Hence, $\mathscr{F}_M \rightarrow \Tilde{M}$ is an isomorphism.
\end{proof}

It is worth noting that we also have a mapping property, which follows directly from the mapping property on $\mathscr{F}_M$: for any other sheaf of $\mathcal{O}_X$-modules $\mathscr{G}$, we have

\begin{equation}\label{qcohaffine3}
    \Hom _{\mathcal{O}_X} (\Tilde{M}, \mathscr{F}) = \Hom _{R} (M, \Gamma(X, \mathscr{F}))
\end{equation}
that is, a map $\Tilde{M} \rightarrow \mathscr{F}$ corresponds to its effect on the global sections.

We have proved that any quasicoherent sheaf over $\operatorname{Spec} R$ is of the form $\Tilde{M}$, for some $M$. The last thing that is left to do is, given some quasicoherent sheaf $\mathscr{F}$, determine to which module $M$ it is associated to:

\begin{proposition}\label{qcohaffine4}
Let $\mathscr{F}$ be a quasicoherent sheaf of $\mathcal{O}_X$-modules over $\operatorname{Spec} R$. Then $\mathscr{F}$ is isomorphic to the sheaf $\Tilde{M}$, where $M = \Gamma(X, \mathscr{F})$ is the module of global sections of $\mathscr{F}$.
\end{proposition}

\begin{proof}
Let $\operatorname{Spec} R = \bigcup_{i = 1}^n D(f_i)$ be an open covering by distinguished open sets on $\operatorname{Spec} R$. By Proposition \ref{qcohaffine1}, we have $R_{f_i}$-modules, say $M_i$ such that over $D_{f_i}$, we have isomorphisms $\varphi_i$ of $\mathcal{O}_X$-modules

\begin{equation}
    \varphi_i: \left. \mathscr{F} \right|_{D(f_i)} \rightarrow \mathscr{F}_{M_i}.
\end{equation}

We now have to look at the overlaps $D(f_i f_j)$. Since the $D(f_i)$ are also affine, by the same argument, we get

\begin{equation}
    g_{ij} = \left. \varphi_{j} \right|_{D(f_i f_j)} \circ  \left. \varphi_{i}^{-1} \right|_{D(f_i f_j)} : \left. \mathscr{F}_{M_i} \right|_{D(f_i f_j)} \rightarrow \left. \mathscr{F}_{M_j} \right|_{D(f_i f_j)}
\end{equation}
and the $g_{ij}$ satisfy the cocycle condition over $D(f_i f_j f_k)$. Again, since all of these overlaps are affine, Theorem \ref{qcohaffine2} tells us that $\mathscr{F}_{M_i} \cong \Tilde{M_i}$, and over $D(f_i f_j)$, $\mathscr{F}_{M_i}(D(f_i f_j))$ is isomorphic to the localized module $(M_i)_{f_j}$. Moreover, the adjunction estabilished in (\ref{qcohaffine3}) tells us that each $g_{ij}$ correspond to an $R_{f_i f_j}$-linear morphism

\begin{equation}
    \Tilde{g}_{ij}: (M_i)_{f_j} \rightarrow (M_j)_{f_i}
\end{equation}
which also satisfy the cocycle condition over triple intersections $D(f_i f_j f_k)$. So the only thing that we are left to do is to glue up all of these modules to an $R$-module $M$ such that each $M_i \cong M_{f_i}$.

Consider the map

\begin{equation}
    \bigoplus_{i=1}^n M_i \rightarrow \bigoplus_{i = 1, j = 1}^n (M_i)_{f_j}
\end{equation}
mapping $(m_1,...,m_n)$ to $\frac{m_i}{1} - \Tilde{g}_{ji}(\frac{m_j}{1})$ on the $i, j$ coordinate. Then the module 

\begin{equation}
\begin{tikzcd}
  0 \ar{r} & M \ar{r} & \bigoplus_{i=1}^n M_i \ar{r} & \bigoplus_{i = 1, j = 1}^n (M_i)_{f_j}
\end{tikzcd}
\end{equation}
which is the kernel of this map is the module satisfying this property (\cite{stacks-project}, Lemma 10.23.5).

So now we want to show that this module $M$ is isomorphic to $\Gamma(X, \mathscr{F})$. Every $m \in M$ maps to $m_i = \frac{m}{1} \in M_i$, by construction of $M$. By construction of each $M_i$, each $m_i$ corresponds to a section $s_i$ of $\mathscr{F}(D(f_i))$ via $\varphi_i^{-1}(s_i) = m_i$, since each $\varphi_i$ is an isomorphism, and moreover, this image agrees on the overlaps $\mathscr{F}(D(f_i f_j))$, since $g_{ij}(m_i) = m_j$. The gluability axiom then tells us that the collection of the sections $s_i$ glue to a global section $s$ of $\mathscr{F}$. Moreover, since the map $M \rightarrow \bigoplus_{i=1}^n M_i$ is $R$-linear, and each $\varphi_i$ is $R$-linear, the association that maps $m$ to $s$ that we just constructed gives an $R$-linear map

\begin{equation}
    M \rightarrow \Gamma(X, \mathscr{F}).
\end{equation}
Again, (\ref{qcohaffine3}) tells us that this corresponds to a map of sheaves of $\mathcal{O}_X$-modules

\begin{equation}
    \Tilde{M} \rightarrow \mathscr{F}
\end{equation}
and so it suffices to check that it is an isomorphism by checking locally. But locally over the distinguished open sets $D(f_i)$, the map is given by $\varphi_i^{-1}$ which is an isomorphism by construction, and this concludes the proof.
\end{proof}

We conclude that the functors given $M \mapsto \Tilde{M}$ and $\mathscr{F} \mapsto \Gamma(X, \mathscr{F})$ define an equivalence of categories between the categories of quasicoherent $\mathcal{O}_X$-modules and the categories of $R$-modules.

\subsection{Derived Functors}

\begin{center}
    \textit{
    “The question you raise 'how can such a formulation lead to computations' doesn't bother me in the least! Throughout my whole life as a mathematician, the possibility of making explicit, elegant computations has always come out by itself, as a byproduct of a thorough conceptual understanding of what was going on. Thus I never bothered about whether what would come out would be suitable for this or that, but just tried to understand - and it always turned out that understanding was all that mattered. [Often the route to solving problems is] to bring new concepts out of the dark."}
    
    \medskip

-Alexander Grothendieck, in a letter to Ronnie Brown, 1983.

Available at https://mathshistory.st-andrews.ac.uk/Biographies/Grothendieck/

\end{center}

We begin motivating our discussion about derived functors with a proposition:

\begin{proposition}
Let $X$ be a topological space and

\begin{equation}
\begin{tikzcd}
   0 \ar{r} & \mathscr{F} \ar{r}{f} & \mathscr{G} \ar{r}{g} & \mathscr{H} \ar{r} & 0
\end{tikzcd}
\end{equation}
be an \textit{exact sequence of sheaves}. Then the sequence

\begin{equation}
\begin{tikzcd}
   0 \ar{r} & \Gamma(X, \mathscr{F}) \ar{r}{\Tilde{f}} & \Gamma(X, \mathscr{G}) \ar{r}{\Tilde{g}} & \Gamma(X, \mathscr{H})
\end{tikzcd}
\end{equation}
is an exact sequence of abelian groups, i.e., the \textbf{global section functor $\Gamma(X, -)$ is left-exact}.
\end{proposition}

\begin{proof}
We have already discussed in the definition of injective morphisms of sheaves that an injective morphism of sheaves induces injective morphisms on sections. So we need to prove exactness at $\Gamma(X, \mathscr{G})$. First, note that the sequence induces an exact sequence on the stalks

\begin{equation}
\begin{tikzcd}
   0 \ar{r} & \mathscr{F}_{x} \ar{r}{f_{x}} & \mathscr{G}_{x} \ar{r}{g_{x}} & \mathscr{H}_{x} \ar{r} & 0.
\end{tikzcd}
\end{equation}
First, we show that $\Ima \Tilde{f} \subseteq \ker \Tilde{g}$. Computing on the stalks, we get

\begin{equation}
    (\Tilde{g} \circ \Tilde{f}(s))_{x} = g_{x} \circ f_{x}(s) = 0.
\end{equation}
where the first equality is by definition and the second is by the exactness of the induced sequence on the stalks. Since $x$ is arbitrary, we conclude that $g_{x} \circ f_{x} = 0$ and therefore $\Tilde{g} \circ \Tilde{f} = 0$, and so $\Ima \Tilde{f} \subseteq \ker \Tilde{g}$. For the other inclusion $\Ima \Tilde{f} \supseteq \ker \Tilde{g}$, let $r \in \ker \Tilde{g}$. So $\Tilde{g}(r) = 0$. Since the sequence on the stalks is exact and $\Ima f_{x} = \ker g_{x}$ for any $x$, we can find some $s_{x} \in \mathscr{F}_x$ such that $f_{x}(s_x) = r_{x} \in \ker g_{x} = \Ima f_{x}$. Choose a class representative of this $s_{x}$, and call it $(s^x, V^x)$, with $s^x \in \mathscr{F}(V^x)$. So $f(s^x)$ and $\rho_{XV^x}(r) = \left. r \right|_{V^x}$ are both elements of $\mathscr{G}(V^x)$ whose stalks at $x$ coincide, and therefore there exists some open neighborhood of $x$ in $V^x$, which we can assume without loss of generality to be $V^x$ itself, in which 
\begin{equation}
f(s^x) = \left. r \right|_{V^x} \in \mathscr{G}(V^x).
\end{equation}
Since for every $x \in X$ we have such a $V^x$, they form a cover of $X$, and so there exists $s^x \in \mathscr{F}(V^x)$ and $s^y \in \mathscr{F}(V^y)$ that on $\mathscr{F}(V^{x} \cap V^{y})$ are mapped by $\Tilde{f}$ to the same $\left. r \right|_{V^x \cap V^y}$. Since $\Tilde{f}$ is injective as we have already discussed, this implies that $s^x = s^y$ in $\mathscr{F}(V^{x} \cap V^{y})$, and so we use the \textbf{gluability axiom} to find $s \in \mathscr{F}$ such that $\left. s \right|_{V^x} = s^x$ for each $x \in X$. The only thing that is left to do is to check that $\Tilde{f}(s) = r$. For this, note that $\rho_{XV^{x}}(\Tilde{f}(s) - r) = \left. (\Tilde{f}(s) - r) \right|_{V^x} = \Tilde{f}(s^x) - \left. r \right|_{V^x} = 0$ for every $x \in X$, and so by the \textbf{identity axiom}, $\Tilde{f}(s) - r = 0$. So $\Tilde{f}(s) = r$ and we are done.
\end{proof}

We would like to, in the spirit of Theorem \ref{lescohomology1} think about this sequence as the beginning of a long exact sequence, so that we can retrieve more information about the sheaves and about the topology of the space. This is done with the machinery of \textbf{derived functors}. Besides the usual references of $\cite{hartshorne}, \cite{vakil}$ for this section we also follow $\cite{ahtopos}$.

\medskip

\noindent\textbf{Abelian Categories and Exact Functors:} So far, through our discussion we have encountered some abelian categories and functors that were either left or right exact (or both). We give a quick review of the main properties of such objects:


\begin{definition}
    Let $\mathcal{C}$ be a category and let $B$ and $C$ be objects of $\mathcal{C}$. We say that a morphism $f:B \rightarrow C$ is a \textit{monomorphism} if for any object $A$ and any two morphisms $g_1 : A \rightarrow B$ and $g_2: A \rightarrow B$ such that $f \circ g_1 = f \circ g_2$ satisfy $g_1 = g_2$. Equivalently, that is the same as the map $\Hom (A, B) \rightarrow \Hom (A, C)$ being injective.
\end{definition}

Dually, we define the notion of \textit{epimorphism}, by reversing all the arrows in the above definition.

\begin{definition}
We say that a category $\mathcal{C}$ is \textit{additive} if it satisfies the following axioms:

\begin{enumerate}
    \item For each pair of objects $A$ and $B$ of $\mathcal{C}$, we have that $\Hom (A, B)$ has the structure of an \textit{abelian group}, and the composition law is \textit{bilinear}: that is, if we write  the group $\Hom (A, B)$ additively, we have that, for $f, f' \in \Hom (A, B)$ and $g, g' \in \Hom (B, C)$
    
    \begin{equation}
        (f + f') \circ (g + g') = f \circ g + f \circ g' + f' \circ g + f' \circ g'
    \end{equation}
    is an element of $\Hom (A, C)$, that is, \textbf{composition distributes over the group operation};
    \item $\mathcal{C}$ has a \textit{zero object} denoted by $0$, i.e., an object that is both \textit{initial} and \textit{final};
    \item For each pair of objects $A$ and $B$ of $\mathcal{C}$, the \textit{product} of $A$ and $B$ exists, and consequently, any \textit{finite product} of elements of $\mathcal{C}$ exists.
\end{enumerate}
\end{definition}

\begin{remark}
    If $\mathcal{C}$ is an additive category, then
    \begin{enumerate}
        \item the opposite category $\mathcal{C}^{\circ}$ is an additive category;
        \item if $\mathcal{D}$ is any category, the category $\mathcal{C}^{\mathcal{D}}$ whose objects are the \textit{functors from $\mathcal{D}$ to $\mathcal{C}$} and the morphisms are \textit{natural transformations}, is an additive category.
    \end{enumerate}
\end{remark}

\begin{remark}
It follows that an additive category also has all \textit{finite coproducts}. The coproduct in an additive category is denoted by $\bigoplus$.
\end{remark}

\begin{definition}
    Let $\mathcal{C}$ be an additive category, $B$ and $C$ be objects of $\mathcal{C}$ and $f: B \rightarrow C$ be a morphism.
    
    \begin{enumerate}
        \item a \textit{kernel} of $f$ is a morphism $i: A \rightarrow B$ such that $f \circ i = 0$, and $i$ is universal with respect to this property - that is, for any other morphism $i':A' \rightarrow B$ such that $f \circ i' = 0$, there is a unique morphism $g: A' \rightarrow A$ such that $i' = i \circ g$. If the kernel of $f$ exists, we denote it by $\ker f \rightarrow B$;
        \item dually, a \textit{cokernel} of $f$ is a morphism $p: C \rightarrow D$ such that $p \circ f = 0$, and is universal with respect to this property - that is, for any other morphism  $p': C \rightarrow D'$ such that $p' \circ f = 0$, there is a unique morphism $g: D \rightarrow D'$ such that $p' = g \circ p$. If the cokernel of $f$ exists, we denote it by $C \rightarrow \coker f$;
        \item the \textit{image} of $f$ is the kernel of $C \rightarrow \coker f$, provided that the cokernel exists. That is, $\Ima f = \ker(C \rightarrow \coker f)$.
    \end{enumerate}
\end{definition}

\begin{definition}
    Let $\mathcal{C}, \mathcal{C}'$ be additive categories. A functor $F: \mathcal{C} \rightarrow \mathcal{C}'$ is said to be \textit{additive} if it preserves the group structure of $\Hom$, that is, the map $\Hom _{\mathcal{C}}(A, B) \rightarrow \Hom _{\mathcal{C}'} (F(A), F(B))$ is a group homomorphism.
\end{definition}

\begin{definition}
    An \textit{abelian category} $\mathcal{C}$ is an additive category satisfying the additional three axioms:
    \begin{enumerate}
        \item every morphism has a kernel and a cokernel;
        \item every \textit{monomorphism} is the \textit{kernel of its cokernel};
        \item every \textit{epimorphism} is the \textit{cokernel of its kernel}.
    \end{enumerate}
\end{definition}

Examples of abelian categories are \textbf{Ab}, the category of abelian groups, and more generally \textbf{Mod$_{R}$} the category of modules over some ring $R$. For examples that we have already encountered, we have the categories $\textbf{Sh(X)}$ of sheaves of abelian groups over some topological space $X$, \textbf{Mod$_{\mathcal{O}_X}$} of sheaves of $\mathcal{O}_X$-modules on some ringed space $(X, \mathcal{O}_X)$, \textbf{QCoh($X$)} of quasicoherent sheaves of $\mathcal{O}_X$-modules and $\textbf{Coh(X)}$ of coherent sheaves of $\mathcal{O}_X$-modules.

\begin{remark}
Similarly to the case of additive categories, a category $\mathcal{C}$ is abelian if, and only if the opposite category $\mathcal{C}^{\circ}$ is abelian, and so we make use of duality arguments freely. 
\end{remark}

\begin{definition}
    Let $\mathcal{C}, \mathcal{C}'$ be abelian categories, and let
    
    \begin{equation}
    \begin{tikzcd}
    0 \ar{r} & A \ar{r} & B \ar{r} & C \ar{r} & 0.
    \end{tikzcd}
    \end{equation}
    be an exact sequence of elements of $\mathcal{C}$. An \textbf{additive functor} $F: \mathcal{C} \rightarrow \mathcal{C}'$ is said to be
    
    \begin{enumerate}
        \item \textit{left exact} if the sequence
            \begin{equation}
            \begin{tikzcd}
            0 \ar{r} & F(A) \ar{r} & F(B) \ar{r} & F(C)
            \end{tikzcd}
            \end{equation}
        is exact in $\mathcal{C}'$;
        \item \textit{right exact} if the sequence
            \begin{equation}
            \begin{tikzcd}
            F(A) \ar{r} & F(B) \ar{r} & F(C) \ar{r} & 0
            \end{tikzcd}
            \end{equation}
        is exact in $\mathcal{C}'$;
        \item \textit{exact} if the sequence
            \begin{equation}
            \begin{tikzcd}
            0 \ar{r} & F(A) \ar{r} & F(B) \ar{r} & F(C) \ar{r} & 0.
            \end{tikzcd}
            \end{equation}
        is exact in $\mathcal{C}'$.
    \end{enumerate}
\end{definition}

\begin{definition}
    We say that an object $I$ of an abelian category $\mathcal{C}$ is \textit{injective} if, for any monomorphism $f: A \rightarrow B$ and \textbf{any} morphism $\alpha: A \rightarrow I$, there exists a morphism $\beta: B \rightarrow I$ such that $\beta \circ f = \alpha$, i.e.,
    
    \begin{center}

    \begin{tikzpicture}[>=triangle 60]
    \matrix[matrix of math nodes,column sep={60pt,between origins},row
    sep={60pt,between origins},nodes={asymmetrical rectangle}] (s)
    {
    |[name=0]| 0 &|[name=A]| A &|[name=B]| B\\
    &|[name=I]| I\\
    };
    \draw[overlay,->, font=\scriptsize,>=latex] 
             (0) edge node[auto] {\(\)} (A)
             (A) edge node[auto] {\(f\)} (B)
             (A) edge node[auto] {\(\alpha\)} (I)
             (B) edge node[auto] {\(\exists \beta\)} (I)
    ;
    
    ;

    \end{tikzpicture}

    \end{center}
    is commutative. Dually, we define a \textit{projective} object $P$ in $\mathcal{C}$
    
    \begin{center}

    \begin{tikzpicture}[>=triangle 60]
    \matrix[matrix of math nodes,column sep={60pt,between origins},row
    sep={60pt,between origins},nodes={asymmetrical rectangle}] (s)
    {
    &|[name=P]| P\\
    |[name=B]| B &|[name=C]| C &|[name=0]| 0\\ 
    };
    \draw[overlay,->, font=\scriptsize,>=latex] 
             (B) edge node[auto] {\(f\)} (C)
             (C) edge node[auto] {\(\)} (0)
             (P) edge node[auto][swap] {\(\exists \delta\)} (B)
             (P) edge node[auto] {\(\gamma\)} (C)
    ;
    
    ;

    \end{tikzpicture}

    \end{center}
\end{definition}

\begin{proposition}\label{hominj}
Let $\mathcal{C}$ be an abelian category, and let $I$ be an object of $\mathcal{C}$. Then $I$ is injective if, and only if the functor $\Hom _{\mathcal{C}} (-, I)$ is exact.
\end{proposition}

\begin{proof}
Let

    \begin{equation}
    \begin{tikzcd}
    0 \ar{r} & A \ar{r}{f} & B \ar{r}{g} & C \ar{r} & 0.
    \end{tikzcd}
    \end{equation}
be an exact sequence in $\mathcal{C}$. Recall that the functor $\Hom _{\mathcal{C}} (-, I)$ is a left exact functor $\mathcal{C}^{\circ} \rightarrow \textbf{Ab}$.

$(\Longrightarrow)$ Let $I$ be injective. We need to prove that the induced map $f_{*}: \Hom _{\mathcal{C}} (B, I) \rightarrow \Hom _{\mathcal{C}} (A, I)$ is surjective. Since $I$ is injective and $f$ is a monomorphism by hypothesis, given any $\alpha \in \Hom _{\mathcal{C}} (A, I)$, there exists some $\beta \in \Hom _{\mathcal{C}} (B, I)$ such that $\alpha = \beta \circ f = f_{*} \beta$, and so $f_{*}$ is surjective.

$(\Longleftarrow)$ Now, let $\Hom _{\mathcal{C}} (-, I)$ be exact, and let $\alpha \in \Hom _{\mathcal{C}} (A, I)$. We have that

    \begin{equation}
    \begin{tikzcd}
    0 \ar{r} & Hom _{\mathcal{C}} (C, I) \ar{r}{g_{*}} & Hom _{\mathcal{C}} (B, I) \ar{r}{f_{*}} & Hom _{\mathcal{C}} (A, I) \ar{r} & 0.
    \end{tikzcd}
    \end{equation}
is exact, and so $f_{*}$ is surjective. So for $\alpha \in Hom _{\mathcal{C}} (A, I)$, there exists some $\beta \in Hom _{\mathcal{C}} (B, I)$, such that $\alpha = f_{*}\beta = \beta \circ f$, and therefore $I$ is injective.
\end{proof}

\begin{proposition}\label{dirsuminj}
Let $\mathcal{C}$ be an abelian category, and let $I_1$ and $I_2$ be two objects. Then their direct sum $I_1 \oplus I_2$ is injective if, and only if $I_1$ and $I_2$ are both injective.
\end{proposition}

\begin{proof}
$(\implies)$ Suppose that $I_1 \oplus I_2$ is injective. Let $f: A \rightarrow B$ be a monomorphism in $\mathcal{C}$, and let $\alpha: A \rightarrow I_1$ be any morphism. Consider the canonical map $s_1: I_1 \rightarrow I_1 \oplus I_2$. Since $I_1 \oplus I_2$ is injective by hypothesis, there exists $\beta \in \Hom (B, I_1 \oplus I_2)$ such that $s_1 \circ \alpha = \beta \circ f$. Composing with the projection in the first coordinate $p_1: I_1 \oplus I_2 \rightarrow I_1$, we get that $p_1 \circ s_1 \circ \alpha = p_1 \circ \beta \circ f$. Then, the left hand side is $\alpha$, and $p_1 \circ \beta \in \Hom(B, I_1)$ is such that $\alpha = p_1 \circ \beta \circ f$, and so $I_1$ is injective. Doing a symmetrical argument for $I_2$, we conclude that it is also injective.

$(\Longleftarrow)$ Now suppose that $I_1$ and $I_2$ are injective. Let $f: A \rightarrow B$ be a monomorphism in $\mathcal{C}$, and let $\alpha: A \rightarrow I_1 \oplus I_2$ be any morphism. Since $I_1$ is injective, composing with the projection in the first coordinate $p_1: I_1 \oplus I_2 \rightarrow I_1$, $p_1 \circ \alpha$, we get a morphism in $\Hom (A, I_1)$, and so there is some $\beta_1 \in \Hom (B, I_1)$ such that $p_1 \circ \alpha = \beta_1 \circ f$. Doing a symmetrical argument for $I_2$, we obtain $\beta_2 \in \Hom (B, I_2)$ such that $p_2 \circ \alpha = \beta_2 \circ f$. We then have two morphisms, $\beta_1$ and $\beta_2$, and so by the universal property of the coproduct, there exists some $\beta \in \Hom (B, I_1 \oplus I_2)$ such that both $\beta_1$ and $\beta_2$ factor through it. So we have a morphism $\beta: B \rightarrow I_1 \oplus I_2$ such that $\alpha = \beta \circ f$, and $I_1 \oplus I_2$ is injective.
\end{proof}

\begin{remark}
Note that the implication $(\Longleftarrow)$ above is true for any family of injective objects, that is $\{I_j\}_{j \in J}$ is such that for every $j \in J$, $I_j$ is injective, then $I_{j}^{\oplus J}$ is injective.
\end{remark}

Recall that given an abelian category $\mathcal{C}$, we can construct the category \textbf{Com}$_{\mathcal{C}}$ of complexes in $\mathcal{C}$.

\begin{definition}
    Let $\mathcal{C}$ be an abelian category and $A$ be an object in $\mathcal{C}$. Let $I^{\bullet} \in \textbf{Com}_{\mathcal{C}}$ be such that $I^i = 0$ for all $i < 0$, and let $A \rightarrow I^0$ be a morphism. We say that $I^{\bullet}$ is a \textit{resolution} of $A$, if
    
    \begin{equation}
    \begin{tikzcd}
    0 \ar{r} & A \ar{r} & I^0 \ar{r} & I^1 \ar{r} & I^2 \ar{r} &{\ldots}
    \end{tikzcd}
    \end{equation}
    is exact. Furthermore, if each $I^i$ is injective, we say that $I^{\bullet}$ is an \textit{injective resolution} of A.
\end{definition}

\begin{definition}
    We say that an abelian category $\mathcal{C}$ has \textit{enough injectives} if for every object $A$ of $\mathcal{C}$ there exists a monomorphism $A \rightarrow I$ with $I$ injective.
\end{definition}

\begin{theorem}
Let $F: \mathcal{C} \rightarrow \mathcal{C}'$ be a left exact functor between two abelian categories, and suppose that the source category $\mathcal{C}$ has enough injectives. Then

\begin{enumerate}
    \item For every $i \ge 0$, there exists additive functors $R^{i}F: \mathcal{C} \rightarrow \mathcal{C}'$;
    \item $F \cong R^{0}F$ is an isomorphism;
    \item For each short exact sequence
    
    \begin{equation}
    \begin{tikzcd}
    0 \ar{r} & A_1 \ar{r} & A_2 \ar{r} & A_3 \ar{r} & 0
    \end{tikzcd}
    \end{equation}
    and each $i \ge 0$, there exists a morphism $\delta^i: R^{i}F(A_3) \rightarrow R^{i+1}F(A_1)$ such that the sequence
    
    \begin{equation}
    \begin{tikzcd}
    {\ldots} \ar{r} & R^{i}F(A_2) \ar{r} & R^{i}F(A_3) \ar{r}{\delta^i} & R^{i+1}F(A_1) \ar{r} & R^{i+1}F(A_2) \ar{r} & {\ldots}
    \end{tikzcd}
    \end{equation}
    is exact.
\end{enumerate}
\end{theorem}

The last item says that, for each short exact sequence, we get connecting morphisms that induce a long exact sequence on the $R^{i}F$. We will not present the proof of the theorem in this text, because it needs some more homological algebra to discuss the homotopic invariance of the constructed functors, which we will not develop here (for the proof, we refer the reader to \cite{ahtopos}, 1, Teorema 1.4.9). Despite not presenting the full proof here, it deserves some comments: let $A$ be an arbitrary object of $\mathcal{C}$, and let $I^{\bullet}$ be an injective resolution for $A$, that we denote by $I^{\bullet}_A$. Since $F$ is left exact, $F(I^{\bullet}_A)$ is a complex in $\mathcal{C}'$. For each $i \ge 0$, define $R^{i}F$ to be the cohomology at the $i$ position of the complex $F(I^{\bullet}_A)$, that is

\begin{equation}
    R^{i}F(A) = \mathrm{H}^i(F(I^{\bullet}_A)).
\end{equation}

Looking at this definition, the first thing that one would need to check is that it does not depend on the choice of injective resolution $I^{\bullet}$. This requires a lemma in homological algebra that states that if $I^{\bullet}$ and $J^{\bullet}$ are two injective resolutions of the same object, then they are homotopy-equivalent. For a reference, cf. \cite{ahtopos} 1, Teorema 1.4.6 or \cite{weibel_1994}, Comparison Theorem 2.2.6. One would then proceeds to argue that homotopy equivalent complexes induce isomorphic cohomology theories, and so the definition of the $R^{i}F$ independs of $I^{\bullet}$.

\begin{definition}
    The functors $R^{i}F$ defined above are called the \textit{right derived functors} of the left exact functor $F$. Dually, if $G$ is a right exact functor, we construct the \textit{left derived functors} of $G$, and denote them $L_{n}G$.
\end{definition}

Note that if the functor $F$ is also right exact (i.e., is an exact functor), then 

\begin{equation}
    R^{i}F = \mathrm{H}^i(F(I^{\bullet})) = 0
\end{equation}
for all $i > 0$. To see this, let $Q$ be an injective object such that there is a monomorphism $A \rightarrow Q$, and take the resolution to be

\begin{equation}
\begin{tikzcd}
    0 \ar{r} & A \ar{r} & Q \ar{r} & Q \ar{r} & 0.
\end{tikzcd}
\end{equation}

\begin{example}
Let $A$ be a ring, $M$ be an $A$-module. We know that the functor $M \otimes_{A} -$ is a \textit{right exact functor} $\textbf{Mod$_{A}$} \rightarrow \textbf{Mod$_{A}$}$, and so we can construct its left derived functors, as discussed above. These functors are denoted $\operatorname{Tor}^{A}_{i}(M, -)$. An interesting application of the $\operatorname{Tor}^{A}_{i}(M, -)$ functors is to find the obstruction for the \textit{flatness} of an $A$-module. 

Recall that an $A$-module $M$ if \textbf{flat} if the functor $M \otimes_{A} -$ is exact, i.e., if it preserves monomorphisms (since it is always right exact, and so preserves epimorphisms anyway). The notion of flatness is an algebraic condition with deep geometric consequences. As was discussed above, we already know that if $M$ is a flat $A$-module then it follows that for any $A$-module $N$, $\operatorname{Tor}^{A}_{i}(M, N) = 0$ for all $i > 0$. Recall the commutative algebra result of the ideal-theoretic criterion for flatness:

\begin{theorem}
Let $M$ be an $A$-module. Then $M$ is flat if, and only if for every ideal $I$ of $A$, the canonical morphism $I \otimes_{A} M \rightarrow A \otimes_{A} M$ is an injection, and therefore $I \otimes_{A} M \rightarrow IM$ is an isomorphism.
\end{theorem}
The proof of the criterion can be found for example in \cite{liu}, 1.2, Theorem 2.4 or \cite{matsumura_1987}, 3, Theorem 7.7. In fact, one can prove that it suffices to check the condition for \textit{finitely generated} ideals only (cf. \cite{vakil}, Unimportant remark 24.4.3), using a colimit argument (cf. \cite{rotman}, 3, Theorem 3.53, Theorem 3.54). We then have the following theorem, from \cite{matsumura_1987}:

\begin{theorem}
Let $M$ be an $A$-module. The following are equivalent

\begin{enumerate}
    \item $M$ is flat;
    \item for every $A$-module $N$, we have $\operatorname{Tor}^{A}_{1}(M, N) = 0$;
    \item $\operatorname{Tor}^{A}_{1}(M, A/I) = 0$ for every finitely generated ideal $I$.
\end{enumerate}
\end{theorem}

\begin{proof}
We have already discussed the first implication, and the second implication is obvious since $A/I$ is an $A$-module. So we're left to check that $(3) \implies (1)$. Consider the canonical short exact sequence

\begin{equation}
\begin{tikzcd}
    0 \ar{r} & I \ar{r} & A \ar{r} & A/I \ar{r} & 0.
\end{tikzcd}
\end{equation}
We then use the induced long exact sequence

\begin{equation}
\begin{tikzcd}
    \operatorname{Tor}^{A}_{1}(M, A/I) \ar{r} & \operatorname{Tor}^{A}_{0}(M, I) \ar{r} & \operatorname{Tor}^{A}_{0}(M, A) \ar{r} & \operatorname{Tor}^{A}_{0}(M, A/I) \ar{r} & 0.
\end{tikzcd}
\end{equation}
Since by hypothesis $\operatorname{Tor}^{A}_{1}(M, A/I) = 0$ and $\operatorname{Tor}^{A}_{0}(M, N)$ is $M \otimes_{A} N$, the above sequence becomes

\begin{equation}
\begin{tikzcd}
    \operatorname{Tor}^{A}_{1}(M, A/I) = 0 \ar{r} & I \otimes_{A} M \ar{r} & A \otimes_{A} M = M \ar{r} & M \otimes_{A} A/I \ar{r} & 0
\end{tikzcd}
\end{equation}
and so $I \otimes_{A} M \rightarrow A \otimes_{A} M$ is injective, and $M$ is flat.
\end{proof}
\end{example}

\begin{example}\label{univcoeffthm}
Dually, we know that the functor $\Hom(-,N)$ is left-exact. Its right derived functor are the $\operatorname{Ext}_{A}^i(-, N)$ functors. One application of the $\operatorname{Ext}_{A}^i$ functors is the Universal Coefficient Theorem in algebraic topology, which relates the singular cohomology groups with coefficients in different groups. It states that, for some coefficient group $G$ we have a short exact sequence

\begin{equation}
\begin{tikzcd}
    0 \ar{r} & \operatorname{Ext}_{\mathbb{Z}}^1(\mathrm{H}_{k-1}(X, \mathbb{Z}), G) \ar{r} & \mathrm{H}^{k}_{Sing}(X, G) \ar{r} & \Hom(\mathrm{H}_k(X, \mathbb{Z}), G) \ar{r} & 0.
\end{tikzcd}
\end{equation}

We also have the dual version, relating the singular homology groups with coeffficients in different groups, which states that the following sequence is exact

\begin{equation}
\begin{tikzcd}
    0 \ar{r} & \mathrm{H}_k(X, \mathbb{Z}) \otimes G \ar{r} & \mathrm{H}_k(X, G) \ar{r} & \operatorname{Tor}^{\mathbb{Z}}_{1}(\mathrm{H}_{k-1}(X, \mathbb{Z}), G) \ar{r} & 0.
\end{tikzcd}
\end{equation}
This last sequence tells us, for example, that if $G$ is a field of characteristic $0$, then it is a flat $\mathbb{Z}$-module, and so by last example, the group $\operatorname{Tor}^{\mathbb{Z}}_{1}(\mathrm{H}_{k-1}(X, \mathbb{Z}), G)$ is $0$. Therefore the vector spaces $\mathrm{H}_k(X, \mathbb{Z}) \otimes G$ and $\mathrm{H}_k(X, G)$ have the same dimension.
\end{example}


We now wish to define the cohomology groups $\mathrm{H}^i(X, \mathscr{F})$ of a space $X$ with values in a sheaf $\mathscr{F}$ as the right derived functors of the global section functor $\Gamma(X, -)$.

\begin{theorem}\label{eninj}
Let $X$ be a topological space. Then the category $\textbf{Sh(X)}$ of sheaves of abelian groups on $X$ has enough injectives.
\end{theorem}

For the proof of the above, we will need some preliminary results

\begin{lemma}
An abelian group $D$ is an injective object in \textbf{Ab} if, and only if $D$ is a \textit{divisible group}, that is, if it is divisible when seen as a $\mathbb{Z}$-module: for all $y \in D$ and every $n \in \mathbb{Z}$, there exists some $x \in D$ such that $y = nx$.
\end{lemma}

This is a well known fact in group theory, and its proof uses the Axiom of Choice, disguised as Zorn's Lemma.

\begin{lemma}
The category \textbf{Ab} = \textbf{Mod$_{\mathbb{Z}}$} has enough injectives.
\end{lemma}

\begin{proof}
Consider the $\mathbb{Z}$-module (abelian group) $\mathbb{Q}/\mathbb{Z}$. It obviously is a divisible $\mathbb{Z}$-module, and so by the above lemma, it is injective. For any abelian group $H$, denote by $H^{\lor}$ the group $\Hom_{\mathbb{Z}}(H, \mathbb{Q}/\mathbb{Z})$. Now let $G$ be any abelian group. There is a natural map $\psi: G \rightarrow G^{\lor \lor}$. Now, consider the projection from some free module $\bigoplus_{j \in J}\mathbb{Z} \rightarrow G^{\lor}$. Then, applying the contravariant functor $\Hom_{\mathbb{Z}}(-, \mathbb{Q}/\mathbb{Z})$, we get

\begin{equation}
    G^{\lor \lor} = \Hom_{\mathbb{Z}}(G^{\lor}, \mathbb{Q}/\mathbb{Z}) \rightarrow \Hom_{\mathbb{Z}}(\bigoplus_{j \in J}\mathbb{Z}, \mathbb{Q}/\mathbb{Z}) = (\mathbb{Q}/\mathbb{Z})^{\oplus J}.
\end{equation}
By Proposition \ref{dirsuminj}, $(\mathbb{Q}/\mathbb{Z})^{\oplus J}$ is injective. Since $\bigoplus_{j \in J}\mathbb{Z} \rightarrow G^{\lor}$ is a surjection,

\begin{equation}
    \Hom_{\mathbb{Z}}(G^{\lor}, \mathbb{Q}/\mathbb{Z}) \rightarrow \Hom_{\mathbb{Z}}(\bigoplus_{j \in J}\mathbb{Z}, \mathbb{Q}/\mathbb{Z})
\end{equation}
is an injection since $\mathbb{Q}/\mathbb{Z}$ is injective (Proposition \ref{hominj}), and so we're left to prove that the map $\psi: G \rightarrow G^{\lor \lor}$ is injective, and we will have injected our group $G$ in the injective $\mathbb{Z}$-module $(\mathbb{Q}/\mathbb{Z})^{\oplus J}$. The injectivity of $\psi$ follows from \textit{Baer's Criterion}, see cf. \cite{weibel_1994}, Chap. 2, Section 2.3.
\end{proof}

\begin{proof}[Proof of Theorem \ref{eninj}:]
We begin by defining a pair of adjoint functors

\begin{equation}
    x^{*}: \textbf{Sh(X)} \rightarrow \textbf{Ab}
\end{equation}
and

\begin{equation}
      x_{*}: \textbf{Ab} \rightarrow \textbf{Sh(X)}
\end{equation}
in the following way: given a sheaf of abelian groups $\mathscr{F} \in \textbf{Sh(X)}$, $x^{*}\mathscr{F}$ will be the stalk of $\mathscr{F}$ at $x$, that is, $x^{*}\mathscr{F} = \mathscr{F}_x$. Given an abelian group $G$, $x_{*}G$ will be the \textbf{skyscrapper sheaf} with value $G$, centered at $x$ (see Example \ref{skyscrapper}), which we denote $i_{x,*}G$ (this notation will soon become clear). We first prove that they are indeed adjoint: consider the inclusion of the point $x$ in $X$
\begin{equation}
    i_{x}: \{x\} \rightarrow X.
\end{equation}
Let's compute the pullback of $\mathscr{F}$ by the morphism $i_{x}$, that is, $i_{x}^{-1}\mathscr{F}$. By definition

\begin{equation}
i_{x}^{-1}\mathscr{F} = (U \mapsto \varinjlim_{i(U) \subset V} \mathscr{F}(V))^{+} = (\{x\} \mapsto \varinjlim_{i(\{x\}) \subset V} \mathscr{F}(V))^{+} = (\{x\} \mapsto \varinjlim_{x \in V} \mathscr{F}(V))^{+}.
\end{equation}
That is, $i_{x}^{-1}\mathscr{F}$ is given by the sheafification of the stalk functor, which is already a sheaf. So on one hand, we have checked that we can see \textbf{the stalk functor as the pullback by the inclusion $i_x$}. On the other hand, let's now compute the pushforward of the constant sheaf $\underbar{G}$ (which by abuse of notation we denote by $G$) by $i_{x}$, that is $i_{x,*}G$: by definition, for $U \subseteq X$ open

\begin{equation}
i_{x,*}G(U) = G(i_{x}^{-1}(U)).
\end{equation}
If $x \in U$, then we get

\begin{equation}
i_{x,*}G(U) = G(i_{x}^{-1}(U)) = G(\{x\}) = G.
\end{equation}
On the other hand, if $x \notin U$, then

\begin{equation}
i_{x,*}G(U) = G(i_{x}^{-1}(U)) = G(\emptyset)
\end{equation}
which we have defined to be the one point space $\{*\}$. Therefore, we checked that we can see the \textbf{skyscrapper sheaf with value $G$, centered at $x$ as the pushforward of the constant sheaf $\underbar{G}$ by the inclusion $i_{x}$}. So by the adjoint property between $i_{x}^{-1}$ and $i_{x, *}$, we conclude that the \textbf{stalk sheaf and the skyscrapper sheaf are adjoint}, and the pair of functors $x^{*}$ and $x_{*}$ that we have just defined are indeed adjoint.

Now, choose for each $x \in X$ an injective object $D_x$ in \textbf{Ab} such that there is a monomorphism $\mathscr{F}_x \rightarrow D_x$. Define a sheaf of abelian groups $\mathscr{D}$ in $X$ in the following way

\begin{equation}
    \mathscr{D} = \prod_{x \in X} x_{*}D_{x}.
\end{equation}

Its stalk at $x$ is $D_x$, and so the monomorphism $\mathscr{F}_x \rightarrow D_x$ induces a monomorphism $\mathscr{F} \rightarrow D$. Now, let $\mathscr{G}$ and $\mathscr{G}'$ be two sheaves in $\textbf{Sh(X)}$. We have

\begin{align*}
    \prod_{x \in X} \Hom_{\textbf{Ab}}(\mathscr{G}_x, D_x) &\cong \prod_{x \in X} \Hom_{\textbf{Ab}}(x^{*} \mathscr{G}, D_x) \\
    &\cong \prod_{x \in X} \Hom_{\textbf{Sh(X)}}(\mathscr{G}, x_{*}D_x) \\
    &\cong \Hom_{\textbf{Sh(X)}} (\mathscr{G}, \prod_{x \in X} x_{*}D_x) \\
    &\cong \Hom_{\textbf{Sh(X)}} (\mathscr{G}, \mathscr{D})
\end{align*}
where the first isomorphism is by definition, the second is by adjointness of $x^{*}$ and $x_{*}$, the third is because $\Hom$ commutes with limits, and the fourth is by definition. That tells us that given a commutative diagram

    \begin{center}

    \begin{tikzpicture}[>=triangle 60]
    \matrix[matrix of math nodes,column sep={60pt,between origins},row
    sep={60pt,between origins},nodes={asymmetrical rectangle}] (s)
    {
    |[name=0]| 0 &|[name=A]| \mathscr{G}_x &|[name=B]| \mathscr{G}'_x\\
    &|[name=I]| D_x\\
    };
    \draw[overlay,->, font=\scriptsize,>=latex] 
             (0) edge node[auto] {\(\)} (A)
             (A) edge node[auto] {\(\)} (B)
             (A) edge node[auto] {\(\)} (I)
             (B) edge node[auto] {\(\)} (I)
    ;
    
    ;

    \end{tikzpicture}
    \end{center}
    for any $x$, we can always complete the diagram of the induced maps to get a commutative diagram
    
        \begin{center}

    \begin{tikzpicture}[>=triangle 60]
    \matrix[matrix of math nodes,column sep={60pt,between origins},row
    sep={60pt,between origins},nodes={asymmetrical rectangle}] (s)
    {
    |[name=0]| 0 &|[name=A]| \mathscr{G} &|[name=B]| \mathscr{G}'\\
    &|[name=I]| \mathscr{D}\\
    };
    \draw[overlay,->, font=\scriptsize,>=latex] 
             (0) edge node[auto] {\(\)} (A)
             (A) edge node[auto] {\(\)} (B)
             (A) edge node[auto] {\(\)} (I)
             (B) edge node[auto] {\(\)} (I)
    ;
    
    ;

    \end{tikzpicture}
    \end{center}
    and so $\mathscr{D}$ is injective.
\end{proof}

\subsection{Comparison Theorems and Examples}

In this section we discuss some comparison theorems between the \v{C}ech cohomology theory and the derived functor cohomology theory, mostly without proofs, and present some examples. This section is intended just as an illustration of the settings in which the theories coincide, as a way to showcase the computational strength of the \v{C}ech theory.

\begin{example}\label{exsphere}
Let $S^1$ be again the circle with the usual topology, and take again the cover $\mathcal{U} = \{U, V\}$, with $U$ and $V$ being two open semi-circles overlapping on each end, so $U \cap V$ consists of two disjoint small open intervals as in Example (\ref{circlecohomology}). Consider the sheaf \underbar{$\mathbb{Z}$} on $S^1$. We have

\begin{align*}
    &C^{0}(\mathcal{U}, \underbar{$\mathbb{Z}$}) = \Gamma(U, \underbar{$\mathbb{Z}$}) \times \Gamma(V, \underbar{$\mathbb{Z}$}) = \mathbb{Z} \times \mathbb{Z} \\
    &C^{1}(\mathcal{U}, \underbar{$\mathbb{Z}$}) = \Gamma(U \cap V, \underbar{$\mathbb{Z}$}) = \mathbb{Z} \times \mathbb{Z}.
\end{align*}
The differential map $d^0: C^{0}(\mathcal{U}, \underbar{$\mathbb{Z}$}) \rightarrow C^{1}(\mathcal{U}, \underbar{$\mathbb{Z}$})$ maps $(a, b) \in \mathbb{Z}$ to $(b-a, b-a) \in \mathbb{Z}$, and so its kernel is $\mathbb{Z}$, and $\check{\mathrm{H}}^{1}(\mathcal{U}, \underbar{$\mathbb{Z}$}) = \mathbb{Z}$. So $\check{\mathrm{H}}^{0}(\mathcal{U}, \underbar{$\mathbb{Z}$}) = \Gamma(S^1, \underbar{$\mathbb{Z}$}) = \mathbb{Z}$ and $\check{\mathrm{H}}^{1}(\mathcal{U}, \underbar{$\mathbb{Z}$}) = \mathbb{Z}$.

On the other hand, computing the derived functor cohomology of \underbar{$\mathbb{Z}$} on $S^1$ is much harder. It is left as an exercise in Hartshorne (\cite{hartshorne}, Chap. 3, Exercise 2.7) to prove that $\mathrm{H}^{1}(S^1, \underbar{$\mathbb{Z}$}) = \mathbb{Z}$. It can be computed by constructing an injective resolution of \underbar{$\mathbb{Z}$}. This is done using a similar construction to the one presented in Theorem \ref{eninj}, in the context of \textbf{Mod$_{\mathcal{O}_X}$} (cf. \cite{hartshorne}, Chap. 3, Proposition 2.2): denote again the \textbf{skyscrapper sheaf centered at $x$ with value $\mathbb{Z}$} by $i_{x,*}\mathbb{Z}$. Define $I_0 = \prod_{x \in S^1} i_{x,*}\mathbb{Z}$, $I_1 = \prod_{x \in S^1} I_{0, x}/\mathbb{Z}$, $I_2 = \prod_{x \in S^1} I_{1, x}/I_{0, x}$, and so on, where $I_{i, x}$ is the stalk of $I_i$ at $x$. One can then check that this is indeed an injective resolution, and use it to conclude that $\mathrm{H}^{1}(S^1, \underbar{$\mathbb{Z}$}) = \mathbb{Z}$.
\end{example}

\begin{example}
Let $A \rightarrow B$ be an $A$-algebra, and $M$ be a $B$-module. A \textit{derivation} of $B$ over $A$ taking values in $M$ is an \textit{$A$-linear map}

\begin{equation}
    d: B \rightarrow M
\end{equation}
such that for all, $a, b \in B$, it satisfies the \textit{Leibniz rule}:

\begin{equation}
    d(ab) = d(a)b + ad(b).
\end{equation}

One can construct a \textit{universal} derivation $d$ and a $B$-module $\Omega^1_{B/A}$, that is, for any other derivation $\Tilde{d}$ and $B$-module $M$, we have a commutative diagram

\begin{center}

\begin{tikzpicture}[>=triangle 60]
\matrix[matrix of math nodes,column sep={60pt,between origins},row
sep={60pt,between origins},nodes={asymmetrical rectangle}] (s)
{
|[name=A]| A &|[name=B]| \Omega^1_{B/A} \\
&|[name=A']| M\\
};
\draw[overlay,->, font=\scriptsize,>=latex] 
          (A) edge node[auto] {\(d\)} (B)
          (A) edge node[auto] {\(\Tilde{d}\)} (A')
          (B) edge node[auto] {\(\)} (A')
;

;

\end{tikzpicture}

\end{center}
such that $\Tilde{d}$ factors uniquely through $\Omega^1_{B/A}$. This construction is done purely in a formal algebraic way: we define the module $\Omega^1_{B/A}$ to be the free $A$-module generated by the symbols $da$ for every $a \in A$, subject to the relations of linearity and the Leibniz rule. The module $\Omega^1_{B/A}$ is called the module of \textit{Kahler differentials}.

Localizing $B$ with respect to some multiplicative subset $S$ we have the following isomorphism

\begin{equation}
    \Omega^1_{B/A} \otimes_B S^{-1}B \cong \Omega^{1}_{S^{-1}B/A}
\end{equation}
and the sheaf obtained from the association

\begin{equation}
    D(f) \mapsto \Omega^{1}_{B_{f}/A} \cong \Omega^1_{B/A} \otimes_B B_f
\end{equation}
using Lemma \ref{sheafonabasis} is a \textit{quasicoherent sheaf} $\Omega^1_{\operatorname{Spec} B/\operatorname{Spec} A}$ on $\operatorname{Spec} B$ over $\operatorname{Spec} A$. These sheaves can then be glued to a global quasicoherent sheaf on \textit{any} scheme $X$. Its sections are called \textit{regular} or \textit{algebraic differential forms on $X$}.

Let's compute the cohomology of this sheaf in the case of the scheme $X = \mathbb{P}^1$ over $\mathbb{C}$: a regular differential form $\omega \in \Omega^1_{X/\mathbb{C}}$ is determined by its restrictions to the standard affine open subsets as in Example \ref{P1} and Example \ref{Pn}, that is,

\begin{equation}
    U_1 = \mathbb{P}^1 \setminus \{(1:0)\}
\end{equation}
and

\begin{equation}
    U_2 = \mathbb{P}^1 \setminus \{(0:1)\}.
\end{equation}
So over $U_1$ we have $\operatorname{Spec} \mathbb{C}[t]$, and over $U_2$ we have $\operatorname{Spec} \mathbb{C}[s]$. Recall that on the overlap $U_1 \cap U_2$, these two affine subsets agree via the identification $z \mapsto \frac{1}{z}$, and so we have

\begin{equation}\label{ts1}
    t\cdot s = 1
\end{equation}
over $U_1 \cap U_2$.
 
So if $\left. \omega \right|_{U_1}$ is written as $\left. \omega \right|_{U_1} = f(t)dt$ and $\left. \omega \right|_{U_2}$ is written as $\left. \omega \right|_{U_2} = g(s)ds$, derivating the expression (\ref{ts1}), we get

\begin{align*}
    0 &= d(st) \\
      &= tds + sdt\\
\end{align*}
and so $ds = \frac{-s}{t}dt$. Taking $s = \frac{1}{t}$, we get that over $U_1 \cap U_2$

\begin{equation}
    ds = -\frac{1}{t^2}dt.
\end{equation}
Therefore over $U_1 \cap U_2$, $f(t)dt = \frac{-g(s)}{t^2}dt$, and

\begin{equation}
    t^2f(t) = -g(1/t)
\end{equation}
which cannot happen since $f$ and $g$ are polynomials, unless they are 0. So $\omega = 0$ and $\mathbb{P}^1$ \textit{does not have any globally defined algebraic differential forms}, that is $\check{\mathrm{H}}^0(X, \Omega^1_{X/\mathbb{C}}) = 0$.

Now let's compute the \v{C}ech complex of $X$, with respect to the covering $\mathcal{U} = \{U_1, U_2\}$: we know that $\Omega^1_{X/\mathbb{C}}(U_1) = \mathbb{C}[t]$, $\Omega^1_{X/\mathbb{C}}(U_2) = \mathbb{C}[s]$ and $\Omega^1_{X/\mathbb{C}}(U_1 \cap U_2) = \mathbb{C}[t, 1/t]$, and so

\begin{align}
    C^0(\mathcal{U}, X) &= \Omega^1_{X/\mathbb{C}}(U_1) \times \Omega^1_{X/\mathbb{C}}(U_2)\\
                        &= \mathbb{C}[t] \times \mathbb{C}[s]
\end{align}
and

\begin{equation}
    C^1(\mathcal{U}, X) = \Omega^1_{X/\mathbb{C}}(U_1 \cap U_2) = \mathbb{C}[t, 1/t].
\end{equation}

The kernel of the differential map $\delta_{0}: C^0(\mathcal{U}, X) \rightarrow C^1(\mathcal{U}, X)$ is given by $\check{\mathrm{H}}^0(X, \Omega^1_{X/\mathbb{C}}) = 0$. Its image is 

\begin{equation}\label{imad0p1}
    \left. (f(t)dt - g(s)(ds)) \right|_{U_1 \cap U_2}
\end{equation}
and using the change of variables $ds = -\frac{1}{t^2}dt$ over $U_1 \cap U_2$, the image is given by all the combinations

\begin{equation}
    \left. (f(t)dt + \frac{1}{t^2}g(t)dt) \right|_{U_1 \cap U_2}.
\end{equation}
The kernel of $d_1: C^1(\mathcal{U}, X) \rightarrow C^2(\mathcal{U}, X)$ is all of $ C^1(\mathcal{U}, X)$, and so

\begin{equation}
    \check{\mathrm{H}}^1(X, \Omega^1_{X/\mathbb{C}}) \cong \frac{\mathbb{C}[t, \frac{1}{t}]}{\Ima(\delta_0)}
\end{equation}
and we see from equation (\ref{imad0p1}) that in particular, the image $\Ima(\delta_0)$ contains all elements of the form $t^n dt$ for $n \neq 1$, $n \in \mathbb{Z}$, and so it annihilates all nonzero order terms in $\mathbb{C}[t, \frac{1}{t}]$. Therefore,

\begin{equation}
    \check{\mathrm{H}}^1(X, \Omega^1_{X/\mathbb{C}}) \cong \frac{\mathbb{C}[t, \frac{1}{t}]}{\Ima(\delta_0)} \cong \mathbb{C}.
\end{equation}

The fact that this result agrees with the result from Propostion \ref{cohdiffPn} is not a coincidence. It is a realization of an important theorem of Serre, called the \textbf{GAGA principle}.

\end{example}

\begin{example}\label{affcohomology}
\textbf{Cohomology of an Affine Noetherian Scheme:} We will now compute the cohomology of an affine noetherian scheme, following \cite{hartshorne}, Chap. 3, §. 3. Some of the proofs of the preliminary results are omitted.

\begin{definition}
A sheaf $\mathscr{F}$ on $X$ is called \textit{flasque} (or sometimes \textit{flabby}), if for every $U \subseteq X$, the restriction map $\rho_{XU}$ is surjective. Equivalently, $\mathscr{F}$ is flasque if for any inclusion $U \subseteq V$ of open sets, the map $\rho_{VU}$ is surjective.
\end{definition}

\begin{proposition}
If $(X, \mathcal{O}_X)$ is a ringed space, then any \textit{injective $\mathcal{O}_X$-module is flasque}.
\end{proposition}

\begin{proof}[Idea of proof:]
If we denote by $\mathcal{O}_U$ the restriction of the structure sheaf to an open set $U \subseteq X$, and $U \subseteq V$ is an inclusion of open sets, we have an inclusion of sheaves of $\mathcal{O}_X$-modules

\begin{equation}
\begin{tikzcd}
    0 \ar{r} & \mathcal{O}_U \ar{r} & \mathcal{O}_V.
\end{tikzcd}
\end{equation}
Now suppose $\mathscr{I}$ is an injective sheaf of $\mathcal{O}_X$-modules. Since it is injective, the above sequence induces an exact sequence, i.e., a surjection 

\begin{equation}
\begin{tikzcd}
    \Hom (\mathcal{O}_V, \mathscr{I}) \ar{r} & \Hom (\mathcal{O}_U, \mathscr{I}) \ar{r} & 0
\end{tikzcd}
\end{equation}
and since $\Hom (\mathcal{O}_V, \mathscr{I}) = \mathscr{I}(V)$ and $\Hom (\mathcal{O}_U, \mathscr{I}) = \mathscr{I}(U)$, $\mathscr{I}$ is flasque.
\end{proof}

\begin{lemma}
Let

\begin{equation}
\begin{tikzcd}
    0 \ar{r} & \mathscr{F} \ar{r} & \mathscr{G} \ar{r} & \mathscr{H} \ar{r} & 0
\end{tikzcd}
\end{equation}
be an exact sequence of sheaves, and suppose that $\mathscr{F}$ and $\mathscr{G}$ are flasque. Then $\mathscr{H}$ is flasque.
\end{lemma}

\begin{proof}
\cite{ahtopos}, 3, Proposição 3.2.9.
\end{proof}

\begin{lemma}
Suppose that 

\begin{equation}
\begin{tikzcd}
    0 \ar{r} & \mathscr{F} \ar{r} & \mathscr{G} \ar{r} & \mathscr{H} \ar{r} & 0
\end{tikzcd}
\end{equation}
is an exact sequence of sheaves, with $\mathscr{F}$ flasque. Then, the sequence

\begin{equation}
\begin{tikzcd}
    0 \ar{r} & \Gamma(X, \mathscr{F}) \ar{r} & \Gamma(X, \mathscr{G}) \ar{r} & \Gamma(X, \mathscr{H}) \ar{r} & 0
\end{tikzcd}
\end{equation}
is exact.
\end{lemma}

\begin{proof}
One can find a discussion in \cite{stacks-project}, Lemma 20.12.3.
\end{proof}

Our goal is to use flasque sheaves to compute cohomology. For this, we use a result from Hartshorne (\cite{hartshorne}, Chap. III, Proposition 1.2A) that states that if we have a complex

\begin{equation}
   0 \longrightarrow A \longrightarrow J^{\bullet} 
\end{equation}
with each $J^k$ acyclic for a functor $F$ (i.e., $R^{i}F(J^k) = 0$ for all $i >0$), then for each $i \ge 0$, $R^{i}F(A)$ is canonically isomorphic to the ith cohomology group of the complex. So we're left to prove that flasque sheaves are acyclic for the global section functor $\Gamma(X,-)$:

\begin{proposition}\label{flasqueisacyclic}
Let $\mathscr{F}$ be a flasque sheaf on $X$. Then $\mathscr{F}$ is $\Gamma(X, -)$-acyclic, i.e., $\mathrm{H}^i(X, \mathscr{F}) = 0$ for all $i > 0$.
\end{proposition}

\begin{proof}[Idea of proof:]
Let $\mathscr{I}$ be an injective sheaf of abelian groups such that there is an inclusion $\mathscr{F} \rightarrow \mathscr{I}$. Consider the canonical sequence

\begin{equation}
\begin{tikzcd}
    0 \ar{r} & \mathscr{F} \ar{r}{f} & \mathscr{I} \ar{r}{g} & \mathscr{H} \ar{r} & 0
\end{tikzcd}
\end{equation}
with $\mathscr{H}$ being the quotient sheaff, so as to make it exact. Since $\mathscr{I}$ is injective, we know that it is flasque, and so by the above Lemma, $\mathscr{H}$ is also flasque. So the induced sequence

\begin{equation}
\begin{tikzcd}
    0 \ar{r} & \Gamma(X, \mathscr{F}) \ar{r}{\Gamma_f} & \Gamma(X, \mathscr{I}) \ar{r}{\Gamma_g} & \Gamma(X, \mathscr{H}) \ar{r} & 0
\end{tikzcd}
\end{equation}
is exact. We then look at part of the long exact sequence on cohomology

\begin{equation}
\begin{tikzcd}
    \mathrm{H}^0(X, \mathscr{I}) \ar{r} & \mathrm{H}^0(X, \mathscr{H}) \ar{r}{\delta^0} & \mathrm{H}^1(X, \mathscr{F}) \ar{r} & \mathrm{H}^1(X, \mathscr{I})
\end{tikzcd}
\end{equation}
which, since $\mathscr{I}$ is injective, becomes

\begin{equation}
\begin{tikzcd}
    \Gamma(X, \mathscr{I}) \ar{r}{\Gamma_g} & \Gamma(X, \mathscr{H}) \ar{r}{\delta^0} & \mathrm{H}^1(X, \mathscr{F}) \ar{r} & 0.
\end{tikzcd}
\end{equation}
Since the sequence is exact, we have that $\mathrm{H}^1(X, \mathscr{F}) \cong \Gamma(X, \mathscr{H})/ \ker \delta^0$ by the isomorphism theorem. However, since $\mathscr{F}$, $\mathscr{G}$ and $\mathscr{H}$ are flasque, $\Gamma_g$ is surjective, and so $\ker \delta^0 = \Ima \Gamma_g = \Gamma(X, \mathscr{H})$. And so $\mathrm{H}^1(X, \mathscr{F}) = 0$. The result then follows by induction on $i$.
\end{proof}

\begin{proposition}
Let $A$ be a noetherian ring, and let $I$ be an $A$-module. If $I$ is injective, then the sheaf $\Tilde{I}$ associated to $I$ is flasque on $X = \operatorname{Spec} A$.
\end{proposition}

\begin{proof}
\cite{hartshorne}, Chap. 3, Proposition 3.4.
\end{proof}

\begin{theorem}\label{cohaffnoetscheme}
Let $A$ be a noetherian ring, and let $X = \operatorname{Spec} A$. Then for all quasicoherent sheaves $\mathscr{F}$ on $X$ and all $i > 0$, we have that $\mathrm{H}^i (X, \mathscr{F}) = 0$.
\end{theorem}

\begin{proof}
Given a quasicoherent sheaf $\mathscr{F}$, set $M = \Gamma(X, \mathscr{F})$. Take an injective resolution 

\begin{equation}
    0 \longrightarrow M \longrightarrow I^{\bullet}
\end{equation}
of $M$ in \textbf{Mod$_{A}$}. This induces a flasque resolution

\begin{equation}
    0 \longrightarrow \Tilde{M} \longrightarrow \Tilde{I}^{\bullet}
\end{equation}
of $\Tilde{M}$ on \textbf{QCoh($X$)}. Now, we have discussed that $\Tilde{M} = \mathscr{F}$ on Proposition \ref{qcohaffine4}. So applying the functor $\Gamma(X, -)$ to the sequence

\begin{equation}
    0 \longrightarrow \mathscr{F} \longrightarrow \Tilde{I}^{\bullet}
\end{equation}
we get back the exact sequence of $A$-modules $0 \longrightarrow M \longrightarrow I^{\bullet}$. Therefore, $\mathrm{H}^0 (X, \mathscr{F}) = \Gamma(X, \mathscr{F}) = M$ and $\mathrm{H}^i (X, \mathscr{F}) = 0$ for all $i > 0$ since we proved that each $\Tilde{I}$ is flasque and that flasque sheaves are acyclic for $\Gamma(X, -)$. This concludes the proof.

\begin{remark}
This result is valid more generally, without the need for the noetherian hypothesis.
\end{remark}
\end{proof}

\begin{corollary}\label{embbedflasque}
Let $X$ be a noetherian scheme, and let $\mathscr{F}$ be a quasicoherent sheaf on $X$ Then $\mathscr{F}$ can be embedded in a flasque quasicoherent sheaf $\mathscr{G}$.
\end{corollary}

\begin{proof}
\cite{hartshorne}, Chap. 4, Corollary 3.6.
\end{proof}

\end{example}

In fact, we will need a slightly more general result: given any sheaf $\mathscr{F}$ on a topological space $X$, we can find a flasque sheaf $\mathscr{G}$ such that $\mathscr{F}$ is a subsheaf of $\mathscr{G}$. In fact, if we define $\mathscr{G}$ as the sheaf given by

\begin{equation}\label{godementflasquesheaf}
    \mathscr{G}(U) = \prod_{x \in U} = \mathscr{F}_x
\end{equation}
over $U \subset X$ open, it is a flasque sheaf called the \textbf{Godement flasque sheaf}. It was used by Godement to discuss the existence flasque resolutions for any sheaf (cf. \cite{hossein}, Chap. 2, Example 2.4).

For the following result, we will need to define the \textit{sheaf version of the \v{C}ech complex}, that is, we will need to be able to realize the \v{C}ech groups $C^{\bullet}(\mathcal{U}, \mathscr{F})$ as a complex of sheaves: Let $X$ be a topological space, $\mathscr{F}$ a sheaf on $X$ and $\mathcal{U}$ be cover of $X$. Denote by $U_{\alpha_{0},...,\alpha_{i}}$ any intersection $U_{\alpha_{0}} \cap ... \cap U_{\alpha_{i}}$ of open sets of $\mathcal{U}$, and by 

\begin{equation}
    f_{\alpha_{0},...,\alpha_{i}}: U_{\alpha_{0},...,\alpha_{i}} \rightarrow X
\end{equation}
the inclusion map. We construct a complex $\mathscr{C}^{\bullet}(\mathcal{U}, \mathscr{F})$ of sheaves in the following way: for each $i \ge 0$

\begin{equation}
    \mathscr{C}^{i}(\mathcal{U}, \mathscr{F}) = \prod_{|I| = i+1} \left. f_{\alpha_{0},...,\alpha_{i},*}\mathscr{F} \right|_{U_{\alpha_{0},...,\alpha_{i}}}
\end{equation}
and the same differential operator as the classical \v{C}ech complex, since it is defined over sections (\cite{hartshorne}, Chap. 3, §4, \cite{stacks-project}, Lemma 20.24.1).
Note that, this construction gives

\begin{equation}
    \Gamma(X, \mathscr{C}^{i}(\mathcal{U}, \mathscr{F})) = \prod_{|I| = i+1} \left. i_{*}\mathscr{F} \right|_{U_{\alpha_{0},...,\alpha_{i}}}
\end{equation}
with $i$ being the identity map $X \rightarrow X$. And so 

\begin{align}
    \Gamma(X, \mathscr{C}^{i}(\mathcal{U}, \mathscr{F})) &= \prod_{|I| = i+1} \left. \mathscr{F} \right|_{U_{\alpha_{0},...,\alpha_{i}}}(X)\\ &= \prod_{|I| = i+1} \mathscr{F}( U_{\alpha_{0},...,\alpha_{i}} \cap X)\\ &= \prod_{|I| = i+1} \mathscr{F}( U_{\alpha_{0},...,\alpha_{i}}) = C^{i}(\mathcal{U}, \mathscr{F}).
\end{align}

The advantage of this point of view is that it allows us to realize the \textbf{\v{C}ech complex $\mathscr{C}^{\bullet}(\mathcal{U}, \mathscr{F})$ as a resolution of the sheaf $\mathscr{F}$}, so there exists a natural map $\mathscr{F} \rightarrow \mathscr{C}^{0}(\mathcal{U}, \mathscr{F})$ such that the sequence of sheaves $0 \longrightarrow \mathscr{F} \longrightarrow \mathscr{C}^{0}(\mathcal{U}, \mathscr{F}) \longrightarrow \mathscr{C}^{1}(\mathcal{U}, \mathscr{F}) \longrightarrow ...$ is exact (\cite{hartshorne}, Chap. 3, Lemma 4.2, \cite{stacks-project}, Lemma 20.24.1) and therefore, all its higher cohomology groups vanish, and so it is acyclic.

\begin{theorem}\label{cechdermaps}
Let $X$ be a topological space, and let $\mathcal{U}$ be an open covering of $X$. Then for each $i \ge 0$, there is a natural map

\begin{equation}
    \check{\mathrm{H}}^i(\mathcal{U}, \mathscr{F}) \rightarrow {\mathrm{H}}^i(X, \mathscr{F})
\end{equation}
and this map is functorial in $\mathscr{F}$.
\end{theorem}

\begin{proof}[Idea of Proof:]
The proof of this theorem again makes use of a comparison theorem of complexes in homological algebra that states that, if $C^{\bullet}$ is acyclic and $D^{\bullet}$ is injective, then for every morphism

\begin{equation}
    \varphi: \mathrm{H}^0(C^{\bullet}) \rightarrow \mathrm{H}^0(D^{\bullet})
\end{equation}
there exists a map of complexes 

\begin{equation}
    \Tilde{\varphi}: C^{\bullet} \rightarrow D^{\bullet}
\end{equation}
that induces it (cf. \cite{hiltonstammbach}, Chap. IV, Theorem 4.4).

So since the complex $0 \longrightarrow \mathscr{F} \longrightarrow \mathscr{C}^{\bullet}(\mathcal{U}, \mathscr{F})$ is acyclic, if we take any injective resolution $0 \longrightarrow \mathscr{F} \longrightarrow \mathscr{I}^{\bullet}$ of $\mathscr{F}$, the theorem states that there exists a map of complexes $\mathscr{C}^{\bullet}(\mathcal{U}, \mathscr{F}) \rightarrow \mathscr{I}^{\bullet}$ inducing the identity map on $\mathscr{F}$. So applying the global section functor to both complexes, we get maps

\begin{equation}
    C^{\bullet}(\mathcal{U}, \mathscr{F}) \rightarrow \Gamma(X, \mathscr{I}^{\bullet})
\end{equation}
and taking cohomology, we get the induced maps

\begin{equation}
    \check{\mathrm{H}}^i(\mathcal{U}, \mathscr{F}) \rightarrow {\mathrm{H}}^i(X, \mathscr{F}).
\end{equation}
\end{proof}

However it is worth noticing that the maps $\mathscr{C}^{\bullet}(\mathcal{U}, \mathscr{F}) \rightarrow \mathscr{I}^{\bullet}$ above need not be a homotopic equivalence between the resolutions $0 \longrightarrow \mathscr{F} \longrightarrow \mathscr{C}^{\bullet}(\mathcal{U}, \mathscr{F})$ and $0 \longrightarrow \mathscr{F} \longrightarrow \mathscr{I}^{\bullet}$, and so the induced maps $\check{\mathrm{H}}^i(\mathcal{U}, \mathscr{F}) \rightarrow {\mathrm{H}}^i(X, \mathscr{F})$ need not be isomorphisms. We now present a result which gives a condition for when they do agree. In particular, it is no coincidence that we got the same result in Example \ref{exsphere}:

\begin{theorem}\label{hartshorneexercthm}
Let $X$ be a topological space, $\mathscr{F}$ a sheaf of abelian groups and $\mathcal{U}$ an open cover of $X$. If for every intersection $U_{\alpha_{0},...,\alpha_{i}} = U_{\alpha_{0}} \cap ... \cap U_{\alpha_{i}}$ of elements of $\mathcal{U}$, and for any $p > 0$, we have that

\begin{equation}
    \mathrm{H}^p(U_{\alpha_{0},...,\alpha_{i}}, \left. \mathscr{F} \right|_{\alpha_{0},...,\alpha_{i}}) = 0
\end{equation}
then for all $q \ge 0$, the maps

\begin{equation}
    \check{\mathrm{H}}^i(\mathcal{U}, \mathscr{F}) \rightarrow {\mathrm{H}}^i(X, \mathscr{F})    
\end{equation}
given in Theorem \ref{cechdermaps} are isomorphisms.
\end{theorem}

This theorem is given as an exercise in Hartshorne's book (\cite{hartshorne}, Chap. 3, Exercise 4.11) and it can be proved with the machinery of \textit{spectral sequences}. In particular, it states that Leray's Theorem can be used interchengeably with the two theories.

As a consequence of Theorem \ref{hartshorneexercthm}, we can prove

\begin{theorem}
Let $X$ be a \textit{noetherian separated scheme},\footnote{Like a variety over an algebraically closed field $k$.} let $\mathcal{U}$ be an open affine cover of $X$ and let $\mathscr{F}$ be a quasicoherent sheaf on $X$. Then, for all $i \ge 0$, the natural maps $\check{\mathrm{H}}^i(\mathcal{U}, \mathscr{F}) \rightarrow {\mathrm{H}}^i(X, \mathscr{F})$
of the above theorem give isomorphisms $\check{\mathrm{H}}^i(\mathcal{U}, \mathscr{F}) \cong {\mathrm{H}}^i(X, \mathscr{F})$.
\end{theorem}

It follows directly from Theorem \ref{cohaffnoetscheme}, and from the fact that any intersection of affine schemes is an affine scheme.

So now we finally give a proof of Leray's Theorem. This proof was taken from \cite{hossein}:

\begin{proof}[Proof of Leray's Theorem:]
Let $\mathscr{F}$ be a sheaf on $X$, and let $\mathcal{U}$ be an $\mathscr{F}$-acyclic cover of $X$. Let $\mathscr{G}$ be the Godement flasque sheaf constructed from $\mathscr{F}$, as in (\ref{godementflasquesheaf}). Let $\mathscr{H}$ be the quotient sheaf, so as to make the short exact sequence

\begin{equation}\label{sesflasqueleray}
\begin{tikzcd}
    0 \ar{r} & \mathscr{F} \ar{r} & \mathscr{G} \ar{r} & \mathscr{H} \ar{r} & 0
\end{tikzcd}
\end{equation}
exact. We will proceed by induction on $q$. The case $q = 0$ was proved in Propostion \ref{zerocohglobalsecn}. Suppose the result holds for $q-1$. For alll $p \ge 1$, it follows from Proposition \ref{flasqueisacyclic} that

\begin{equation}\label{gisflasque}
    \check{\mathrm{H}}^p(\mathcal{U}, \mathscr{G}) \cong \check{\mathrm{H}}^p(X, \mathscr{G}) = 0.
\end{equation}

Looking at the long exact in cohomology associated to the short exact sequence (\ref{sesflasqueleray}) restricted to each intersection $U_{\alpha_{0},...,\alpha_{i}}$ in $\mathcal{U}$, we get

\begin{equation}
\begin{tikzcd}
    \check{\mathrm{H}}^p(U_{\alpha_{0},...,\alpha_{i}}, \mathscr{G}) \ar{r} & \check{\mathrm{H}}^p(U_{\alpha_{0},...,\alpha_{i}}, \mathscr{H}) \ar{r} & \check{\mathrm{H}}^{p+1}(U_{\alpha_{0},...,\alpha_{i}}, \mathscr{F}) \ar{r} & \check{\mathrm{H}}^{p+1}(U_{\alpha_{0},...,\alpha_{i}}, \mathscr{G})
\end{tikzcd}
\end{equation}
and since $\check{\mathrm{H}}^p(U_{\alpha_{0},...,\alpha_{i}}, \mathscr{G}) \cong \check{\mathrm{H}}^{p+1}(U_{\alpha_{0},...,\alpha_{i}}, \mathscr{G}) = 0$, we conclude that

\begin{equation}\label{hisacyclic}
    \check{\mathrm{H}}^p(U_{\alpha_{0},...,\alpha_{i}}, \mathscr{H}) \cong \check{\mathrm{H}}^{p+1}(U_{\alpha_{0},...,\alpha_{i}}, \mathscr{F})
\end{equation}
for all $p \ge 1$, and for all intersections $U_{\alpha_{0},...,\alpha_{i}}$. Since by hypothesis

\begin{equation}
    \check{\mathrm{H}}^{p+1}(U_{\alpha_{0},...,\alpha_{i}}, \mathscr{F}) = 0
\end{equation}
we look at the following diagram with exact rows

\begin{equation}
\begin{tikzcd}
     \check{\mathrm{H}}^{q-1}(\mathcal{U}, \mathscr{G}) \ar{r} \ar[d, "\sim" {anchor=south, rotate=90, inner sep=.5mm}] & \check{\mathrm{H}}^{q-1}(\mathcal{U}, \mathscr{H}) \ar{r} \ar[d, "\sim" {anchor=south, rotate=90, inner sep=.5mm}] &  \check{\mathrm{H}}^{q}(\mathcal{U}, \mathscr{F}) \ar{r} \ar[d] & \check{\mathrm{H}}^{q}(\mathcal{U}, \mathscr{G}) \ar{r} \ar[d, "\sim" {anchor=south, rotate=90, inner sep=.5mm}] & \check{\mathrm{H}}^{q}(\mathcal{U}, \mathscr{H}) \ar[d, "\sim" {anchor=south, rotate=90, inner sep=.5mm}] \\
    \check{\mathrm{H}}^{q-1}(X, \mathscr{G}) \ar{r} & \check{\mathrm{H}}^{q-1}(X, \mathscr{H}) \ar{r} & \check{\mathrm{H}}^{q}(X, \mathscr{F}) \ar{r} & \check{\mathrm{H}}^{q}(X, \mathscr{G}) \ar{r} & \check{\mathrm{H}}^{q}(X, \mathscr{H})
\end{tikzcd}
\end{equation}
where the first and second isomorphism follow from the induction hypothesis, since we proved in (\ref{hisacyclic}) that $\mathscr{H}$ is $\mathcal{U}$-acyclic, the third isomorphism follows from (\ref{gisflasque}) and the fourth isomorphism follows from the long exact sequence in cohomology and (\ref{gisflasque}). Then applying the $5$-Lemma, we get the desired isomorphism

\begin{equation}
    \check{\mathrm{H}}^{q}(\mathcal{U}, \mathscr{F}) \cong \check{\mathrm{H}}^{q}(X, \mathscr{F}).
\end{equation}
\end{proof}

To end, this section we present an example where the two theories \textbf{do not} agree. This example is due to Grothendieck, and can be found in full detail, with all computations in \cite{tohoku}, section 3.8, Example 3.8.3:

\begin{example}\label{extohoku}
Let $X$ be $\mathbb{A}^{2}_{k}$, the affine plane over a field $k$. Let $Y \subset X$ be the union of two curves in $X$, meeting in \textbf{two distinct points}. Denote by $\mathbb{Z}_{X}$ and $\mathbb{Z}_Y$ the constant sheaf with value $\mathbb{Z}$ in $X$ and $Y$ respectively, and let

\begin{equation}
\begin{tikzcd}
    0 \ar{r} & F\ar{r} & \mathbb{Z}_{X} \ar{r} & \mathbb{Z}_{Y} \ar{r} & 0
\end{tikzcd}
\end{equation}
i.e., let $F$ be the sheaf kernel of the restriction map $\mathbb{Z}_X \rightarrow \mathbb{Z}_Y$. Grothendieck than proceeds to compute $\check{\mathrm{H}}^{2}(X, F)$ and $\mathrm{H}^2(X, F)$, and concludes that $\mathbb{Z} = \mathrm{H}^{2}(X, F) \neq \check{\mathrm{H}}^{2}(X, F) = 0$.
\end{example}

\begin{remark}
One can also find in \cite{bruzzootero} a proof of Theorem \ref{paracompact}. It uses an argument that involves the \textbf{universality} of a certain \textbf{$\delta$-functor} - the generalization proposed by Grothendieck of the notion of derived functor, which was not discussed in this text. But maybe some other time...
\end{remark}

\section{Cohomological Formulation of the Mittag-Leffler \\ Problem}

In this section we tie some loose ends, as well as prove some interesting and remarkable theorems using the theories we've constructed so far in this text. The discussed theorems will then later be related to the main problem, of finding meromorphic functions with specified singular parts on Riemann surfaces.

\subsection{Via Dolbeault cohomology}

We now return to the discussion of finding meromorphic functions with specified poles and principal parts on Riemann surfaces. Suppose that we're given the data of a Mittag-Leffler problem on a Riemann surface $X$. By an arguments similar to that on Theorem \ref{mlplane}, since charts are trivially holomorphic functions, we can find a covering $\mathcal{U} = \{U_{\alpha}\}$ of $X$ and functions $f_{\alpha} \in \mathcal{K}_{X}(U_{\alpha})$ solving the problem locally. So the whole discussion in this section will be about formulating the problem as a problem of passage of data from local to global.

The first approach that one can take is complex analytical, via Dolbeault cohomology: suppose $f_{\alpha} \in \mathcal{K}_{X}(U_{\alpha})$ is a local solution to the problem, and suppose that $\mathcal{U}$ is fine enough so as to have at most one $p_n$ in each $U_{\alpha}$. 

For this approach, we have to introduce the notion of a \textit{$C^{\infty}$-bump function}: a  $C^{\infty}$-bump function or simply a bump function is a function $\rho_{\alpha}$ such that its value is identically $1$ in a neighbourhood $U \subset U_{\alpha}$ of $p_n$, and moreover it is smooth and its support is compact and contained in $U$. So intuitively a bump function $\rho_{\alpha}$ is a function that is identically $1$ in some neighbourhood of our pole and that glues smoothly with the complement of its support.

In particular, if we look at the product $\rho_{\alpha}f_{\alpha}$, we get a smooth function on the domain of definition of $f_{\alpha}$ whose support is compact and contained in $U$. So consider

\begin{equation}
    \varphi = \sum_{\alpha} \Bar{\partial}(\rho_{\alpha}f_{\alpha}) \in \mathcal{A}_{\Bar{\partial}}^{(0,1)}(X).
\end{equation}
We start by noting that that \textbf{$\varphi$ is a $\Bar{\partial}$-closed $C^{\infty}$ $(0,1)$-form on $X$}, since

\begin{equation}
    \Bar{\partial}\varphi = \sum_{\alpha} \Bar{\partial}\Bar{\partial}(\rho_{\alpha}f_{\alpha}) = 0.
\end{equation}

Also, since $\rho_\alpha = 1$ identically on a neighborhood of each $p_n$, we have that $\varphi$ is identically $0$ in a neighborhood of each $p_n$ by construction. So, if $\varphi$ is in the image of $\left. \Bar{\partial} \right|_{\mathcal{A}^{(0,0)}(X) = C^{\infty}(X)}$, that is, if there exists some smooth function $\eta \in C^{\infty}(X)$ on $X$ such that $\Bar{\partial}\eta = \varphi$, then the function defined as

\begin{equation}
    f = \bigl(\sum \rho_{\alpha}f_{\alpha}\bigr) - \eta
\end{equation}
is a global solution to the Mittag-Leffler problem given by the data $(f_{\alpha})$. So in particular, we see that we can solve the problem if, and only if

\begin{equation}
    0 = [\varphi] \in \mathrm{H}^{(0,1)}_{\Bar{\partial}}(X) = \frac{\ker \Bar{\partial_{1}}}{\Ima \Bar{\partial_0}}
\end{equation}
that is, any obstruction for the solution of some Mittag-Leffler problem on $X$ is an element of $\mathrm{H}^{(0,1)}_{\Bar{\partial}}(X)$, and so every problem has a solution if, and only if $\mathrm{H}^{(0,1)}_{\Bar{\partial}}(X) = 0$.

\subsection{Via Sheaf Cohomology}

Our second approach uses sheaf cohomology and the theory of \textit{residues}. We begin with some preliminaries, starting with some more generalities on Riemann surfaces. Reference is \cite{forster}, Chapter 2.

\medskip

\noindent\textbf{Residues:} Let $Y \subset X$ be an open subset, of some Riemann surface $X$, $a \in Y$, and $\omega \in \Omega_{X}^{1}(Y \setminus \{a\})$. Let $(U,z)$ be a coordinate system around $a$ such that $U \subset Y$ and $z(a) = 0$. On $U \setminus \{a\}$ one may write $\omega = fdz$ for some holomorphic function $f \in \mathcal{O}_{X}(U \setminus \{a\})$. Now let

\begin{equation}
    f = \sum_{-\infty}^{+\infty} c_n z^n
\end{equation}
be the Laurent expansion of $f$ about $a$ with respect to the coordinate function $z$. If $c_n = 0$ for all $n < 0$, then $f$ has a \textit{removable singularity} on $a$, and thus can be holomorphically continued to some holomorphic function in $\mathcal{O}_{X}(U)$. Consequently, $\omega$ can be extended to some $1$-form in $\Omega_{X}^{1}(U)$. If there exists some $k < 0$ such that $c_k \neq 0$ and $c_n = 0$ for all $n < k$, then $\omega$ has a \textit{pole of order $k$} at $a$. If $c_n \neq 0$ for infinitely many $n < 0$ then $\omega$ has an essential singularity at $a$. The coefficient $c_{-1}$ of the Laurent expansion of $f$ is called the \textit{residue} of $\omega$ at $a$ and is denoted by $Res_{a}(\omega)$. The definition of the residue \textbf{does not depend on the choice of chart $(U, z)$} (cf. \cite{forster}, 9.9).

\medskip

\noindent\textbf{Meromorphic $1$-forms:} We can give an explicit description of meromorphic $1$-forms on Riemann surfaces, similar to that of meromorphic functions: a differential $1$-form $\omega$ on a Riemann surface $X$ is said to be meromorphic if there is an open subset $Y$ of $X$ such that

\begin{enumerate}
    \item $\omega$ is holomorphic on $X \setminus Y$;
    \item $X \setminus Y$ consists of isolated points;
    \item the coefficient of $\omega$ has a pole at every point of $X \setminus Y$.
\end{enumerate}
We denote the sheaf of meromorphic $1$-forms by $\mathcal{K}^{(1)}_X$.

\medskip

\begin{theorem}
Let $X$ be a compact Riemann surface. Then we have

\begin{equation}
    \dim{\mathrm{H}^1(X, \mathcal{O}_X)} < \infty
\end{equation}
that is, the first cohomology group of $X$ is a finite dimensional complex vector space.
\end{theorem}

If $X$ is a compact riemann surface, we denote $g = \dim{\mathrm{H}^1(X, \mathcal{O}_X)}$, and call it the \textit{genus} of $X$ The finiteness of the genus of a compact Riemann surface can be proven using functional analysis.

\medskip

\noindent\textbf{The degree of a divisor and the Riemann-Roch formula:} Let $X$ be a Riemann surface, and let $D = \sum n_i[x_i]$.\footnote{Note that on a Riemann surface, a divisor is just a linear combination of points.} The \textit{degree} of $D$ is the integer

\begin{equation}
    deg(D) = \sum n_i.
\end{equation}

The degree defines a surjective group homomorphism $deg: Div(X) \rightarrow \mathbb{Z}$.

\begin{proposition}
If $(f)$ is a principal divisor on a Riemann surface $X$, then $deg(f) = 0$.
\end{proposition}
Therefore \textbf{linearly equivalent divisors have the same degree}.

\begin{theorem}[The Riemann-Roch Theorem]
Let $X$ be a compact Riemann surface of genus $g$, and let $D$ be a divisor on $X$. Then, $\dim{\mathrm{H}^{0}(X, \mathcal{O}(D))}$ and $\dim{\mathrm{H}^1(X, \mathcal{O}(D))}$ are finite and

\begin{equation}\label{rrformula}
    \dim{\mathrm{H}^{0}(X, \mathcal{O}(D))} - \dim{\mathrm{H}^1(X, \mathcal{O}(D))} = 1 - g + deg(D).
\end{equation}
\end{theorem}

A proof of the Riemann-Roch formula can be found in the appendix at the end of this part.

\medskip

\noindent\textbf{Mittag-Leffler Distributions:} We now discuss the Mittag-Leffler problem first on the case of \textbf{compact Riemann surfaces}. We begin by reformulating it in the language of \v{C}ech theory. As we noted in the Dolbeault solution, every Mittag-Leffler problem is trivial locally, and so it is a problem of whether we can glue this locally defined functions around each pole to a global meromorphic function. With this in mind, we generalize our discussion by making the following definition:

\begin{definition}\label{mldistribution}
Let $X$ be any Riemann surface, and let $\mathcal{U} = \{U_i\}$ be an open covering of $X$. A \textit{$0$-cochain} $\mu = (f_i) \in C^{0}(\mathcal{U}, \mathcal{K}_X)$ is called a \textit{Mittag-Leffler distribution} if the diferences $f_i - f_j$ are holomorphic in $U_i \cap U_j$, that is, \textbf{if $\delta\mu$ is an element of $Z^1(\mathcal{U}, \mathcal{O}_X)$} (note that the condition requires that it be a \textbf{holomorphic} 1-cocycle, and so an element of $Z^1(\mathcal{U}, \mathcal{O}_X)$).
\end{definition}

Note that the condition for a $0$-cochain to be a Mittag-Leffler distribution implies that since the differences are holomorphic, they have the \textbf{same principal parts}. A \textit{solution} of $\mu$ is a global meromorphic function $f \in \mathcal{K}_{X}(X)$ with this same principal part, that is $\left. f \right|_{U_{i}} - f_i \in \mathcal{O}_{X}(U_i)$.

So giving a Mittag-Leffler distribution on $X$ is the same as giving the data that defines a Mittag-Leffler problem on $X$, such that the local solutions are encoded in the $0$-cochain $\mu$. This allows us to solve the problem using \v{C}ech cohomology. In fact, we can prove

\begin{theorem}
A Mittag-Leffler distribution $\mu$ has a solution if, and only if $[\delta\mu] = 0$, where $[\delta\mu] \in \check{\mathrm{H}}^1(X, \mathcal{O}_X)$ is the cohomology class represented by the cocycle $\delta\mu$.
\end{theorem}

\begin{proof}
For the \textit{if} part: suppose that $f \in \mathcal{K}_{X}(X)$ is a solution of $\mu$. Set $g_i = f_i - f$. Then $g_i \in \mathcal{O}_{X}(U_i)$ by hypothesis. On $U_i \cap U_j$ we have

\begin{equation}
    f_j - f_i = (f_j - f) - (f_i - f) = g_j - g_i
\end{equation}
and so $\delta\mu$ is a $1$-coboundary, that is $\delta\mu \in B^1(\mathcal{U}, \mathcal{O}_X)$, and $[\delta\mu] = 0$ in $\check{\mathrm{H}}^1(X, \mathcal{O}_X)$.

For the \textit{only if} part: suppose $\delta \mu \in B^1(\mathcal{U}, \mathcal{O}_X)$. Then there exists a $0$-cochain $(g_i) \in C^0(\mathcal{U}, \mathcal{O}_X)$ such that on $U_i \cap U_j$

\begin{equation}
    f_i - f_j = g_i - g_j
\end{equation}
and so, $f_i - g_i = f_j - g_j$ on $U_i \cap U_j$. Thus, by the \textbf{gluability axiom}, the differences $f_i - g_i$ glue together to some global section $\mathcal{K}_{X}(X)$. Moreover, since $\left. f \right|_{U_i} = f_i - g_i \in \mathcal{O}_{X}(U_i)$ for every $i$, then by the \textbf{identity axiom} $f$ is a solution of $\mu$.
\end{proof}

Now, lets suppose that $X$ is a compact Riemann surface, and lets look at the exact sequence

\begin{equation}\label{sheafprincipal}
\begin{tikzcd}
    0 \arrow{r} & \mathcal{O}_X \arrow{r} & \mathcal{K}_X \arrow{r} & \mathcal{K}_X / \mathcal{O}_X \arrow{r} & 0.
\end{tikzcd}
\end{equation}

The quotient sheaf $\mathcal{K}_X / \mathcal{O}_X$ is the sheaf whose sections are equivalence classes of meromorphic functions such that $f_i - f_j = 0$ if the difference is holomorphic. That is, we can think of \textbf{the sheaf $\mathcal{K}_X / \mathcal{O}_X$ as the sheaf of principal parts on $X$}.\footnote{We will discuss more about this formulation of the data of a Mittag-Leffler problem below.} So a $0$-cochain $\mu \in C^0(\mathcal{U}, \mathcal{K}_X / \mathcal{O}_X)$ is nothing but a Mittag-Leffler distribution on $X$. Now, lets look at part of the long exact sequence in cohomology induced by (\ref{sheafprincipal}):

\begin{equation}
\begin{tikzcd}
    \mathrm{H}^0(X, \mathcal{K}_X) \arrow{r} & \mathrm{H}^0(X, \mathcal{K}_X/ \mathcal{O}_X) \arrow{r} & \mathrm{H}^1(X, \mathcal{O}_X) \arrow{r} & \mathrm{H}^1(X, \mathcal{K}_X).
\end{tikzcd}
\end{equation}

One can prove using Serre Duality that $\mathrm{H}^1(X, \mathcal{K}_X) = 0$ (cf. \cite{forster}, Corollary 17.17) and so the map $\mathrm{H}^0(X, \mathcal{K}_X/ \mathcal{O}_X) \longrightarrow \mathrm{H}^1(X, \mathcal{O}_X)$ is surjective. Therefore, for every $\xi \in \mathrm{H}^1(X, \mathcal{O}_X)$ there exists a Mittag-Leffler distribution $\mu$ such that $[\delta\mu] = \xi$, and so \textbf{if $X$ is compact there exists Mittag-Leffler problems on $X$ which have no solution}.

We can prove a criterion for when a Mittag-Leffler distribution $\mu$ on a compact Riemann surface $X$ has a solution. For this, we need to introduce the Serre Duality Theorem.

However, before discussing the Serre Duality Theorem, we would like to point that it is not a coincidence that so far we have found the obstructions to the solutions of Mittag-Leffler problems to be both on $\mathrm{H}^1(X, \mathcal{O}_X)$ and $\mathrm{H}^{(0,1)}_{\Bar{\partial}}(X)$. This is a particular case of the following theorem, which is the holomorphic analog of Theorem \ref{derhamthm}:

\begin{theorem}[Dolbeault Theorem]\label{dolbeaultthm}
Let $X$ be a complex manifold. We have

\begin{equation}
    \mathrm{H}^{q}(X, \Omega_{X}^p) = \mathrm{H}^{(p,q)}_{\Bar{\partial}}(X).
\end{equation}
\end{theorem}

To prove this theorem we need two facts, whose proofs can be found in \cite{griffithsharris}, Chapter 0:

\begin{lemma}\label{cohomcinftyforms}
If $X$ is a complex manifold, then

\begin{equation}
    \mathrm{H}^r(X, \mathcal{A}^{(p,q)}) = 0
\end{equation}
for all $r > 0$.
\end{lemma}

\begin{lemma}\label{conseqpoincarelemma}
If $X$ is a complex manifold, then the sequence of sheaves

\begin{equation}
\begin{tikzcd}
    0 \arrow{r} & \Omega^{p}_{X} \arrow{r}{\Bar{\partial}} & \mathcal{A}^{(p,0)} \arrow{r}{\Bar{\partial}} & \mathcal{A}^{(p,1)} \arrow{r}{\Bar{\partial}} & \ldots \arrow{r}{\Bar{\partial}} & \mathcal{A}^{(p,q)} \arrow{r}{\Bar{\partial}} & \ldots 
\end{tikzcd}
\end{equation}
is exact.
\end{lemma}

Lemma \ref{conseqpoincarelemma} is a direct consequence of the \textbf{$\Bar{\partial}$-Poincaré Lemma}, which is the complex analog of Theorem \ref{poincarelemma}.

\begin{proof}[Proof of Dolbeault Theorem]
Let $\mathcal{A}^{(p,q)}_{\Bar{\partial}}(X)$ denote the $\Bar{\partial}$-closed $(p,q)$-forms on $X$. Then, by Lemma \ref{conseqpoincarelemma}, we get short exact sequences

\begin{equation}
\begin{tikzcd}
    0 \arrow{r} & \Omega^{p}_{X} \arrow{r} & \mathcal{A}^{(p,0)} \arrow{r}{\Bar{\partial}} & \mathcal{A}^{(p,1)}_{\Bar{\partial}} \arrow{r} & 0.
\end{tikzcd}
\end{equation}
and

\begin{equation}
\begin{tikzcd}
    0 \arrow{r} & \mathcal{A}^{(p,q)}_{\Bar{\partial}} \arrow{r} & \mathcal{A}^{(p,q)} \arrow{r}{\Bar{\partial}} & \mathcal{A}^{(p,q+1)}_{\Bar{\partial}} \arrow{r} & 0.
\end{tikzcd}
\end{equation}
of sheaves, for every $p, q$. Looking at the long exact sequences in cohomology induced by these short exact sequences, we get

\begin{equation}
\begin{tikzcd}
    0 \arrow{r} & \Omega^{p}_{X} \arrow{r} & \mathcal{A}^{(p,0)} \arrow{r}{\Bar{\partial}} & \mathcal{A}^{(p,1)}_{\Bar{\partial}} \arrow{r} & 0
\end{tikzcd}
\end{equation}
induces

\begin{equation}
\begin{tikzcd}
    \mathrm{H}^{q-1}(X, \Omega^{p}_{X}) \arrow{r} & \mathrm{H}^{q-1}(X, \mathcal{A}^{(p,0)}) \arrow{r} & \mathrm{H}^{q-1}(X, \mathcal{A}^{(p,1)}_{\Bar{\partial}}) \\ \arrow{r} & \mathrm{H}^{q}(X, \Omega^{p}_{X}) \arrow{r} & \mathrm{H}^{q}(X, \mathcal{A}^{(p,0)}) \arrow{r} & \mathrm{H}^{q}(X, \mathcal{A}^{(p,1)}_{\Bar{\partial}})
\end{tikzcd}
\end{equation}
and since by Lemma \ref{cohomcinftyforms} $\mathrm{H}^{q-1}(X, \mathcal{A}^{(p,0)})$ and $\mathrm{H}^{q}(X, \mathcal{A}^{(p,0)})$ are $0$, we conclude that

\begin{equation}
\mathrm{H}^{q}(X, \Omega^{p}_{X}) \cong \mathrm{H}^{q-1}(X, \mathcal{A}^{(p,1)}_{\Bar{\partial}}).
\end{equation}
If we repeat this process for the next sequence, we get that 

\begin{equation}
\begin{tikzcd}
    0 \arrow{r} & \mathcal{A}^{(p,1)}_{\Bar{\partial}} \arrow{r} & \mathcal{A}^{(p,1)} \arrow{r}{\Bar{\partial}} & \mathcal{A}^{(p,2)}_{\Bar{\partial}} \arrow{r} & 0
\end{tikzcd}
\end{equation}
induces

\begin{equation}
\begin{tikzcd}
    \mathrm{H}^{q-2}(X, \mathcal{A}^{(p,1)}_{\Bar{\partial}}) \arrow{r} & \mathrm{H}^{q-2}(X, \mathcal{A}^{(p,1)}) \arrow{r} & \mathrm{H}^{q-2}(X, \mathcal{A}^{(p,2)}_{\Bar{\partial}}) \\ \arrow{r} & \mathrm{H}^{q-1}(X, \mathcal{A}^{(p,1)}_{\Bar{\partial}}) \arrow{r} & \mathrm{H}^{q-1}(X, \mathcal{A}^{(p,1)}) \arrow{r} & \mathrm{H}^{q-1}(X, \mathcal{A}^{(p,2)}_{\Bar{\partial}})
\end{tikzcd}
\end{equation}
and again, conclude that $\mathrm{H}^{q}(X, \Omega^{p}_{X}) \cong \mathrm{H}^{q-1}(X, \mathcal{A}^{(p,1)}_{\Bar{\partial}}) \cong \mathrm{H}^{q-2}(X, \mathcal{A}^{(p,2)}_{\Bar{\partial}})$. Repeating this process $q$ times, we get to

\begin{equation}
\begin{tikzcd}
    0 \arrow{r} & \mathrm{H}^{0}(X, \mathcal{A}^{(p,q-1)}) \arrow{r}{\Bar{\partial}} & \mathrm{H}^{0}(X, \mathcal{A}^{(p,q)}_{\Bar{\partial}}) \\ \arrow{r} & \mathrm{H}^{1}(X, \mathcal{A}^{(p,q-1)}_{\Bar{\partial}}) \arrow{r} & \mathrm{H}^{1}(X, \mathcal{A}^{(p,q-1)}) \arrow{r} & \mathrm{H}^{1}(X, \mathcal{A}^{(p,q)}_{\Bar{\partial}})
\end{tikzcd}
\end{equation}
and so

\begin{equation}
    \mathrm{H}^{q}(X, \Omega^{p}_{X}) \cong \frac{\mathrm{H}^0(X, \mathcal{A}^{(p,q)}_{\Bar{\partial}})}{\Bar{\partial}\mathrm{H}^0(X, \mathcal{A}^{(p,q-1)})} = \frac{\mathcal{A}^{(p,q)}_{\Bar{\partial}}(X)}{\Bar{\partial}\mathcal{A}^{(p,q-1)}(X)} = \mathrm{H}^{(p,q)}_{\Bar{\partial}}(X).
\end{equation}
\end{proof}

In particular, Theorem \ref{dolbeaultthm} tells us that for a complex manifold $X$ we have

\begin{equation}
    \mathrm{H}^1(X, \mathcal{O}_X) \cong \mathrm{H}^{(0,1)}_{\Bar{\partial}}(X).
\end{equation}

\noindent\textbf{The Serre Duality Theorem:}  A proof of the version of the duality theorem of Serre that we will present can be found in \cite{forster}, Chap. 2, §17.\footnote{More general versions of both the Serre Duality Theorem and the Riemann-Roch Theorem as well as their proofs can be found in \cite{hartshorne}, Chapters III and IV respectively.} The reason for this section is just to be able to introduce the objects necessary to state the Serre Duality Theorem, and most of it is not necessary outside of this context. For this reason, most proofs are omitted.

We begin by generalizing the notion of residue of a meromorphic function. We want to define a linear functional $Res: \mathrm{H}^1(X, \Omega^{1}_{X}) \rightarrow \mathbb{C}$. Suppose $X$ is a compact Riemann surface. By Theorem \ref{dolbeaultthm}, we have a short exact sequence

\begin{equation}
\begin{tikzcd}
    0 \arrow{r} & \Omega^{1}_{X} \arrow{r} & \mathcal{A}^{(1,0)} \arrow{r}{\Bar{\partial}} & \mathcal{A}^{(1,1)} \arrow{r} & 0
\end{tikzcd}
\end{equation}
of sheaves on $X$, and an isomorphism $\mathrm{H}^1(X, \Omega^{1}_{X}) \cong \frac{\mathcal{A}^{(1,1)}}{\Bar{\partial}\mathcal{A}^{(1,0)}}$. Let $\xi \in \mathrm{H}^1(X, \Omega^{1}_{X})$ be a cohomology class, and $\omega \in \mathcal{A}^{(1,1)}(X)$ be a representative of $\xi$ via the above isomorphism. Define

\begin{equation}
    Res(\xi) = \frac{1}{2 \pi i} \int\int_{X} \omega.
\end{equation}
This definition is independent of the choice of representative $\omega$ of $\xi$ (cf. \cite{forster}, Chap. 2, §17).\footnote{The ideia from this definition comes from the theory of functions of one complex variable, where the residue is the value of the contour integral of a meromorphic function. For background on integration of forms on Riemann surfaces, see \cite{forster}, Chap. 1, §10.}

We now extend Definition \ref{mldistribution} to meromorphic $1$-forms: let $X$ be a Riemann surface, and denote by $\mathcal{K}_{X}^{(1)}$ the sheaf of meromorphic 1-forms on $X$. Let $\mathcal{U} = \{U_i\}$ be an open covering of $X$. A $0$-cochain $\mu = (\omega_i) \in C^0(\mathcal{U}, \mathcal{K}_{X}^{(1)})$ is called a Mittag-Leffler distribution of meromorphic $1$-forms if $\delta\mu \in Z^1(\mathcal{U}, \Omega_{X}^{1})$.

We can now define the \textit{residue of a Mittag-Leffler distribution}:

\begin{definition}
Let $X$ be a Riemann surface, and let $a \in X$. Let $\mu = (\omega_i)$ be a Mittag-Leffler distribution of meromorphic $1$-forms on $X$. Let $U_i \in \mathcal{U}$ be such that $a \in U_i$. Then the \textit{residue of $\mu$ at $a$} is

\begin{equation}
    Res_{a}(\mu) = Res_{a}(\omega_i).
\end{equation}
In any intersection such that $a \in U_i \cap U_j$ we have that $\omega_i - \omega_j$ is holomorphic, and so since they have the same principal part at $a$, their residue is the same. So the definition is independent of the index $i$ such that $a \in U_i$.
\end{definition}

Now, if $X$ is compact, we may assume that the covering is finite, and so we have finitely many meromorphic $1$-forms, and $Res_{a}(\mu) \neq 0$ for only a finite number of points of $X$. So if $X$ is compact, we can define the \textbf{residue of $\mu$} as

\begin{equation}
    Res(\mu) = \sum Res_{a}(\mu)
\end{equation}
and this sum is finite.

The notion of the residue of a Mittag-Leffler distribution and the map $Res$ defined above are related in the following way:

\begin{theorem}\label{relatingres}
Let $X$ be a compact Riemann surface, and let $\mu$ be a Mittag-Leffler distribution of meromorphic $1$-forms. Denote by $[\delta\mu]$ the cohomology class of $\delta\mu$ in $\mathrm{H}^1(X, \Omega^{1}_X)$. Then

\begin{equation}
    Res(\mu) = Res([\delta\mu]).
\end{equation}
\end{theorem}

The proof of this theorem can be found in \cite{forster}, Chap. 2, §17, Theorem 17.3.

The main goal of the Serre Duality Theorem is to be able to understand the groups $\mathrm{H}^1(X, \mathcal{O}(D))$ for some divisor $D$ in terms of some simpler cohomology groups. We now have to develop an analog of the sheaves $\mathcal{O}(D))$ for sheaves of meromorphic forms:

\begin{definition}
Let $\omega \in \mathcal{K}_{X}^{(1)}(U)$ be a meromorphic $1$-form on some open set $U$. We can define the order of $\omega$ at some point $a \in U$ in the following way: let $(V, z)$ be a coordinate neighborhood around $a$. Then on $U \cap V$, $\omega = fdz$ for some meromorphic function $f$. Set the order of $\omega$ at $a$ $ord_{a}(\omega)$ to be $ord_{a}(f)$. This definition is independent of the choice of coordinate neighborhood $(V, z)$. In this way, we can define the divisor associated to $\omega$ in the same way that we defined the divisor associated to a meromorphic function $f$. This divisor will be denoted $(\omega)$. Now let $D = \sum n_{x}[x]$ be a divisor. We then define the subsheaf $\Omega_{D}$ of $\mathcal{K}_{X}^{(1)}(U)$ to be the sheaf given by

\begin{equation}
    \Omega_{D}(U) = \{\omega \in \mathcal{K}_{X}^{(1)}(U): ord_{x}(\omega) \ge -n_x, x \in U\}.
\end{equation}
\end{definition}

In particular, note that if $D = 0$ then $\Omega_{D} = \Omega_{X}^{1}$, the sheaf of holomorphic $1$-forms, since there are no poles.

Now, let $X$ be a compact Riemann surface, $D \in Div(X)$ be a divisor on $X$. The product map

\begin{equation}
    \Omega_{-D} \times \mathcal{O}(D) \rightarrow \Omega_{X}^{1}, \text{ } (\omega, f) \mapsto \omega f
\end{equation}
induces a map on cohomology

\begin{equation}
    \mathrm{H}^0(X, \Omega_{-D}) \times \mathrm{H}^1(X, \mathcal{O}(D)) \rightarrow \mathrm{H}^1(X, \Omega_{X}^{1}).
\end{equation}

Composing this map with the $Res$ map defined above, we get a bilinear form

\begin{equation}
    \mathrm{H}^0(X, \Omega_{-D}) \times \mathrm{H}^1(X, \mathcal{O}(D)) \rightarrow \mathbb{C}
\end{equation}
given by $(\omega, \xi) = Res(\omega\xi)$. So by the representation theorem in linear algebra, this bilinear form induces a map

\begin{equation}
    \mathrm{H}^0(X, \Omega_{-D}) \rightarrow \mathrm{H}^1(X, \mathcal{O}(D))^{*}
\end{equation}
to the dual of $\mathrm{H}^1(X, \mathcal{O}(D))$. \textbf{The Serre Duality Theorem states that this map is in fact an isomorphism}.\footnote{In other words, it states that the \textbf{pairing} $\mathrm{H}^0(X, \Omega_{-D}) \times \mathrm{H}^1(X, \mathcal{O}(D)) \rightarrow \mathbb{C}$ is \textbf{perfect}.}

\begin{theorem}[Serre Duality Theorem]\label{serreduality}
Let $X$ be a compact Riemann surface, and let $D \in Div(X)$ be any divisor on $X$. Then the map

\begin{equation}
    \mathrm{H}^0(X, \Omega_{-D}) \rightarrow \mathrm{H}^1(X, \mathcal{O}(D))^{*}
\end{equation}
defined above is an isomorphism.
\end{theorem}

The Serre Duality Theorem allows us to prove the following theorem that characterizes when a Mittag-Leffler problem on a compact Riemann surface has a solution:

\begin{theorem}
Let $X$ be a compact Riemann surface, and let $\mathcal{U}$ be an open covering of $X$. Let $\mu \in C^0(\mathcal{U}, \mathcal{K}_{X})$ be a Mittag-Leffler distribution of meromorphic functions on $X$. Then $\mu$ has a solution if, and only if

\begin{equation}
    Res(\omega\mu) = 0
\end{equation}
for every $\omega \in \Omega^{1}_{X}$.
\end{theorem}

Note that if $\omega \in \Omega^{1}_{X}$ then the product $\omega \mu$ is defined and is a Mittag-Leffler distribution of meromorphic $1$-forms, $\omega\mu \in C^0(\mathcal{U}, \mathcal{K}_{X}^{(1)})$. Then its residue is defined, as we have discussed above.

\begin{proof}
Recall that we have proved that $\mu$ has a solution if, and only if $[\delta\mu] \in \mathrm{H}^1(X, \mathcal{O}_X)$ vanishes. That happens precisely if for every linear functional $\lambda \in \mathrm{H}^1(X, \mathcal{O}_X)^{*}$, we have that $\lambda([\delta\mu]) = 0$. But $\mathrm{H}^1(X, \mathcal{O}_X)$ is the same as $\mathrm{H}^1(X, \mathcal{O}(0))$, and so, by Serre Duality we have that $\mathrm{H}^1(X, \mathcal{O}_X)^{*} \cong \mathrm{H}^0(X, \Omega^{1}_{X})$, and so every linear functional on $\mathrm{H}^1(X, \mathcal{O}_X)$ is equivalent to some holomorphic $1$-form $\omega \in \mathrm{H}^0(X, \Omega^{1}_{X})$ on $X$. Moreover, by the discussion leading to Theorem \ref{serreduality}, applying the linear functional $\lambda$ is the same as computing the bilinear form

\begin{equation}
    (\omega, [\delta\mu]) = Res(\omega[\delta\mu])
\end{equation}
and by Theorem \ref{relatingres}, $Res(\omega[\delta\mu]) = Res(\omega\mu)$. So $[\delta\mu] = 0$ if, and only if for every $\omega \in \Omega^{1}_{X}$ we have that $Res(\omega\mu) = 0$.
\end{proof}

To end this section, we quickly turn our attention to the noncompact case. We can prove that if $X$ is a noncompact Riemann surface, then every Mittag-Leffler problem on $X$ has a solution:

\begin{theorem}\label{mlncpct}
Let $X$ be a noncompact Riemann surface. Then

\begin{equation}
    \mathrm{H}^1(X, \mathcal{O}_X) = 0.
\end{equation}
\end{theorem}

By Theorem \ref{dolbeaultthm}, we know that $\mathrm{H}^1(X, \mathcal{O}_X) = \mathrm{H}^{(0,1)}_{\Bar{\partial}}(X) = \frac{\mathcal{A}^{(0,1)}(X)}{\Bar{\partial}\mathcal{A}^{(0,0)}(X)}$. So to prove that $\mathrm{H}^1(X, \mathcal{O}_X) = 0$, one has to prove that every $(0,1)$-form $\omega$ on a noncompact Riemann surface $X$ is $\Bar{\partial}$-exact, that is, for every $\omega \in \mathcal{A}^{(0,1)}(X)$ there exists a smooth function $f \in C^{\infty}(X)$ such that $\Bar{\partial}f = \omega$. This is done with an analog of Theorem \ref{runge}, and so we're back to complex analysis. The full proof can be found in \cite{forster}, Chap. 3, §25, Theorem 25.6.

The proof begins with the definition of a \textit{Runge domain}: a subset $Y$ of a Riemann surface $X$ is called a Runge domain if $X \setminus Y$ has no compact connected components.

One then proceeds to prove that if $X$ is noncompact, then there exists an \textit{exhaustion} on $X$ by Runge domains, that is, there exists a sequence

\begin{equation}
    Y_1 \Subset Y_2 \Subset Y_3 \Subset ... \Subset Y_k \Subset
\end{equation}
of relatively compact Runge domains such that $X = \bigcup Y_n$. Such an exhaustion in can be used to prove Theorem \ref{rungesurfaces}. The next step is again a special case of a result from complex analysis:

\begin{theorem}[Dolbeault's lemma]\label{dolbeaultlemma}
Let $g$ be a smooth function on an open disc in the plane. Then the \textit{inhomogenious Cauchy-Riemann equation}

\begin{equation}
    \frac{\partial f}{\partial\overline{z}} = g
\end{equation}
has a smooth solution $f$.
\end{theorem}

Then Theorem \ref{dolbeaultlemma} allows one to prove that if $X$ is a noncompact Riemann surface and $Y \Subset \Tilde{Y} \subset X$ are open subsets, then for every form $\omega \in \mathcal{A}^{0,1}(\Tilde{Y})$ there exists $f \in C^{\infty}(Y)$ such that $\bar{\partial}f = \left. \omega \right|_{Y}$. With this local solutions in hand, we then take an exhaustion of $X$ by relatively compact Runge domains and prove Theorem \ref{mlncpct} by induction: let $f_0 \in C^{\infty}(Y_0)$ be a local solution of

\begin{equation}
    \bar{\partial}f_0 = \left. \omega \right|_{Y_0}.
\end{equation}
Now suppose that we have constructed solutions for $Y_0,...,Y_n$. We again find a local solution $g_{n+1}$ such that

\begin{equation}
    \bar{\partial}g_{n+1} = \left. \omega \right|_{Y_{n+1}}.
\end{equation}
Since

\begin{equation}
    Y_{n} \Subset Y_{n+1},
\end{equation}
over $Y_n$ we have $\bar{\partial}g_{n+1} = \bar{\partial}f_{n}$, and so $\bar{\partial}(g_{n+1} - f_n) = 0$ and $g_{n+1} - f_n \in \mathcal{O}_{Y_n}(Y_n)$. We then use Theorem \ref{rungesurfaces} and find $\Tilde{g} \in \mathcal{O}_{Y_{n+1}}(Y_{n+1})$ such that 

\begin{equation}
    \lVert (g_{n+1} - f_n) - \Tilde{g} \rVert \le \frac{1}{2^n}
\end{equation}
where the norm is that of uniform convergence over compact sets on $C^{\infty}(Y_{n+1})$. Setting

\begin{equation}
    g_{n+1} - \Tilde{g} = f_{n+1}
\end{equation}
since $\bar{\partial}(\Tilde{g}) = 0$, $\bar{\partial}f_{n+1} = \bar{\partial}g_{n+1}$ and $\bar{\partial}g_{n+1} = \left. \omega \right|_{Y_{n+1}}$. One then proves that as $n$ grows larger, the difference

\begin{equation}
    \lVert f_{n+1} - f_n \rVert \le \frac{1}{2^n}
\end{equation}
becomes smaller, and that the sequence $\{f_n\}$ converges to a global solution $f$ of $\bar{\partial}f = \omega$, and so $\omega$ is $\bar{\partial}$-closed.

This theorem is a special case of a very important theorem, which we will discuss below.

\subsection{A Glimpse into Higher Dimensions: Analytic Spaces and The First Cousin Problem}

So far, we have asked ourselves whether we could construct a specific type of global meromorphic function on a Riemann surface $X$. The purpose of this section is to showcase the extension of this discussion to higher dimensional spaces. The theory which we will discuss is beyond the scope of this work, and so this section will have an expository flavour. The references are \cite{gunningrossi} and \cite{grauertremmert}.

\medskip

\noindent\textbf{Analytic Spaces:} Analytic spaces generalize analytic varieties and complex manifolds very similarly to how schemes generalize algebraic varieties. In what follows, $\mathcal{O}_n$ denotes the sheaf of holomorphic functions in $\mathbb{C}^n$.

Let $\mathscr{I}$ be a sheaf of ideals of $\mathcal{O}_n$ such that for every $z \in \mathbb{C}^n$ we have an open neighborhood $z \in U \subseteq \mathbb{C}^n$ and holomorphic functions $f_1,...,f_k \in \mathcal{O}_n(U)$ such that the ideal $\mathscr{I}(U)$ is \textit{finitely generated as an ideal} by the functions $f_1,...,f_k$. Consider the set $Y$ of points $z \in \mathbb{C}^n$ such that the stalk $\mathscr{I}_z$ of $\mathscr{I}$ at $z$ is \textit{not the unit ideal} in $\mathscr{O}_{n, z}$, or equivalently, the set of points in which the stalk of the quotient sheaf

\begin{equation}
    (\mathcal{O}_n / \mathscr{I})_z
\end{equation}
is not the zero ring.

For each $z \in Y$ we can find an open neighborhood $W$ of $z$ such that $Y \cap W$ is given by the vanishing locus of the functions $f_1,..., f_k$, and so locally, $Y$ \textit{is given by the zero locus of finitely many holomorphic functions}. The restriction of this quotient

\begin{equation}
    \mathcal{O}_Y = \left. \mathcal{O}_n/\mathscr{I} \right|_{Y}
\end{equation}
is a sheaf of rings on $Y$.

\begin{definition}
A ringed space $(X, \mathcal{O}_X)$ is called an \textit{analytic space} if for every $x \in X$ we can find an open neighbourhood $U$ of $x$ in $X$ such that $(U, \mathcal{O}_{\left. X \right|_{U}})$ is isomorphic as a ringed space to a ringed space $(Y, \mathcal{O}_Y)$ as above.
\end{definition}

One may observe that analytic subvarieties of $\mathbb{C}^n$ as defined in Section \ref{divlb} are examples of analytic spaces.

A section of $\mathcal{O}_X$ is called a \textbf{holomorphic function on $X$}. A map of analytic spaces is a map between the underlying ringed spaces. A map of analytic spaces is called a \textit{holomorphic map}. An isomorphism of analytic spaces is called a \textit{biholomorphic map}.

Complex manifolds are a special case of analytic spaces: we may define a complex manifold as a ringed space $(X, \mathcal{O}_X)$ if for every $x \in X$ we can find an open neighbourhood $U$ of $x$ in $X$ such that $(U, \mathcal{O}_{\left. X \right|_{U}})$ is isomorphic as a ringed space to a ringed space $(Y, \mathcal{O}_Y)$ where $Y \subset 
\mathbb{C}^n$ is a domain (i.e., an open connected subset) and $\mathcal{O}_Y$ is the sheaf of germs of holomorphic functions on $Y$.

We can then redefine meromorphic functions on complex manifolds in this new language of analytic spaces:

\begin{proposition}
Let $(X, \mathcal{O}_X)$ be a connected complex manifold. Then $\Gamma(X, \mathcal{O}_X)$ is a unique factorization domain.
\end{proposition}

A proof and a discussion of this fact can be found at the beginning of \cite{griffithsharris}, Chap. 0.

\begin{definition}
Let $(X, \mathcal{O}_X)$ be a complex manifold, and let $\mathcal{U}$ be an open covering of $X$ such that every $U \in \mathcal{U}$ is connected. Denote by $\mathcal{K}_{X}(U)$ the fraction field of $\mathcal{O}_{X}(U)$. The collection $\mathcal{K}_{X}(U)$ determines a presheaf on $X$. Its sheafification, which we also denote by $\mathcal{K}_X$ is called the \textit{sheaf of germs of meromorphic functions} on $(X, \mathcal{O}_X)$.
\end{definition}

In particular, we have that for every $x \in X$, the stalk $\mathcal{K}_{X,x}$ of $\mathcal{K}_{X}$ at $x$ is the fraction field of $\mathcal{O}_{X,x}$ (cf. \cite{gunningrossi}, Chapter VIII, Section B, Proposition 2).

Since $\mathcal{O}_{X,x}$ is an UFD, any $m \in \mathcal{K}_{X,x}$ can be written uniquely, up to multiplication by unit, as $m = \frac{f}{g}$, with $f, g \in \mathcal{O}_{X,x}$ relatively prime.

We can now formulate the analog of the Mittag-Leffler problem for higher dimensional manifolds, the First Cousin Problem:

\medskip

\noindent\textbf{First Cousin Problem:} Let $(X, \mathcal{O}_X)$ be a complex manifold, $\mathcal{U} = \{U_{\alpha}\}$ be an open covering of $X$, and let $m_{\alpha}$ be the germ of a meromorphic function on $U_{\alpha}$, that is, $m_{\alpha} \in \Gamma(U_{\alpha}, \mathcal{K}_X)$. Suppose that $m_{\alpha} - m_{\beta} \in \Gamma(U_{\alpha} \cap U_{\beta}, \mathcal{O}_X)$ for all $\alpha, \beta$. Can we find $m \in \Gamma(X, \mathcal{K}_X)$ such that $m - m_{\alpha} \in \Gamma(U_{\alpha}, \mathcal{O}_X)$ for all $\alpha$?

We can solve this problem for a nice class of complex manifolds, namely that of \textbf{Stein spaces}. The definition of Stein spaces is rather technical and analytical, and can be found in \cite{gunningrossi}, Chapter VII, Section A, Definition 2. 

Examples of Stein spaces include $\mathbb{C}^n$ and noncompact Riemann surfaces. For a more elaborate example, consider an \textit{analytic subspace $(X, \mathcal{O}_X)$ of $(\mathbb{C}^n, \mathcal{O}_n)$}, and let $f_1,...,f_l \in \mathcal{O}_X(X)$ be such that the set

\begin{equation}
    \{z \in X : |f_i(z)| \le 1\}
\end{equation}
is compact. Then the set $\{z \in X : |f_i(z)| < 1\}$ is a Stein space (cf. \cite{gunningrossi}, Chapter VII, Section A, Examples).

What we will need here however is an important feature of Stein spaces, the celebrated

\begin{theorem}[Cartan's Theorem B]\label{cartanb}
Let $(X, \mathcal{O}_X)$ be a Stein space, and let $\mathscr{F}$ be a coherent sheaf on $X$. Then $\mathrm{H}^q(X, \mathscr{F}) = 0$ for all $q \ge 1$.
\end{theorem}

This theorem is one of the most important results in the theory of analytic spaces. In particular, it generalizes Theorem \ref{mlncpct}. Its proof is (way) beyond the scope of this document. It can be found in \cite{gunningrossi}, Chapter VIII, Section A, Theorem 14.

With Theorem \ref{cartanb}, we can prove

\begin{theorem}
The First Cousin Problem has a solution if $X$ is Stein.
\end{theorem}

\begin{proof}
Consider again the exact sequence of sheaves

\begin{equation}
\begin{tikzcd}
    0 \arrow{r} & \mathcal{O}_X \arrow{r} & \mathcal{K}_{X} \arrow{r}{\pi} & \mathcal{K}_{X}/\mathcal{O}_X \arrow{r} & 0
\end{tikzcd}
\end{equation}
where $\pi$ is the projection. Looking at the long exact sequence in cohomology, we have

\begin{equation}
\begin{tikzcd}
    \mathrm{H}^{0}(X, \mathcal{O}_X) \arrow{r} & \mathrm{H}^{0}(X, \mathcal{K}_X) \arrow{r} & \mathrm{H}^{0}(X, \mathcal{K}_{X}/\mathcal{O}_X) \\ \arrow{r} & \mathrm{H}^{1}(X, \mathcal{O}_X) \arrow{r} & \mathrm{H}^{1}(X, \mathcal{K}_X) \arrow{r} & \mathrm{H}^{1}(X, \mathcal{K}_{X}/\mathcal{O}_X)
\end{tikzcd}
\end{equation}
and since $X$ is Stein by hypothesis, Theorem \ref{cartanb} tells us that $\mathrm{H}^{1}(X, \mathcal{O}_X) = 0$, and so

\begin{equation}
    \pi: \mathrm{H}^{0}(X, \mathcal{K}_X) \rightarrow \mathrm{H}^{0}(X, \mathcal{K}_{X}/\mathcal{O}_X)
\end{equation}
is surjective. Denote by $\{m_{\alpha} \in \Gamma(U_{\alpha}, \mathcal{K}_X)\}$ the given data of the First Cousin Problem on $X$. Then for every $\alpha$,

\begin{equation}
    \pi (m_{\alpha}) \in \mathrm{H}^{0}(X, \mathcal{K}_{X}/\mathcal{O}_X)
\end{equation}
and since in $U_{\alpha} \cap U_{\beta}$ the difference $m_{\alpha} - m_{\beta}$  is holomorphic by hypothesis, that is, is an element of $\mathrm{H}^{0}(U_{\alpha} \cap U_{\beta}, \mathcal{O}_X)$ we have that $\pi(m_{\alpha}) = \pi(m_{\beta})$ in $\mathrm{H}^{0}(U_{\alpha} \cap U_{\beta}, \mathcal{K}_{X}/\mathcal{O}_X)$. So let

\begin{equation}
    \sigma \in \mathrm{H}^{0}(X, \mathcal{K}_{X}/\mathcal{O}_X)
\end{equation}
be such that $\sigma = \pi(m_{\alpha})$ in $\mathrm{H}^{0}(U_{\alpha}, \mathcal{K}_{X}/\mathcal{O}_X)$. Since $\pi$ is surjective, there exists some $m \in \mathrm{H}^{0}(X, \mathcal{K}_X)$ such that $\sigma = \pi(m)$. This way, we have that $\pi(m) = \pi(m_{\alpha}) \in \mathrm{H}^{0}(U_{\alpha}, \mathcal{K}_{X}/\mathcal{O}_X)$, so $\pi(m) - \pi(m_{\alpha}) = 0$, and $m - m_{\alpha} \in \Gamma(U_{\alpha}, \mathcal{O}_X)$.
\end{proof}

\pagebreak

\begin{appendices}

\section{Proof of the Riemann-Roch formula}

\begin{theorem}[The Riemann-Roch Theorem]
Let $X$ be a compact Riemann surface of genus $g$, and let $D$ be a divisor on $X$. Then, $\dim{\mathrm{H}^{0}(X, \mathcal{O}(D))}$ and $l(D)$ are finite and

\begin{equation}\label{rrformula}
    \dim{\mathrm{H}^{0}(X, \mathcal{O}(D))} - l(D) = 1 - g + deg(D).
\end{equation}
\end{theorem}

\begin{proof}
We begin by noting that the formula (\ref{rrformula}) holds if $D$ is the zero divisor: indeed, since we showed that $\mathcal{O}(D) = \mathcal{O}_X$, by taking global sections we have 

\begin{equation}
   \mathrm{H}^{0}(X, \mathcal{O}(D)) = \mathcal{O}_{X}(X)
\end{equation}
and since $X$ is compact by hypothesis, it has no nonconstant globally defined holomorphic functions. Then

\begin{equation}
    \mathcal{O}_{X}(X) = \mathrm{H}^{0}(X, \mathcal{O}_X) \cong \mathbb{C}
\end{equation}
and $\dim{\mathrm{H}^{0}(X, \mathcal{O}(D))} = \dim{\mathrm{H}^{0}(X, \mathcal{O}_X)} = 1$. Moreover, $\mathrm{H}^{1}(X, \mathcal{O}(D)) = \mathrm{H}^{1}(X, \mathcal{O}_X)$ and $\dim{\mathrm{H}^{1}(X, \mathcal{O}_X)} = g$ by definition, so

\begin{equation}
    \dim{\mathrm{H}^{0}(X, \mathcal{O}(D))} - l(D) = 1 - g + deg(D)
\end{equation}
becomes

\begin{equation}
    1 - g = 1 - g + 0
\end{equation}
and the formula indeed holds for $D = 0$.

Now, by an abuse of notation, denote by $P$ the divisor $\sum n_x [x]$, such that $n_P = 1$, and $n_x = 0$ for every $x \neq P$. Then obviously, if we denote the coefficient of $P$ in $D' = D + P$ by $P'$, and the coefficient of $P$ in $D$ by $n_P$, we have $n_P \le P'$ and we have a natural inclusion $\mathcal{O}(D) \rightarrow \mathcal{O}(D+P)$. Now let $(V, z)$ be a coordinate neighborhood on $X$ about $P$ with $z(P) = 0$, and let $i_{P,*}\mathbb{C}$ denote the \textit{skyscrapper sheaf} (of fields) with value $\mathbb{C}$ centered at $P$. Define a sheaf morphism

\begin{equation}
    \beta : \mathcal{O}(D') \rightarrow i_{P,*}\mathbb{C}
\end{equation}
as follows: let $U \subseteq X$ be an open subset. If $P \notin U$, then $\beta (U): \Gamma(U, \mathcal{O}(D')) \rightarrow \Gamma(U, i_{P,*}\mathbb{C})$ is the zero morphism. If $P \in U$ and $f \in \Gamma(U, \mathcal{O}(D'))$, then $f$ admits a Laurent series expansion about $P$, with respect to $z$, say

\begin{equation}
    f = \sum_{n=-k-1}^{+\infty}c_n z^n
\end{equation}
where $k = n_P$ as above. Define $\beta (U)(f)$ to be

\begin{equation}
    \beta (U)(f) = c_{-k-1} \in \mathbb{C} = \Gamma(U, i_{P,*}\mathbb{C}).
\end{equation}

We have that $\beta$ is a sheaf epimorphism, and so we get a short exact sequence

\begin{equation}
\begin{tikzcd}
    0 \arrow{r} & \mathcal{O}(D) \arrow{r} & \mathcal{O}(D') \arrow{r}{\beta} & i_{P,*}\mathbb{C} \arrow{r} & 0
\end{tikzcd}
\end{equation}
from which we get a \textit{long exact sequence in cohomology}:

\begin{equation}
\begin{tikzcd}
    0 \arrow{r} & \mathrm{H}^{0}(X, \mathcal{O}(D)) \arrow{r} & \mathrm{H}^{0}(X, \mathcal{O}(D')) \arrow{r} & \mathrm{H}^{0}(X, i_{P,*}\mathbb{C}) \\ \arrow{r} & \mathrm{H}^{1}(X, \mathcal{O}(D)) \arrow{r} & \mathrm{H}^{1}(X, \mathcal{O}(D')) \arrow{r} & \mathrm{H}^{1}(X, i_{P,*}\mathbb{C})
\end{tikzcd}
\end{equation}
that becomes

\begin{equation}\label{lesrr}
\begin{tikzcd}
    0 \arrow{r} & \mathrm{H}^{0}(X, \mathcal{O}(D)) \arrow{r} & \mathrm{H}^{0}(X, \mathcal{O}(D')) \arrow{r} & \mathbb{C} \\ \arrow{r} & \mathrm{H}^{1}(X, \mathcal{O}(D)) \arrow{r} & \mathrm{H}^{1}(X, \mathcal{O}(D')) \arrow{r} & \mathrm{H}^{1}(X, i_{P,*}\mathbb{C}).
\end{tikzcd}
\end{equation}

Suppose that $[\xi] \in \mathrm{H}^{1}(X, i_{P,*}\mathbb{C})$ is a cohomology class. Let $\mathcal{U}$ be a covering of $X$, and we can suppose that $\mathcal{U}$ is fine enough so as to guarantee that $P$ is only in one element of $\mathcal{U}$. Let $\eta \in Z^1(\mathcal{U}, i_{P,*}\mathbb{C})$ be a $1$-cocycle representing $[\xi]$. But then by construction, $P$ is not in any intersection and so $Z^1(\mathcal{U}, i_{P,*}\mathbb{C}) = 0$, and $[\xi] = 0$. So $\mathrm{H}^{1}(X, i_{P,*}\mathbb{C}) = 0$ and we have

\begin{equation}
\begin{tikzcd}
    0 \arrow{r} & \mathrm{H}^{0}(X, \mathcal{O}(D)) \arrow{r} & \mathrm{H}^{0}(X, \mathcal{O}(D')) \arrow{r} & \mathbb{C} \\ \arrow{r} & \mathrm{H}^{1}(X, \mathcal{O}(D)) \arrow{r} & \mathrm{H}^{1}(X, \mathcal{O}(D')) \arrow{r} & 0.
\end{tikzcd}
\end{equation}

In this way, we can split the sequence (\ref{lesrr}) into two short exact sequences: denote by $V$ the image of $\mathrm{H}^{0}(X, \mathcal{O}(D'))$ in $\mathbb{C}$, that is $V = \Ima (\mathrm{H}^{0}(X, \mathcal{O}(D')) \rightarrow \mathbb{C})$, and by $W$ the quotient $\mathbb{C}/V$. Then (\ref{lesrr}) becomes

\begin{equation}
\begin{tikzcd}
    0 \arrow{r} & \mathrm{H}^{0}(X, \mathcal{O}(D)) \arrow{r} & \mathrm{H}^{0}(X, \mathcal{O}(D')) \arrow{r} & V \arrow{r} & 0
\end{tikzcd}
\end{equation}
and

\begin{equation}
\begin{tikzcd}
    0 \arrow{r} & W \arrow{r} & \mathrm{H}^{1}(X, \mathcal{O}(D)) \arrow{r} & \mathrm{H}^{1}(X, \mathcal{O}(D')) \arrow{r} & 0.
\end{tikzcd}
\end{equation}

So we have that $\dim V + \dim W = 1 = deg(D') - deg(D)$, and that the sequences above are exact by construction.

Now, suppose that the formula (\ref{rrformula}) holds for one of the divisors $D, D'$. Then, the short exact sequences imply that all of the complex vector spaces are finite dimensional, and we have the following relations

\begin{equation}
    \dim V = \dim \mathrm{H}^0(X, \mathcal{O}(D')) - \dim \mathrm{H}^0(X, \mathcal{O}(D))
\end{equation}
and

\begin{equation}
    \dim \mathrm{H}^1(X, \mathcal{O}(D')) = \dim \mathrm{H}^1(X, \mathcal{O}(D)) - \dim W.
\end{equation}

Adding these equations, we get

\begin{align}
    \dim \mathrm{H}^0(X, \mathcal{O}(D')) - &\dim \mathrm{H}^1(X, \mathcal{O}(D)) = \\ = &\dim \mathrm{H}^0(X, \mathcal{O}(D)) - \dim \mathrm{H}^1(X, \mathcal{O}(D')) + \dim V + \dim W
\end{align}
and by $\dim V + \dim W = 1 = deg(D') - deg(D)$

\begin{align}
    \dim \mathrm{H}^0(X, \mathcal{O}(D')) - &\dim \mathrm{H}^1(X, \mathcal{O}(D')) - deg(D') = \\ = &\dim \mathrm{H}^0(X, \mathcal{O}(D)) - \dim \mathrm{H}^1(X, \mathcal{O}(D)) - deg(D)
\end{align}
And so, supposing that (\ref{rrformula}) holds for either $D$ or $D'$, we get

\begin{align}
    \dim \mathrm{H}^0(X, \mathcal{O}(D')) - &\dim \mathrm{H}^1(X, \mathcal{O}(D')) - deg(D') = \\ = &\dim \mathrm{H}^0(X, \mathcal{O}(D)) - \dim \mathrm{H}^1(X, \mathcal{O}(D)) - deg(D) \\
    =& 1 - g
\end{align}
and the formula holds for the other one. So since we proved that the formula holds for $D = 0$, we get that the formula holds for every effective divisor $D'$.

Now note that any arbitrary divisor $D$ can be written as a finite sum

\begin{equation}
    D = P_1 +...+ P_m - P_{m+1} -...-Pn.
\end{equation}

We prove the general case by induction. The base case $D = 0$ is already proved. Suppose that (\ref{rrformula}) holds for a divisor $D = P_1 +...+ P_m - P_{m+1} -...-Pn$. Then, (\ref{rrformula}) holds for $D' = D + P_{n+1}$, since our induction hypothesis guarantees that it holds for $D$, and by the above argument it must also hold for $D'$. Also, it holds for $D' = D - P_{n+1}$, since since our induction hypothesis guarantees that it holds for $D$, and by the above argument it must also hold for $D'$, since $D' + P_{n+1} = D$. So by induction, the formula holds for any $D$, and this completes the proof.
\end{proof}

Using the Riemann-Roch Theorem, one can prove that

\begin{theorem}\label{rsac}
Every compact Riemann surface is a complex algebraic curve.
\end{theorem}

This rather remarkable fact has some very disctinct proofs. It can be proved as a somewhat hard theorem in analysis, using some functional analytic mahinery to prove the existance of meromorphic functions on compact Riemann surfaces. It can also be proved using the before mentioned Kodaira's Vanishing Theorem. A proof through this approach can be found in \cite{griffithsharris}. Finally, it can also be proved as a direct application of the Riemann-Roch formula. A discussion about the Riemann-Roch Theorem and its relation with the algebraicity of compact Riemann surfaces can be found in \cite{miranda}.

\end{appendices}

\pagebreak

\bibliography{bibliography}
\bibliographystyle{alpha}

\end{document}